\documentclass[12pt,reqno,english]{amsart}
\usepackage{fullpage}
\usepackage{amsfonts}
\usepackage{amssymb}
\usepackage[T1]{fontenc}
\usepackage[utf8]{inputenc}
\usepackage{babel}
\usepackage{csquotes}
\usepackage{times}
\usepackage{amsmath}
\usepackage{url}
\usepackage{blindtext}
\usepackage{tikz-cd}
\usepackage{mathtools}
\usepackage{graphicx}
\usepackage{amsmath,calligra,mathrsfs}
\newtheorem{thm}{Theorem}[section]
\newtheorem{cor}[thm]{Corollary}
\newtheorem{lem}[thm]{Lemma}
\newtheorem{prop}[thm]{Proposition}
\theoremstyle{definition}\newtheorem{ex}{Example}
\newtheorem{Def}[thm]{Definition}
\newtheorem{fact}[thm]{Fact}
\newtheorem{conj}[thm]{Conjecture}
\newtheorem{const}[thm]{Construction}
 \newtheorem{remark}[thm]{Remark}
 \newtheorem{conv}[thm]{Convention}
\newcommand{\shA}{\textbf{A}}

\newcommand{\gerbeE}{\mathcal{E}}
\newcommand{\Sn}{S^{\bigotimes_{R}n}}
\newcommand{\Snp}{S^{\bigotimes_{R}(n+1)}}
\newcommand{\GE}{G_{\mathcal{E}}}
\newcommand\restr[2]{{
  \left.\kern-\nulldelimiterspace 
  #1 
    \vphantom{\big|} 
  \right|_{#2} 
  }}

\newcommand{\Hom}{\text{Hom}}
\newcommand*{\sheafhom}{\mathrm{H}\kern -.5pt om}
\newcommand{\Spec}{\text{Spec}}
\DeclareMathOperator{\sHom}{\mathscr{H}\text{\kern -3pt {\calligra\large om}}\,}

\newcommand{\Z}{\mathbb{Z}}

\begin{document}

\title{Rigid inner forms over local function fields}

\author{Peter Dillery}
\thanks{This research was conducted while author was affiliated with: University of Michigan, Department of Mathematics, 530 Church St, Ann Arbor, MI 48109, USA}
\thanks{This research was partially supported by NSF grant DMS-1840234}
\address{Department of Mathematics,
  University of Maryland, 4176 Campus Drive,
  College Park, MD 20742-4015, USA}
\email{dillery@umd.edu}
\maketitle

\begin{abstract}
We generalize the concept of rigid inner forms, defined by Kaletha in \cite{Tasho}, to the setting of a local function field $F$ in order state the local Langlands conjectures for arbitrary connected reductive groups over $F$. To do this, we define for a connected reductive group $G$ over $F$ a new cohomology set $H^{1}(\gerbeE, Z \to G) \subset H_{\text{fpqc}}^{1}(\gerbeE, G)$ for a gerbe $\gerbeE$ attached to a class in $H_{\text{fppf}}^{2}(F, u)$ for a certain canonically-defined profinite commutative group scheme $u$, building up to an analogue of the classical Tate-Nakayama duality theorem. We define a relative transfer factor for an endoscopic datum serving a connected reductive group $G$ over $F$, and use rigid inner forms to extend this to an absolute transfer factor, enabling the statement of endoscopic conjectures relating stable virtual characters and $\dot{s}$-stable virtual characters for a semisimple $\dot{s}$ associated to a tempered Langlands parameter. 
\end{abstract}

\section{Introduction}

\subsection{Motivation}
The purpose of this paper is to generalize the theory of \textit{rigid inner forms}, introduced in \cite{Tasho} for local fields of characteristic zero, to local function fields. Rigid inner forms allow one to study the representation theory of a connected reductive group $G$ over a local field $F$ by working simultaneously with all inner forms of $G$---in particular, they allow for an unambiguous statement of the endoscopic local Langlands conjectures for arbitrary connected reductive groups over $F$. 

The idea of studying all the inner forms of $G$ simultaneously for endoscopy was first suggested by Adams-Barbasch-Vogan in \cite{ABV}; generally speaking, given a tempered Langlands parameter $\varphi \colon W_{F}' \to \prescript{L}{}G$, we should have a subset of representations of inner forms of $G$, denoted by $\Pi_{\varphi}$, and a bijective map to some set of representations related to $S_{\varphi}$, the centralizer of $\varphi$ in $\widehat{G}$. A fundamental question encountered when treating all inner forms at the same time is when two inner forms should be declared ``the same". Since we are concerned with representation theory, a natural requirement of isomorphisms of inner forms is that an automorphism of an inner form $G'$ of $G$ should preserve the conjugacy classes of $G'(F)$ as well as the representations of $G'(F)$.

In order to ensure that automorphisms of inner twists satisfy the above requirements, Vogan in \cite{Vogan} expanded the data of an inner twist to that of a \textit{pure inner twist}, which gives the desired rigidity. A pure inner twist is a triple $(G', \psi, x)$, where $\psi \colon G \to G'$ is an inner form of $G$, and $x \in Z^{1}(F, G)$ is a 1-cocycle such that $\text{Ad}(x(\sigma)) = \psi^{-1} \circ \prescript{\sigma}{}\psi$ for all $\sigma$ in $\Gamma$. However, not every inner twist can be enriched to a pure inner twist, since in general $H^{1}(F, G) \to H^{1}(F, G_{\text{ad}})$ need not be surjective. The question then becomes: How does one rigidify the notion of inner twists in a way that includes all of them?

The concept of rigid inner forms introduced by Kaletha in \cite{Tasho} answers this question when $F$ is of characteristic zero. Again we take tuples $(G', \psi, z)$, where now $z$ is a 1-cocycle in a new cohomology set, denoted by $H^{1}(u \to W, Z \to G)$, where $Z$ is some finite central $F$-subgroup of $G$. The cohomology set $H^{1}(u \to W, Z \to G)$ carries a canonical surjective map to $H^{1}(F, G/Z)$, which means that such tuples encompass all inner forms of $G$. Moreover, rigid inner forms are rigid enough so that their automorphisms preserve both desired representation-theoretic properties discussed above. We also have an embedding $H^{1}(F, G) \hookrightarrow H^{1}(u \to W, Z \to G)$, connecting rigid inner twists to Vogan's pure inner twists.

Assume that $F$ is a finite extension of $\mathbb{Q}_{p}$ for some $p$, so that  the theory of \cite{Tasho} applies. The following is a short account of the conjectures enabled by rigid inner forms:

We first record the conjectures coming from Vogan's pure inner twists. Fix $\varphi \colon W_{F}' \to \prescript{L}{}G$ a tempered Langlands parameter with centralizer $S_{\varphi} \subset \widehat{G}$, as well as $G^{*}$, a quasi-split pure inner form of $G$. After fixing a Whittaker datum $\mathfrak{w}$ for $G^{*}$, we have a conjectural map $\iota_{\mathfrak{w}}$ and subset $\Pi_{\varphi}^{\text{pure}}$ of the irreducible tempered representations of the pure inner forms of $G^{*}$ making the following diagram commute:
\[
\begin{tikzcd}
\Pi_{\varphi}^{\text{pure}} \arrow["\iota_{\mathfrak{w}}"]{r} \arrow{d} & \text{Irr}(\pi_{0}(S_{\varphi})) \arrow{d} \\
H^{1}(F,G^{*}) \arrow{r} & \pi_{0}(Z(\widehat{G})^{\Gamma})^{*},
\end{tikzcd}
\]
where the left arrow sends a pure inner form representation $(G', \psi, x, \pi)$ to the class $[x]$, the lower arrow is the Kottwitz pairing (see \cite{Kott86}), and the right-hand arrow sends an irreducible representation to its central character. Moreover, the map $\iota_{\mathfrak{w}}$ provides the correct virtual characters which are needed for the endoscopic character identities for a choice of semisimple element $s \in S_{\varphi}(\mathbb{C})$. However, there need not be a quasi-split pure inner form of our general connected reductive $G$.

Now we will see the conjectures allowed by replacing the notion of pure inner forms with rigid inner forms. In addition to the Langlands parameter $\varphi$ with centralizer $S_{\varphi}$, let $Z$ be a fixed finite central $F$-subgroup of $G$. The isogeny $G \to G/Z := \overline{G}$ dualizes to an isogeny $\widehat{\overline{G}} \to \widehat{G}$; let $S_{\varphi}^{+}$ denote the preimage of $S_{\varphi}$ under this isogeny. Then, after fixing a Whittaker datum $\mathfrak{w}$ for $G^{*}$, a quasi-split rigid inner form of $G$ (which always exists), we conjecture the existence of a subset $\Pi_{\varphi}$ of $\Pi_{\varphi}^{\text{temp}}$, the tempered representations of the rigid inner forms of $G^{*}$, and a bijective map $\iota_{\mathfrak{w}}$ making the following diagram commute
\[
\begin{tikzcd}
\Pi_{\varphi} \arrow["\iota_{\mathfrak{w}}"]{r} \arrow{d} & \text{Irr}(\pi_{0}(S_{\varphi}^{+})) \arrow{d} \\
H^{1}(u \to W, Z \to G^{*}) \arrow{r} & \pi_{0}(Z(\widehat{\overline{G}})^{+})^{*}
\end{tikzcd}
\]
where the left map sends a representation of a rigid inner twist to the corresponding class in $H^{1}(u \to W, Z \to G^{*})$, the right map sends a representation to its central character, and the bottom map is an extension of the duality isomorphism $H^{1}(F, G) \xrightarrow{\sim}  \pi_{0}(Z(\widehat{G})^{\Gamma})^{*}$ defined by Kottwitz in \cite{Kott86}; here $Z(\widehat{\overline{G}})^{+}$ denotes the preimage of $Z(\widehat{G})^{\Gamma}$ in $Z(\widehat{\overline{G}})$.

We now turn to endoscopy. Choosing a semisimple $s \in S_{\varphi}(\mathbb{C})$, along with the data of $\varphi$, gives rise to an endoscopic datum $\mathfrak{e} = (H, \mathcal{H}, \eta, s)$ for $G$; for simplicity we will assume that $\mathcal{H} = \prescript{L}{}H$.  Rigid inner forms allow us to define, given a fixed quasi-split rigid inner twist $(G^{*}, \psi, z)$ of $G$, a ($\mathfrak{w}$-normalized) absolute transfer factor $\Delta'[\dot{\mathfrak{e}}, \psi, z, \mathfrak{w}]$ for pairs of related strongly regular semisimple elements of $H(F)$ and $G(F)$---this was only previously possible for quasi-split $G$. The fact that we have replaced $\mathfrak{e}$ by $\dot{\mathfrak{e}}$ corresponds to the necessity of replacing $s$ by a preimage $\dot{s}$ in $S_{\varphi}^{+}(\mathbb{C})$, on which this factor depends. This absolute transfer factor allows for the formulation of endoscopic virtual character identities for the images $\iota_{\mathfrak{w}}(\dot{\pi})$ of representations $\dot{\pi} \in \Pi_{\varphi}$ of rigid inner twists of $G$ in the set $\text{Irr}(\pi_{0}(S_{\varphi}^{+}))$. 

If we want to generalize these conjectures to connected reductive groups over a local function field $F$, a natural question that arises is whether or not an analogue of the theory of rigid inner forms can be developed in this new situation. There are nontrivial obstacles to a direct translation of the theory established in \cite{Tasho}. Notably, the cohomology set $H^{1}(u \to W, Z \to G)$ is defined using the cohomology of a group extension $$0 \to u \to W \to \Gamma \to 0$$ corresponding to a canonical class in $H^{2}(F,u)$ for a special profinite commutative affine group $u$ (where $\Gamma$ denotes the absolute Galois group of $F$). The group $u$ will not be smooth in positive characteristic, and so it is no longer true that $H^{2}(F, u) = H^{2}(\Gamma, u(F^{s}))$ (where $F^{s}$ is a separable closure of $F$), and therefore there is no way of choosing a corresponding group extension in this situation.

We remedy this deficiency by working instead with the fppf cohomology group $H^{2}_{\text{fppf}}(F, u)$, which may be computed using the \v{C}ech cohomology related to the fpqc cover $\Spec(\bar{F}) \to \Spec(F)$. Classes in the group $\check{H}^{2}_{\text{fppf}}(F, u)$ correspond to isomorphism classes of $u$-gerbes over $\Spec(F)$, which means that for a canonical class in $H^{2}_{\text{fppf}}(F, u)$ we get a corresponding $u$-gerbe $\gerbeE$, whose role will replace that of $W$ in \cite{Tasho}. With the gerbe $\gerbeE$ in hand, we investigate its cohomology in a way that parallels the cohomology of the group $W$ in \cite{Tasho}, culminating in the construction of a cohomology set $H^{1}(\gerbeE, Z \to G)$ that is the analogue of $H^{1}(u \to W, Z \to G)$ discussed above. In particular, we will have a Tate-Nakayama type isomorphism for $H^{1}(\gerbeE, Z \to G)$ that will be used to construct a canonical pairing $$H^{1}(\gerbeE, Z \to G) \times \pi_{0}(Z(\widehat{\overline{G}})^{+}) \to \mathbb{C}^{*}$$ extending the positive-characteristic analogue (see \cite{Thang2}) of the Kottwitz pairing in characteristic zero alluded to above. 

Note that if $F$ is a finite extension of $\mathbb{Q}_{p}$, then $u$ is smooth, and in this case our gerbe $\gerbeE$ may be replaced by a group extension of $\Gamma$ by $u(\bar{F})$ using the comparison isomorphism $H^{2}_{\text{fppf}}(F, u) \xrightarrow{\sim} H^{2}_{\text{\'{e}tale}}(F, u) = H^{2}(\Gamma, u(\bar{F}))$. This then recovers the group $W$ used in \cite{Tasho}, cf. the discussion of \textit{Galois gerbes} in \cite{LR}.

The definition of the cohomology set $H^{1}(\gerbeE, Z \to G)$ allows for a completely analogous definition of rigid inner forms, which, when combined with a construction of the relative local transfer factor for local functions fields, allows for the definition of an absolute transfer factor for an endoscopic datum $\mathfrak{e}$ associated to an arbitrary connected reductive group over $F$. The development of this theory culminates in a statement of the above conjectures in the setting of local function fields. This work has been extended to global function fields in  \cite{Dillery} in order to study the multiplicity of discrete automorphic representations, as well as the relationship between the adelic transfer factor and the aforementioned local transfer factors. Moreover, the transfer factors constructed in this paper (and studied further in \cite{Dillery}) will be useful in the study/stabilization of the trace formula, which was recently established for global function fields in \cite{LL}.

\subsection{Overview} We now summarize the structure of this paper. The goal of \S 2 is to obtain a concrete interpretation of torsors on gerbes. It begins by recalling the basic theory of fibered categories, stacks, and gerbes,  progresses to a characterization of torsors on gerbes, and concludes by investigating the analogue of the inflation-restriction sequence in group cohomology in the setting of gerbes. Following this, the next two sections focus solely on tori: in \S 3, we construct the pro-algebraic group $u$, investigate its cohomology, and then define the cohomology set $H^{1}(\gerbeE, Z \to S)$ for an $F$-torus $S$, where $\gerbeE$ is a $u$-gerbe associated to a canonical cohomology class in $H^{2}(F, u)$. We also discuss basic functoriality properties of the cohomology group $H^{1}(\gerbeE, Z \to S)$ using our insight from \S 2. An analogue of the classical Tate-Nakayama isomorphism is constructed for $H^{1}(\gerbeE, Z \to S)$ for $S$ an $F$-torus in \S 4, using an fppf-analogue of the ``unbalanced cup product" (see \cite[\S 4.3]{Tasho}); this is the technical heart of the paper. 

Once the situation for tori is established, the purpose of \S 5 is to define $H^{1}(\gerbeE, Z \to G)$ for a general connected reductive group $G$, and extend all of the previous results to this new situation. There is not much to do here: the bulk of the work is just direct translation of the results in \cite[\S 3, \S 4]{Tasho} to fppf cohomology, using basic theorems about the structure theory of connected reductive groups over local function fields (see \cite{Debacker}, \cite{Thang1}, \cite{Thang2}). In order to apply the first five sections to the local Langlands conjectures, it is necessary to recall and define the (relative) local transfer factor corresponding to an endoscopic datum for a reductive group over a local function field---we do this in \S6. This section is entirely self-contained for expository purposes, and in many cases is just a direct exposition of the constructions stated in \cite{LS}; the only aspects of the arguments loc. cit. that require minor adjustment are those concerning the $\Delta_{I}$ and $\Delta_{III_{1}}$ factors, but we include a discussion of all of the factors for completeness.  

Finally, in \S7 we define rigid inner forms for local function fields. Then we use them to define an absolute local transfer factor for an endoscopic datum associated to an arbitrary connected reductive group over $F$. Once this is done, we give a brief summary of the conjectures stemming from our constructions. This section closely parallels \cite[\S 5]{Tasho}; in many cases, we follow the arguments verbatim, substituting Galois-cohomological calculations with analogous computations in \v{C}ech cohomology.
\subsection{Notation and terminology}
We will always assume that $F$ is a local field of characteristic $p > 0$, although all of the arguments work for $p$-adic local fields, forgetting about the prime-to-$p$ sequence $\{n_{k}'\}$ in \S 4.3. For an arbitrary algebraic group $G$ over $F$, $G^{\circ}$ denotes the identity component. For a connected reductive group $G$ over $F$, $Z(G)$ denotes the center of $G$, and for $H$ a subgroup of $G$, $N_{G}(H), Z_{G}(H)$ denote the normalizer and centralizer group schemes of $H$ in $G$, respectively. We will denote by $\mathscr{D}(G)$ the derived subgroup of $G$, by $G_{\text{ad}}$ the quotient $G/Z(G)$, and if $G$ is semisimple, we denote by $G_{\text{sc}}$ the simply-connected cover of $G$; if $G$ is not semisimple, $G_{\text{sc}}$ denotes $\mathscr{D}(G)_{\text{sc}}$. If $T$ is a maximal torus of $G$, denote by $T_{\text{sc}}$ its preimage in $G_{\text{sc}}$. We fix an algebraic closure $\bar{F}$ of $F$, which contains a separable closure of $F$, denoted by $F^{s}$. For $E/F$ a Galois extension, we denote the Galois group of $E$ over $F$ by $\Gamma_{E/F}$, and we set $\Gamma_{F^{s}/F} =: \Gamma$ (although occasionally we will also use $\Gamma$ to denote a finite Galois group---this will be made clear when relevant). 

We call an affine, commutative algebraic group \textit{multiplicative} if its characters span its coordinate ring over $F^{s}$. For $Z$ a multiplicative group over $F$, we denote by $X^{*}(Z), X_{*}(Z) (=X_{*}(Z^{\circ}))$ the character and co-character modules of $Z$, respectively, viewed as $\Gamma$-modules. Given $T$ a split maximal torus in $G$, we denote by $W(G,T)$ the quotient group $N_{G}(T)/Z_{G}(T)$, and frequently identify it as a subset of $\text{Aut}_{\Z}(X^{*}(T))$.  For a morphism $f \colon A \to B$ of multiplicative groups over $F$, we use $f^{\sharp}$ to denote both induced morphisms $X_{*}(A) \to X_{*}(B)$ and $X^{*}(B) \to X^{*}(A)$. Also, given a morphism $f \colon U \to V$ of two objects in a stack $\mathcal{C}$ and sheaf $\mathcal{F}$ on $\mathcal{C}$, we also use the symbol $f^{\sharp}$ to denote the induced morphism $\mathcal{F}(V) \to \mathcal{F}(U)$; there will be no danger of confusing these two notations. For an $F$-torus $T$, we define the \textit{dual torus} $\widehat{T}$ to be the $\mathbb{C}$-torus with character group $X_{*}(T)$; we equip $\widehat{T}$ with a $\Gamma$-action via the natural $\Gamma$-action on $X_{*}(T)$. We will frequently denote $\widehat{T}(\mathbb{C})^{\Gamma}$ by $\widehat{T}^{\Gamma}$. For two $F$-schemes $X, Y$ and $F$-algebra $R$, we set $X \times_{\Spec(F)} Y =: X \times_{F} Y$, or by $X \times Y$ if $F$ is understood, and set $X \times_{F} \text{Spec}(R) =: X_{R}$. We also set $X(\text{Spec}(R)) =: X(R)$, the set of $F$-morphisms $\{\text{Spec}(R) \to X\}$.

\subsection{Acknowledgements} The author thanks Tasho Kaletha for introducing to him the motivating question of this paper, as well as for his feedback and proof-reading. The author also thanks the anonymous reviewer for their useful comments and corrections. The author gratefully acknowledges the support of NSF grant DMS-1840234.

\tableofcontents

\section{Preliminaries on gerbes}

\subsection{Basics of fibered categories and stacks}

The purpose of this subsection is to briefly review the theory of fibered categories and stacks that will be used later in the paper. For a comprehensive treatment, see for example \cite[Chapter 3]{Olsson}. Let $\mathcal{C}$ denote a category which has finite fibered products. In the later sections, this will be the category $\text{Sch}/S$ of schemes over a fixed scheme $S$, but for now we will allow it to be arbitrary. Let $\mathscr{X} \xrightarrow{\pi} \mathcal{C}$ be a morphism of categories (i.e., a functor). 

\begin{Def} For $X, Y \in \text{Ob}(\mathscr{X})$ denote by $U, V$ (respectively) the objects $\pi(X), \pi(Y)$ in $\mathcal{C}$ (i.e., $X$ and $Y$ \textit{lie above} or \textit{lift} $U$ and $V$); we say that a morphism $f \colon Y \to X$ in $\mathscr{X}$ is \textit{strongly cartesian} if for every pair of a morphism $g \colon Z \to X$ in $\mathscr{X}$ and morphism $h \colon \pi(Z) \to V$ in $\mathcal{C}$ such that $\pi(g) = \pi(f) \circ h$, there is a unique $\tilde{h} \colon Z \to Y$ such that $f \circ \tilde{h} = g$ and $\pi(\tilde{h}) = h$. In this case, we say that $\tilde{h}$ \textit{lifts} $h$. 
\end{Def}

We continue working with a fixed $\mathscr{X} \xrightarrow{\pi} \mathcal{C}$. 

\begin{Def} For a fixed $U \in \text{Ob}(\mathcal{C})$, we define a category $\mathscr{X}(U)$ as follows; its objects will be given by the set $\{X \in \text{Ob}(\mathscr{X}) \colon \pi(X) = U\}$ and its morphisms will be those morphisms $X \xrightarrow{f} X'$ such that $\pi(f) = \text{id}_{U}$. We call this the \textit{fiber category over $U$}, or just the \textit{fiber over $U$}. 
\end{Def}

\begin{Def} We say that $\mathscr{X} \xrightarrow{\pi} \mathcal{C}$ is a \textit{fibered category} over $\mathcal{C}$ if for every $U \in \text{Ob}(\mathcal{C})$, morphism $V \xrightarrow{f} U$ in $\mathcal{C}$, and $X \in \mathscr{X}(U)$, there is an object $Y \in \mathscr{X}(V)$ and strongly cartesian morphism $\tilde{f} \colon Y \to X$ such that $\pi(\tilde{f}) = f$. One checks that if we have another strongly cartesian $Y' \xrightarrow{\tilde{f}'} X$ satisfying the above property, then there is a unique isomorphism $Y' \to Y$ making all the obvious diagrams commute. We define a \textit{morphism of fibered categories } from $\mathscr{X} \xrightarrow{\pi} \mathcal{C}$ to $\mathscr{X}' \xrightarrow{\pi'} \mathcal{C}$ to be a functor $f \colon \mathscr{X} \to \mathscr{X}'$ such that $\pi = \pi' \circ f$. We say that a fibered category $\mathscr{X} \to \mathcal{C}$ is \textit{fibered in groupoids} if for all $U \in \text{Ob}(\mathcal{C})$, $\mathscr{X}(U)$ is a groupoid (recall that a category is a \textit{groupoid} if all morphisms are isomorphisms). We will denote the group $\text{Aut}_{\mathscr{X}(U)}(X)$ simply by $\text{Aut}_{U}(X)$ for ease of notation. 
\end{Def}

\begin{lem}\label{fiberedprod} If $\mathscr{X} \to \mathcal{C}$ is fibered in groupoids, then $\mathscr{X}$ also has finite fibered products.
\end{lem}

\begin{proof} Since we assume that $\mathcal{C}$ has finite fibered products, this follows from \cite[Lemma 06N4]{Stacksproj}.
\end{proof}

In all that follows, given a fibered category $\mathscr{X} \to \mathcal{C}$, for every $U \in \text{Ob}(\mathcal{C})$, $X \in \mathscr{X}(U)$, and morphism $V \xrightarrow{f} U$ in $\mathcal{C}$, we choose some $Y \to X$ satisfying the conditions in the above definition, and will denote this by $f^{*}X \to X$. One checks that for any morphism $X \xrightarrow{\varphi} Y$ in $\mathscr{X}(U)$, a morphism $f \colon V \to U$ induces a canonical morphism $f^{*}X \to f^{*}Y$ in $\mathscr{X}(V)$, which we will denote by $f^{*}\varphi$. 

The choices of pullbacks also define, for any $X \in \mathscr{X}(U)$, a functor $\mathcal{C}/U \to \mathscr{X}/X$ (where $\mathcal{C}/U$ denotes the category of pairs $(V,g)$ where $V \in \text{Ob}(\mathcal{C})$ and $g \colon V \to U$, morphisms given in the obvious way) given by sending the morphism $V \xrightarrow{f} U$ to $f^{*}X \to X$ and the $U$-morphism $h \colon (W \xrightarrow{g} U) \to (V \xrightarrow{f} U)$ to the unique morphism $g^{*}X \to f^{*}X$ lifting $h$ which is compatible with the pullback morphisms to $X$. Such a functor is injective on objects and fully faithful.

\begin{Def} Given a fibered category $\mathscr{X} \xrightarrow{\pi} \mathcal{C}$ and $X,Y \in \mathscr{X}(U)$, we may define a presheaf (of sets), denoted by $\underline{\Hom}(X,Y)$, on the category $\mathcal{C}/U$ by setting
$$\underline{\Hom}(X,Y)(V \xrightarrow{f} U) := \Hom_{\mathscr{X}(V)}(f^{*}X, f^{*}Y),$$
and for a morphism $(W \xrightarrow{g} U) \xrightarrow{h} (V \xrightarrow{f} U)$, we define the restriction map to be 
$$\Hom_{\mathscr{X}(V)}(f^{*}X, f^{*}Y) \xrightarrow{h^{*}} \Hom_{\mathscr{X}(W)}(h^{*}(f^{*}X), h^{*}(f^{*}Y)) \cong \Hom_{\mathscr{X}(W)}(g^{*}X, g^{*}Y),$$ where the first map above sends $\varphi$ to $h^{*}\varphi$, and the second map is the canonical isomorphism induced by the canonical identifications $h^{*}(f^{*}X) \cong g^{*}X$, $h^{*}(f^{*}Y) \cong g^{*}Y$. For the remainder of this paper, it will be harmless to make such identifications, and we do so without comment. If $\mathscr{X} \to \mathcal{C}$ is fibered in groupoids and $Y=X$, we denote the above presheaf by $\underline{\text{Aut}}_{U}(X)$---this is a presheaf of groups. It will play an important role in what follows.
\end{Def}
\noindent We will now assume that we may endow $\mathcal{C}$ with the structure of a site, denoted by $\mathcal{C}_{\tau}$, so that it makes sense to talk about sheaves on $\mathcal{C}_{\tau}$.

\begin{Def} We say that a fibered category is a \textit{prestack} (over $\mathcal{C}_{\tau}$) if for all $U \in \text{Ob}(\mathcal{C})$ and $X, Y \in \mathscr{X}(U)$, the presheaf $\underline{\Hom}(X,Y)$ is a sheaf on $(\mathcal{C}/U)_{\tau}$. 
\end{Def}

\begin{Def} Fix $U \in \text{Ob}(\mathcal{C})$, a covering $\{V_{i} \xrightarrow{h_{i}} U\}_{i \in I}$ of $U$ (here $I$ denotes the indexing set), and a subset $\{X_{i} \in \mathscr{X}(V_{i})\}_{i \in I}$ of $\text{Ob}(\mathscr{X})$. The fibered product $V_{ij} := V_{i} \times_{U} V_{j}$ has two projections; we will denote the one to $V_{i}$ by $p_{1}$ and the one to $V_{j}$ by $p_{2}$. We say that this subset, together with a collection of isomorphisms $\{f_{ij} \colon p_{1}^{*}X_{i} \xrightarrow{\sim} p_{2}^{*}X_{j} \colon f_{ij} \in \Hom(\mathscr{X}(V_{ij}))\}_{i,j \in I}$ is a \textit{descent datum} (for this fixed covering of $U$) if the following diagram commutes for all $i,j,k \in I$:
\[
\begin{tikzcd} 
p_{12}^{*}p_{1}^{*}X_{i} \arrow["p_{12}^{*}f_{ij}"]{r} \arrow[equals]{d} & p_{12}^{*}p_{2}^{*}X_{j} \arrow[equals]{r} & p_{23}^{*}p_{1}^{*}X_{j} \arrow["p_{23}^{*}f_{jk}"]{d} \\
p_{13}^{*}p_{1}^{*}X_{i} \arrow["p_{13}^{*}f_{ik}"]{r} & p_{13}^{*}p_{2}^{*}X_{k} \arrow[equals]{r} & p_{23}^{*}p_{2}^{*}X_{k}, 
\end{tikzcd}
\]
where the equalities denote the canonical isomorphisms discussed above, $p_{ij}$ denotes the projection $V_{ijk}:= V_{i} \times_{U} V_{j} \times_{U} V_{k} \to V_{ij}$, and analogously for the other projections. Given another descent datum $\{Y_{i} \in \mathscr{X}(V_{i})\}_{i \in I}$, $\{g_{ij}\}_{i,j \in I}$, we say that it is \textit{isomorphic} to our above datum if there are isomorphisms $\phi_{i} \colon X_{i} \to Y_{i}$ in $\mathscr{X}(V_{i})$ which for all $i,j$ satisfy $p_{2}^{*}\phi_{j}^{-1} \circ g_{ij} \circ p_{1}^{*}\phi_{i} = f_{ij}$.

\end{Def}

Continuing the notation of the above definition, note that if $X \in \mathscr{X}(U)$, then we get a descent datum for free via setting $X_{i} := h_{i}^{*}X$ and $f_{ij} \colon p_{1}^{*}h_{i}^{*}X \to p_{2}^{*}h_{j}^{*}X$ the canonical isomorphism between these two pullbacks to $V_{ij}$ of $X$. We denote this descent datum by $X_{\text{canon}}$. 

\begin{Def} We say that a descent datum $\{X_{i}\}_{i \in I}$, $\{f_{ij}\}_{i,j \in I}$ for $U$ with respect to the cover $\{V_{i} \to U\}$ is \textit{effective} if there is an object $X \in \mathscr{X}(U)$ such that $\{X_{i}\}_{i \in I}$, $\{f_{ij}\}_{i,j \in I}$ is isomorphic to $X_{\text{canon}}$ (in this case, the pair of $X$ and a choice of isomorphism between $X_{\text{canon}}$ and $\{X_{i}\}$ is unique up to canonical isomorphism). We say that a prestack $\mathscr{X} \to \mathcal{C}_{\tau}$ is a \textit{stack} if all descent data (for all objects of $\mathcal{C}$ and their covers) are effective. We define a \textit{morphisms of stacks over $\mathcal{C}_{\tau}$} to be a morphism between their underlying fibered categories.
\end{Def}

The following proposition shows that whether or not a morphism between two stacks over $\mathcal{C}_{\tau}$ is an equivalence can be checked over a cover of $\mathcal{C}_{\tau}$. We will assume that $\mathcal{C}$ has a final object $U$ and that our cover consists of one element $U_{0} \to U$ (this will be our general situation for the rest of the paper). It is easy to check that if $\mathscr{X} \to \mathcal{C}_{\tau}$ is a stack, then restricting $\mathscr{X}$ to the category whose objects are pairs $(X, V \xrightarrow{f} U_{0})$, where $X \in \mathscr{X}(V)$, and whose morphisms are the morphisms in $\mathscr{X}$ that lift $U_{0}$-morphisms, defines a stack over $(\mathcal{C}/U_{0})_{\tau}$,  denoted by $\mathscr{X}_{U_{0}}$, with structure map given in the canonical manner. This may also be viewed as the fibered product of categories $\mathscr{X} \times_{\mathcal{C}} (\mathcal{C}/U_{0})$, for the definition of this, see e.g. \cite[\S 003O]{Stacksproj}. We set $U_{1}:= U_{0} \times_{U} U_{0}$.

\begin{prop}\label{localequivalence} Let $U_{0} \to U$ be a cover of $\mathcal{C}_{\tau} = (\mathcal{C}/U)_{\tau}$, and $\phi \colon \mathscr{X} \to \mathscr{X}'$ be a morphism of stacks over $\mathcal{C}_{\tau}$; we have an induced morphism of stacks over $(\mathcal{C}/U_{0})_{\tau}$, denoted by $\phi_{U_{0}} \colon \mathscr{X}_{U_{0}} \to \mathscr{X}'_{U_{0}}$. Then $\phi$ is an equivalence of categories if and only if $\phi_{U_{0}}$ is.
\end{prop}

\begin{proof} One direction is trivial. For the other, if $X'$ is an object of $\mathscr{X}'$, then we may find an object $\tilde{X}$ of $\mathscr{X}$ and $f$ a morphism in $\mathscr{X}'(U_{0})$ such that $\phi(\tilde{X}) \xrightarrow{f,\sim} X'_{U_{0}}$ (where we are denoting the pullback of $X'$ to $U_{0}$ by $X'_{U_{0}}$). We may also find objects $\tilde{X}_{1}$, $\tilde{X}_{2}$ in $\mathscr{X}(U_{1})$ and morphisms $f_{i}$ in $\mathscr{X}'(U_{1})$ with $\phi(\tilde{X}_{i}) \xrightarrow{f_{i}, \sim} p_{i}^{*}(X'_{U_{0}})$ for $i=1,2$, which, since $\phi_{U_{0}}$ is an equivalence, are such that we have isomorphisms $\tilde{X}_{i} \xrightarrow{\tilde{f}_{i}, \sim} p_{i}^{*}\tilde{X}$ with $p_{i}^{*}f \circ \phi_{U_{0}}(\tilde{f}_{i}) =f_{i}$ as well as an isomorphism $h \colon \tilde{X}_{1} \to \tilde{X}_{2}$ such that $f_{2} \circ \phi(h)\circ f_{1}^{-1}$ is the canonical identification $p_{1}^{*}X'_{U_{0}} \cong p_{2}^{*}X'_{U_{0}}$. It is straightforward to check that $\mathscr{D} := \{\tilde{X}\}$, $\{\tilde{f}_{2} \circ h \circ \tilde{f}_{1}^{-1}, \tilde{f}_{1} \circ h^{-1} \circ \tilde{f}_{2}^{-1}\}$ is a descent datum on $\mathscr{X}$, and hence (since $\mathscr{X}$ is a stack)  there is some $X \in \mathscr{X}(U)$ with  $X_{\text{canon}}$ isomorphic to $\mathscr{D}$ as descent data. Then since $\mathscr{X}'$ is a prestack, the local isomorphism $\phi(X)_{U_{0}} \xrightarrow{\sim} X'_{U_{0}}$ induced by $f$ and the isomorphism of descent data glues to an isomorphism $\phi(X) \xrightarrow{\sim} X'$, as desired. The analogous argument for morphisms is similar, and left as an exercise.
\end{proof}

\subsection{Cohomological basics} We briefly recall some results on the \v{C}ech cohomology of sites. We now work in a less general setting, assuming that $\mathcal{C} = \text{Sch}/\Spec(F)$ (for brevity of notation, we will denote this category by $\text{Sch}/F$), and $\tau \in \{\text{fppf},\text{fpqc} \}$ for $F$ a field. As a matter of convention, all group schemes in this paper will be affine. Fix a cover $U_{0} \to \Spec(F)$ in $(\text{Sch}/F)_{\tau}$, which for our purposes will frequently be $\Spec(\bar{F})$ (when dealing with the fpqc site). Following the notation in \cite{Lieblich}, we denote $U_{0} \times_{F} U_{0}$ by $U_{1}$, with the two projection maps $p_{1}, p_{2}: U_{1} \to U_{0}$, and $U_{0} \times_{F} U_{0} \times_{F} U_{0}$ by $U_{2}$, with the three projection maps $p_{12}, p_{13}, p_{23}: U_{2} \to U_{1}$ and three projection maps $q_{1}, q_{2}, q_{3}: U_{2} \to U_{0}$. Define $U_{n}$ to be the $(n+1)$-fold fibered product of $U_{0}$ over $F$. We continue with the notation from \S 2.1. 

The category of abelian sheaves on $(\text{Sch}/F)_{\text{fppf}}$ is an abelian category with enough injectives, and for an abelian $F$-group scheme $\shA$ we may thus define the cohomology groups $H^{i}((\text{Sch}/F)_{\text{fppf}}, \shA)$ for $i \geq 0$ by taking the derived functors of the global section functor on this abelian category, viewing $\shA$ as a sheaf on $(\text{Sch}/F)_{\text{fppf}}$. We will denote $H^{i}((\text{Sch}/F)_{\text{fppf}}, \shA)$ by $H^{i}(F,\shA)$, or sometimes by $H^{i}_{\text{fppf}}(F, \shA)$ when we want to emphasize our use of the fppf topology. We could also define these cohomology groups for the fpqc topology, but there are set-theoretic issues that make this complicated, so we avoid doing so. Here is a useful vanishing result:

\begin{prop}\label{vanishingcohomology} For $G$ a finite type commutative $F$ group scheme and $U_{0} = \Spec(\bar{F})$, we have $H^{i}(U_{j}, G_{U_{j}}) = 0$ for all $i > 0$ and all $j \geq 0$. 
\end{prop}

\begin{proof} This is \cite[Lemma 2.9.4]{Rosengarten}.
\end{proof}

We also have the following alternative notion of cohomology:

\begin{Def} For $\shA$ an abelian $F$-group scheme, the \textit{\v{C}ech cohomology of $\shA$ with respect to the (fpqc or fppf) cover $U_{0} \to \Spec(F)$} is the cohomology of the following complex:
\begin{center}
$\shA(U_{0}) \xrightarrow{d} \shA(U_{1}) \xrightarrow{d} \shA(U_{2}) \xrightarrow{d} \dots$
\end{center}
where the map $d \colon \shA(U_{i-1}) \to \shA(U_{i})$ for $i \geq 1$ sends $a$ to $$\prod_{1 \leq j \leq i+1} (-1)^{j+1}p_{1, \dots, \widehat{j}, \dots, i+1}^{\sharp}(a),$$
where $p_{1, \dots, \widehat{j}, \dots, i+1} \colon U_{i} \to U_{i-1}$ is the projection map given by forgetting the $j$th factor. These groups will be denoted by $\check{H}^{i}(U_{0} \to \Spec(F), \shA)$ for $i \geq 0$, or simply by $\check{H}^{i}(R/F, \shA)$ if $U_{0} = \Spec(R)$ for an $F$-algebra $R$. For another cover $V_{0} \to U_{0} \to \Spec(F)$, we have a morphism of complexes which induces homomorphisms $\check{H}^{i}(U_{0} \to \Spec(F), \shA) \to \check{H}^{i}(V_{0} \to \Spec(F), \shA)$ for all $i$. We then define the \textit{\v{C}ech cohomology of $\shA$ on $(\text{Sch}/F)_{\tau}$} to be the colimit $$\varinjlim_{V_{0} \to \Spec(F)} \check{H}^{i}(V_{0} \to \Spec(F), \shA)$$ over all covers $V_{0} \to \Spec(F)$ in $(\text{Sch}/F)_{\tau}$, and denote this group by $\check{H}_{\tau}^{i}(F, \shA)$. 
\end{Def}

The following result gives the relationship between \v{C}ech and sheaf cohomology for the site $(\text{Sch}/F)_{\text{fppf}}$ and $G$ a commutative finite type $F$-group scheme: 

\begin{prop}\label{newcomparison3} For $G$ a commutative group scheme of finite type over $F$, we have a canonical isomorphism $\check{H}^{i}_{\text{fppf}}(F,G) \xrightarrow{\sim} H_{\text{fppf}}^{i}(F,G)$ for all $i \geq 0$.
\end{prop}
\begin{proof} This is \cite[Proposition 2.9.5]{Rosengarten}.
\end{proof}

The following result allows us to compare \v{C}ech cohomology of fppf covers and derived fppf cohomology:

\begin{lem} For an fppf cover $\text{Spec}(R) \to \text{Spec}(F)$ and abelian group sheaf $\textbf{A}$ on $(\text{Sch}/F)_{\text{fppf}}$, there is a spectral sequence $\{E_{r}, d_{r}\}_{r \geq 0}$ with $E_{2}^{p,q} = \check{H}^{p}(R/F, \underline{H}^{q}(\textbf{A}))$ which converges to $H^{p+q}(F, \textbf{A})$, where $\underline{H}^{q}(\textbf{A})$ is the presheaf given by $U \mapsto H^{q}(U, \restr{\textbf{A}}{U})$, 
\end{lem}

\begin{proof} This is \cite[Lemma 03AZ]{Stacksproj}.
\end{proof}

The following generalization of this result will be of crucial importance to us: Since the direct limit functor is exact on the category of direct systems of abelian groups, if $\{\text{Spec}(R_{i}) \to \text{Spec}(F)\}_{i \geq 0}$ is an inverse system of fppf covers of $\text{Spec}(F)$, then we have a spectral sequence $\{E_{r}, d_{r}\}_{r \geq 0}$ with $E_{2}^{p,q} = \varinjlim_{i} \check{H}^{p}(R_{i}/F, \underline{H}^{q}(\textbf{A}))$ converging to $\varinjlim_{i} H^{p+q}(F, \textbf{A}) = H^{p+q}(F, \textbf{A})$. Note that, by definition, $\varinjlim_{i} \check{H}^{p}(R_{i}/F, \underline{H}^{q}(\textbf{A})) =  \check{H}^{p}(R/F, \underline{H}^{q}(\textbf{A}))$, where $R:= \varinjlim_{i} R_{i}$. We thus obtain:

\begin{cor}\label{spectralsequence1} For an inverse system of fppf covers $\{\text{Spec}(R_{i}) \to \text{Spec}(F)\}_{i \geq 0}$ and abelian group sheaf $\textbf{A}$ on $(\text{Sch}/F)_{\text{fppf}}$, there is a spectral sequence $\{E_{r}, d_{r}\}_{r \geq 0}$ with $E_{2}^{p,q} = \check{H}^{p}(R/F, \underline{H}^{q}(\textbf{A}))$ which converges to $H^{p+q}(F, \textbf{A})$, where $R := \varinjlim_{i} R_{i}$. 
\end{cor}

We may now compare the \v{C}ech cohomology of systems of fppf covers and derived fppf cohomology (cf. \cite[Lemma 03F7]{Stacksproj}):

\begin{cor}\label{cechtoderived} For an inverse system of fppf covers $\{\text{Spec}(R_{i}) \to \text{Spec}(F)\}_{i \geq 0}$ and fppf abelian group sheaf $\textbf{A}$, if $H^{i}(\tilde{R}_{j}, \textbf{A}_{\tilde{R}_{j}}) = 0$ for all $i \geq 1$, $j \geq 0$, where $\tilde{R}_{j} := R^{\bigotimes_{F} (j+1)}$, then the above spectral sequence gives a canonical isomorphism $\check{H}^{p}(R/F, \textbf{A}) \to H^{p}(F, \textbf{A})$ for all $p \in \mathbb{N}$.
\end{cor}

\begin{proof} The above hypotheses imply that the spectral sequence of Corollary \ref{spectralsequence1} degenerates at $E_{2}$, giving the result.
\end{proof}

We immediately obtain:

\begin{prop}\label{fpqccover} For $G$ a commutative group scheme of finite type over $F$, we have a natural isomorphism $\check{H}^{i}(\bar{F}/F, G) \xrightarrow{\sim} H^{i}(F, G)$ for all $i \geq 0$. 
\end{prop}

\begin{proof} Combine Proposition \ref{vanishingcohomology} with Corollary \ref{cechtoderived} for $\{R_{i} = E_{i}\}$ a tower of finite extensions of $F$ whose union is $\bar{F}$.
\end{proof}

In order to effectively use the above result to compute derived functor cohomology, we need a \v{C}ech interpretation of the connecting homomorphism $\delta$ for a short exact sequence of commutative finite-type $F$-group schemes $0 \to K \to G \to Q \to 0$. We do this for degree 1-2 and abelian $G$. Setting $U_{0} = \Spec(\bar{F})$, let $x \in Q(U_{1})$ be a \v{C}ech 1-cocycle. By Proposition \ref{vanishingcohomology}, we may find a lift $y \in G(U_{1})$ of $x$, and then we have $dy \in K(U_{2})$, since $Q(U_{2}) = G(U_{2})/K(U_{2})$, also using \ref{vanishingcohomology}. Moreover, $dy \in K(U_{2})$ is a \v{C}ech 2-cocycle. One checks that this defines a map $\check{\delta} \colon \check{H}^{1}(\bar{F}/F, Q) \to \check{H}^{2}(\bar{F}/F, K)$, and we have the following result:

\begin{prop}\label{connecting} Under the comparison isomorphisms of Proposition \ref{fpqccover}, we have $\check{\delta} = \delta$, where $\delta$ is the usual connecting homomorphism arising from the derived functor formalism.
\end{prop}

\begin{proof} This is \cite[Proposition E.2.1]{Rosengarten}.
\end{proof}

Recall that for a site $\mathcal{C}_{\tau}$ and $\mathscr{G}$ a group sheaf on $\mathcal{C}_{\tau}$, \textit{a $\mathscr{G}$-torsor $\mathscr{T}$} is a sheaf on $\mathcal{C}_{\tau}$ equipped with a right group (sheaf) action $\mathscr{T} \times \mathscr{G} \to \mathscr{T}$ (satisfying the usual group action axioms) such that for every object $X$ of $\mathcal{C}_{\tau}$, there is some cover $\{Y_{i} \to X\}$ such that $\mathscr{T}_{Y_{i}} := \mathscr{T} \times_{\mathcal{C}} (\mathcal{C}/Y_{i})$ is ($\mathscr{G}$-equivariantly) isomorphic to the trivial $\mathscr{G}_{Y_{i}}$-torsor $\mathscr{G}_{Y_{i}}$, that is, the group sheaf $\mathscr{G}_{Y_{i}}$ equipped with the right translation action.

\begin{Def}\label{torsors} For $G$ a finite-type $F$-group scheme (not necessarily abelian), we call an $F$-scheme $S$ equipped with a right $G$-action a \textit{$G$-torsor over $F$} if there is an fppf cover $V \to \Spec(F)$ such that $S_{V}$ is $G$-equivariantly isomorphic to $G_{V}$ equipped with the trivial $G_{V}$-action. We will see shortly that this is equivalent to only requiring $V \to \Spec(F)$ be fpqc. We define $G_{U}$-torsors over any scheme $U \to \Spec(F)$ completely analogously. 
\end{Def}

\begin{prop}\label{representability} Let $G$ be a finite-type group scheme over $F$ (recall that we assume it is affine), with $\underline{G}$ the associated sheaf on $(\text{Sch}/F)_{\text{fpqc}}$. For every $\underline{G}$-torsor $P$ on $(\text{Sch}/F)_{\text{fpqc}}$, $P$ is representable (as a torsor) by a $G$-torsor $S \to \Spec(F)$.
\end{prop}

\begin{proof} To begin with, let $\mathcal{V} = \{V_{i} \to \Spec(F)\}$ be an fpqc cover of $\Spec(F)$ trivializing $P$. Choosing trivializations $h_{i} \colon P_{V_{i}} \xrightarrow{\sim} \underline{G}_{V_{i}}$ (as $\underline{G}_{V_{i}}$-torsors) gives an element $x = (x_{ij}) \in \prod_{i,j} G(V_{i} \times_{F} V_{j})$ satisfying the 1-cocycle condition. This furnishes us with an fpqc descent datum of torsors on the site $(\text{Sch}/F)_{\text{fpqc}}$ via the cover $\{V_{i} \to \Spec(F)\}$, objects $\{G_{V_{i}}\}$ (with trivial right $G_{V_{i}}$-action), and isomorphisms $m_{x_{ij}} \colon p_{1}^{*}(G_{V_{j}}) \xrightarrow{\sim} p_{2}^{*}(G_{V_{i}})$ of $G_{V_{ij}}$-torsors given by left-translation by $x_{ij}$. Now, since the morphisms $G_{V_{i}} \to V_{i}$ are quasi-affine (indeed, they are the base change of the affine morphism $G \to \Spec(F)$), by \cite[Lemma 0247]{Stacksproj}, this descent datum is effective, and hence we get an $F$-scheme $S$ with a $G$-action such that $\underline{S} \xrightarrow{\sim} P$ as fpqc $G$-sheaves. Now, since $G \to \Spec(F)$ is fppf and $S$ is isomorphic to $G$ after an fpqc base-change, it is also fppf over $F$, and we have a section $S \xrightarrow{\Delta} S \times_{F} S$ given by the diagonal, showing that $S$ is trivialized over an fppf cover of $\Spec(F)$.
\end{proof}

\begin{remark}\label{repremark}
We will frequently use this proposition without comment in order to identify $G$-torsors over $F$ and $\underline{G}$-torsors on $(\text{Sch}/F)_{\text{fpqc}}$. Because of this, it is harmless to abuse notation and denote the sheaf $\underline{G}$ by $G$. When $G$ is not of finite type over $F$, we will make sure to specify the topology on $\text{Sch}/F$. However, even if $G$ is not of finite type, the first part of the above proof (which only uses affineness) shows that an fpqc $\underline{G}$-torsor is represented by an $F$-scheme $S$ equipped with a $G$-action which is trivialized over an fpqc (but not necessarily fppf) cover of $\Spec(F)$---we call such an $S$ a \textit{$G$-torsor over $F$ with respect to the \text{fpqc} topology}, and freely identify $\underline{G}$-torsors on $(\text{Sch}/F)_{\text{fpqc}}$ and $G$-torsors over $F$ with respect to the fpqc topology. Also, if we replace $F$ by $R$ an arbitrary commutative $F$-algebra, the proof carries over verbatim to show that every fpqc $\underline{G}$-torsor is representable by an $R$-scheme $S$.
\end{remark}

We may also define the \v{C}ech cohomology groups $\check{H}^{0}(U_{0} \to \Spec(F), G)$ and $\check{H}^{1}(U_{0} \to \Spec(F), G)$ for an arbitrary (possibly non-abelian) group scheme $G$, using the conventions of \cite[III.3.6]{Giraud}, which agree with our previous \v{C}ech cohomology conventions if $G$ is abelian. Namely, we define differentials from $G(U_{0})$ to $G(U_{1})$ and from $G(U_{1})$ to $G(U_{2})$, given (respectively) by 
\begin{equation}\label{cechboundaries} g \mapsto p_{1}^{\sharp}(g)^{-1}p_{2}^{\sharp}(g), \hspace{2mm} g \mapsto p_{12}^{\sharp}(g)p_{23}^{\sharp}(g)p_{13}^{\sharp}(g)^{-1}.
\end{equation}
We may then take $\check{H}^{0}(U_{0} \to \Spec(F), G)$ to be the fiber over the identity of the degree-zero differential, and $\check{H}^{1}(U_{0} \to \Spec(F), G)$ to be the pointed set consisting of the fiber over the identity of the degree-one differential modulo the equivalence relation given by declaring $a$ and $b$ equivalent if there exists $g \in G(U_{0})$ with $a = p_{1}^{\sharp}(g)^{-1}bp_{2}^{\sharp}(g)$. Denote the inverse limit of these sets over all $\tau$-covers by $\check{H}^{1}_{\tau}(F, G)$. The set $\check{H}^{1}_{\tau}(F, G)$ is evidently in bijection with isomorphism classes of $\underline{G}$-torsors on $(\text{Sch}/F)_{\tau}$, by gluing of sheaves on sites (see \cite[\S 04TP]{Stacksproj}).

\begin{Def}\label{cohomologous} When we say that two \v{C}ech $n$-cocycles $a, b \in \shA(U_{n})$ are \textit{cohomologous}, we always mean that $[a] = [b]$ in the set $\check{H}^{n}(U_{0} \to \Spec(F), \shA)$.
\end{Def}

For $G$ an arbitrary $F$-group scheme of finite type, we (temporarily) denote by $\tilde{H}^{1}(F,G)$ the pointed set of isomorphism classes of $G$-torsors over $F$ (see Definition \ref{torsors}---we don't specify the topology here in light of Proposition \ref{representability}). For $G$ abelian, $\tilde{H}^{1}(F,G)$ equals the group $H^{1}_{\text{fppf}}(F,G) = H^{1}(F,G)$, since we have a canonical isomorphism $\tilde{H}^{1}(F,G) = \check{H}^{1}_{\text{fppf}}(F, G)$, using that every fppf $\underline{G}$-torsor is representable by a scheme (by Proposition \ref{representability}), and the canonical isomorphism $\check{H}^{1}_{\text{fppf}}(F, G) = H^{1}_{\text{fppf}}(F, G)$, from Proposition \ref{newcomparison3}. Because of this, we may replace the notation $\tilde{H}^{1}(F,G)$ by $H^{1}(F,G)$ without causing any ambiguity. It is also clear that for general finite type $G$, many of our cohomology sets are redundant, since we have identifications $H^{1}(F,G) = \check{H}^{1}_{\text{fppf}}(F, G) = \check{H}^{1}_{\text{fpqc}}(F, G)$, by Proposition \ref{representability}.

In light of the above discussion, we will frequently identify $H^{1}(F,G)$ with isomorphism classes of fpqc $\underline{G}$-torsors on $(\text{Sch}/F)_{\text{fpqc}}$. We conclude this subsection with a short discussion of sheaf cohomology on a stack $\mathscr{X} \to (\text{Sch}/F)_{\tau}$ fibered in groupoids; first, we discuss how to give $\mathscr{X}$ the structure of a site.   

\begin{Def} For a stack $\mathscr{X} \xrightarrow{\pi} (\text{Sch}/F)_{\tau} $ fibered in groupoids, we give $\mathscr{X}$ the structure of a site via the \textit{$\tau$ topology on $\mathscr{X}$} (where $\tau \in \{\text{fppf},\text{fpqc} \}$). First, recall that $\mathscr{X}$ has finite fibered products, by Lemma \ref{fiberedprod}; to define this $\tau$ topology, for $X \in \text{Ob}(\mathscr{X})$ say that a collection of morphisms $\{X_{i} \xrightarrow{f_{i}} X \}$ in $\mathscr{X}$ is a cover if and only if $\{\pi(X_{i}) \xrightarrow{\pi(f_{i})} \pi(X) \}$ is a cover in $(\text{Sch}/F)_{\tau}$. This endows $\mathscr{X}$ with the structure of a site such that $\mathscr{X} \xrightarrow{\pi} (\text{Sch}/F)_{\tau}$ is a continuous functor. Whenever we talk about a Grothendieck topology on a stack over $F$, we will always be assuming that it is fibered in groupoids.
\end{Def}

\begin{conv}\label{settheoreticconv} Assume that for a group sheaf $\mathscr{G}$ on our site $\mathscr{X} \to (\text{Sch}/F)_{\tau}$ we have that the collection of isomorphism classes of $\mathscr{G}$-torsors on $\mathscr{X}$ is a set---of course this is not a problem unless $\tau = \text{fpqc}$, which is where the set-theoretic difficulties emerge, see for example \cite[\S 022A]{Stacksproj}. This will always be the case for our applications in this paper, and will be explained when relevant.
\end{conv}

\begin{Def} For a commutative group sheaf $\mathscr{G}$ on our site $\mathscr{X} \to (\text{Sch}/F)_{\tau}$ (which for us will always be the pullback to $\mathscr{X}$ of the fpqc sheaf associated to a finite type $F$-group scheme $G$; this sheaf on $\mathscr{X}$ will be denoted by $G_{\mathscr{X}}$) we define $H_{\tau}^{1}(\mathscr{X}, \mathscr{G})$ to be the group of isomorphism classes of $\mathscr{G}$-torsors on $\mathscr{X}$ equipped with the $\tau$ topology, with group operation induced by sending the torsors $\mathscr{T}, \mathscr{S}$ to the contracted product $\mathscr{T} \times^{\mathscr{G}} \mathscr{S}$; this contracted product always exists, see \cite[III.1.3]{Giraud}. If $\mathscr{G}$ is non-abelian, we define $H_{\tau}^{1}(\mathscr{X}, \mathscr{G})$ to be the pointed set of isomorphism classes of $\mathscr{G}$-torsors on $\mathscr{X}$ (with the induced $\tau$ topology). 
\end{Def}


\subsection{Gerbes}

We continue with the notation of \S 2.1 and \S 2.2, and fix $\shA$ a commutative $F$-group scheme, not necessarily of finite type.

\begin{Def} A stack $\gerbeE \xrightarrow{\pi} (\text{Sch}/F)_{\tau}$ fibered in groupoids is called a \textit{gerbe} if every object $U$ of $(\text{Sch}/F)_{\tau}$ has a cover $\{V_{i} \to U\}$ such that every $V_{i}$ has a lift in $\gerbeE$, and for any two objects $X,Y \in \text{Ob}(\gerbeE(U))$, there is a cover $\{V_{i} \xrightarrow{f_{i}} U\}$ such that $f_{i}^{*}X$ and $f_{i}^{*}Y$ are isomorphic in $\gerbeE(V_{i})$ for all $i$. When we want to emphasize the topology of $(\text{Sch}/F)_{\tau}$, we say that $\gerbeE$ is a $\tau$-gerbe ($\tau \in \{\text{fppf}, \text{fpqc} \}$).
\end{Def}

\begin{ex} The \textit{classifying stack of $\shA$ over $F$}, denoted by $B_{F}\shA \to (\text{Sch}/F)_{\tau}$, has fiber category $B_{F}\shA(U)$, for $U \in \text{Ob}(\text{Sch}/F)$ an $F$-scheme, the category of all $\shA_{U}$ torsors $T$ (with respect to the $\tau$ topology), with morphisms being isomorphisms of $\shA_{U}$ torsors. For $V \xrightarrow{f} U$ in $\text{Sch}/F$ and $T, S$ fixed $\shA_{U}, \shA_{V}$-torsors (respectively), a morphism $(V, S) \to (U,T)$ lifting $f$ is an isomorphism of $\shA_{V}$-torsors $S \to f^{*}T$. One verifies easily that this is a gerbe over $(\text{Sch}/F)_{\tau}$.
\end{ex}

By abuse of notation, for a stack $\mathscr{X} \to (\text{Sch}/F)_{\tau}$, we frequently write $\mathscr{X}$ to denote the site of $\mathscr{X}$ with the Grothendieck topology induced by $(\text{Sch}/F)_{\tau}$; that is to say, we write $\mathscr{X}$ to denote $\mathscr{X}_{\tau}$, $\tau \in \{\text{fppf}, \text{fpqc} \}$.

\begin{Def} As we discussed in \S 2.1, for any $X \in \gerbeE(U)$, the functor on $\text{Sch}/U$ given by sending $V \xrightarrow{f} U$ to $\text{Aut}_{U}(f^{*}X)$ defines a sheaf of groups on $(\text{Sch}/U)_{\tau}$, denoted by $\underline{\text{Aut}}_{U}(X)$. We call our gerbe $\gerbeE$ \textit{abelian} if this group sheaf is abelian for all $X$.  
\end{Def}

\begin{lem}\label{bandlemma} If $\gerbeE$ is an abelian gerbe, then the sheaves $\underline{\text{Aut}}_{U}(X)$, as $X$ varies through all objects of $\gerbeE$, glue to define an abelian group sheaf on $(\text{Sch}/F)_{\tau}$, called the \textbf{band} of $\gerbeE$ and denoted by $\text{Band}(\gerbeE)$. Moreover, we have for any $X \in \gerbeE(U)$ an isomorphism $\restr{\text{Band}(\gerbeE)}{U} \xrightarrow{h_{X}} \underline{\text{Aut}}_{U}(X)$ of sheaves on $(\text{Sch}/U)_{\tau}$ such that for any $X,Y \in \gerbeE(U)$ and isomorphism $\varphi \colon X \to Y$ in $\gerbeE(U)$, the following diagram commutes
\[
\begin{tikzcd}
\restr{\text{Band}(\gerbeE)}{U} \arrow[equals]{r} \arrow[d, "h_{X}"] & \restr{\text{Band}(\gerbeE)}{U} \arrow[d, "h_{Y}"] \\
\underline{\text{Aut}}_{U}(X) \arrow["f \mapsto \varphi \circ f \circ \varphi^{-1}"]{r} & \underline{\text{Aut}}_{U}(Y)
\end{tikzcd}
\]
\end{lem}

\begin{proof} This is \cite[Lemma 0CJY]{Stacksproj}.
\end{proof}
In fact, following the setup of the above lemma, even if $X$ and $Y$ are not isomorphic in $\gerbeE(U)$, since they are locally isomorphic (by the definition of a gerbe), we may find a cover $\{V_{i} \to U\}$ such that the pullbacks of $X$ and $Y$ to each $V_{i}$ are isomorphic via some $\phi_{i}$, so that we get an isomorphism $\restr{\underline{\text{Aut}}_{U}(X)}{V_{i}} \xrightarrow{\sim} \restr{\underline{\text{Aut}}_{U}(Y)}{V_{i}}$ for all $i$ of sheaves on $(\text{Sch}/V_{i})_{\tau}$ which is independent of the choice of $\phi_{i}$ in view of the above lemma, and hence glues to a \textit{canonical} isomorphism $\underline{\text{Aut}}_{U}(X) \xrightarrow{\sim} \underline{\text{Aut}}_{U}(Y)$ of sheaves on $(\text{Sch}/U)_{\tau}$ (which is the same as $h_{Y} \circ h_{X}^{-1}$). 

For the rest of this paper, all gerbes will be assumed to be abelian, and when we refer to a ``gerbe," we always mean an abelian gerbe.

\begin{conv}\label{conventions} A simplifying convention we will use in this paper is that, when discussing an abelian $F$-group scheme $\shA$ and a $\tau$-cover $U_{0} \to \Spec(F)$, we will always assume that $\check{H}^{1}_{\tau}(U_{n}, \shA_{U_{n}}) = 0$ for all $n \geq 0$.  Equivalently, every $\shA_{W}$-torsor over $W$ has a $W$-trivialization (see Remark \ref{repremark}) for $W = U_{n}$. If $\shA$ is of finite-type, this condition holds for $U_{0} = \Spec(\bar{F})$ and $\tau = \text{fpqc}$, see \cite[\S 2.9]{Rosengarten}.
\end{conv}

\begin{Def}\label{bandedgerbe} We call a pair $(\gerbeE, \theta)$ of a gerbe $\gerbeE$ and an isomorphism $\theta \colon \shA \xrightarrow{\sim} \text{Band}(\gerbeE)$ an \textit{$\shA$-gerbe}. In practice, $\theta$ will be a way for us to identify automorphisms of objects in $\gerbeE$ with elements of $\shA$ in a manner that does not depend on isomorphism classes in the fibers; we will frequently omit explicit mention of the map $\theta$. For $X \in \gerbeE(V)$, we denote the isomorphism $h_{X} \circ \theta_{U}$ from Lemma \ref{bandlemma} by $\theta_{X}$. Any morphism of stacks over $(\text{Sch}/F)_{\tau}$ between two gerbes $\gerbeE$ and $\gerbeE'$ induces a morphism of group sheaves over $F$ between the corresponding bands. If both can be given the structure of $\shA$-gerbes, then we say that such a morphism of $\text{Sch}/F$ categories between two $\shA$-gerbes is a \textit{morphism of $\shA$-gerbes} if it is the identity on bands (via the identifications of both bands with $\shA$). By \cite[Lemma 12.2.4]{Olsson}, any morphism of $\shA$-gerbes is an equivalence, and so we will also call such a functor an \textit{equivalence of $\shA$-gerbes}. 
\end{Def}

\begin{ex} The gerbe $B_{F}\shA \to (\text{Sch}/F)_{\tau}$ may be canonically given the structure of an $\shA$-gerbe, since for an abelian group sheaf $\shA$ and $\shA$-torsor $T$, the automorphism sheaf defined by $T$ is canonically isomorphic to $\shA$. 
\end{ex}
We say that an $\shA$-gerbe $\gerbeE$ is \textit{split} over the $\tau$-cover $V \to \Spec(F)$ if $\gerbeE_{V} := \gerbeE \times_{\text{Sch}/F} \text{Sch}/V$ is equivalent as an $\shA_{V}$-gerbe to $B_{V}(\shA_{V})$. The following result is a useful alternative characterization of an $\shA$-gerbe $\gerbeE$ splitting over a cover $V \to \Spec(F)$:
\begin{prop}\label{splittingchar} The gerbe $\gerbeE \to (\text{Sch}/F)_{\tau}$ is split over $V \to \Spec(F)$ if and only if there is an object $X \in \gerbeE(V)$.
\end{prop}

\begin{proof} It is clear that if an $\shA$-gerbe $\gerbeE$ is split over $V$, we have such an object. For the other direction, see \cite[Remark 2.4]{Vistoli}.
\end{proof}

\begin{fact}\label{gerbecocycle} Gerbes are closely related to \v{C}ech 2-cocycles of $\shA$ with respect to covers of $(\text{Sch}/F)_{\tau}$, and in this sense are natural analogues of the group extensions that arise in the study of 2-cocycles from Galois cohomology. Indeed, let $(\gerbeE, \theta)$ be an $\shA$-gerbe over $(\text{Sch}/F)_{\tau}$, and take some $U_{0} \to \Spec(F)$ a cover in $\tau$ such that we have some $X \in \gerbeE(U_{0})$ with $p_{1}^{*}X \xrightarrow{\varphi, \sim} p_{2}^{*}X$ for some $\varphi$ an isomorphism in $\gerbeE(U_{1})$ (because of Convention \ref{conventions}, given any $X \in \gerbeE(U_{0})$, we can always find such a $\varphi$). We extract a \v{C}ech 2-cocycle $a \in \shA(U_{2})$ in the following manner: $\varphi$ defines an automorphism of $q_{1}^{*}X$ over $U_{2}$ via the composition $$d\varphi:= (p_{13}^{*}\varphi)^{-1} \circ (p_{23}^{*}\varphi) \circ (p_{12}^{*} \varphi) \in \text{Aut}_{U_{2}}(q_{1}^{*}X),$$ and we set $a = \theta_{q_{1}^{*}X}^{-1}(d \varphi) \in \shA(U_{2})$. Then $a$ is a \v{C}ech 2-cocycle, whose class in $\check{H}^{2}(U_{0} \to \Spec(F), \shA)$ is independent of the choice of $\varphi$ and $X$ (see \cite[\S3]{Moerdijk}). We denote by $[\gerbeE] \in \check{H}^{2}(U_{0} \to \Spec(F), \shA)$ the \v{C}ech cohomology class obtained from $\gerbeE$ as above, and call $[\gerbeE]$ the \textit{\v{C}ech class corresponding to $\gerbeE$}.
\end{fact}

\begin{cor}\label{equivalence} We have a well-defined map from the set of $\shA$-gerbes split over $V$ to the group $\check{H}^{2}(V \to \Spec(F), \shA)$ defined by $\gerbeE \mapsto [\gerbeE]$. Moreover, if $(\gerbeE, \theta)$ is equivalent to $(\gerbeE', \theta')$, then $[\gerbeE] = [\gerbeE']$.
\end{cor}

\begin{proof} The first statement is immediate. The second statement is a straightforward exercise using Lemma \ref{bandlemma} and pullbacks in fibered categories.
\end{proof}

\begin{remark} One could just as well replace the isomorphism $\varphi$ with an isomorphism from $p_{2}^{*}X$ to $p_{1}^{*}X$, but this creates sign issues that for clarity we will choose to avoid. This is the reason that the isomorphism $\varphi$ goes in the opposite direction from our twisted gluing isomorphisms defined below (for example, in Definition \ref{explicitgerbe}).
\end{remark}

\begin{ex} For $\gerbeE = B_{F}\shA$, one can take $U_{0} = \Spec(F)$, $X = \shA$, and $\varphi = \text{id}$, yielding the trivial class in $\check{H}^{2}(U_{0} \to \Spec(F), A)$ via the above correspondence.
\end{ex}

\begin{Def} Let $\mathscr{F}$ be a sheaf (of sets) on an $\shA$-gerbe $(\gerbeE, \theta)$ with the induced $\tau$ topology. Recall that for $G$ a finite type $F$-group scheme (thought of as an fpqc sheaf over $F$), we denote by $\GE$ the pullback of $G$ to a sheaf on $\gerbeE$ (with the induced $\tau$ topology). We have a morphism of sheaves on $\gerbeE$ denoted by $$\iota \colon \shA_{\gerbeE} \times_{\gerbeE} \mathscr{F} \to \mathscr{F},$$ called the \textit{inertial action}, which for an object $X$ of $\gerbeE(U)$ and $a \in \shA_{\gerbeE}(X) = \shA(U)$ is defined by the automorphism $\mathscr{F}(X) \xrightarrow{\theta_{X}(a)^{\sharp}} \mathscr{F}(X)$. This gives an action of the group sheaf $\shA_{\gerbeE}$ on the sheaf $\mathscr{F}$, see \cite[2.3]{Shin}.
\end{Def}

\begin{Def}\label{explicitgerbe} Fix a \v{C}ech 2-cocycle $a$ of $\shA$ taking values in the cover $U_{0} \to \Spec(F)$, that is to say, $a \in \shA(U_{2})$. Then we may define an $\shA$-gerbe as follows: take the fibered category $\gerbeE_{a} \to \text{Sch}/F$ whose fiber over $V$ is defined to be the category of pairs $(T,\psi)$, where $T$ is a (right) $\shA_{V \times_{F} U_{0}}$-torsor on $V \times_{F} U_{0}$ with $\shA$-action $m$ (in the $\tau$ topology), along with an isomorphism of $\shA_{V \times_{F} U_{1}}$-torsors $\psi \colon p_{2}^{*}T \xrightarrow{\sim} p_{1}^{*}T$, called a \textit{twisted gluing map}, satisfying the following ``twisted gluing condition" on the $\shA_{V \times_{F} U_{2}}$-torsor $q_{1}^{*}T$: $$(p_{12}^{*}\psi) \circ (p_{23}^{*} \psi ) \circ (p_{13}^{*}\psi)^{-1} = m_{a},$$ where $m_{a}$ denotes the automorphism of the torsor $q_{1}^{*}T$ given by right-translation by $a$. A morphism $(T,\psi_{T}) \to (S, \psi_{S})$ in $\gerbeE_{a}$ lifting the morphism of $F$-schemes $V \xrightarrow{f} V'$ is a morphism of $\shA_{V \times U_{0}}$-torsors $T \xrightarrow{h} f^{*}S$ satisfying, on $V \times_{F} U_{1}$, the relation $f^{*}\psi_{S} \circ p_{2}^{*}h = p_{1}^{*}h \circ \psi_{T}$. We will call such a pair $(T, \psi)$ in $\gerbeE_{a}(V)$ a \textit{$a$-twisted torsor over $V$} when $\shA$ is understood. We call $\gerbeE_{a}$ the \textit{gerbe corresponding to $a$}.
\end{Def}
When working with the fibered category $\gerbeE_{a} \to \text{Sch}/F$ there is an obvious canonical choice of pullbacks. Indeed, for $(T, \psi) \in \gerbeE_{a}(U)$ and $f \colon V \to U$, we set $f^{*}(T, \psi) := (f^{*}T, f^{*}\psi)$, and the strongly cartesian morphism $f^{*}(T, \psi) \to (T, \psi)$ to be the one induced by the identity. We always work with this choice of pullbacks. 

\begin{prop}\label{expectedclass}The category $\gerbeE_{a} \to (\text{Sch}/F)_{\tau}$ may be canonically given the structure of an $\shA$-gerbe $(\gerbeE_{a}, \theta)$, with an object $X \in \gerbeE_{a}(U_{0})$ and an isomorphism $\varphi \colon p_{1}^{*}X \to p_{2}^{*}X$ satisfying $\theta_{q_{1}^{*}X}^{-1}(d \varphi) = a \in \shA(U_{2})$. In particular, $\gerbeE_{a}$ is split over $U_{0}$ and $[\gerbeE_{a}] = [a]$. 
\end{prop}

\begin{proof} If we prove that there is such an object $X \in \gerbeE_{a}(U_{0})$, it will follow immediately that $\gerbeE_{a}$ defines a gerbe, provided that any two objects are locally isomorphic (the proof of this last fact will be postponed). Moreover, we have that $\text{Band}(\gerbeE_{a})$ is canonically isomorphic to $\shA$, since (for $V = \Spec(F)$, the general case is identical) any automorphism of an $a$-twisted torsor $(T,\psi)$ is given by a unique element $x \in \shA(\bar{F})$, and since $\psi$ is a morphism of $\shA_{U_{1}}$-torsors commuting with this chosen automorphism, we in fact have that $x \in \shA(F)$ (using fpqc descent, cf. the proof of Lemma \ref{twisthom} below). All that's left to show is the existence of $X$ and $\varphi$. This follows from the following lemma (which is important in its own right).
\end{proof}

\begin{lem}\label{section} We have a canonical section $x: \text{Sch}/U_{0} \to \gerbeE_{a}$ such that the two pullbacks $x_{1}$ and $x_{2}$ to $\text{Sch}/U_{1}$ are isomorphic via $\varphi \colon x_{1} \xrightarrow{\sim} x_{2}$ satisfying $d \varphi := (p_{13}^{*}\varphi)^{-1} \circ (p_{23}^{*}\varphi) \circ (p_{12}^{*} \varphi)  = \iota_{a}$ as a natural transformation from $q_{1}^{*}x: \text{Sch}/U_{2} \to \gerbeE_{a}$ to itself, where we are using $\iota_{a}$ to denote the natural transformation from the identity functor $(\gerbeE_{a})_{U_{2}} \to (\gerbeE_{a})_{U_{2}}$ to itself given by the automorphism $\theta_{Z}(a_{V}) \colon Z \xrightarrow{\sim} Z$ for all $Z \in \gerbeE_{U_{2}}(V)$.
\end{lem}

\begin{proof}

Define the $a$-twisted torsor on $(\text{Sch}/U_{0})_{\text{fpqc}}$ to be (as an $\shA_{U_{0}}$-torsor) $\shA_{U_{0}}$; we will define the twisted gluing map after a short discussion. The gluing map should be an isomorphism of $\shA_{U_{1}}$-torsors: $\psi:  \tilde{p_{2}}^{*}(\shA_{U_{0}}) \to \tilde{p_{1}}^{*}(\shA_{U_{0}})$, where $\tilde{p_{2}}: U_{0} \times U_{1} \to U_{0} \times U_{0}$ is $\text{id}_{U_{0}} \times p_{2}$ and $\tilde{p_{1}}: U_{0} \times U_{1} \to U_{0} \times U_{0}$ is $\text{id}_{U_{0}} \times p_{1}$. We have that $U_{0} \times U_{1} = U_{2}$, and $U_{0} \times U_{0} = U_{1}$, and then $\tilde{p_{1}}$ equals $p_{12}$, $\tilde{p_{2}}$ equals $p_{13}$. So, giving $\psi$ reduces to giving a morphism of $\shA_{U_{0} \times U_{1}} = \shA_{U_{2}}$-torsors $p_{13}^{*}(\shA_{U_{1}}) \to p_{12}^{*}(\shA_{U_{1}})$. Both sides are canonically equal to $\shA_{U_{2}}$, because $\shA$ is a sheaf on $(\text{Sch}/F)_{\text{fpqc}}$ so its value on a $U_{1}$-object only depends on the map to $\Spec(F)$, which is the same regardless of the map from $U_{2}$ to $U_{1}$. So we may take $\psi$ to be $m_{a}$, which makes sense since $a \in \shA(U_{2})$; this is $\shA$-equivariant since $\shA$ is commutative. We need to check that $\psi$ satisfies the twisted cocycle condition. 

The above paragraph relied on the equalities $U_{0} \times U_{1} = U_{2}$ and $U_{0} \times U_{0} = U_{1}$. Continuing these identifications, $\tilde{p_{12}}: U_{0} \times U_{2} \to U_{0} \times U_{1}$ is the map $U_{3} \to U_{2}$ given by $q_{123}$, and similarly $\tilde{p_{13}} = q_{124}, \tilde{p_{23}} = q_{134}$. Whence, $\tilde{p_{13}}^{*}(\psi^{-1}) \circ \tilde{p_{12}}^{*}(\psi) \circ \tilde{p_{23}}^{*}(\psi) = (q_{124}^{*}m_{a}^{-1}) \circ (q_{123}^{*}m_{a}) \circ (q_{134}^{*}m_{a})  = q_{234}^{*}m_{a}$, since $a$ is a \v{C}ech 2-cocycle. Take $\tilde{q_{1}}^{*}(\shA_{U_{0}}), \tilde{q_{1}} = \text{id}_{U_{0}} \times q_{1}$. By construction, after identifying $\tilde{q_{1}}\shA$ with $\shA_{U_{3}}$, we see that the left multiplication map $m_{a_{U_{2},r_{2}}}$, where $a_{U_{2}, r_{2}}$ denotes the image of $a$ in $\shA(U_{0} \times U_{2}) = \shA(U_{3})$ via the map $r_{2} \colon U_{0} \times U_{2} \to U_{2}$ which projects onto the second factor, equals $q_{234}^{*}m_{a}$, as desired. This $a$-twisted $\shA$-torsor on $(\text{Sch}/U_{0})_{\text{fpqc}}$ induces an $a_{V}$-twisted $\shA$-torsor on each $(\text{Sch}/V)_{\text{fpqc}}$, $V \to U$, via pullback, giving our map $x$, which one easily checks is a functor.

We now need to define a natural transformation $\varphi: x_{1} \xrightarrow{\sim} x_{2}$ between the two pullbacks of $x$ to $U_{1}$. It's enough (by taking pullbacks) to define a morphism of $a$-twisted torsors $$\varphi \colon \shA_{(U_{1} \xrightarrow{p_{1}} U_{0}) \times U_{0}} \to \shA_{(U_{1} \xrightarrow{p_{2}} U_{0}) \times U_{0}},$$ which we can take to be translation by $a$, via the same identifications as above.  We will verify shortly that $p_{1}^{*}\varphi \circ \psi_{U_{1} \xrightarrow{p_{1}} U_{0}} =  \psi_{U_{1} \xrightarrow{p_{2}} U_{0}} \circ p_{2}^{*}\varphi$. The same argument showing that $d \psi = m_{a}$ gives that $d\varphi = m_{a}$, which is $\iota_{a}$, by the definition of the inertial action on $\gerbeE_{a}$.

We now justify our above claim that $\varphi$ is a morphism of $a$-twisted torsors. For $V \xrightarrow{f} U_{0}$, the gluing map $\psi_{V}$ is $$(\shA_{V \times U_{0}}) \times_{\text{id} \times p_{2}} (V \times U_{1}) \to (\shA_{V \times U_{0}}) \times_{\text{id} \times p_{1}} (V \times U_{1}),$$ given by left translation by $a \in \shA(U_{2}) \xrightarrow{(f \times \text{id})^{\sharp}} \shA(V \times U_{1})$. As such, we first look at $\psi_{U_{1} \xrightarrow{p_{1}} U_{0}}$. This is the map on $\shA_{U_{3}}$ given by left translation by the image of $a$ in $\shA(U_{1} \times U_{1})$ via $\shA(U_{2}) \xrightarrow{(p_{1} \times \text{id})^{\sharp}} \shA(U_{3})$, which is evidently $p_{134}^{\sharp}(a)$. 

We also have the map $$\varphi: \shA_{(U_{1} \xrightarrow{p_{1}} U_{0}) \times U_{0}} \to \shA_{(U_{1} \xrightarrow{p_{2}} U_{0}) \times U_{0}} $$ which is also left translation by $a \in \shA(U_{2})$. Thus, $p_{1}^{*}\varphi$ is the map  $$\varphi: \shA_{(U_{1} \xrightarrow{p_{1}} U_{0}) \times U_{0}} \times_{\text{id} \times p_{1}} (U_{1} \times U_{1}) \to \shA_{(U_{1} \xrightarrow{p_{2}} U_{0}) \times U_{0}} \times_{\text{id} \times p_{1}} (U_{1} \times U_{1}),$$ which is left translation by the image of $a$ in $\shA(U_{3})$ via $U_{3} \xrightarrow{\text{id} \times p_{1}} U_{2}$, which is $p_{123}^{\sharp}(a)$. 

On the other hand, the map $$p_{2}^{*}\varphi = \varphi \colon \shA_{(U_{1} \xrightarrow{p_{1}} U_{0}) \times U_{0}} \times_{\text{id} \times p_{2}} (U_{1} \times U_{1}) \to \shA_{(U_{1} \xrightarrow{p_{2}} U_{0}) \times U_{0}} \times_{\text{id} \times p_{2}} (U_{1} \times U_{1}),$$

\noindent corresponds on $\shA_{U_{3}}$ to translation by $(\text{id} \times p_{2})^{\sharp}(a) = p_{124}^{\sharp}(a)$, and, finally, we have $$\psi_{U_{1} \xrightarrow{p_{2}} U_{0}} \colon \shA_{(U_{1} \xrightarrow{p_{2}} U_{0}) \times U_{0}} \times_{\text{id} \times p_{2}} (U_{1} \times U_{1}) \to \shA_{(U_{1} \xrightarrow{p_{2}} U_{0}) \times U_{0}} \times_{\text{id} \times p_{1}} (U_{1} \times U_{1})$$ given by $(p_{2} \times \text{id})^{\sharp}(a) = p_{234}^{\sharp}(a)$. The desired equality holds since $p_{234}^{\sharp}(a) \cdot p_{124}^{\sharp}(a) = p_{134}^{\sharp}(a) \cdot p_{123}^{\sharp}(a)$, since $a$ is a 2-cocycle.
\end{proof}

We still need to prove that any two objects in $\gerbeE_{a}$ are locally isomorphic before we know that it is a gerbe (right now all we have shown is that it's a fibered category with an object in the fiber over $U_{0}$). One can see this using a similar calculation as in the proof of Lemma \ref{section}, but there is a more conceptual argument coming from the following discussion of functoriality:



\begin{const}\label{changeofgerbe} Let $\shA \xrightarrow{f} \textbf{B}$ be an $F$-morphism of commutative group schemes, $U_{0} \to \Spec(F)$ a cover in $\tau$, and $a, b \in \shA(U_{2}), \textbf{B}(U_{2})$ two \v{C}ech 2-cocycles such that $[f(a)]= [b]$ in $\check{H}^{2}(U_{0} \to \Spec(F), \textbf{B})$. Then for any $x \in \textbf{B}(U_{1})$ satisfying $d(x)\cdot b = f(a)$, we may define a morphism of $(\text{Sch}/F)_{\tau}$-stacks $\gerbeE_{a} \xrightarrow{\phi_{a,b,x}} \gerbeE_{b}$.

For any $V \in \text{Ob}(\text{Sch}/F)$, given a $a$-twisted torsor $(T,\psi)$ over $V$, we define a $b$-twisted torsor $(T',\psi')$ over $V$ as follows. Define the $\textbf{B}_{V \times_{F} U_{0}}$ torsor $T'$ to be $T \times^{\textbf{A}_{V \times U_{0}},f} \textbf{B}_{V \times U_{0}}$, and take the gluing map to be $\psi' := \overline{m_{x^{-1}} \circ \psi}$, where $\overline{m_{x^{-1}} \circ \psi}$ denotes the isomorphism of contracted products $$p_{2}^{*}(T \times^{\textbf{A}_{V \times U_{0}},f} \textbf{B}_{V \times U_{0}}) = (p_{2}^{*}T) \times^{\textbf{A}_{V \times U_{1}},f} \textbf{B}_{V \times U_{1}} \to (p_{1}^{*}T) \times^{\textbf{A}_{V \times U_{1}},f} \textbf{B}_{V \times U_{1}} = p_{1}^{*}(T \times^{\textbf{A}_{V \times U_{0}},f} \textbf{B}_{V \times U_{0}})$$ induced by $(m_{x^{-1}} \circ \psi) \times \text{id}_{\textbf{B}}$ (and we are implicitly identifying $x$ with its image in $\textbf{B}(V \times_{F} U_{1})$). 

We compute that 
$$(p_{12}^{*}\psi') \circ (p_{23}^{*} \psi' ) \circ (p_{13}^{*}\psi')^{-1} = m_{p_{12}^{\sharp}(x)^{-1}p_{23}^{\sharp}(x)^{-1}p_{13}^{\sharp}(x) \cdot f(a)} = m_{b},$$ so that $\phi_{a,b,x}((T,\psi)):= (T', \psi')$ indeed defines an element of $\gerbeE_{b}(V)$. From here, one checks that any morphism $\varphi \colon (S,\psi_{S}) \to (T,\psi_{T})$ of $a$-twisted torsors induces a morphism of the corresponding $b$-twisted torsors by means of the map on contracted products induced by $\varphi \times \text{id}$, giving the desired morphism of $(\text{Sch}/F)_{\tau}$-stacks.

Note that the above morphism does in general depend on the choice of $x$; indeed, any two such morphisms differ by post-composing by an automorphism of $\gerbeE_{b}$ determined by a \v{C}ech 1-cocycle $z$ with respect to the cover $U_{0} \to \Spec(F)$. It is clear from the above discussion that if $\shA = \textbf{B}$ and $f = \text{id}_{\shA}$ and $a$ and $b$ are cohomologous, then this construction gives an isomorphism of categories between $\gerbeE_{a}$ and $\gerbeE_{b}$.
\end{const}

We can finally prove that $\gerbeE_{a}$ is a gerbe:

\begin{cor} Any two objects in $\gerbeE_{a}$ are locally isomorphic.
\end{cor}

\begin{proof}
It is enough to prove the result after replacing $\gerbeE_{a}$ with $(\gerbeE_{a})_{U_{0}}$. This replaces the cover $U_{0} \to \Spec(F)$ with $U_{0}' := U_{1} \xrightarrow{p_{1}} U_{0}$, and each $n$-fold product $U_{n}$ with $U_{n}' := U_{n+1}$. Using the definition of the gerbes $\mathcal{E}_{c}$, the base-change $(\gerbeE_{a})_{U_{0}}$ becomes canonically identified with $\gerbeE_{a'}$, where $a' := p_{234}^{\sharp}(a) \in \shA_{U_{0}}(U_{2}') = \shA(U_{3})$. But now $a'$ is a coboundary---we can take $x = a \in \shA_{U_{0}}(U_{1}') = \shA(U_{2})$ and, since $a$ is a \v{C}ech $2$-cocycle, we compute that
$$dx = p_{123}^{\sharp}(a) \cdot p_{134}^{\sharp}(a)^{-1} \cdot p_{124}^{\sharp}(a) = p_{234}^{\sharp}(a).$$
From here we can apply Construction \ref{changeofgerbe} with $f = \text{id}$ and $x$ as above to identify $(\gerbeE_{a})_{U_{0}} = \mathcal{E}_{a'}$ with $B_{F}\shA_{U_{0}}$, where the desired result is clear.
\end{proof}

\begin{prop}\label{gerbebijection} Suppose that $\gerbeE \to (\text{Sch}/F)_{\tau}$ and $\gerbeE' \to (\text{Sch}/F)_{\tau}$ are two $\shA$-gerbes split over $U_{0} \to \Spec(F)$ a cover in $\tau$. Then $[\gerbeE] = [\gerbeE']$ in $\check{H}^{2}(U_{0} \to \Spec(F), \shA)$ if and only if $\gerbeE$ is $\shA$-equivalent to $\gerbeE'$.
\end{prop}

\begin{proof} We already know the ``if" direction from Corollary \ref{equivalence}. Let $\gerbeE$ be a $\shA$-gerbe with $X \in \gerbeE(U_{0})$ and $\varphi \colon p_{1}^{*}X \xrightarrow{\sim} p_{2}^{*}X$ in $\gerbeE(U_{1})$. By Definition \ref{bandedgerbe}, it's enough to construct a $(\text{Sch}/F)$-morphism $\gerbeE \to \gerbeE'$ which is the identity on bands. If we show that $\gerbeE$ is $\shA$-equivalent to $\gerbeE_{a}$ for $a \in \shA(U_{2})$ giving $d \varphi$, then the result will follow from applying Construction \ref{changeofgerbe} to cohomologous cocycles (with $f = \text{id}_{\shA}$, in the notation of the construction).

At the level of objects, send $Y \in \gerbeE(V \xrightarrow{f} \Spec(F))$ to the sheaf $\underline{\text{Isom}}_{\gerbeE(V \times U_{0})}(\tilde{p}_{2}^{*}X, \tilde{p}_{1}^{*}Y)$, on $\text{Sch}/(V \times_{F} U_{0})$, where $\tilde{p}_{i}$ is the $i$th projection of $V \times U_{0}$. We claim that this sheaf is an $a$-twisted torsor over $V$. First, it is easy to see that the above sheaf is an $\shA_{V \times U_{0}}$-torsor, by means of the action of the band of $\gerbeE$ on either side of the isomorphism (it doesn't matter which by Lemma \ref{bandlemma}). We need to define an isomorphism of $\shA_{V \times U_{1}}$-torsors $$\psi \colon p_{2}^{*}[\underline{\text{Isom}}_{\gerbeE(V \times U_{0})}(\tilde{p}_{2}^{*}X, \tilde{p}_{1}^{*}Y)] \xrightarrow{\sim} p_{1}^{*}[\underline{\text{Isom}}_{\gerbeE(V \times U_{0})}(\tilde{p}_{2}^{*}X, \tilde{p}_{1}^{*}Y)]$$ satisfying the twisted gluing condition with respect to $a$. We may take this to be the isomorphism obtained by pre-composing by $\tilde{p}_{2, U_{1}}^{*}\varphi$ (after making appropriate canonical identifications which we leave to the reader, where $\tilde{p}_{i, U_{1}}$ is the $i$th projection for $V \times U_{1}$). This defines our equivalence on the level of objects. 

At the level of morphisms, for $Y \xrightarrow{\tilde{f}} Z$ lifting $V \xrightarrow{f} W$, we have an induced morphism $\tilde{p}_{1}^{*}Y \to \tilde{p}_{1}^{*}Z$, and post-composing by this map gives a morphism $$(\underline{\text{Isom}}_{\gerbeE(V \times U_{0})}(\tilde{p}_{2}^{*}X, \tilde{p}_{1}^{*}Y), \psi_{Y}) \to (\underline{\text{Isom}}_{\gerbeE(V \times U_{0})}(\tilde{p}_{2}^{*}X, \tilde{p}_{1}^{*}Z), \psi_{Z}),$$ which is a morphism of $a$-twisted torsors. The induced morphism of bands is the identity by definition of the $\shA$-action on each $\underline{\text{Isom}}_{\gerbeE(V \times U_{0})}(\tilde{p}_{2}^{*}X, \tilde{p}_{1}^{*}Y)$.
\end{proof}

\begin{remark} The above equivalence $\gerbeE \to \gerbeE_{a}$ sends $X$ to the isomorphism class of the \textit{trivial $a$-twisted torsor}, which we define to be the $a$-twisted torsor over $U_{0}$ given by $x(U_{0})$, where $x$ is the section constructed in Lemma \ref{section}.
\end{remark}

We conclude this subsection with an important result concerning torsors on gerbes which will be crucial for our later cohomological constructions. In what follows, we fix a finite type $F$-group scheme $G$. Recall that if $\gerbeE \xrightarrow{\pi} (\text{Sch}/F)_{\tau}$ is a gerbe, then $\GE$ denotes the corresponding group sheaf on $\gerbeE$ with the induced $\tau$ topology. 

\begin{lem}\label{altres} Assume that $G$ is abelian. If $\mathscr{T}$ is a $G_{\gerbeE}$-torsor on the $\shA$-gerbe $\gerbeE \xrightarrow{\pi} (\text{Sch}/F)_{\tau}$ split over $U_{0}$, then there is a unique map $\phi \in \Hom_{F}(\shA, G)$ such that the inertial action $\iota \colon \shA_{\gerbeE} \times_{\gerbeE} \mathscr{T} \to \mathscr{T}$ is induced by $\phi_{\gerbeE} := \pi^{*}\phi \colon \shA_{\gerbeE} \to G_{\gerbeE}$. We denote this homomorphism by $\text{Res}(\mathscr{T})$. 
\end{lem}

\begin{proof} If such a map exists, uniqueness is clear. For $V \to \Spec(F)$, $X \in \gerbeE(V)$, and $x \in \shA_{\gerbeE}(X) = \shA(V)$, the induced automorphism of sheaves $\iota_{x} \colon \restr{\mathscr{T}}{\gerbeE/X} \to \restr{\mathscr{T}}{\gerbeE/X}$ is $\restr{\GE}{\gerbeE/X}$-equivariant, since the $\GE$-action $\mathscr{T} \times_{\gerbeE} \GE \to \mathscr{T}$ is a morphism of sheaves on $\gerbeE$ and $\iota_{x}$ is induced by an automorphism of $X$ (by definition). It follows that $\iota_{x}$ must be given by right-translation by a unique element $g_{x} \in \GE(X) = G(V)$, defining a map $\shA_{\gerbeE}(X) = \shA(V) \xrightarrow{\Phi_{X}} G(V) = \GE(X)$. Moreover, if $X \xrightarrow{\nu} Y$ is a morphism in $\gerbeE$, lifting $V \xrightarrow{f} U$, then the square 
\[
\begin{tikzcd}
\shA(U) \arrow["\Phi_{X}"]{r} & G(U) \\
\shA(V) \arrow["f^{\sharp}"]{u} \arrow["\Phi_{Y}"]{r} & G(V) \arrow["f^{\sharp}"]{u}
\end{tikzcd}
\]
commutes because of the commutativity of the squares 
\[
\begin{tikzcd}
\restr{\shA_{\gerbeE}}{\gerbeE/X} \times \restr{\mathscr{T}}{\gerbeE/X} \arrow["\iota_{X}"]{r} &  \restr{\mathscr{T}}{\gerbeE/X} & &  \restr{\mathscr{T}}{\gerbeE/X}  \times \restr{\GE}{\gerbeE/X} \arrow{r} &  \restr{\mathscr{T}}{\gerbeE/X} \\
\restr{\shA_{\gerbeE}}{\gerbeE/Y} \times \restr{\mathscr{T}}{\gerbeE/Y} \arrow["\nu^{\sharp}"]{u} \arrow["\iota_{Y}"]{r} & \restr{\mathscr{T}}{\gerbeE/Y} \arrow["\nu^{\sharp}"]{u} & &  \restr{\mathscr{T}}{\gerbeE/Y}  \times \restr{\GE}{\gerbeE/Y} \arrow{r}  \arrow["\nu^{\sharp}"]{u} & \restr{\mathscr{T}}{\gerbeE/Y}  \arrow["\nu^{\sharp}"]{u}.
\end{tikzcd}
\]
It follows that the above maps glue across all objects in $\gerbeE$ to define a homomorphism of group sheaves $\Phi_{\gerbeE} \colon \shA_{\gerbeE} \to G_{\gerbeE}$. This defines a homomorphism $\bar{\Phi} \colon \shA_{U_{0}} \to G_{U_{0}}$ via pulling back by a section $s \colon \text{Sch}/U_{0}\to \gerbeE$, and the above argument shows that the two pullbacks to $\shA_{U_{1}}$ coincide by setting $\nu = \varphi \colon p_{1}^{*}X \xrightarrow{\sim} p_{2}^{*}X$ any isomorphism in $\gerbeE(U_{1})$ for $X := s(U_{0})$, showing that $\bar{\Phi}$ descends to an $F$-homomorphism $\Phi$, whose pullback by $\pi$ is $\Phi_{\gerbeE}$.
\end{proof}

Let $\mathcal{A}$ be the category of monomorphisms $Z \to G$ defined over $F$, where $G$ is either a commutative algebraic group of finite type over $F$ or a connected reductive group defined over $F$, and $Z$ is a finite multiplicative group defined over $F$ (usually thought of as a subgroup of $G$) whose image in $G$ is central. We define the set of morphisms $\mathcal{A}(Z_{1} \to G_{1}, Z_{2} \to G_{2})$ to be the set of commutative diagrams
\[
\begin{tikzcd}
Z_{1}  \arrow{r} \arrow[swap]{d} & Z_{2}  \arrow{d}\\
G_{1}  \arrow{r} & G_{2},
\end{tikzcd}
\]
where the horizontal maps are morphisms of algebraic groups defined over $F$. Set $\mathcal{T} \subset \mathcal{A}$ (resp. $\mathcal{R} \subset \mathcal{A}$) to be the subcategory where $[Z \to G]$ belongs to $\mathcal{T}$ (resp. $\mathcal{R}$) if $G$ is a torus (resp. a connected reductive group).

For $G$ abelian with finite $F$-subgroup $Z$ we define the set $H^{1}(\gerbeE, Z \to G)$ to be the group of isomorphism classes of (fpqc) $\GE$-torsors on $\gerbeE$ such that $\text{Res}(\mathscr{T})$ factors through $Z$. For arbitrary $[Z \to G] \in \mathcal{A}$, we define $H^{1}(\gerbeE, Z \to G)$ to be the pointed set of all isomorphism classes of $\GE$-torsors on $\gerbeE$ such that the inertial action is induced by an $F$-homomorphism $\phi \colon \shA \to Z$; note that this agrees with our previous definition if $G$ is abelian, and define the set $H^{1}_{\text{bas}}(\gerbeE, G_{\gerbeE})$ to be $\varinjlim_{Z} H^{1}(\gerbeE, Z \to G)$, where the direct limit is over all finite central subgroups of $G$ (for arbitrary $G$). If $\mathscr{T}$ is a $\GE$-torsor whose isomorphism class lies in $H^{1}(\gerbeE, Z \to G)$, we say that $\mathscr{T}$ is \textit{$Z$-twisted}. The map $[Z \to G] \mapsto H^{1}(\gerbeE, Z \to G)$ defines a functor from $\mathcal{A}$ to the category of pointed sets (abelian groups if $G$ is abelian).

\subsection{Torsors on $\gerbeE_{a}$}
We use the same notation as in the previous subsections. Fix the fpqc cover $\Spec(\bar{F}):= U_{0} \to \Spec(F)$, and fix a group scheme $G$ of finite type over $F$, viewed as a sheaf on $(\text{Sch}/F)_{\tau} =(\text{Sch}/F)_{\text{fpqc}}$. We will not assume that $G$ is abelian. For $U \in \text{Ob}(\text{Sch}/F)$, recall that $\gerbeE \times_{\text{Sch}/F} (\text{Sch}/U)$ is denoted by $\gerbeE_{U}$. 

We say again for emphasis that in this subsection, our site will be the big fpqc site $(\text{Sch}/F)_{\tau} = (\text{Sch}/F)_{\text{fpqc}}$, not $(\text{Sch}/F)_{\text{fppf}}$.

We first record a characterization of sheaves on the site $B_{F}\shA \to (\text{Sch}/F)_{\tau}$ (with the induced $\tau$ topology). Consider the category of sheaves on $B_{F}\shA$, as well as the category of sheaves on $(\text{Sch}/F)_{\tau}$ equipped with an $\shA$-action, where we require morphisms in this latter category to be $\shA$-equivariant. There is a canonical section $s \colon \text{Sch}/F \to B_{F}\shA$ sending $U \to \Spec(F)$ to the trivial $\shA_{U}$-torsor $\shA_{U}$. Define the map between the above two categories to be the one which sends a sheaf $\mathscr{F}$ on $B_{F}\shA$ to the sheaf $s^{*}\mathscr{F}$ on $(\text{Sch}/F)_{\tau}$ with $\shA$-action given by $$\shA \times_{F} s^{*}\mathscr{F} \xrightarrow{s^{*}\iota} s^{*}\mathscr{F},$$ and sends the morphism of sheaves $\mathscr{F} \xrightarrow{f} \mathscr{F}'$ to $s^{*}f$, where in the definition of the action we are making the identification $s^{*}(\shA_{\gerbeE}) = \shA$.

\begin{prop}\label{classifyingsheaf} The above map defines an equivalence of categories. 
\end{prop}

\begin{proof} See \cite[Remark 2.6]{Shin}.
\end{proof}

\begin{Def} For our fixed $G$ and $\gerbeE \to (\text{Sch}/F)_{\text{fpqc}}$ a gerbe, define the fibered category $\textbf{Tors}(G,\gerbeE)$ over $(\text{Sch}/F)_{\text{fpqc}}$, where the fiber over $U \in \text{Ob}(\text{Sch}/F)$ is the category of (fpqc) $G_{\gerbeE_{U}}$-torsors on $\gerbeE_{U}$, with a morphism from $\mathscr{T}$ to $\mathscr{S}$ lying above $f: V \to U$ given by a morphism of $G_{\gerbeE_{V}}$-torsors $\mathscr{T} \to f^{*}\mathscr{S}$. Here $f^{*}\mathscr{S}$ denotes the pullback of the $G_{\gerbeE_{U}}$-torsor $\mathscr{S}$ to $\gerbeE_{V}$ via the morphism $\gerbeE_{V} := \gerbeE \times_{\text{Sch}/F} (\text{Sch}/V) \to \gerbeE \times_{\text{Sch}/F} (\text{Sch}/U) =: \gerbeE_{U}$ induced by the functor $\text{Sch}/V \to \text{Sch}/U$ sending $W \to V$ to $W \to V \xrightarrow{f} U$.

\end{Def}

\begin{prop} The fibered category $\textbf{Tors}(G,\gerbeE) \to (\text{Sch}/F)_{\text{fpqc}}$ is a stack.
\end{prop}

\begin{proof}
Our above construction is clearly a fibered category, and the remaining conditions, namely that the isomorphism functor associated to the fiber over $U \in \text{Ob}(\text{Sch}/F)$ is a sheaf and that all descent data from $(\text{Sch}/F)_{\text{fpqc}}$ are effective, follow from (respectively) gluing of morphisms of torsors and gluing of torsors on stacks over $(\text{Sch}/F)_{\text{fpqc}}$ with the induced fpqc topology, which follow easily from the discussion in \cite[\S 04TP]{Stacksproj} (with our stack being $\gerbeE \to (\text{Sch}/F)_{\text{fpqc}}$).
\end{proof}

We now introduce the category of $a$-twisted $G$-torsors on the site $(\text{Sch}/F)_{\text{fpqc}}$, corresponding to a \v{C}ech 2-cocycle $a \in \shA(U_{2})$, whose purpose is to give a concrete interpretation of the above stack in the case where $\gerbeE = \gerbeE_{a}$. This definition is a generalization of \cite[Definition 1.2.1]{Caldararu}.

\begin{Def}\label{twistedtorsordefinition} An \textit{$a$-twisted $G$-torsor over $F$} is a quadruple $(T, \psi, m, n)$ consisting of a $G_{U_{0}}$-torsor $m: T \times G_{U_{0}} \to T$ over $U_{0}$, an $\shA_{U_{0}}$-action $n: \shA_{U_{0}} \times_{U_{0}} T \to T$ which commutes with $m$, and an $\shA$-equivariant isomorphism of $G_{U_{1}}$-torsors $\psi: p_{2}^{*}T \to p_{1}^{*}T$ satisfying the \textit{twisted cocycle condition} 
$$(p_{12}^{*} \psi )\circ (p_{23}^{*}\psi) \circ (p_{13}^{*}\psi)^{-1}= n_{a}$$ 
on $q_{1}^{*}T$. We occasionally abbreviate the quadruple $(T,\psi,m,n)$ by $(T,\psi)$ (in such cases there will be no ambiguity regarding the associated actions). A morphism of $a$-twisted $G$-torsors on $(\text{Sch}/F)_{\text{fpqc}}$, $h: (T, \psi_{T}, m_{T}, n_{T}) \to (S,\psi_{S}, m_{S}, n_{S})$ is an $\shA$-equivariant morphism of $G_{U_{0}}$-torsors over $U_{0}$, $h: T \to S$, satisfying $\psi_{S} \circ p_{2}^{*}h = p_{1}^{*}h \circ \psi_{T}$. We get an associated fibered category over $\text{Sch}/F$, denoted by $\textbf{Tors}_{a}(G, \shA, F)$, by letting the fiber over $V$ be all \textit{$a_{V}$-twisted-torsors over $V$}, where $a_{V}$ is the image of $a$ in $\shA(V \times U_{2})$, defined identically as above after replacing $U_{0}, U_{1}$ by $V \times U_{0}$ and $V \times U_{1} = (V \times U_{0}) \times_{V} (V \times U_{0})$.
\end{Def}

The following lemma provides a different way to interpret some aspects of the above definition.

\begin{lem}\label{twisthom} Assume that $G$ is abelian. For a $G_{U_{0}}$-torsor $T$, having a $G_{U_{0}}$-equivariant $\shA_{U_{0}}$-action on $T$ is equivalent to requiring that the $\shA_{U_{0}}$-action be induced by a group homomorphism $\shA_{U_{0}} \to G_{U_{0}}$, and insisting further that there is a twisted gluing map giving $T$ (along with the two given actions) the structure of an $a$-twisted $G$-torsor implies that this homomorphism is defined over $F$. 
\end{lem}

\begin{proof}

For $V \to U_{0}$, if we fix $x \in \shA(V)$, then $n_{x} \colon T_{V} \xrightarrow{\sim} T_{V}$ is an automorphism of $G_{V}$-torsors, and is thus right-translation $m_{g_{x}}$  by some unique $g_{x} =: f(x) \in G(V)$, and the assignment $a \mapsto g_{x}$ is functorial in $V$ by uniqueness of $g_{x}$, and hence we get a group homomorphism $\shA_{U_{0}} \xrightarrow{f} G_{U_{0}}$ giving the $\shA$-action. 

To show that the homomorphism $f$ descends to a morphism $\shA \to G$, it is enough by fpqc descent of morphisms to show that $p_{1}^{*}f(x) = p_{2}^{*}f(x)$ for fixed $x \in \shA(V)$ with $V \to U_{1}$. Since $p_{1}^{*}f$ is induced by the $\shA_{U_{1}}$-action on $p_{1}^{*}T$, $p_{2}^{*}f$ by the $\shA_{U_{1}}$-action on $p_{2}^{*}T$, and we have an $\shA_{U_{1}}$-equivariant morphism of $G_{U_{1}}$-torsors $\psi \colon p_{2}^{*}T \xrightarrow{\sim} p_{1}^{*}T$, this means that since the diagram
\[
\begin{tikzcd}
 (p_{2}^{*}T)_{V} \arrow[r, "\psi"] \arrow[d, "m_{p_{2}^{*}f(x)}"] &  (p_{1}^{*}T)_{V} \arrow[d, "m_{p_{1}^{*}f(x)}"] \\
( p_{2}^{*}T)_{V} \arrow[r, "\psi"] &  (p_{1}^{*}T)_{V}  
\end{tikzcd}
\]
commutes and $\psi$ is $G_{U_{1}}$-equivariant, the right-hand translation $m_{p_{1}^{*}f(x)}$ equals $\psi \circ m_{p_{2}^{*}f(x)} \circ \psi^{-1} = m_{p_{2}^{*}f(x)}$, giving the result.
\end{proof}

\begin{prop} The fibered category $\textbf{Tors}_{a}(G, \shA, F) \to (\text{Sch}/F)_{\text{fpqc}}$ is a stack.
\end{prop}

\begin{proof}
The isomorphism functor on $V \in \text{Ob}(\text{Sch}/F)$ associated to the fiber category over $V$ is evidently a sheaf, by gluing of morphism of sheaves (again, see \cite[\S 04TP]{Stacksproj}), and if the equivariance conditions hold on an fpqc cover, they hold on $V$. Thus, all that remains to check is effectivity of descent data. This follows because of gluing of $G$-torsors on $(\text{Sch}/U_{0})_{\text{fpqc}}$ with the fpqc topology, and the $\shA$-action on compatible torsors defined on any cover $\{V_{i} \to V \}$ extends to a $\shA$-action of the glued torsor on $V$ by gluing of morphisms (using $\shA$-equivariance of morphisms in $\textbf{Tors}_{a}(G, \shA, F)$). Again, the commutation relations can be checked locally.
\end{proof}

The next fundamental result shows that the above two notions of torsors actually coincide. We begin with a lemma that addresses the case when $\gerbeE = B_{F}\shA$. 

\begin{lem}\label{classtors} There is an equivalence of categories over $\text{Sch}/F$ $$\eta: \textbf{Tors}(G,B_{F}\shA) \to \textbf{Tors}_{e_{\shA}}(G, \shA, F).$$
\end{lem}

\begin{proof} Set $\gerbeE := B_{F}\shA \xrightarrow{\pi} \text{Sch}/F$. If we start with the data of an object $(T,\psi)$ in $\textbf{Tors}_{e_{\shA}}(G, \shA, F)$, the map $\psi$ furnishes $T$ with a descent datum (of torsors, not just sheaves) with respect to the fpqc cover $U_{0} \to \Spec(F)$. By gluing of fpqc sheaves (see \cite[\S 04TP]{Stacksproj}) such an object then gives a $G$-torsor on $(\text{Sch}/F)_{\text{fpqc}}$ with $G$-equivariant $\shA$-action. By Proposition \ref{classifyingsheaf}, this defines a sheaf $\mathscr{T}$ on $B_{F}\shA$, so all we need to do is define the $\GE$-action, $\tilde{m}: \mathscr{T} \times \GE \to \mathscr{T}$. 

Denote by $s \colon \text{Sch}/F \to B_{F}\shA$ the canonical section. Denote by $\mathcal{C}$ the (categorical) image of this embedding of categories. We may define a morphism of sheaves 
\begin{equation}\label{gaction}
\restr{\mathscr{T}}{\mathcal{C}} \times \restr{\GE}{\mathcal{C}} \to \restr{\mathscr{T}}{\mathcal{C}}
\end{equation}
by applying $\pi^{p}$ (notation as in Stacksproject, 00WU) to the action $T \times G \to T$. 

For an arbitrary $\shA$-torsor over $V$, say $X$, we may find an fpqc cover $\{V_{i} \xrightarrow{f_{i}} V\}$ such that we have isomorphisms of $\shA_{V_{i}}$-torsors $h_{X_{i}} \colon X_{i} := f_{i}^{*}X \xrightarrow{\sim} \shA_{V_{i}}$, and, if $s_{X_{i}} \colon \text{Sch}/V_{i} \to \gerbeE$ denotes the embedding of categories induced by $X_{i}$, we can define an action
\begin{equation}\label{gaction2}
\restr{\mathscr{T}}{s_{X_{i}}(\text{Sch}/V_{i})} \times \restr{\GE}{s_{X_{i}}(\text{Sch}/V_{i})} \to \restr{\mathscr{T}}{s_{X_{i}}(\text{Sch}/V_{i})}
\end{equation}
by conjugating our above action \eqref{gaction} by $h_{X_{i}} \colon s_{X_{i}}(\text{Sch}/V_{i}) \xrightarrow{\sim} \mathcal{C}/\shA_{V_{i}}$ (where we identify $h_{X_{i}}$ with the induced equivalence between the embedded categories). To check that this glues to give a morphism of sheaves $\mathscr{T} \times \GE \to \mathscr{T}$, it's enough to check that the action defined in \eqref{gaction2} is independent of the choice of $h_{X_{i}}$, which is equivalent to showing that the action in \eqref{gaction} is equivariant under the inertial action. This follows because the $G$-action on $T$ is $\shA$-equivariant. 

We have thus constructed a map  $\textbf{Tors}_{e_{\shA}}(G, \shA, F) \to \textbf{Tors}(G,B_{F}\shA)$ which is the inverse of the map $\textbf{Tors}(G,B_{F}\shA) \to \textbf{Tors}_{e_{\shA}}(G, \shA, F)$ obtained by pulling back by the section $s$.
\end{proof}

\begin{prop}\label{gerbesheaf}
For $\gerbeE = \gerbeE_{a}$, there is a canonical equivalence of categories $\eta: \textbf{Tors}(G,\gerbeE) \to \textbf{Tors}_{a}(G, \shA, F)$.
\end{prop}

\begin{proof} The argument largely follows that in \cite[\S 2.1.3]{Lieblich} (where we replace the action via a character $\chi$ by the inertial action). Let $x: \text{Sch}/U_{0} \to \gerbeE$ be the section constructed in Lemma \ref{section}; let $X$ be the corresponding lift of $U_{0}$. This same lemma also tells us that the two pullbacks of $x$ to $U_{1}$, the maps $x_{1}$ and $x_{2}$, are isomorphic via $\varphi$; this means that for every $V \xrightarrow{f} U_{1}$, we have an isomorphism $\varphi_{V}: (p_{1} \circ f)^{*}X \to  (p_{2} \circ f)^{*}X$ in $\gerbeE(V)$.

Let $\mathscr{T} \in \textbf{Tors}(G,\gerbeE)(\Spec(F))$ (the argument is identical for a $G_{\gerbeE_{V}}$-torsor, $V \in \text{Ob}(\text{Sch}/F)$). Then define the $G$-torsor over $(\text{Sch}/U_{0})_{\text{fpqc}}$ to be $T:= x^{*}\mathscr{T}$ (sending $V \xrightarrow{f} U_{0}$ to $\mathscr{T}(f^{*}X)$). We know that $\shA_{\gerbeE}$ acts on $\mathscr{T}$ via the inertial action, denoted by $\iota: \shA_{\gerbeE} \times \mathscr{T} \to \mathscr{T}$. As such, we get an $\shA$-action on $T$ via taking $x^{*}\iota$ (using that $x^{*}\shA_{\gerbeE} = \shA$). Similarly, we can set $\psi$ to be the $U_{1}$-sheaf isomorphism $p_{2}^{*}x^{*}\mathscr{T} \to p_{1}^{*}x^{*}\mathscr{T}$ induced by the natural transformation $\varphi: x \circ p_{1} \xrightarrow{\sim} x \circ p_{2}$. One sees that $\psi$ satisfies the twisted cocycle condition, since the map from $(q_{1}^{*}x)(\text{Sch}/U_{2})$ to itself given by the natural transformation of $q_{1}^{*}x$: $$d\varphi = (p_{13}^{*}\varphi)^{-1} \circ (p_{23}^{*} \varphi) \circ (p_{12}^{*}\varphi)$$ equals $\iota_{a}$, so that the induced map $q_{1}^{*}T \to q_{1}^{*}T$ is exactly translation by $a$.  Note that $\psi$ is $\shA$-equivariant for our $\shA$-action, since for $z \in \shA(U_{0})$, we can identify $z$ with $\theta_{p_{1}^{*}X}(z), \theta_{p_{2}^{*}X}(z) \in \text{Aut}_{U_{1}}(p_{1}^{*}X),$ $\text{Aut}_{U_{1}}(p_{2}^{*}X)$, and then $\varphi _{U_{0}} \circ  \theta_{p_{2}^{*}X}(z) =  \theta_{p_{1}^{*}X}(z) \circ \varphi_{U_{0}}$, as $ \theta_{p_{1}^{*}X}(z)= \varphi_{U_{0}} \circ  \theta_{p_{2}^{*}X}(z) \circ \varphi_{U_{0}}^{-1}$ (by Lemma \ref{bandlemma}).

We take $m: T \times G_{U_{0}} \to T$ to be the pullback of the $\GE$-action $\tilde{m}$ on $\mathscr{T}$ by $x$. Fixing $V \xrightarrow{f} U_{0}$, since $\tilde{m}: \mathscr{T} \times \GE \to \mathscr{T}$ is a morphism of sheaves on $\gerbeE$, it commutes with the restriction maps $\varphi_{V}^{\sharp}$, giving  the $G$-equivariance of $\psi$. One checks via an identical argument that $m$ commutes with the $\shA_{U_{0}}$-action (since it acts via the band of $\gerbeE$), and that if $\mathscr{T} \to \mathscr{S}$ is a morphism in $\textbf{Tors}(G,\gerbeE)(U_{0})$, the induced maps $\mathscr{T}(f^{*}X) \to \mathscr{S}(f^{*}X)$ give a morphism in $\textbf{Tors}_{a}(G, \shA, F)(U_{0})$. We thus obtain our functor $\eta$ (after applying the above construction with $U_{0}$ replaced by an arbitrary $V \to U_{0}$, which proceeds identically as above).

Since both $\textbf{Tors}(G,\gerbeE)$ and $\textbf{Tors}_{a}(G, \shA, F)$ are stacks over $(\text{Sch}/F)_{\text{fpqc}}$, it's enough to check that $\eta$ is locally an equivalence, by Proposition \ref{localequivalence} (where we are using that we are working with the fpqc sites). By base-changing to $U_{0}$, we may assume that $a$ is a $1$-coboundary; one checks easily (using an argument similar to the one used in Construction \ref{changeofgerbe}) that if $a$ is cohomologous to $b$, then $\textbf{Tors}_{a}(G, \shA, F)$ and $\textbf{Tors}_{b}(G, \shA, F)$ are equivalent, and we know from Construction \ref{changeofgerbe} that $\gerbeE_{b}$ and $\gerbeE_{a}$ are isomorphic (not just equivalent). Hence, we may assume that $a = e_{\shA}$, and $\gerbeE = B_{F}\shA$, and now we may apply Lemma \ref{classtors}.
\end{proof}

The following two results follow immediately from the above proof, pulling back functors between the categories $\textbf{Tors}(G,\gerbeE)$ (with varying $G$ and/or $\gerbeE$) by the section $x$:

\begin{cor}\label{functwist} Let $G \xrightarrow{f} H$ be a morphism of finite type $F$-group schemes, giving the usual functor $$\textbf{Tors}(G,\gerbeE_{a}) \to \textbf{Tors}(H,\gerbeE_{a}),$$ which sends $\mathscr{T}$ to $\mathscr{T} \times^{\GE, f_{\gerbeE}} H_{\gerbeE}$. Then this corresponds via the equivalence $\eta$ to the functor $$\textbf{Tors}_{a}(G, \shA, F) \to  \textbf{Tors}_{a}(H, \shA, F)$$ sending $(T, \psi, m, n)$ to the $H_{U_{0}}$-torsor $T \times^{G,f} H$, with $\shA$-action induced by $n \times \text{id}$; when $G$ is abelian, this is the same as replacing  the homomorphism $\shA \to G$ giving $n$ with its post-composition by $f$. The new gluing map $\tilde{\psi}$ is obtained by applying $- \times^{G,f} H$ and taking the morphism induced by $\psi \times \text{id}$.

\end{cor}

\begin{cor}\label{pulltwist} Let $\phi_{a,b,x} \colon \gerbeE_{a} \to \gerbeE_{b}$ be the morphism of stacks over $F$ defined in Construction \ref{changeofgerbe} between the $\shA$-gerbe $\gerbeE_{a}$ corresponding to the \v{C}ech 2-cocycle $a \in \shA(U_{2})$, the $\textbf{B}$-gerbe $\gerbeE_{b}$, corresponding to the \v{C}ech 2-cocycle $b \in \textbf{B}(U_{2})$, induced by a homomorphism $\shA \xrightarrow{h} \textbf{B}$ such that $[h(a)] = [b] \in H^2(F, \textbf{B})$ and $x \in \textbf{B}(U_{1})$ such that $d(x) \cdot b = h(a)$. Then the functor $$\textbf{Tors}(G,\gerbeE_{b}) \to \textbf{Tors}(G, \gerbeE_{a})$$ induced by pullback by $\phi_{a,b,x}$ corresponds via $\eta$ to the functor $$\textbf{Tors}_{b}(G, \textbf{B}, F) \to  \textbf{Tors}_{a}(G, \shA, F)$$ sending the object $(T, \psi, m, n)$ to the $a$-twisted $G$-torsor with underlying $G_{U_{0}}$-torsor $T$, $\shA$-action given by mapping to $\textbf{B}$ by $h$, and gluing map $\tilde{\psi}$ given by translating $\psi$ by $x$.

\end{cor}

\subsection{Twisted cocycles}
In this section, we introduce the notion of twisted cocycles, which facilitate computations involving torsors on gerbes. We continue with the notation and conventions of \S 2.4. Fix a \v{C}ech 2-cocycle $a \in \shA(U_{2})$.

\begin{Def}\label{twistedcocycle} An \textit{$a$-twisted \v{C}ech 1-cocycle valued in $G$} (or just an \textit{$a$-twisted cocycle} if $G$ is understood) is a pair $(x, \phi)$, where $\phi \colon \shA \to Z(G)$ is an $F$-homomorphism and $x \in G(U_{1})$ satisfies $dx = \phi(a)$. We say that $(x, \phi)$ and $(y, \phi')$ are \textit{equivalent} if $\phi = \phi'$ and there exists $z \in G(\bar{F})$ such that $p_{1}^{\sharp}(z)^{-1} y p_{2}^{\sharp}(z) = x$ (in this case, we say that $z$ \textit{realizes the equivalence} of $(x, \phi)$ and $(y, \phi')$). This clearly defines an equivalence relation. We denote the set of all $a$-twisted cocycles by $Z^{1}(\gerbeE_{a}, G_{\gerbeE_{a}})$, and the set of all equivalence classes by $H^{1,*}(\gerbeE_{a}, G_{\gerbeE_{a}})$. For some fixed finite central $Z$ in $G$, we say that an $a$-twisted cocycle $(x,\phi)$ is an \textit{$a$-twisted $Z$-cocycle} if $\phi$ factors through $Z$. We denote the set of all $a$-twisted $Z$-cocycles of $G$ for a fixed $Z$ by $Z^{1}(\gerbeE_{a}, Z \to G)$. If $(x, \phi)$ is in $Z^{1}(\gerbeE_{a}, Z \to G)$, then evidently its whole equivalence class is as well. Denote the set of equivalence classes of $a$-twisted $Z$-cocycles by $H^{1,*}(\gerbeE_{a}, Z \to G)$, and the set $\varinjlim_{Z} H^{1,*}(\gerbeE_{a}, Z \to G)$ by $H^{1,*}_{\text{bas}}(\gerbeE_{a}, G_{\gerbeE_{a}})$, where the direct limit is over all finite central $F$-subgroups.
\end{Def}

We get the following expected result:

\begin{prop}\label{correspondence1} For $G$ a finite-type $F$-group and $Z$ a finite central $F$-subgroup, we have a canonical bijection from $H^{1}(\gerbeE_{a}, Z \to G)$ to $H^{1,*}(\gerbeE_{a}, Z \to G)$ which is functorial in $[Z \to G]$. Taking direct limits, this induces a canonical bijection $H^{1}_{\text{bas}}(\gerbeE_{a}, G_{\gerbeE_{a}}) \to H^{1,*}_{\text{bas}}(\gerbeE_{a}, G_{\gerbeE_{a}})$, functorial in the group $G$. If $G$ is abelian, we also have a canonical isomorphism of abelian groups $H^{1}(\gerbeE_{a}, G_{\gerbeE_{a}}) \to H^{1,*}(\gerbeE_{a}, G_{\gerbeE_{a}})$.
\end{prop}

\begin{proof} Let $\mathscr{T}$ be a $Z$-twisted $\GE$-torsor. Set $\phi := \text{Res}(\mathscr{T})$; it remains to construct the appropriate $x \in G(U_{1})$.  Let $X := s(\Spec(\bar{F}))$, where $s$ denotes the canonical section $\text{Sch}/\bar{F} \to \gerbeE_{a}$ constructed in Lemma \ref{section}, and let $\varphi$ the isomorphism $p_{1}^{*}X \to p_{2}^{*}X$ in $\gerbeE_{a}(U_{1})$ from the same Lemma. Setting $T := s^{*}\mathscr{T}$ gives a $G_{\bar{F}}$-torsor---choose a $\bar{F}$-trivialization $h$ of $T$. Taking $p_{1}^{*}h \circ \varphi^{\sharp} \circ p_{2}^{*}h^{-1}$ defines an automorphism $G_{U_{1}} \to G_{U_{1}}$ which is given by left-translation by a unique element $x \in G(U_{1})$, and this $x$ satisfies $dx = \phi(a)$, as desired (we leave the details to the reader, cf. the proof of Proposition \ref{gerbesheaf}). Note that choosing a different $h$ gives an equivalent twisted cocycle. 

Moreover, given any isomorphism $\Psi \colon \mathscr{T}_{1} \to \mathscr{T}_{2}$, fixing trivializations $h_{i} \colon T_{i} \to G_{\bar{F}}$ as above gives the isomorphism $h_{1} \circ s^{*}\Psi \circ h_{2}^{-1} \colon G_{\bar{F}} \xrightarrow{\sim} G_{\bar{F}}$, which is left-translation by a unique $y \in G(\bar{F})$, which realizes the equivalence between the twisted cocycles obtained using $h_{1}$ and $h_{2}$. Thus, we have a canonical well-defined map $H^{1}(\gerbeE_{a}, Z \to G) \to H^{1,*}(\gerbeE_{a}, Z \to G)$. The fact that this is a bijection is immediate from Proposition \ref{gerbesheaf}. Functoriality in $[Z \to G] \in \mathcal{A}$ is trivial. The proof of the last statement follows by replacing $Z$ by $G$ in the above argument for abelian $G$. 
\end{proof}

We thus get a concrete interpretation of $H^{1}(\gerbeE_{a}, Z \to G)$ and $H^{1}_{\text{bas}}(\gerbeE_{a}, G_{\gerbeE_{a}})$; in light of the above results, we denote $H^{1,*}(\gerbeE_{a}, Z \to G)$ simply by $H^{1}(\gerbeE_{a}, Z \to G)$ and $H^{1,*}(\gerbeE_{a}, G_{\gerbeE_{a}})$ by $H^{1}(\gerbeE_{a}, G_{\gerbeE_{a}})$ (we make this latter identification only for abelian $G$)---the above identifications are implicit in this notation. 

To extend this to an arbitrary $\shA$-gerbe $\gerbeE$ split over $\bar{F}$, we need the following result:

\begin{prop}\label{trivialize}
Let $(\gerbeE, \theta)$ be an \textit{arbitrary} $\shA$-gerbe and $a \in \shA(U_{2})$ such that $[a] = [\gerbeE]$ in $\check{H}^{2}(\bar{F}/F, \shA)$. If $\check{H}^{1}(\bar{F}/F, \shA) = 0$, we have a canonical functorial bijection between $H^{1}(\gerbeE, \GE)$ and $H^{1}(\gerbeE_{a}, G_{\gerbeE_{a}})$. 
\end{prop}

\begin{proof} We have an equivalence of $\shA$-gerbes $\eta_{a} \colon \gerbeE \to \gerbeE_{a}$ for $a \in \shA(U_{2})$ representing $[\gerbeE] \in \check{H}^{2}(\bar{F}/F, \shA)$. This means we have a quasi-inverse $\nu_{a} \colon \gerbeE_{a} \to \gerbeE$ of $\shA$-gerbes, so that pullback by $\nu_{a}$ and $\eta_{a}$ induce the claimed bijection; if $\Psi$ is the natural isomorphism $\nu_{a} \circ \eta_{a} \xrightarrow{\sim} \text{id}_{\gerbeE}$, then $\Psi^{\sharp}$ gives an isomorphism from $\eta_{a}^{*}(\nu_{a}^{*}\mathscr{T})$ to $\mathscr{T}$.  To check that the above map is independent of the choice of $\nu_{a}$, it's enough to show that if $\eta \colon \gerbeE \to \gerbeE$ is an auto-equivalence of $\shA$-gerbes, then the induced map $H^{1}(\gerbeE, \GE) \to H^{1}(\gerbeE, \GE)$ is the identity.  This is the content of the following lemma, which will be useful later. 
\end{proof}

\begin{lem}\label{uniqueness} If $\shA$ is such that $\check{H}^{1}(\bar{F}/F, \shA) = 0$, then for any $\shA$-gerbe $\gerbeE$ split over $\bar{F}$ and $\shA$-equivalence $\eta \colon \gerbeE \to \gerbeE$, the induced map $\eta^{*} \colon H^{1}(\gerbeE, \GE) \to H^{1}(\gerbeE, \GE)$ is the identity for any $F$-group scheme $G$.
\end{lem}

\begin{proof}  The first step is to extract a \v{C}ech 1-cocycle from $\eta$. Let $Y \in \gerbeE(U)$; note that for any morphism $f \colon V \to U$, we have a unique isomorphism $\phi_{f}$ making the diagram 
\[
\begin{tikzcd}
\eta(f^{*}Y) \arrow{dr} \arrow["\phi_{f}"]{r} & f^{*}\eta(Y) \arrow{d} \\
& \eta(Y)
\end{tikzcd}
\]
commute. This means, for an object $X \in \gerbeE(\Spec(\bar{F}))$, we have canonical identifications $p_{k}^{*}\eta(X) \xrightarrow{\sim} \eta(p_{k}^{*}X)$, $p_{ij}^{*}\eta(p_{k}^{*}X) \xrightarrow{\sim} \eta(p_{ij}^{*}p_{k}^{*}X)$, and (combining the previous two) $p_{ij}^{*}p_{k}^{*}\eta(X) \xrightarrow{\sim} \eta(p_{ij}^{*}p_{k}^{*}X)$ for all $1 \leq i, j \leq 3$, $1 \leq k \leq 2$. We make these identifications without comment in what follows. 

Picking an isomorphism $\varphi \colon p_{1}^{*}X \xrightarrow{\sim} p_{2}^{*}X$ in $\gerbeE(U_{1})$, these identifications allow us to view the isomorphism $\eta(\varphi)$ as an isomorphism from $p_{1}^{*}\eta(X)$ to $p_{2}^{*}\eta(X)$. Choosing an isomorphism $h \colon X \xrightarrow{\sim} \eta(X)$ in $\gerbeE(\Spec(\bar{F}))$ (possible because of Convention \ref{conventions}), the map $[p_{1}^{*}h^{-1} \circ \eta(\varphi)^{-1} \circ p_{2}^{*}h] \circ \varphi$ lies in $\text{Aut}_{U_{1}}(p_{1}^{*}X)$ and thus (via $\theta_{p_{1}^{*}X}^{-1}$) gives an element $x \in \shA(U_{1})$. We claim that $x$ is a \v{C}ech 1-cocycle. This follows from repeated use of Lemma \ref{bandlemma} and the fact that, on $q_{1}^{*}X$, we may use the above identifications and the fact that $\eta$ is the identity on bands to deduce that $ \theta_{q_{1}^{*}\eta(X)}^{-1}(d \eta(\varphi)) =   \theta_{q_{1}^{*}X}^{-1}(d \varphi) \in \shA(U_{2})$. It is important to note that the 1-cocycle $x$ does \textit{not} depend on the choice of $\varphi$, since if $\varphi'$ is obtained by precomposing $\varphi$ by an automorphism $y$ of $p_{1}^{*}X$, then the extra $y$ cancels out, again using Lemma \ref{bandlemma} and the fact that $\eta$ is the identity on bands.

With this in hand, since we assume that $\check{H}^{1}(\bar{F}/F, \shA) = 0$, we get that $x = dy$ for some $y \in \shA(\bar{F})$. We will show that any $G_{\gerbeE}$-torsor $\mathscr{T}$ is isomorphic to $\eta^{*}\mathscr{T}$, which gives the result. It's enough to construct a 2-isomorphism $\mu \colon \text{id}_{\gerbeE} \xrightarrow{\sim} \eta$, since then $\mu^{\sharp}$ will give the desired isomorphism of $\GE$-torsors (for any choice of $G$). This will just consist of a compatible system of isomorphisms $X \xrightarrow{\mu_{X}} \eta(X)$ in $\gerbeE(U)$ for every $X \in \gerbeE(U)$. The argument will be similar to the proof of Lemma \ref{classtors}; we will first construct such a system of isomorphisms on $\gerbeE_{\bar{F}}$, which we will descend to a system of isomorphisms on $\gerbeE$ using the fact that $\gerbeE \to (\text{Sch}/F)_{\text{fpqc}}$ is a stack.

We first define this system of isomorphisms on the embedded subcategory $\mathcal{C} := s(\text{Sch}/\bar{F}) \subset \gerbeE$, where $s$ is the section induced by $X$. For $f^{*}X \in \gerbeE(V)$, $\mu_{f^{*}X}$ is given by $f^{*}h$ post-composed with $\theta_{f^{*}\eta(X)}(y_{V}^{-1}) \in \text{Aut}_{V}(f^{*}\eta(X))$. It is a straightforward exercise to verify that for any object $Z \in \gerbeE(W \xrightarrow{g} \Spec(\bar{F}))$ such that we have a (non-canonical) isomorphism $Z \xrightarrow{\lambda, \sim} g^{*}X$ in $\gerbeE(W)$, the isomorphism $Z \to \eta(Z)$ in $\gerbeE(W)$ given by $\eta(\lambda^{-1}) \circ \mu_{g^{*}X} \circ \lambda$ is independent of the choice of $\lambda$, and so we set $\mu_{Z} := \eta(\lambda^{-1}) \circ \mu_{g^{*}X} \circ \lambda$. By taking common refinements of fpqc covers (since $\gerbeE \to (\text{Sch}/F)_{\text{fpqc}}$ is a gerbe), this implies that $\mu_{X}$ induces a natural isomorphism $\restr{\text{id}}{\gerbeE_{\bar{F}}} \xrightarrow{\bar{\mu}} \restr{\eta}{\gerbeE_{\bar{F}}}$.

To show that $\bar{\mu}$ descends to $\gerbeE$, we need to show (by gluing of morphisms) that $p_{1}^{*}(\bar{\mu}) = p_{2}^{*}(\bar{\mu})$ on $\gerbeE_{U_{1}}$. Let $Y \in \gerbeE(V \xrightarrow{f} U_{1})$; there is an fpqc cover $\{V_{i} \xrightarrow{f_{i}} V \}$ such that we have isomorphisms $f_{i}^{*}Y \xrightarrow{\Psi_{i,1}} f_{i}^{*}f^{*}p_{1}^{*}X$ in $\gerbeE(V_{i})$, as well as isomorphisms $\{\Psi_{i,2}\}$ defined analogously.  For each $i$, we have the following diagram
\[
\begin{tikzcd}
f_{i}^{*}Y \arrow["\Psi_{i,1}"]{r} \arrow["\Psi_{i,2}"]{rd} & f_{i}^{*}f^{*}p_{1}^{*}X \arrow["\Psi_{i,1,2}"]{d} \arrow["f_{i}^{*}f^{*}p_{1}^{*}h"]{r} & f_{i}^{*}f^{*}p_{1}^{*}\eta(X) \arrow["\eta(\Psi_{i,1,2})"]{d} \arrow["\eta(\Psi_{i,1})^{-1}"]{r} & f_{i}^{*}\eta(Y) \\
& f_{i}^{*}f^{*}p_{2}^{*}X \arrow["f_{i}^{*}f^{*}p_{2}^{*}h"]{r}  & f_{i}^{*}f^{*}p_{2}^{*}\eta(X), \arrow[swap, "\eta(\Psi_{i,2})^{-1}"]{ur} &
\end{tikzcd}
\]
where we have made the canonical identifications mentioned at the beginning of the proof in several places and $\Psi_{i,1,2} := \Psi_{i,2} \circ \Psi_{i,1}^{-1}$. The diagram does not commute because of the middle square. Indeed, starting at the top-left corner, going right then down then left yields $ f_{i}^{*}f^{*}p_{2}^{*}h^{-1} \circ \eta(\Psi_{i,1,2}) \circ f_{i}^{*}f^{*}p_{1}^{*}h = \theta_{f_{i}^{*}Y}(x_{V_{i}}^{-1}) \circ \Psi_{i,1,2}$, where $x_{V_{i}}$ denotes the image of $x \in \shA(U_{1}) \to \shA(V_{i} \xrightarrow{f \circ f_{i}} U_{1})$ (using that $x$ does not depend on the choice of $\varphi$, see beginning of the proof). But now replacing $f_{i}^{*}f^{*}p_{k}^{*}h$ with $\bar{\mu}_{f_{i}^{*}f^{*}p_{k}^{*}X}$ for $k=1,2$ serves to replace the above composition with $$\theta_{f_{i}^{*}Y}(p_{1,V_{i}}^{\sharp}(y)^{-1}) \circ \theta_{f_{i}^{*}Y}(p_{2,V_{i}}^{\sharp}(y)) \circ \theta_{f_{i}^{*}Y}(x_{V_{i}}^{-1}) \circ \Psi_{i,1,2} = \theta_{f_{i}^{*}Y}(d(y_{V_{i}}) \cdot x_{V_{i}}^{-1}) \circ \Psi_{i,1,2} = \Psi_{i,1,2},$$ where $p_{k, V_{i}}$ for $k=1,2$ denotes the map $V_{i} \xrightarrow{p_{k} \circ f \circ f_{i}} \Spec(\bar{F})$, since $dy = x$ by construction. This gives the main result, since if $(p_{1}^{*}\bar{\mu})_{Y}$ and $(p_{2}^{*}\bar{\mu})_{Y}$ coincide on an fpqc cover of $Y$, they coincide on $Y$ as well.
\end{proof}

One consequence is:

\begin{cor} For $\mathcal{E}$, $G$ as above, $H^{1}(\mathcal{E}, G_{\mathcal{E}})$ is a set (cf. Convention \ref{settheoreticconv}).
\end{cor}

\begin{proof} By the above result, it's enough to prove the statement for $\mathcal{E} = \mathcal{E}_{a}$, where it follows from Proposition \ref{gerbesheaf} and the fact that isomorphism classes of $a$-twisted $G$-torsors over $F$ evidently form a set. 
\end{proof}

The above proof also gives a useful blueprint for constructing isomorphisms of $\GE$-torsors. We give one application here, using it to explain how to explicitly construct an isomorphism of $G_{\gerbeE_{a}}$-torsors $\mathscr{T}_{1} \to \mathscr{T}_{2}$ given an equivalence between their corresponding (class of) $a$-twisted cocycles $(x_{1}, \phi)$, $(x_{2}, \phi)$ coming from trivializations $h_{1}, h_{2}$, realized by the element $y \in G(\bar{F})$. Namely, we first define the map $\restr{\mathscr{T}_{1}}{\mathcal{C}} \to \restr{\mathscr{T}_{2}}{\mathcal{C}}$ on the category $\mathcal{C} := s(\text{Sch}/\bar{F})$ by taking $h_{2}^{-1} \circ m_{y} \circ h_{1}$, and then extend this to all of $\gerbeE_{\bar{F}}$ by conjugating by fpqc-local isomorphisms to objects in $\mathcal{C}$ (as in the above proof, cf. also the proof of Lemma \ref{classtors}). The fact that $dy \in G(U_{1})$ is 1-cocycle implies that this isomorphism descends to an isomorphism of $G_{\gerbeE_{a}}$-torsors $\mathscr{T}_{1} \to \mathscr{T}_{2}$.

The punchline of this entire subsection is the following result:

\begin{cor}\label{punchline} For any $\shA$-gerbe $\gerbeE$ split over $\bar{F}$ and $a \in \shA(U_{2})$ with $[a] = [\gerbeE] \in \check{H}^{2}(\bar{F}/F, \shA)$, if $\check{H}^{1}(\bar{F}/F, \shA) = 0$, then we have a canonical functorial bijection between $H^{1}(\gerbeE, Z \to G)$ and $H^{1,*}(\gerbeE_{a}, Z \to G)$ for any $[Z \to G]$ in $\mathcal{A}$. 
\end{cor}

\subsection{Inflation-restriction}

We continue with the notation of the previous sections; in particular, $\gerbeE \xrightarrow{\pi} (\text{Sch}/F)_{\text{fpqc}}$ is a fixed $\shA$-gerbe split over $\bar{F}$, and we will now assume that $\check{H}^{1}(\bar{F}/F, \shA) = 0$. In this section, we discuss the analogue of the inflation-restriction exact sequence in the setting of gerbes. Again $G$ will be a fixed finite type $F$-group scheme; recall that when $G$ is abelian, this implies that $H^{2}_{\text{fppf}}(F,G) = \check{H}^{2}_{\text{fppf}}(F,G) = \check{H}^{2}(\bar{F}/F, G)$, see \S 2.2. Our goal is to define a functorial ``inflation-restriction" sequence for any $[Z \to G] \in \mathcal{A}$:
\[
\begin{tikzcd}
  0 \arrow[r] & H^{1}(F, G)  \arrow[r, "\text{Inf}"] & H^{1}(\gerbeE, Z \to G)  \arrow[r, "\text{Res}"] & \Hom_{F}(\shA, Z)   \arrow[r, "tg"] & H^{2}(F, G),
\end{tikzcd}
\]
where the $H^{2}$-term is to be ignored if $G$ is non-abelian. In order to define this sequence, we may assume that $\gerbeE = \gerbeE_{a}$ for some $a \in \shA(U_{2})$, due to Corollary \ref{punchline}, and take $H^{1}(\gerbeE_{a}, Z \to G)$ to be equivalence classes of $a$-twisted $Z$-cocycles valued in $G$. This makes computations significantly simpler.

We take the first map, called \textit{inflation}, to be the one induced by sending the 1-cocycle $x \in G(U_{1})$ to $(x, 0) \in Z^{1}(\gerbeE_{a}, G_{\gerbeE_{a}})$, we take the second map, called \textit{restriction}, to be the one that sends the $a$-twisted cocycle $(a, \phi)$ to $\phi$, and we take the third map, called \textit{transgression} to be the one that sends $\phi \in \Hom_{F}(\shA, Z)$ to $[\phi(a)] \in \check{H}^{2}(\bar{F}/F, G)$; we could just as well take the image of tg to be $[\phi(a)] \in \check{H}^{2}(\bar{F}/F, Z)$ as well, but we use $\check{H}^{2}(\bar{F}/F, G)$ to match more closely with \cite{Tasho}. We leave it to the reader to check that these maps are well-defined.  

\begin{prop} The image of the class $[\mathscr{T}] \in H^{1}(\gerbeE, Z \to G)$ under the restriction map defined above equals the unique $F$-homomorphism $\shA \to Z$ inducing the inertial action on $\mathscr{T}$ (see Lemma \ref{altres}).
\end{prop}

\begin{proof} We leave this as an exercise, using the proof of Proposition \ref{gerbesheaf} for the case $\gerbeE = \gerbeE_{a}$.
\end{proof}

\begin{prop}\label{infres} The above maps define a functorial exact sequence of pointed sets (groups if $G$ is abelian, where the $H^{2}$ term is to be ignored if $G$ is non-abelian).
\end{prop}

\begin{proof} Clearly the image of the first set is contained in the fiber over identity of the second map. Conversely, if we have some $(x, \phi)$ with $\phi = 0$, then the twisted cocycle condition on $x \in G(U_{1})$ is just the usual cocycle condition, and hence $[x] \in H^{1}(F, G)$ maps to $(x, 0)$. Taking a twisted cocycle $(x, \phi)$ already gives an element $x \in G(U_{1})$ such that $dx = \phi(a)$, so that evidently $[\phi(a)] = 0$ in $\check{H}^{2}(\bar{F}/F, G)$. Finally, if $\phi \in \Hom_{F}(\shA, Z)$ is such that $\phi(a) = dx$ for $x \in G(U_{1})$, then $(x, \phi)$ defines a twisted cocycle, completing the proof. We leave functoriality in $G$ as an exercise. 
\end{proof}

For $[Z \to G]$ in $\mathcal{A}$, denote by $G \xrightarrow{\pi} \overline{G}$ the quotient of $G$ by $Z$. The following version of the long exact sequence in fpqc cohomology will be useful later:

\begin{lem}\label{LES} For $[Z \to G] \in \mathcal{A}$ we have an exact sequence of pointed sets (abelian groups if $G$ is abelian):
\[
\begin{tikzcd}
\overline{G}(F) \arrow{r} & H^{1}(\gerbeE, Z_{\gerbeE}) \arrow{r} & H^{1}(\gerbeE, \GE) \arrow{r} & H^{1}(\gerbeE, \overline{G}_{\gerbeE})
\end{tikzcd}
\]
\end{lem}

\begin{proof} Again, we may work with $a$-twisted cocycles. The first map is defined to be the composition $\overline{G}(F) \xrightarrow{\delta} H^{1}(F, Z) \xrightarrow{\text{Inf}} H^{1}(\gerbeE_{a}, Z_{\gerbeE_{a}})$ from the short exact sequence of fppf group schemes associated to $Z \to G$, and the second and third maps come from functoriality. The first map lands in the kernel of the second because the composition of the first two maps may be factored as $\overline{G}(\bar{F}) \xrightarrow{\delta} H^{1}(F, Z) \to  H^{1}(F, G) \xrightarrow{\text{Inf}} H^{1}(\gerbeE_{a}, G_{\gerbeE_{a}})$. Moreover, if $(x, \phi) \in Z^{1}(\gerbeE_{a}, Z_{\gerbeE_{a}})$ has trivial image in $H^{1}(\gerbeE_{a}, G_{\gerbeE_{a}})$, then $\phi=0$ and hence $(x, \phi)$ lies in the image of the inflation map $H^{1}(F, Z) \to H^{1}(\gerbeE_{a}, Z_{\gerbeE_{a}})$, and again the composition $H^{1}(F, Z) \to H^{1}(\gerbeE_{a}, Z_{\gerbeE_{a}}) \to H^{1}(\gerbeE_{a}, G_{\gerbeE_{a}})$ factors as $H^{1}(F, Z) \to H^{1}(F, G) \hookrightarrow H^{1}(\gerbeE_{a}, G_{\gerbeE_{a}})$, giving the other containment.

For exactness at the second spot, if $(x, \phi)$ is such that $[\pi(x, \phi)] = 0$, then $\pi \circ \phi = 0$, and so by basic properties of quotients, this happens if and only if $\phi$ factors through $Z$. Given this, the class maps to the identity if and only if (using centrality) we have that $x \in G(U_{1})$ is such that there is some $z \in Z(U_{1})$ with $x = p_{1}^{\sharp}(g)^{-1} z p_{2}^{\sharp}(g)$, and this $z$ necessarily satisfies $dz = dx = \phi(a)$ (since it's cohomologous to $x$). This holds if and only if $[(x, \phi)] = [(z, \phi)]$ in $H^{1}(\gerbeE_{a}, G_{\gerbeE_{a}})$, and $(z, \phi) \in Z^{1}(\gerbeE_{a}, Z_{\gerbeE_{a}})$, as desired.
\end{proof}

One checks easily (using Construction \ref{changeofgerbe}) that if $\gerbeE$ is an $\shA$-gerbe split over $\bar{F}$ and $\gerbeE'$ is a $\textbf{B}$-gerbe split over $\bar{F}$, and we have a morphism $\nu \colon \gerbeE \to \gerbeE'$ of categories over $\text{Sch}/F$ inducing the map $f \in \Hom_{F}(\shA, \textbf{B})$, then  the following diagram also commutes for any finite type $G$ (ignoring $H^{2}$ if $G$ is non-abelian):
\[
\begin{tikzcd}
0 \arrow{r}  & H^{1}(F, G) \arrow{r} & H^{1}(\gerbeE, \GE) \arrow{r} & \Hom_{F}(\shA, Z(G)) \arrow{r} & H^{2}(F,G) \arrow[equals]{d} \\
0 \arrow{r}  & H^{1}(F, G) \arrow[equals]{u} \arrow{r} & H^{1}(\gerbeE', G_{\gerbeE'}) \arrow["\nu^{*}"]{u} \arrow{r} & \Hom_{F}(\textbf{B}, Z(G)) \arrow{r} \arrow["f^{*}"]{u} & H^{2}(F,G).
\end{tikzcd}
\]

\subsection{Addendum: inverse limits of gerbes}
In this section we present a few elementary results concerning inverse limits of gerbes. We keep all of the previous notation of \S 2. The new assumptions of this subsection are as follows: We have a system $\{u_{n}\}_{n \in \mathbb{N}}$ of finite-type commutative affine groups over $F$ with transition maps $p_{n+1,n} \colon u_{n+1} \to u_{n}$ (defined over $F$) which are epimorphisms. We also assume that we have systems of elements $\{a_{n} \in u_{n}(U_{2})\}$ and $\{x_{n} \in u_{n}(U_{1}) \}$ such that $a_{n}$ are \v{C}ech 2-cocycles and $a_{n} \cdot dx_{n} = p_{n+1,n}(a_{n+1})$. This gives rise to a system of (fpqc) gerbes $\{\gerbeE_{n}:= \gerbeE_{a_{n}} \to (\text{Sch}/F)_{\text{fpqc}}\}_{n \in \mathbb{N}}$ (abbreviated as just $\{\gerbeE_{n}\}$) with morphisms of $\text{Sch}/F$-categories $\pi_{n+1,n} \colon \gerbeE_{n+1} \to \gerbeE_{n}$, where $\pi_{n+1,n} := \phi_{a_{n+1},a_{n},x_{n}}$, see Construction \ref{changeofgerbe}. 

\begin{Def} Define the \textit{inverse limit} of the system $\{\gerbeE_{n}\}$, denoted by $\varprojlim_{n} \gerbeE_{n} \to (\text{Sch}/F)_{\text{fpqc}}$, to be the category with fiber over $U \in \text{Ob}(\text{Sch}/F)$ with objects given by systems of pairs $(X_{n}, i_{n})_{n \in \mathbb{N}}$ of an object $X_{n} \in \gerbeE_{n}(U)$ and an isomorphism $i_{n} \colon \pi_{n+1,n}(X_{n+1}) \xrightarrow{\sim} X_{n}$ in $\gerbeE_{n}(U)$, and morphisms $(X_{n},i_{n}) \to (Y_{n},j_{n})$ given by a system of morphisms $\{f_{n} \colon X_{n} \to Y_{n} \}$ such that $j_{n} \circ \pi_{n+1,n}f_{n+1} = f_{n} \circ i_{n}$ for all $n$ (we extend this definition to morphisms between objects in different fibers in the obvious way). We call such a system of morphisms \textit{coherent}. It is clear that we have a compatible system of canonical morphisms of $\text{Sch}/F$-categories $\pi_{m} \colon \varprojlim_{n} \gerbeE_{n} \to \gerbeE_{m}$ for all $m$.
\end{Def}

It will turn out that the category $\gerbeE := \varprojlim_{n} \gerbeE_{n} \to (\text{Sch}/F)_{\text{fpqc}}$ is canonically an fpqc $u:= \varprojlim_{n} u_{n}$-gerbe, split over $\Spec(\bar{F})$. Denote the projection map $u \to u_{n}$ by $p_{n}$. Note that we have maps $\check{H}^{i}(\bar{F}/F, u_{n+1}) \to \check{H}^{i}(\bar{F}/F, u_{n})$ induced by $p_{n+1,n}$, and thus also a map 
\begin{equation}\label{projection}
\check{H}^{i}(\bar{F}/F,  u) \to \varprojlim_{n} \check{H}^{i}(\bar{F}/F, u_{n})
\end{equation}
for all $i \geq 0$. Recall from Proposition \ref{gerbebijection} that the fpqc $u$-gerbe $\gerbeE$ corresponds to a class in $\check{H}^{2}(\bar{F}/F, u)$. We give one preliminary result to show that our Convention \ref{conventions} applies for the group $u$ (which is not in general of finite type):

\begin{lem} Using the notation as above, $\check{H}_{\text{fpqc}}^{1}(U_{n}, u_{U_{n}}) = 0$ for all $n \geq 0$.
\end{lem}

\begin{proof} Let $R$ denote $\bar{F}^{\bigotimes_{F} n}$ for $n \geq 1$, and let $P$ be an fpqc $u_{R}$-torsor over $R$ (cf. Remark \ref{repremark}). Then for all $n$ we obtain a $u_{n,R}$-torsor by taking $P_{n}:= P \times^{u_{R}, p_{n}} u_{n,R}$. Moreover, by Proposition \ref{vanishingcohomology}, we have an isomorphism of $u_{1,R}$-torsors $P_{1} \xrightarrow{h_{1}} u_{1,R}$. Similarly, we have a trivialization $h_{2} \colon P_{2} \xrightarrow{\sim} u_{2,R}$, and the induced isomorphism $P_{1} = P_{2} \times^{u_{2,R},p_{2,1}} u_{1,R} \to u_{1,R}$ differs from $h_{1}$ by post-composing by an automorphism of the trivial $u_{1,R}$-torsor $u_{1,R}$ which must be translation by some $y_{1} \in u_{1}(R)$, which we may lift to $\tilde{y}_{1} \in u_{2}(R)$. We may then replace $h_{2}$ by its post-composition with translation by $\tilde{y}_{1}$ to assume that, via $p_{2,1}$, it induces $h_{1}$. Proceeding inductively in this manner, we obtain trivializations $h_{n} \colon P_{n} \xrightarrow{\sim} u_{n,1}$ such that $h_{n-1}$ is induced by $h_{n}$ via $p_{n,n-1}$, as above. This allows us to define a morphism of $u$-torsors $h \colon P \to u$ by applying $\varprojlim_{n}$ to the ($u$-equivariant) composition $P \xrightarrow{\text{id} \times e_{u_{n,R}}} P_{n} \xrightarrow{h_{n}} u_{n,R}$ (where $e_{u_{n,R}} \colon \Spec(R) \to u_{n,R}$ is the identity section), which is automatically an isomorphism.
\end{proof}

 The main result of this subsection is:

\begin{prop}\label{invlim1} With the setup as above, the category $\gerbeE := \varprojlim_{n} \gerbeE_{n} \to (\text{Sch}/F)_{\text{fpqc}}$ can be given the structure of an fpqc $u$-gerbe, split over $\Spec(\bar{F})$. Moreover, the map \eqref{projection} sends the class in $\check{H}^{2}(\bar{F}/F, u)$ corresponding to $\gerbeE$ to the element $([a_{n}])  \in \varprojlim_{n} \check{H}^{2}(\bar{F}/F, u_{n})$.
\end{prop}

\begin{proof} Existence of an object in $X \in \gerbeE(\Spec(\bar{F}))$  is clear, since by Lemma \ref{section} we have the ``trivial" $a_{n}$-twisted torsors $X_{n}$ in each $\gerbeE_{n}(\Spec(\bar{F}))$, and any two objects in $\gerbeE_{n}(\Spec(\bar{F}))$ are isomorphic (since $H^{1}(\bar{F}, u_{n}) = 0$ for all $n$). It is straightforward to check that, after choosing $i_{n}$ appropriately, we can take $\varphi_{n} \colon p_{1}^{*}X_{n} \to p_{2}^{*}X_{n}$ to be translation by a $2$-cocycle cohomologous to $a_{n}$ (cf. the proof of Lemma \ref{section}) to obtain an isomorphism $\varphi \colon p_{1}^{*}X \to p_{2}^{*}X$.

We now check that any two objects of $\gerbeE$ are locally isomorphic, following the proof suggested by the anonymous reviewer. Let $(X_{n}, i_{n})$ and $(Y_{n}, j_{n})$ be two objects in $\gerbeE(V)$; each $\Hom_{\gerbeE_{n}(V)}(X_{n}, Y_{n}) =: S_{n}$ is a $u_{n,V}$-torsor for any $n$. Pre-composing by $i_{n}$ and post-composing by $j_{n}^{-1}$ gives an isomorphism of $u_{n,V}$ torsors from $S_{n+1}\times^{\pi_{n+1,n}} u_{n,V}$ to $S_{n}$, and so we may define the inverse limit $S := \varprojlim_{n}  S_{n} \to V$, which, since each $S_{n} \to V$ is affine, is representable by an affine scheme $S$ over $V$, by \cite[\S 01YV]{Stacksproj}, which is also an fpqc cover of $V$, because each $S_{n} \to V$ is faithfully flat. It follows immediately that the pullbacks of $X$ and $Y$ to $S$ are isomorphic, as desired (all isomorphism sheaves $S_{n}$ are compatibly trivialized).

The band of $\gerbeE$ is canonically isomorphic to $u$, since for $V \to \Spec(F)$, any automorphism of the object $(X_{n}, i_{n})_{n}$ is given by a compatible system of automorphisms $X_{n} \xrightarrow{\sim} X_{n}$; for each $n$ we have a canonical identification of the band of $\gerbeE_{n}$ with $u_{n}$, and the compatibility hypothesis exactly says that we have a coherent system of elements with respect to the projective system $\{u_{n}(U)\}$ for any such system of automorphisms. This finishes the proof of the first claim.

For the second claim, we may use the lift $X$ and isomorphism $\varphi \colon p_{1}^{*}X \to p_{2}^{*}X$ constructed above to compute the class $[\gerbeE] \in \check{H}^{2}(\bar{F}/F, u)$ (see Fact \ref{gerbecocycle}). It is clear from our above construction that, via the natural projection map $u = \text{Band}(\gerbeE) \to \text{Band}(\gerbeE_{n}) = u_{n}$, the differential of $\varphi$ maps to the differential of $\varphi_{n}$, which one checks gives translation by an element that is cohomologous to $a_{n}$, as desired.
\end{proof}

We immediately get the following result:
\begin{cor}\label{invlimcor} The map $\check{H}^{2}(\bar{F}/F, u) \to \varprojlim_{n} \check{H}^{2}(\bar{F}/F, u_{n})$ defined above is surjective.
\end{cor}

\section{The cohomology set $H^{1}(\gerbeE, Z \to G)$}

This section largely follows \cite[\S 3]{Tasho}. We fix a local field $F$ of characteristic $p > 0$, and continue with the notation of \S 2. We make extensive use of the equivalence of categories between multiplicative $F$-groups of finite type and discrete, finitely-generated $\Gamma$-modules, see for example \cite{Borel}, Chapter 8. For the rest of this paper, we will always have (in the notation of \S 2) $U_{0} = \Spec(\bar{F})$; we continue using the notation of $U_{i}$ for $i > 0$ to represent the $(i+1)$st fibered products of $\Spec(\bar{F})$ over $F$.

\subsection{The multiplicative pro-algebraic group $u$}

For a finite Galois extension $E/F$, we consider the algebraic group $R_{E/F}[n] := \text{Res}_{E/F}\mu_{n}$, which is a multiplicative $F$-group with character group $X^{*}(R_{E/F}[n]) = \Z/n\Z[\Gamma_{E/F}]$ with $\Gamma_{E/F}$ acting by left-translation. We have the diagonal embedding $\mu_{n} \to R_{E/F}[n]$ induced by the $\Gamma$-homomorphism $\Z/n\Z[\Gamma_{E/F}] \to \Z/n\Z$ defined by $[\gamma] \mapsto 1$. The kernel of this homomorphism will be denoted by $\Z/n\Z[\Gamma_{E/F}]_{0}$, and is the character group of the multiplicative $F$-group $R_{E/F}[n]/\mu_{n}$, which will denote by $u_{E/F,n}$. Note that $u_{E/F,n}$ is smooth if and only if $n$ is coprime to the characteristic of $F$.

If $K/F$ is a finite Galois extension containing $E$ and $m$ is a multiple of $n$, then the injective morphism of $\Gamma$-modules $\Z/n\Z[\Gamma_{E/F}] \to \Z/m\Z[\Gamma_{K/F}]$ induced by the inclusion $\Z/n\Z \hookrightarrow \Z/m\Z$ and the map 
\begin{equation*}
[\gamma] \mapsto  \sum_{\substack{\sigma \in \Gamma_{K/F} \\ \sigma \mapsto \gamma}} [\sigma] 
\end{equation*}
induces an epimorphism $R_{K/F}[m] \to R_{E/F}[n]$. This maps $R_{K/F}[m]_{0}$ to $R_{E/F}[n]_{0}$ and thus induces an epimorphism $u_{K/F,m} \to u_{E/F,n}$. We define the pro-algebraic multiplicative group $u$ to be the limit $$u:= \varprojlim u_{E/F,n}$$ taken over the index category $\mathcal{I}$ whose objects are tuples $(E/F,n)$ as $n$ ranges through $\mathbb{N}$ and $E/F$ ranges over all finite Galois extensions of $F$ contained in our fixed separable closure $F^{s}$, and where there is at most one morphism $(K/F,m) \to (E/F,n)$ in $\mathcal{I}$ and it exists if and only if $E \subset K$ and $n \mid m$. For every $(E/F, n)$, the canonical map $u \to u_{E/F,n}$ is an epimorphism. Note that $u$ is a commutative affine group scheme over $F$; when taking the cohomology of $u$, we view it as a commutative fpqc group sheaf on $(\text{Sch}/F)_{\text{fpqc}}$ (and thus also a sheaf on $(\text{Sch}/F)_{\text{fppf}}$).

For a finite multiplicative algebraic group $Z$ over $F$, any $F$-homomorphism $u \to Z$ factors through an $F$-homomorphism $u_{E/F,n} \to Z$ for some $(E/F,n) \in \mathcal{I}$. We also have the ``evaluation at $e$" map $\delta_{e}: \mu_{n,F^{s}} \to u_{E/F,n,F^{s}}$, which is induced by the corresponding morphism of character groups from $\Z/n\Z[\gamma_{E/F}]_{0}$ to $\Z/n\Z$ sending $\sum_{\gamma \in \Gamma_{E/F}} c_{\gamma} [\gamma]$ to $c_{e}.$ It's easy to check that, for $E$ splitting $Z$, we have an isomorphism
\begin{equation}\label{normisomorphism}
\Hom_{F}(u_{E/F,n}, Z) \to \Hom(\mu_{n}, Z)^{N_{E/F}},\hspace{1cm}  f \mapsto f \circ \delta_{e},
\end{equation}
where the superscript $N_{E/F}$ denotes the kernel of the norm map and for two algebraic $F$-groups $A, B$, $\Hom(A,B)$ denotes the abelian group $\Hom_{F^{s}}(A_{F^{s}}, B_{F^{s}})$, which carries a natural $\Gamma$-action.

When taking inverse limits of the groups $u_{E/F, n}$ (and computing any cohomology groups) we may replace the category $\mathcal{I}$ with any co-final subcategory $\{E_{k}/F, n_{k}\}$ in $\mathcal{I}$, which we do in what follows by taking a tower $F = E_{0} \subset E_{1} \subset E_{2} \subset \dots$ of finite Galois extensions of $F$ with the property that $\cup E_{k} = F^{s}$ and a co-final sequence $\{n_{k}\} \subset \mathbb{N}^{\times}$. We set $R_{k}:= R_{E_{k}/F}[n_{k}]$ and $u_{k}:= u_{E_{k}/F, n_{k}}.$


\begin{prop}\label{profinite} The canonical map $H^{i}(F, u) \to \varprojlim H^{i}(F, u_{k})$ is an isomorphism for $i=1,2$.
\end{prop}

\begin{proof} First we fix $(E/F, n) \in \mathcal{I}$. We know from Hilbert's Theorem 90 that $H^{1}(F, \mu_{n_{k}}) = F^{*}/F^{*,n_{k}}$, from Shapiro's lemma that $H^{1}(F, R_{k}) = E_{k}^{*}/E^{*,n_{k}}$, and from local class field theory that $H^{2}(F, \mu_{n_{k}}) = \Z/n_{k}\Z$. Under these correspondences, the map $H^{1}(F, \mu_{n_{k}}) \to H^{1}(F, R_{k})$ corresponds to the obvious map $F^{*}/F^{*,n_{k}} \to E_{k}^{*}/E_{k}^{*,n_{k}}$, and so we have a short exact sequence of groups

\begin{equation}\label{LimitSES}
\begin{tikzcd}
  0 \arrow[r] & E^{*}/(F^{*} \cdot E^{*,n_{k}}) \arrow[r] & H^{1}(F, u_{k})  \arrow[r] & C_{k} \arrow[r] & 0  ,
\end{tikzcd}
\end{equation}

\noindent where $C_{k}$ is the image of $H^{1}(F, u_{k}) \to H^{2}(F, \mu_{n_{k}})$. 

By \cite[Lemma 0D6K]{Stacksproj}, we have the short exact sequences
\begin{center}
\begin{tikzcd}
  0 \arrow[r] & \varprojlim^{(1)}H^{0}(F, u_{k}) \arrow[r] & H^{1}(F, u)  \arrow[r] & \varprojlim H^{1}(F, u_{k}) \arrow[r] & 0  ; \\
  0 \arrow[r] & \varprojlim^{(1)}H^{1}(F, u_{k}) \arrow[r] & H^{2}(F, u)  \arrow[r] & \varprojlim H^{2}(F, u_{k}) \arrow[r] & 0  ,
\end{tikzcd}
\end{center}
and so it suffices to show that both left-hand terms vanish: the first vanishes because it's an inverse system of finite groups, and we will show that the second vanishes as well. Because of \eqref{LimitSES}, it is enough to show that $\varprojlim^{(1)} E_{k}^{*}/(F^{*} \cdot E_{k}^{*,n_{k}})$ vanishes, which follows if we can show that $\varprojlim^{(1)} E_{k}^{*}/E_{k}^{*,n_{k}}$ vanishes. To this end, consider the exact sequence induced by the valuation map $v$:
\begin{center}
\begin{tikzcd}
  0 \arrow[r] & \mathcal{O}_{k}^{\times}/( \mathcal{O}_{k}^{\times})^{n_{k}} \arrow[r] & E_{k}^{*}/E_{k}^{*,n_{k}}  \arrow[r, "v"] & \Z/ n_{k} \Z \arrow[r] & 0.

\end{tikzcd}
\end{center}
Note that, under the above transition maps, $\varprojlim \Z/ n_{k}\Z = \varprojlim^{(1)} \Z/n_{k}\Z = 0$; to see the first equality, fix $l \in \mathbb{N}$: we know from basic number theory that if $\varpi_{k}$ is a uniformizer of $E_{k}$, then $v_{l}(N_{E_{k}/E_{l}}(\pi_{k})) = f_{E_{k}/E_{l}}$, where $f_{E_{k}/E_{l}}$ denotes the degree of the associated extension of residue fields. Whence, we may choose $k > > l$ so that $n_{l} \mid f_{E_{k}/E_{l}}$, and so the transition map $\Z / n_{k} \Z \to \Z / n_{l} \Z$ is zero, as desired. 

It thus suffices to show that $\varprojlim^{(1)} \mathcal{O}_{k}^{\times}/( \mathcal{O}_{k}^{\times})^{n_{k}}$ vanishes, and, the final reduction, that $\varprojlim^{(1)} \mathcal{O}_{k}^{\times}$ vanishes. From here, it is enough by the explicit resolution computing the derived inverse limit (cf. \cite[\S 3.5]{Weibel}) to show that the map $$\Pi_{k \geq 1} \mathcal{O}_{k}^{\times} \xrightarrow{\text{id}-\alpha} \Pi_{k \geq 1} \mathcal{O}_{k}^{\times} $$ is surjective, where $\alpha$ on $\mathcal{O}_{l+1}^{\times}$ is the map $$\mathcal{O}_{l+1}^{\times} \xrightarrow{N_{E_{l+1}/E_{l}}} \mathcal{O}_{l}^{\times}  \hookrightarrow \Pi_{k \geq 1} \mathcal{O}_{k}^{\times}.$$ To this end, let $(x_{k}) \in \Pi_{k \geq 1} \mathcal{O}_{k}^{\times}$, and fix $l \in \mathbb{N}$. Consider the sequence inside $\mathcal{O}_{l}^{\times}$ given by $$x_{l}, x_{l}N_{E_{l+1}/E_{l}}(x_{l+1}), x_{l}N_{E_{l+1}/E_{l}}(x_{l+1})N_{E_{l+2}/E_{l}}(x_{l+2}), \dots ;$$ since the norm groups are open subgroups shrinking to the identity inside $\mathcal{O}_{l}^{\times}$, it follows that this sequence converges to some $y_{l} \in \mathcal{O}_{l}^{\times}$, and it is clear that $(y_{k})_{k \geq 1}$ gives a preimage of $(x_{k})$ under $\text{id}-\alpha$, as desired.
\end{proof}
The following result will be important when discussing the uniqueness of our constructions. 

\begin{lem}\label{spectralcollapses} We have $H^{i}(U_{n}, u) = 0$ for all $i >0$, $n \geq 0$.
\end{lem}

\begin{proof} First note that $H^{i}(U_{n}, u_{k}) = 0$ for any $i,k>0$, $n \geq 0$, by Proposition \ref{vanishingcohomology}. Thus, the result is clear if we can show that $H^{i}(U_{n}, u) = \varprojlim H^{i}(U_{n}, u_{k})$ for all $i > 0$, $n \geq 0$. Using the same short exact sequence for inverse limits and cohomology used in the proof of Proposition \ref{profinite}, it's enough to show that $\varprojlim^{(1)} H^{j}(U_{n}, u_{k}) = 0$ for all $j \geq 0$. 

For $j \geq 1$ this is immediate, since all the groups in the system are zero, by above. Thus, all that's left is showing $\varprojlim^{(1)} H^{0}(U_{n}, u_{k}) = 0$ for all $n$, which follows from the Mittag-Leffler condition, since $H^{1}(U_{n}, \text{ker}(u_{l} \to u_{k}))$ vanishes for any $k,l$.
\end{proof}

\begin{cor}\label{cechidentification} We have canonical isomorphisms $\check{H}^{p}(\bar{F}/F, u) \to H^{p}(F, u)$ for all $p \in \mathbb{N}$.
\end{cor}

\begin{proof} This is an immediate consequence of combining Lemma \ref{spectralcollapses} with Corollary \ref{cechtoderived}.
\end{proof}

Next, we prove the basic result about the cohomology of $u$. 

\begin{thm}\label{maincohom} We have $H^{1}(F, u) = 0$ and a canonical isomorphism $H^{2}(F, u) = \widehat{\mathbb{Z}}$.
\end{thm}

\begin{proof} As above, we fix a co-final subcategory $\{(E_{k}, n_{k})\}$ of $\mathcal{I}$. By Proposition \ref{profinite}, $H^{i}(F, u) = \varprojlim H^{i}(F, u_{E_{k}/F, n_{k}})$ for $i=1,2$. 

The argument for $i=2$ is identical to that in \cite{Tasho}, with a few minor adjustments---we have the functorial isomorphism 
\[
H^{2}(F, u_{k}) \cong H^{0}(F, \underline{X}^{*}(u_{k}))^{*} = H^{0}(\Gamma, X^{*}(u_{k}))^{*} \cong  \left[ \frac{n_{k}}{(n_{k}, [E_{k}: F])} \Z / n_{k} \Z \right]^{*} \cong \Z / (n_{k}, [E_{k}: F]) \Z ,
\]

\noindent where for an abelian group $M$, $M^{*}$ denotes the group $\Hom_{\Z}(M, \mathbb{Q}/\Z)$, $\underline{X}^{*}(u_{k})$ denotes the \'{e}tale group scheme associated to the $\Gamma$-module $X^{*}(u_{k})$, and the first isomorphism is given by the analogue of Poitou-Tate duality for fppf cohomology of finite group schemes over a local field of positive characteristic, see for example \cite[III.6.10]{Milne}. For $k > l$, the transition map $H^{2}(p): H^{2}(F, u_{k}) \to H^{2}(F, u_{l})$ is translated by this isomorphism to the natural projection map $$\Z / (n_{k}, [E_{k}: F]) \Z \to \Z / (n_{l}, [E_{l}: F]) \Z.$$ We may then set $n_{k} = [E_{k}: F]$ for all $k$, giving $(n_{k}, [E_{k}: F]) = n_{k}$, settling the case $i=2$.

For $i=1$, we consider the short exact sequence \eqref{LimitSES} from the proof of Proposition \ref{profinite}. Note that since the composition $H^{2}(F, \mu_{n_{k}}) \to H^{2}(F, R_{k}) \xrightarrow{\sim} H^{2}(E, \mu_{n_{k}})$ is the restriction map, after identifying both sides with $\frac{1}{n_{k}}\Z/\Z$ via the invariant map, it is multiplication by $[E_{k} \colon F]$ and so, by choosing $E_{k}$ such that $n_{k}$ divides $[E_{k} \colon F]$ for all $k$, $C_{k}$ is just $H^{2}(F, \mu_{n_{k}}) = \frac{1}{n_{k}}\Z/\Z$. Our desired vanishing will then follow from the vanishing of $\varprojlim H^{2}(F, \mu_{n_{k}})$ and $\varprojlim E_{k}^{*}/(F^{*} \cdot E_{k}^{*,n_{k}})$.

To show the former, it's enough to find, for every $l$ fixed, some $k > l$ such that the transition map $H^{2}(F, \mu_{n_{k}}) \to H^{2}(F, \mu_{n_{l}})$ is zero. For this, note that, at the level of character modules, the map $p_{k,l}^{\sharp}: X^{*}(R_{l}) \to X^{*}(R_{k})$ induces a map on quotients by the subgroups $X^{*}(R_{l})_{0}$, $X^{*}(R_{k})_{0}$ (respectively) that's identified with the map $\Z / n_{l} \Z \to \Z / n_{k} \Z$ sending $[1]$ to $[(\frac{n_{k}}{n_{l}})^{2}]$, and we may choose $k$ so that $n_{k}/n_{l}$ is a  multiple of $n_{l}$.

It remains to show that $\varprojlim E_{k}^{*}/(F^{*} \cdot E_{k}^{*,n_{k}}) = 0$---after noting that $\varprojlim^{(1)} F^{*}/(F^{*} \cap E_{k}^{*,n_{k}}) = 0$ (the transition maps are induced by multiplication by $[E_{k} \colon F]$, so for any fixed $k$ their image is eventually trivial), this reduces to showing $\varprojlim E_{k}^{*}/E_{k}^{*,n_{k}} = 0$. As in the proof of Proposition \ref{profinite}, the valuation short exact sequence lets us reduce further to showing $\varprojlim \mathcal{O}_{k}^{\times}/(\mathcal{O}_{k}^{\times})^{n_{k}} = 0$, and the identical argument loc. cit. for the vanishing of $\varprojlim^{(1)} \mathcal{O}_{k}^{\times}$ shows that $\varprojlim^{(1)} (\mathcal{O}_{k}^{\times})^{n_{k}} = 0$ as well, so we may make the final reduction to showing the triviality of $\varprojlim \mathcal{O}_{k}^{\times}$, an immediate consequence of the fact that the norm groups shrink to the identity.
\end{proof}

Combining the above result with Corollary \ref{cechidentification} immediately gives:

\begin{cor} We have  $\check{H}^{1}(\bar{F}/F, u) = 0$ and a canonical identification $\check{H}^{2}(\bar{F}/F, u) \xrightarrow{\sim} \widehat{\Z}$. In particular, the natural map $\check{H}^{p}(\bar{F}/F, u) \to \varprojlim_{k} \check{H}^{p}(\bar{F}/F, u_{k})$ is an isomorphism for $p=1,2$. 
\end{cor}

We denote by $\xi \in H^{2}(F, u)$ the element corresponding to $-1 \in \widehat{\Z}$. Later in the paper we will also refer to this class simply by $-1$, when there is no danger of confusion. For any multiplicative algebraic group $Z$ defined over $F$, we obtain a map 
\begin{equation}\label{alphamap} \xi^{*} \colon \Hom_{F}(u,Z) \to H^{2}(F, Z)
\end{equation}
via taking the image of $\xi$ under the map $H^{2}(F, u) \to H^{2}(F, Z)$ induced by $\phi \in \Hom_{F}(u,Z)$.

\begin{prop} If $Z$ is any finite multiplicative algebraic group defined over $F$, then $\xi^{*}$ is surjective. If $Z$ is also split, then $\xi^{*}$ is also injective.
\end{prop}

\noindent The identical proof as in \cite[Proposition 3.1]{Tasho} works here, with the only difference being the replacement of the classical local Poitou-Tate with the version for finite groups schemes over local fields of positive characteristic, which does not affect the rest of the argument. 

\subsection{Basic properties of $H^{1}(\gerbeE, Z \to G)$} 
Fix a $u$-gerbe $(\gerbeE, \theta)$ split over $\bar{F}$ corresponding to the class $\xi \in H^{2}(F, u)$, where by ``corresponding" we mean $[\gerbeE] \in \check{H}^{2}(\bar{F}/F, u) \xrightarrow{\sim} \varprojlim_{n} \check{H}^{2}(\bar{F}/F, u_{n}) =  \widehat{\Z}$ maps to $\xi$ (see Proposition \ref{fpqccover}, Corollary \ref{cechidentification}, and Proposition \ref{profinite}). 

Given $[Z \to G]$ in $\mathcal{A}$, recall the definition of $H^{1}(\gerbeE, Z \to G)$ from the end of \S 2.3, which is functorial $[Z \to G]$. For any other choice of $\shA$-gerbe $\gerbeE'$ with $[\gerbeE'] = -1$, we know from Proposition \ref{cechidentification} that $[\gerbeE] = [\gerbeE'] \in \check{H}^{2}(\bar{F}/F, u)$, and hence by Proposition \ref{gerbebijection} we have a $u$-equivalence $\eta \colon \gerbeE \to \gerbeE'$, which (via pullback) induces a map $H^{1}(\gerbeE', G_{\gerbeE'}) \to H^{1}(\gerbeE, \GE)$, and it is straightforward to verify that this map further gives rise to a map $H^{1}(\gerbeE', Z \to G) \to H^{1}(\gerbeE, Z \to G)$ for any $[Z \to G] \in \mathcal{A}$, which by Lemma \ref{uniqueness}, is independent of the choice of $\eta$. 

\begin{lem}\label{translem} For abelian $G$, the transgression map $\Hom_{F}(u,Z) \to H^{2}(F, G)$ is equal to the composition of the map $\xi^{*}$ defined in \eqref{alphamap} and the natural map $H^{2}(F, Z) \to H^{2}(F, G)$. 
\end{lem}

\begin{proof} We may work with $a$-twisted cocycles valued in $G$ for a choice of $a \in \shA(U_{2})$ with $[a] = \xi$. By the functoriality of our inflation-restriction sequence, we may replace $G$ by $Z$, and we are reduced to showing that the transgression map $\Hom_{F}(u,Z) \to H^{2}(F, Z)$ equals the map $\xi^{*}$. Recall that $\xi^{*}$ is defined by mapping a homomorphism to the image of $\xi$ under the induced map $H^{2}(F, u) \to H^{2}(F, Z)$. By construction, the image of $f \in \Hom_{F}(u, Z)$ under the transgression map is the class $[f(a)] \in H^{2}(F, Z)$, which is exactly the statement of the lemma, since $[a] = \xi$. 
\end{proof}

\begin{remark} The above proof does not use anything specific about the group $u$; the result holds when $u$ is replaced by any commutative $F$-group scheme $\shA$, $a$ by a \v{C}ech 2-cocycle $c$, and $\gerbeE$ by $\gerbeE_{c}$. We will use this in \S 4 without comment.
\end{remark}

For $[Z \to G]$ in $\mathcal{A}$, recall that $G \xrightarrow{\pi} \overline{G} := G/Z$.

\begin{lem}\label{mapb} There is a morphism of pointed sets $b \colon H^{1}(\gerbeE, Z \to G) \to H^1(F, \overline{G})$, and when $G$ is abelian it is a group homomorphism.
\end{lem}

\begin{proof}
If the inertial action on a $G_{\gerbeE}$-torsor $\mathscr{T}$ is induced by a homomorphism $u \to Z$ then evidently the inertial action is trivial on the $\overline{G}_{\gerbeE}$-torsor $\overline{\mathscr{T}} := \mathscr{T} \times^{G_{\gerbeE}} \overline{G}_{\gerbeE}$, and so by inflation-restriction $\overline{\mathscr{T}}$ descends to a $\overline{G}$-torsor $\overline{T}$ over $F$ whose isomorphism class is uniquely determined by $\overline{\mathscr{T}}$. We may thus define the map as $[\mathscr{T}] \mapsto [\overline{T}]$.
\end{proof}

The following is the most important proposition of the section, and will be used extensively in the next section. 

\begin{prop}\label{bigdiag} Let $[Z \to G] \in \mathcal{A}$ be abelian. Put $\overline{G} = G/Z$. Then (discarding equalities) we have the commutative diagram with exact rows and columns:
\[
\begin{tikzcd}
  & \overline{G}(F)  \arrow[d] \arrow[equals]{r} & \overline{G}(F) \arrow[d] & & \\  
  0 \arrow[r] & H^{1}(F, Z) \arrow[d]  \arrow[r, "\text{Inf}"] & H^{1}(\gerbeE_{a}, Z \to Z) \arrow[d]  \arrow[r, "\text{Res}"] & \Hom_{F}(u, Z) \arrow[equals]{d}  & \\
   0 \arrow[r] & H^{1}(F, G)  \arrow[r, "\text{Inf}"] \arrow[equals]{d}& H^{1}(\gerbeE_{a}, Z \to G) \arrow[d, "b"]  \arrow[r, "\text{Res}"] & \Hom_{F}(u, Z) \arrow[d, "\xi^{*}"] \arrow[r, "tg"] & H^{2}(F, G) \arrow[equals]{d} \\  
  & H^{1}(F, G) \arrow[r] & H^{1}(F, \overline{G}) \arrow[d]  \arrow[r] & H^{2}(F, Z) \arrow[d]  \arrow[r] & H^{2}(F, G) \\  
  & & 0 & 0. &
\end{tikzcd}
\]
\end{prop}

\begin{proof} We may work with $a$-twisted cocycles for an appropriate choice of $a$. The second and third rows come from the already-established inflation-restriction result, the fourth row and left column come from the long exact sequence for fppf cohomology associated to the short exact sequence $0 \longrightarrow Z \longrightarrow G \longrightarrow \overline{G} \longrightarrow 0$, and the middle column is from Lemma \ref{LES} and Lemma \ref{mapb}. It follows from the same lemmas that the middle column is exact, except for possibly the surjectivity of $b$, which we will show later in the proof. The commutativity of all squares is obvious, except for the bottom right one, which is exactly Lemma \ref{translem}, and the bottom middle one, which we will show now. 

The map $H^{1}(\gerbeE_{a}, G_{\gerbeE_{a}}) \to  H^1(\gerbeE_{a}, \overline{G}_{\gerbeE_{a}})$ sends the class of the $a$-twisted cocycle $(x, \phi)$ to the class of $(\pi(x), \pi \circ \phi)$. Since we assume that $\phi$ factors through $Z$, the class $[(\pi(x), \pi \circ \phi)]$ is the class of $[\pi(x), 0]$, where $\pi(x) \in \overline{G}(U_{1})$ is an actual 1-cocycle (because $\pi(\phi(a)) = e_{\overline{G}}$). We want to look at the image of the class $[\pi(x)]$ under the connecting homomorphism $\delta \colon H^{1}(F, \overline{G}) \to H^{2}(F, Z)$ (computed as in Proposition \ref{connecting}).  

To compute $\delta([\pi(x)])$, we first lift $\pi(x)$ to $G(U_{1})$; the natural element to pick here is $x \in G(U_{1})$. Then $\delta([\pi(x)])$ is exactly $[dx] \in H^{1}(\bar{F}/F, Z)$, which, by assumption, equals $[\phi(a)]$, which gives the desired commutativity, since the class $[\text{Res}[(x, \phi)](a)] = [\phi(a)]$ is exactly the element of $H^{2}(F, Z)$ obtained by following the square in the other direction, see the proof of Lemma \ref{translem}.

The last thing to show is the surjectivity of $b$, which follows immediately from the surjectivity of $\xi^{*}$, using the commutativity of the bottom right and middle squares and the four-lemma. We will address the non-abelian case analogue of this result in Proposition \ref{bigdiag2}.
\end{proof}

\section{Extending Tate-Nakayama}
Let $S$ be an $F$-torus and $E/F$ a finite Galois extension. As in \cite[\S 4]{Tasho}, the goal of this section is to extend the notion of the classical Tate-Nakayama isomorphism $$X_{*}(S)_{\Gamma, \text{tor}} = H^{-1}_{\text{Tate}}(\Gamma_{E/F}, X_{*}(S)) \xrightarrow{\sim} H^1(\Gamma, S)$$ to the setting of our cohomology group $H^{1}(\gerbeE, Z \to S)$. Some new notation: for an affine $F$-group scheme $G$, we will denote by $F[G]$ the coordinate ring of $G$. Let $H^{1}(\gerbeE)$ denote the functor from $\mathcal{T}$ to $\text{AbGrp}$ which sends $[Z \to S]$ to the group $H^{1}(\gerbeE, Z \to S)$.

Following \cite{Tasho}, we will first construct a functor $\overline{Y}_{+, \text{tor}} \colon \mathcal{T} \to \text{AbGrp}$ which extends the functor $S \mapsto X_{*}(S)_{\Gamma, \text{tor}}$, as well as a morphism of functors from $\overline{Y}_{+, \text{tor}}$ to the functor $[Z \to S] \mapsto \Hom_{F}(u, Z)$.  Then we will construct a unique isomorphism of functors $$\overline{Y}_{+, \text{tor}} \to H^{1}(\gerbeE)$$ on $\mathcal{T}$ which for objects $[1 \to S] \in \mathcal{T}$ coincides with the Tate-Nakayama isomorphism, and such that the composition $\overline{Y}_{+, \text{tor}}(Z \to S) \to H^{1}(\gerbeE, Z \to S) \to \Hom_{F}(u, Z)$ equals the morphism alluded to above. 

The first subsection is just a summary of \cite[\S 4.1]{Tasho}. 

\subsection{The functor $\overline{Y}_{+, \text{tor}}$}

For $[Z \to S] \in \mathcal{T}$, we set $\overline{S} := S/Z$. Then if $Y := X_{*}(S)$ and $\overline{Y} := X_{*}(\overline{S})$, we have an injective morphism of $\Gamma$-modules $Y \to \overline{Y}$. 

We then have an isomorphism of $\Gamma$-modules $$\overline{Y}/Y \to \Hom(\mu_{n}, Z) \hspace{1cm} \bar{\lambda} \mapsto [x \mapsto (n\lambda)(x)],$$ for any $n \in \mathbb{N}$ such that $[\overline{Y} \colon Y] \mid n$, where for $\lambda \in \overline{Y}$, we identify $n\lambda$ with an element of $Y$. Take any finite Galois extension $E/F$ which splits $S$, and take $I \subset \Z[\Gamma_{E/F}]$ to be the augmentation ideal. Set $\overline{Y}_{+} := \overline{Y}/IY$, and $\overline{Y}_{+}^{N}$ the quotient of $\overline{Y}^{N}$ by $IY$, where the superscript $N$ denotes the kernel of the norm map $N_{E/F}$. 

Then we have $\overline{Y}_{+}^{N} = \overline{Y}_{+, \text{tor}}$ (see \cite[Fact 4.1]{Tasho}), and the natural map $\overline{Y}_{+}^{N} \to [\overline{Y}/Y]^{N}$ post-composed with the isomorphisms $[\overline{Y}/Y]^{N} \xrightarrow{\sim} \Hom(\mu_{n}, Z)^{N}  \xrightarrow{\sim} \Hom_{F}(u_{E/F,n}, Z)$ (this second isomorphism comes from \eqref{normisomorphism} in \S 3) gives a homomorphism $\overline{Y}_{+}^{N} \to \Hom_{F}(u, Z)$. For varying $E/F$ and $n$, these homomorphisms are compatible and splice to a homomorphism $\overline{Y}_{+, \text{tor}} \to \Hom_{F}(u, Z)$. 

Given a morphism $[Z_{1} \to S_{1}] \to [Z_{2} \to S_{2}]$ in $\mathcal{T}$, the induced morphism $\overline{S_{1}} \to \overline{S_{2}}$ induces a $\Gamma$-morphism $X_{*}(\overline{S_{1}})_{+, \text{tor}} \to X_{*}(\overline{S_{2}})_{+, \text{tor}}$, showing that the assignment $[Z \to S] \mapsto \overline{Y}_{+, \text{tor}}$ is functorial in $[Z \to S]$.

\subsection{The unbalanced cup product on fppf cohomology}
Let $K/F$ be a finite field extension, and $E/F$ a Galois extension contained in $K$. In order to construct the isomorphism of functors from $\overline{Y}_{+, \text{tor}}$ to $H^{1}(\gerbeE)$, it is necessary to extend the construction of the unbalanced cup product from Galois cohomology to the setting of fppf cohomology.

Recall that, as in \cite{Tasho}, for a group $G$ with surjective homomorphism $G \xrightarrow{\Delta} H$ and $G$-module $M$, the subgroup $C^{n,1}(G, H, M)$ is the subgroup of all (inhomogeneous) $n$-cochains for $G$ with respect to $M$ such that the last coordinate only depends on the residue modulo the kernel of $\Delta$; i.e., the last argument of any cochain $G^{n} \to M$ is a function on $H$ (keeping other coordinates fixed). One may then define a bilinear $\Z$-pairing $$C^{n,1}(G,H,M) \times C^{-1}_{\text{Tate}}(H, M) \to C^{n-1}(G, M),$$ called the \textit{unbalanced cup product}, where here $C^{-1}_{\text{Tate}}(H, M)$ denotes $M$ (inhomogenous $-1$-cochains). For global applications (which will not be addressed in this paper, but do not change the arguments), we will work with an arbitrary commutative ring $R$ and finite flat extension $R \to S$ such that we have a finite Galois subextension $R \to S' \to S$. For our applications in this paper, $R$ will be $F$, $S$ will be $K$, and $S'$ will be $E$. We remind the reader that all $R$-group schemes are assumed to be affine over $R$.

For $S/R$ a finite flat extension (not necessarily \'{e}tale), $S'/R$ a Galois extension contained in $S$, and $G$ a commutative $R$-group scheme, define the group $C^{n}(S/R, G)$ to be $G(\Snp)$, and  $C^{n,1}(S/R, S', G)$ to be the subgroup $G(\Sn \otimes_{R} S')$. Our goal is to define an unbalanced cup product $$C^{n,1}(S/R, S', G) \times C_{\text{Tate}}^{-1}(\Gamma', H(S')) \xrightarrow{\underset{S'/R}{\sqcup}} C^{n-1}(S/R, J)$$ for two commutative $R$-group schemes $G, H$ and $R$-pairing $P \colon G \times H \to J$ for $J$ another commutative $R$-group scheme, where as above $C_{\text{Tate}}^{-1}(\Gamma', H(S')) = H(S')$ and $\Gamma' := \text{Aut}_{R\text{-alg}}(S')$. We start by defining the standard cup product in \v{C}ech cohomoloogy:

\begin{Def} We recall the general fppf cup product on \v{C}ech cochains as defined in \cite[\S 3]{Shatz}, continuing with the notation from above. For $n,m \geq 1$, we have the morphism of $R$-algebras $S^{\otimes_{R}n} \otimes_{R} S^{\otimes_{R} m}  \xrightarrow{\theta} S^{\otimes_{R} (n+m-1)}$ defined on simple tensors by $$(a_{1} \otimes \dots \otimes a_{n}) \otimes (b_{1} \otimes \dots \otimes b_{m}) \mapsto a_{1} \otimes \dots \otimes a_{n-1} \otimes a_{n}b_{1} \otimes b_{2} \otimes \dots \otimes b_{m}.$$
The cup product is defined to be the composition of $$G(S^{\otimes_{R} n}) \times H(S^{\otimes_{R} m}) \xrightarrow{(\text{id} \otimes 1)^{\sharp} \times (1 \otimes \text{id})^{\sharp}} G(S^{\otimes_{R}n} \otimes_{R} S^{\otimes_{R} m}) \times H(S^{\otimes_{R}n} \otimes_{R} S^{\otimes_{R} m}) = (G \times H)(S^{\otimes_{R}n} \otimes_{R} S^{\otimes_{R} m})$$ with the map $$(G \times H)(S^{\otimes_{R}n} \otimes_{R} S^{\otimes_{R} m}) \xrightarrow{\theta^{\sharp}} (G \times H)(S^{\otimes_{R} (n+m-1)}) \xrightarrow{P} J(S^{\otimes_{R} (n+m-1)}).$$
\end{Def}

With the above definition in hand, we resume our construction of the unbalanced cup product. We have a homomorphism of $R$-algebras $\lambda \colon S^{\bigotimes_{R} n} \otimes_{R} S' \to \prod_{\Gamma'} S^{\bigotimes_{R} n}$ defined on simple tensors by $$a_{i,0} \otimes \dots \otimes a_{i,n} \mapsto (a_{i,0} \otimes \dots \otimes a_{i,n-1} \prescript{\sigma}{}a_{i,n} )_{\sigma \in \Gamma'}.$$ Moreover, for any $R$-group scheme $J$, we have a canonical identification $J(\prod_{\Gamma'} S^{\bigotimes_{R} n}) \to \prod_{\Gamma'} J(S^{\bigotimes_{R} n})$; we then define, for $a \in G(S^{\bigotimes_{R} n} \otimes_{R} S')$ and $b \in H(S')$, $$a \tilde{\underset{S'/R}{\sqcup}} b = \lambda(a \cup b^{(0)}) \in \prod_{\Gamma'} J(S^{\bigotimes_{R} n}).$$ In the above formula we are using the fppf cup product as defined in \cite[\S 3]{Shatz}, (which in this case is the pairing applied to $(a, p_{n+1}^{\sharp}(b^{(0)})) \in G(S^{\bigotimes_{R} (n+1)}) \times H(S^{\bigotimes_{R} (n+1)})$) and $b^{(0)}$ denotes the element $b \in H(S')$ viewed as a $0$-cochain.

We now apply the group homomorphism $N \colon \prod_{\Gamma'} J(S^{\bigotimes_{R} n}) \to J(S^{\bigotimes_{R} n})$ obtained by taking the sum of all elements on the left-hand side, and the resulting pairing 
$$C^{n,1}(S/R, S', G) \times C_{\text{Tate}}^{-1}(\Gamma', H(S')) \to J(S^{\bigotimes_{R} n})$$ is $\Z$-bilinear.  Indeed, $$(a +a') \tilde{\underset{S'/R}{\sqcup}} b = \lambda[(a+a') \cup b^{(0)}] = \lambda(a\cup b^{(0)} + a' \cup b^{(0)}) = \lambda(a \cup b^{(0)}) + \lambda(a' \cup b^{(0)}).$$ This will be our desired pairing, denoted by $a \underset{S'/R}{\sqcup} b$. 

We will now prove some basic properties of this pairing. The first order of business is to show that this agrees with the analogous pairing from \cite{Tasho} in the case that $S/R$ is also finite \'{e}tale, with Galois group $\Gamma$. There is a simple way to compare \v{C}ech cohomology and Galois cohomology in such a case: for any commutative $R$-group $G$, there is a comparison homomorphism $$C^{n}(S/R, G) \to C^{n}(\Gamma, G(S))$$ given as follows: We have a homomorphism of $R$-algebras 
\begin{equation}\label{translate}
t \colon S^{\bigotimes_{F} (n+1)} \to \prod_{\underline{\sigma} \in \Gamma^{n}} S_{\underline{\sigma}},
\end{equation}
induced by the map on simple tensors $$a_{0} \otimes \dots \otimes a_{n} \mapsto (a_{0}\prescript{\sigma_{1}}{}a_{1}\prescript{(\sigma_{1} \sigma_{2})}{}a_{2} \dots \prescript{(\sigma_{1} \dots \sigma_{n})}{}a_{n})_{(\sigma_{1}, \dots ,\sigma_{n})}.$$ 

We immediately get a homomorphism $c: G(S^{\bigotimes_{R} (n+1)}) \to G(\prod_{\underline{\sigma} \in \Gamma^{n}} S_{\underline{\sigma}}) = C^{n}(\Gamma, G(S))$, where the last equality is the obvious identification. Passing to cohomology, this induces a homomorphism $\check{H}^{n}(S/R, G) \to H^{n}(\Gamma, G(S))$. Note that all of these maps are isomorphisms if $S/R$ is a finite Galois extension of fields. This comparison map also respects our special subgroups; that is to say, the homomorphism $G(\Snp) \to \prod_{\Gamma^{n}} G(S)$ maps  $C^{n,1}(S/R, S', G)$ into $C^{n,1}(\Gamma, \Gamma', G(S))$, which again is an isomorphism when $R$ and $S$ are fields.

\begin{prop}\label{comparecup} When $S/R$ is finite Galois, the unbalanced cup product $a \underset{S'/R}{\sqcup} b$ agrees with the unbalanced cup-product from \cite{Tasho} after applying the comparison homomorphism \eqref{translate}.
\end{prop}

\begin{proof} Recall that the pairing from \cite{Tasho} sends $a \in C^{n,1}(\Gamma, \Gamma', G(S))$ and $b \in H(S')$ to the $(n-1)$-cochain $$(\sigma_{1}, \dots, \sigma_{n-1}) \mapsto P[\Sigma_{\sigma \in \Gamma'} a(\sigma_{1}, \dots, \sigma_{n-1}, \sigma) \otimes \prescript{\sigma_{1} \dots \sigma_{n-1} \sigma}{} b] \in J(S),$$ where we are abusing notation and using $P$ to denote the map $G(S) \otimes_{\Z} H(S) \to J(S)$ induced by the pairing $P$.

In the \v{C}ech setting, the point $a \cup b^{(0)}$ corresponds to the $R$-algebra homomorphism $R[J] \to R[G] \otimes_{F} R[H] \to S^{\bigotimes_{R} n+1}$ given by post-composing $P^{\sharp}$ by the map determined by $a$ and $1 \otimes b$ (identifying the points with their ring homomorphisms). Then the map $\lambda$ sends this point to the map $R[J] \to R[G] \otimes_{F} R[H] \to  \prod_{\Gamma'} S^{\bigotimes_{R} n}$ given by post-composing $P^{\sharp}$ by the map $R[G] \otimes_{F} R[H] \to \prod_{\sigma \in \Gamma'} S^{\bigotimes_{R} n}$ determined by $\lambda \circ a $ and $\lambda \circ (1 \otimes b)$. It is straightforward to verify that via the composition $\prod_{\Gamma'} J(S^{\bigotimes_{R} n}) \to \prod_{\Gamma'} \prod_{\Gamma^{n-1}} J(S) \to \prod_{\Gamma^{n-1}} J(S)$, we obtain the $(n-1)$-cochain of $\Gamma$ valued in $J(S)$ sending $(\sigma_{1}, \dots, \sigma_{n-1})$ to $\sum_{\sigma \in \Gamma'} P(a(\sigma_{1}, \dots, \sigma_{n-1}, \sigma), \prescript{\sigma_{1} \dots \sigma_{n-1} \sigma}{}b) \in J(S)$, as desired.
\end{proof}

We now prove an elementary result stating how this map behaves with respect to \v{C}ech differentials. 

\begin{lem} For $a \in C^{n,1}(S/R, S', J)$, we have $N\lambda (da) = (\# \Gamma')(-1)^{n+1}a + d(N\lambda a).$
\end{lem}

\begin{proof} First, note that $d$ maps $C^{n,1}(S/R, S', J)$ into $C^{n+1,1}(S/R, S', J)$, so the statement makes sense. Start with the $\Sn \otimes_{R} S'$-point $a \colon R[J] \to \Sn \otimes_{R} S'$. Applying the differential yields the sum as $i$ ranges from $0$ to $n+1$ of $(-1)^{i}$ times the $\Snp \otimes_{R} S'$-point $p_{\widehat{i}} \circ a \colon R[J] \to \Snp \otimes_{R} S'$, where $0 \leq i \leq n+1$. Note that $\lambda \circ p_{\widehat{i}} \circ a $ equals $(\lambda \circ a)_{i}:= p_{\widehat{i}} \circ \lambda \circ a$ for $i \neq n+1$ and $\lambda \circ p_{\widehat{n+1}} \circ a = (a)_{\sigma \in \Gamma'}$. We conclude that $$N\lambda(da) = (\# \Gamma')(-1)^{(n+1)}a + N [\sum_{0 \leq i \leq n} (-1)^{i} (\lambda \circ a)_{i}]$$ which equals $(\# \Gamma')(-1)^{n+1}a + d(N \lambda a)$. 
\end{proof}

We now reach the key property of our unbalanced cup product.

\begin{prop}\label{bigcup} For $a \in C^{n,1}(S/R, S', G)$ and $b \in C^{-1}_{\text{Tate}}(\Gamma',H(S'))$, we have $$d(a \underset{S'/R}{\sqcup} b) = (da) \underset{S'/R}{\sqcup} b + (-1)^{n} (a \cup db).$$
\end{prop}

\begin{proof} Recall that the differential $db$ of the $-1$-cochain $b$ of $\Gamma'$ is defined to be its $\Gamma'$-norm. The left-hand side equals $d[N(\lambda(a \cup b^{(0)}))] = (\# \Gamma')(-1)^{n}(a \cup b^{(0)}) + N\lambda (d(a \cup b^{(0)}))$, by the above lemma. This in turn equals $(\# \Gamma')(-1)^{n}(a \cup b^{(0)}) + N\lambda[(da)\cup b^{(0)} + (-1)^{n}(a \cup db^{(0)})]$ (by \cite[\S 3]{Shatz}). Thus, the desired equality reduces to $$(\# \Gamma')(a \cup b^{(0)}) + N \lambda(a \cup db^{(0)}) = a \cup db.$$

Now, $db^{(0)} = -p_{1}^{\sharp}(b)+ p_{2}^{\sharp}(b) \in H(S' \otimes_{R} S')$, so that $\lambda(a \cup db^{(0)}) = \lambda(a \cup -p_{1}^{\sharp}(b)) + \lambda(a \cup p_{2}^{\sharp}(b))$, and $\lambda(a \cup -p_{1}^{\sharp}(b)) = (a \cup -b^{(0)})_{\sigma}$, so all we need to show is $N \lambda(a \cup p_{2}^{\sharp}(b)) = a \cup db$. 
Note that $\lambda(a \cup p_{2}^{\sharp}(b)) = (a \cup \prescript{\sigma}{}{b}^{(0)})_{\sigma}$, so applying $N$ gives the desired result.
\end{proof}


The setting we will be concerned with in this paper is the case where $R=F$ our local field, $G = A_{1}$, $J= A_{2}$ are two multiplicative groups over $F$, $H = \underline{\Hom}(A_{1},A_{2})$ is the \'{e}tale $F$-group scheme of morphisms from $A_{1}$ to $A_{2}$, and the pairing $P \colon A_{1} \times \underline{\Hom}(A_{1}, A_{2}) \to A_{2}$ is the canonical one; we switch to multiplicative rather than additive notation for our abelian groups here. The following two elementary results will be used repeatedly in what follows, so we record them here:

\begin{lem}\label{zerocup} For $f \in \underline{\Hom}(A_{1}, A_{2})(\bar{F})$ (viewed as a $0$-cochain) and $x \in A_{1}(K^{\bigotimes_{F}n})$, we have $f \cup x = [p_{1}^{*}f](x)$. In particular, for $\phi \in \underline{\Hom}(A_{1},A_{2})(F)$, we have $x \cup \phi = \phi \cup x = \phi(x)$.
\end{lem}

\begin{proof} This is a straightforward computation.
\end{proof}

\begin{lem}\label{pullcup} Let $A_{3}$ be a third multiplicative group scheme over $F$. If $g \in \underline{\Hom}(A_{3}, A_{2})(F)$ (viewed as a $0$-cochain), then for $f \in \Hom(A_{1}, A_{3})(\bar{F})$ (viewed as a $-1$-cochain) and $x \in C^{n,1}(K/F, E, A_{1})$, we have $$x \underset{E/F}{\sqcup} (g \circ f) = (x \underset{E/F}{\sqcup} f) \cup g,$$
where in the above equality $g \circ f$ is viewed as a $-1$-cochain.
\end{lem}

\begin{proof} 
We have that  $x \cup (g \circ f) = x \cup (f \cup g) = (x \cup f) \cup g = g(x \cup f)$, where we are using the fact that $f \cup g = g \circ f$, and since $g$ is defined over $F$, we have by Lemma \ref{zerocup} that $(x \cup f ) \cup g = g (x \cup f)$. Thus, $\lambda[ x \cup (g \circ f)] = \lambda[g(x \cup f)] = (\prod_{\sigma}g)[\lambda(x \cup f)]$, where this last equality follows from the fact that $g$ is defined over $F$. Finally, applying $N$ gives that $x \underset{E/F}{\sqcup} (g \circ f) = N (\prod_{\sigma}g)[\lambda(x \cup f)] = g(x \underset{E/F}{\sqcup} f) = (x \underset{E/F}{\sqcup} f) \cup g$, where this last equality comes once again from Lemma \ref{zerocup}.
\end{proof}

\subsection{Construction of the isomorphism} We are now ready to construct the isomorphism of functors on $\mathcal{T}$ from $\overline{Y}_{+, \text{tor}}$ to $H^{1}(\gerbeE)$. 

Choose an increasing tower $E_{k}$ of finite Galois extensions of $F$ and cofinal sequence $\{n_{k}\}$ in $\mathbb{N}^{\times}$, with associated prime-to-$p$ sequence $\{n_{k}'\}$. Choose a sequence of 2-cocycles $c_{k}$ representing the canonical classes in each $H^{2}(\Gamma_{E_{k}/F}, E_{k}^{*})$ as in \cite[\S 4.4]{Tasho}, which we will identify with their corresponding \v{C}ech 2-cocycles, and maps $l_{k} \colon (F^{s})^{*} \to (F^{s})^{*}$ satisfying $l_{k}(x)^{n_{k}'} = x$ and $l_{k+1}(x)^{n_{k+1}'/n_{k}'} = l_{k}(x)$ for all $x \in (F^{s})^{*}$. For $K/F$ a finite Galois extension, we may also view $l_{k}$ as a map on \v{C}ech-cochains $C^{n}(K/F, \mathbb{G}_{m}) \to C^{n}(F^{s}/F, \mathbb{G}_{m})$ by identifying the left-hand side with $\prod_{\underline{\sigma} \in \Gamma_{K/F}^{n}} K^{*}_{\underline{\sigma}}$, applying $l_{k}$ to each coordinate, and then mapping by $t^{-1}$ to $L^{\bigotimes_{F}(n+1)}$, where $L/F$ is some finite Galois extension containing all the chosen $n_{k}'$-roots of the entries of $x$. 

As in \S 3, denote $u_{E_{k}/F, n_{k}}$ by $u_{k}$ and $R_{E_{k}/F}[n_{k}]$ by $R_{k}$. Recall the homomorphism $\delta_{e} \colon \mu_{n_{k}} \to R_{k}$ inducing a homomorphism $\delta_{e} \colon \mu_{n_{k}} \to u_{k}$ that is killed by the norm map for the group $\Gamma_{E_{k}/F}$ acting on $\Hom(\mu_{n_{k}}, u_{k})$.

Following \cite[\S 4.5]{Tasho}, we define $$\xi_{k} = d[(l_{k}c_{k})^{(1/p^{m_{k}})}] \underset{E_{k}/F}{\sqcup} \delta_{e} \in C^{2}(\bar{F}/F, u_{k}),$$ where for an $n$-cochain $x \in \mathbb{G}_{m}(\bar{F}^{\bigotimes_{F} (n+1)})$ with $n \geq 1$, we choose for every $p$-power $p^{m_{k}} := n_{k}/n_{k}'$ a $p^{m_{k}}$-root of $x$, denoted by $x^{(1/p^{m_{k}})}$, satisfying $(x^{(1/p^{m_{k+1}})})^{p^{m_{k+1}}/p^{m_{k}}} = x^{(1/p^{m_{k}})}$ and if $x \in \bar{F} \otimes_{F} \bar{F} \otimes_{F} \dots \otimes_{F} \bar{F} \otimes_{F} E$ for $E/F$ a finite Galois extension, then $x^{(1/p^{m_{k}})} \in \bar{F} \otimes_{F} \bar{F} \otimes_{F} \dots \otimes_{F} \bar{F} \otimes_{F} E$ as well. We briefly explain why such a choice of roots exists, starting with a fixed $m = m_{k}$. Denote $F^{(p^{m})}$ the (purely-inseparable) extension of $F$ obtained by adjoining all $p^{m}$-power roots. The key input is the following basic fact about these extensions: 

\begin{lem}\label{purelyinsep} For any uniformizer $\varpi \in F$, we have $F^{(p^{m})} = F(\varpi^{1/p^{m}})$ (in particular, $[F^{(p^{m})} \colon F] = p^{m}$). Moreover, for any finite Galois extension $E/F$, the inclusion $F^{(p^{m})} \cdot E \subset E^{(p^{m})}$ (inside our fixed algebraic closure $\bar{F}$) is an equality. 
\end{lem}

\begin{proof} The first statement in a well-known fact about local function fields (choose a uniformizer and identify $F$ with Laurent series in one variable).  For the second statement, the case $m=1$ is clear since $E^{(p)}/E$ has degree $p$ (by the first part), and $F^{(p)} \cdot E \neq E$. The general case then follows by induction if we replace $E/F$ by $E^{(p^{m-1})}/F^{(p^{m-1})}$.
\end{proof}

Since the $p^{m}$-power map is a ring isomorphism on $\bar{F}^{\bigotimes_{F} (n+1)}$, it's enough to find such a root for a simple tensor $t$ whose last tensor factor lies in $E$. Choosing a uniformizer $\varpi$ of $F$, we have $F^{(p^{m})} = F(\varpi^{1/p^{m}})$ and moreover, since $E/F$ is Galois, we have $E^{(p^{m})} = E(\varpi^{1/p^{m}})$ as well, by Lemma \ref{purelyinsep}. For notational ease, set $ \varpi^{1/p^{m_{k}}} = \tilde{\varpi}_{k} = \tilde{\varpi} $. Evidently there is a $p^{m}$th root $t'$ of $t$ which is a simple tensor in $\bar{F} \otimes_{F} \bar{F} \otimes \dots \otimes_{F} \bar{F} \otimes_{F} E^{(p^{m})}$,  which we write as $$t' = z_{1} \otimes z_{2} \otimes \dots \otimes z_{n} \otimes (\sum_{i=1}^{p^{m}} c_{i}\tilde{\varpi}^{i}),$$ where $c_{i} \in E$. Then the element $$\sum_{i=1}^{p^{m}} (z_{1}\tilde{\varpi}^{i} \otimes z_{2} \otimes \dots \otimes z_{n} \otimes c_{i} )\in \bar{F} \otimes_{F} \bar{F} \otimes \dots \otimes_{F} \bar{F} \otimes_{F} E$$ has the same $p^{m}$th power as $t'$, since $\varpi \in F$, giving the desired root. Compatibility for varying $k$ follows from the fact that $\tilde{\varpi}_{k+1}^{p^{m_{k+1}}/p^{m_{k}}} = \tilde{\varpi}_{k}$ and, for $1 \leq i \leq p^{m_{k}}$, $c_{i,k} = \sum_{j=1}^{p^{m_{k+1}-m_{k}}} c_{i+jp^{m_{k}},k+1}$ (by uniqueness of $p$-power roots in $\bar{F}$).


Now that we have established that the $2$-cochain $\xi_{k}$ is well-defined, we note that it is in fact a $2$-cocycle, since (using Proposition \ref{bigcup} and the fact that $\delta_{e}$ is a \v{C}ech $-1$-cocycle)
$$d\xi_{k} = d(d[(l_{k}c_{k})^{(1/p^{m_{k}})}]) \underset{E_{k}/F}{\sqcup} \delta_{e} + d[(l_{k}c_{k})^{(1/p^{m_{k}})}] \underset{E_{k}/F}{\sqcup} d(\delta_{e}) = 0.$$

For ease of notation, denote $(l_{k}c_{k})^{(1/p^{m_{k}})}$ by $\widetilde{l_{k}c_{k}}$, which we view as a \v{C}ech 2-cochain valued in $\mathbb{G}_{m}(U_{2})$. To ensure that the above definition makes sense, we need to verify that $l_{k}c_{k} \in C^{2,1}(F^{s}/F, E_{k}, \mathbb{G}_{m})$ and $(l_{k}c_{k})^{(1/p^{m_{k}})} \in  C^{2,1}(\bar{F}/F, E_{k}, \mathbb{G}_{m}).$ The first inclusion follows from looking at the corresponding Galois $n$-cochain, as in \cite{Tasho}, and the second inclusion follows from the first and the construction of the $(-)^{(1/p_{m_{k}})}$-maps.

Define the torus $S_{E_{k}/F}$ to be the quotient of $\text{Res}_{E_{k}/F}(\mathbb{G}_{m})$ by the diagonally-embedded $\mathbb{G}_{m}$; it's clear that $u_{k}$ is the subgroup $S_{E_{k}/F}[n_{k}]$ of $n_{k}$-torsion points.  Define $$\alpha_{k}' = (l_{k}c_{k}\underset{E_{k}/F}{\sqcup} \delta_{e,k})^{-1} \cdot p_{k+1,k}'(l_{k+1}c_{k+1} \underset{E_{k+1}/F}{\sqcup} \delta_{e,k+1}) \in C^{1}(F^{s}/F, S_{E_{k}/F})$$ and $$\alpha_{k} = (\widetilde{l_{k}c_{k}}\underset{E_{k}/F}{\sqcup} \delta_{e,k})^{-1} \cdot p_{k+1,k}(\widetilde{l_{k+1}c_{k+1}} \underset{E_{k+1}/F}{\sqcup} \delta_{e,k+1}) \in C^{1}(\bar{F}/F, S_{E_{k}/F}),$$ where by $p_{k+1,k}$ we mean the map from $S_{E_{k+1}/F}$ to $S_{E_{k}/F}$ induced by the homomorphism of $\Gamma$-modules $\Z[\Gamma_{E_{k}/F}]_{0} \to \Z[\Gamma_{E_{k+1}/F}]_{0}$ defined by $[\gamma] \mapsto (n_{k+1}/n_{k})\sum_{\sigma \mapsto \gamma} [\sigma]$, similarly with $p_{k+1,k}'$. By $\delta_{e,k}$ we mean the extension of $\delta_{e} \colon \mu_{n_{k}} \to u_{k}$ to the map $\mathbb{G}_{m} \to S_{E_{k}/F}$ defined on $\Gamma$-modules by $\Z[\Gamma_{E_{k}/F}] \to \Z$ the evaluation at $[e]$ map. Note that this is not in general $\Gamma$-equivariant, but is still killed by the norm $N_{E_{k}/F}$.

\begin{lem}\label{techcup} \begin{enumerate} \item{The cochain $\alpha_{k}$ takes values in $u_{k}$ and the equality $d \alpha_{k} = p_{k+1,k}(\xi_{k+1}) \xi_{k}^{-1}$ holds in $C^{2}(\bar{F}/F, u_{k})$.}
\item{The element $([\xi_{k}])$ of $\varprojlim H^{2}(F, u_{k})$ is equal to the canonical class $\xi$.}
\end{enumerate}
\end{lem}

\begin{proof} We start by proving (1). To show that $\alpha_{k} \in u_{k}(\bar{F} \otimes_{F} \bar{F}) = S_{E_{k}/F}[n_{k}](\bar{F} \otimes_{F} \bar{F})$, it's enough to show that $\alpha_{k}^{p^{m_{k}}} \in S_{E_{k}/F}[n_{k}'](\bar{F} \otimes_{F} \bar{F})$. By construction, $$\alpha_{k}^{p^{m_{k}}} =  (l_{k}c_{k}\underset{E_{k}/F}{\sqcup} \delta_{e,k})^{-1} \cdot p_{k+1,k}([\widetilde{l_{k+1}c_{k+1}}]^{p^{m_{k}}} \underset{E_{k+1}/F}{\sqcup} \delta_{e,k+1}) = \alpha_{k}',$$ since $p_{k+1,k}$ is $p_{k+1,k}'$ pre-composed with the $p^{m_{k+1}}/p^{m_{k}}$-power map on $S_{E_{k+1}/F}$. Thus, it's enough to show that $\alpha_{k}' \in S_{E_{k}/F}[n_{k}'](F^{s} \otimes_{F} F^{s})$, which follows from \cite[Lemma 4.5]{Tasho}.

To show the second part of (1), we note by Proposition \ref{bigcup} that $$d(\widetilde{l_{k}c_{k}}\underset{E_{k}/F}{\sqcup} \delta_{e,k}) = d(\widetilde{l_{k}c_{k}}) \underset{E_{k}/F}{\sqcup} \delta_{e} = \xi_{k},$$ since $\delta_{e,k}$ is killed by $N_{E_{k}/F}$, and now the result follows because $p_{k+1,k}$ commutes with $d$ (it is defined over $F$).

It remains to prove (2). As in the analogous part of  the proof of \cite[Lemma 4.5]{Tasho}, it's enough to show that under the isomorphism $H^{2}(F, u_{k}) \to H^{0}(\Gamma, X^{*}(u_{k}))^{*} \to \Z / (n_{k}, [E_{k} : F]) \Z$ used in the proof of Theorem \ref{maincohom}, the class of $\xi_{k}$ maps to the element $-1$. Consider the cup product of $\xi_{k}$ with the element $\frac{n_{k}}{(n_{k}, [E_{k}: F])} \in \frac{n_{k}}{(n_{k}, [E_{k}: F])} \Z / n_{k} \Z \cong H^{0}(\Gamma, X^{*}(u_{k}))$, which we denote by $\chi \in H^{0}(\Gamma, X^{*}(u_{k}))$. Setting $A_{1} =  \mu_{n_{k}}$, $A_{3} = u_{k}$, and $A_{2} = \mathbb{G}_{m}$ in the context of Lemma \ref{pullcup}, we have $\xi_{k} \cup \chi  = d\widetilde{l_{k}c_{k}} \underset{E_{k}/F}{\sqcup} \chi \circ \delta_{e}$ by that same Lemma.

Note that $\chi \circ \delta_{e} \colon \mu_{n_{k}} \to \mathbb{G}_{m}$ is the $n_{k}/(n_{k}, [E_{k} \colon F])$-power map, which is the restriction of the same map $\mathbb{G}_{m} \xrightarrow{\gamma_{k}} \mathbb{G}_{m}$, and we may replace $d\widetilde{l_{k}c_{k}} \underset{E_{k}/F}{\sqcup} \chi \circ \delta_{e}$ with $d\widetilde{l_{k}c_{k}} \underset{E_{k}/F}{\sqcup} \gamma_{k}$, which (via Proposition \ref{bigcup}) is cohomologous to $\widetilde{l_{k}c_{k}}^{-1} \cup [E_{k} \colon F]  \gamma_{k}$, where in this last cup product we view $\gamma_{k}$ as a $0$-cocycle. Applying Lemma \ref{zerocup} (with $A_{1} = A_{2} = \mathbb{G}_{m})$ to this last expression yields $\widetilde{l_{k}c_{k}}^{-r_{k}}$, where $$r_{k} := \frac{[E_{k} \colon F]n_{k}}{(n_{k}, [E_{k} \colon F])}.$$ Since $n_{k} \mid r_{k}$, we have $\widetilde{l_{k}c_{k}}^{-r_{k}} = c_{k}^{-r_{k}/n_{k}}$, which has invariant $$\frac{-r_{k}}{n_{k}[E_{k} \colon F]} = \frac{-1}{ (n_{k}, [E_{k} \colon F])}.$$ This exactly gives that cupping with $\xi_{k}$ sends $\chi$ to $-1/(n_{k}, [E_{k} \colon F])$ under the pairing used in the proof of Theorem \ref{maincohom}, giving the result.
\end{proof}

For fixed $k \in \mathbb{N}$, denote by $\gerbeE_{k} := \gerbeE_{\xi_{k}}$ the $u_{k}$-gerbe corresponding to the \v{C}ech 2-cocycle $\xi_{k}$.  For any fixed $k$ we have a morphism of $F$-stacks $\pi_{k+1,k} \colon \gerbeE_{k+1} \to \gerbeE_{k}$ given by $\phi_{\xi_{k+1},\xi_{k},\alpha_{k}}$, obtained by combining Lemma \ref{techcup} with Construction \ref{changeofgerbe}. In fact, the systems $(\xi_{k})_{k}$ and $(\alpha_{k})_{k}$, along with the groups $u_{k}$ and gerbes $\gerbeE_{k}$, exactly satisfy the assumptions made in \S 2.7, our subsection on inverse limits of gerbes. Thus, as a consequence of Proposition \ref{invlim1}, we may take $\gerbeE := \varprojlim_{k} \gerbeE_{k}$ to be the gerbe used to define the groups $H^{1}(\gerbeE, Z \to S)$ for $[Z \to S]$ in $\mathcal{T}$. 

We are now ready to begin describing the Tate-Nakayama isomorphism. For a fixed $[Z \to S]$ in $\mathcal{T}$, let $k$ be large enough so that $E_{k}$ splits $S$ and $|Z|$ divides $n_{k}$. Let $\bar{\lambda} \in \overline{Y}^{N_{E_{k}/F}}$, and $\phi_{\bar{\lambda}, k} \in \Hom_{F}(u_{k}, Z)$ be its image under the isomorphism $$[\overline{Y}/Y]^{N_{E_{k}/F}} \to \Hom(\mu_{n_{k}}, Z)^{N_{E_{k}/F}} \to \Hom_{F}(u_{k}, Z).$$

Define a $\xi_{k}$-twisted $S$-torsor on $\bar{F}$ as follows. Take the trivial $S_{\bar{F}}$-torsor $S_{\bar{F}}$, with $u_{k}$-action induced by the homomorphism $u_{k} \xrightarrow{\phi_{\bar{\lambda}, k}} S_{\bar{F}}$ and gluing map $S_{\bar{F} \otimes_{F} \bar{F}} \xrightarrow{\sim} S_{\bar{F} \otimes_{F} \bar{F}}$ given by left-translation by $z_{k, \bar{\lambda}} := \widetilde{l_{k} c_{k}} \underset{E_{k}/F}{\sqcup} n_{k} \bar{\lambda} \in S(\bar{F} \otimes_{F} \bar{F})$, where we view $n_{k} \bar{\lambda}$ as an element of $X_{*}(S)$ (this makes sense since $|Z|$ divides $n_{k}$). This gluing map is trivially $S$- and hence $u_{k}$-equivariant.

\begin{lem} The above $S_{\bar{F}}$-torsor with the specified $u_{k}$ -action and gluing map defines a $\xi_{k}$-twisted $S$-torsor, which we will denote by $Z_{k,\bar{\lambda}}$. Moreover, for every $k$, we have the equality of $\xi_{k+1}$-twisted $S$-torsors $$\pi_{k+1,k}^{*}Z_{k, \bar{\lambda}} = Z_{k+1, \bar{\lambda}}.$$

\end{lem}

\begin{proof} For the first statement, we just need to check that the above $S_{\bar{F}}$-torsor is $\xi_{k}$-twisted with respect to translation by $z_{\bar{\lambda}, k}$ on $S_{\bar{F} \otimes_{F} \bar{F}}$. Since $u_{k}$ acts via $\phi_{\bar{\lambda}, k}$, this is the same as showing that $d(z_{\bar{\lambda}, k}) = \phi_{\bar{\lambda}, k}(\xi_{k})$. Since $\bar{\lambda}$ is killed by $N_{E_{k}/F}$, so is $n_{k} \bar{\lambda} $, and hence by Proposition \ref{bigcup} we have $d(\widetilde{l_{k} c_{k}} \underset{E_{k}/F}{\sqcup} n_{k} \bar{\lambda}) = (d\widetilde{l_{k} c_{k}}) \underset{E_{k}/F}{\sqcup} n_{k} \bar{\lambda}.$

Moreover, $\phi_{\bar{\lambda},k}$ is such that $\phi_{\bar{\lambda},k} \circ \delta_{e}$ equals the restriction of $n_{k} \bar{\lambda}$ to $\mu_{n_{k}}$, and so by Lemma \ref{pullcup} (with $A_{1} = \mu_{n_{k}}$, $A_{2} = S$, and $A_{3} = u_{k}$), since $\phi_{\bar{\lambda},k}$ is defined over $F$, we obtain  $$(d\widetilde{l_{k} c_{k}}) \underset{E_{k}/F}{\sqcup} n_{k} \bar{\lambda} = \phi_{\bar{\lambda},k}[(d\widetilde{l_{k} c_{k}}) \underset{E_{k}/F}{\sqcup} \delta_{e}] = \phi_{\bar{\lambda},k}(\xi_{k}),$$ as desired. We thus get our $\xi_{k}$-twisted $S$-torsor $Z_{\bar{\lambda}, k}$.

We now want to compare the pullback $\pi_{k+1,k}^{*}Z_{\bar{\lambda}, k}$ to $Z_{\bar{\lambda}, k+1}.$ As $S_{\bar{F}}$-torsors, these are both trivial, so it's enough to show that the $u_{k+1}$-actions coincide, and that the difference of the two gluing maps is the identity in $S(\bar{F} \otimes_{F} \bar{F})$. By Corollary \ref{pulltwist}, the $u_{k+1}$-action on $\pi_{k+1,k}^{*}Z_{\bar{\lambda}, k}$ is given by the homomorphism $u_{k+1} \xrightarrow{\phi_{\bar{\lambda},k} \circ p_{k+1,k}} S_{\bar{F}}$ and the $u_{k+1}$-action on $Z_{\bar{\lambda}, k}$ is given by $\phi_{\bar{\lambda},k+1}$. One checks easily that $\phi_{\bar{\lambda},k+1} = \phi_{\bar{\lambda},k} \circ p_{k+1,k}$, so the $u_{k+1}$-actions coincide.

Corollary \ref{pulltwist} also tells us that the twisted gluing map for $\pi_{k+1,k}^{*}Z_{\bar{\lambda}, k}$ is left-translation on $S_{\bar{F}}$ by $\phi_{\bar{\lambda}, k}(\alpha_{k}) \cdot z_{\bar{\lambda},k} \in S(\bar{F} \otimes_{F} \bar{F})$, and for $Z_{\bar{\lambda}, k+1}$ is left-translation by $z_{\bar{\lambda},k+1}$. We want to look at $$z_{\bar{\lambda},k} \cdot  \phi_{\bar{\lambda},k}(\alpha_{k}) \cdot z_{\bar{\lambda},k+1}^{-1}= \phi_{\bar{\lambda},k}(\alpha_{k}) \cdot (\widetilde{l_{k+1} c_{k+1}} \underset{E_{k+1}/F}{\sqcup} n_{k+1} \bar{\lambda})^{-1} \cdot \widetilde{l_{k} c_{k}} \underset{E_{k}/F}{\sqcup} n_{k} \bar{\lambda}.$$

Recall (since $p_{k+1,k}$ is defined over $F$) that $$\alpha_{k} = (\widetilde{l_{k}c_{k}}\underset{E_{k}/F}{\sqcup} \delta_{e,k})^{-1} \cdot (\widetilde{l_{k+1}c_{k+1}} \underset{E_{k+1}/F}{\sqcup} p_{k+1,k}\circ \delta_{e,k+1});$$ note that the map $u_{k} \xrightarrow{\phi_{\bar{\lambda},k}} S$ extends canonically to the map $S_{E_{k}/F} \xrightarrow{\phi_{\bar{\lambda},k}} S$ given by the preimage of $n_{k} \bar{\lambda}$ via the identification $\Hom_{F}(S_{E_{k}/F}, S) \xrightarrow{\sim} Y^{N_{E_{k}/F}}$, which, at the level of functions, sends a cocharacter $f$ to $f \circ \delta_{e,k}$, where $\delta_{e,k}$ is defined analogously on $S_{E_{k}/F}$ as for $u_{k}$. We may pull $\phi_{\bar{\lambda},k}$ inside both cup products to obtain $$\phi_{\bar{\lambda},k}(\alpha_{k}) =  (\widetilde{l_{k}c_{k}}\underset{E_{k}/F}{\sqcup} n_{k}\bar{\lambda})^{-1} \cdot (\widetilde{l_{k+1}c_{k+1}} \underset{E_{k+1}/F}{\sqcup}  \phi_{\bar{\lambda},k} \circ p_{k+1,k} \circ \delta_{e,k+1}).$$ Since $\phi_{\bar{\lambda},k} \circ p_{k+1,k} = \phi_{\bar{\lambda},k+1}$, the above is exactly $z_{\bar{\lambda},k}^{-1} \cdot z_{\bar{\lambda},k+1}$, so we are done.
\end{proof}

Again choosing $k \in \mathbb{N}$ such that $E_{k}$ splits $S$ and $|Z|$ divides $n_{k}$, we may define an $S_{\gerbeE}$-torsor on $\gerbeE$ by pulling back $Z_{k, \bar{\lambda}}$ (identifying this $\xi_{k}$-twisted $S$-torsor with an $S_{\gerbeE_{k}}$-torsor on $\gerbeE_{k}$ as in Proposition \ref{gerbesheaf}) to $\gerbeE$ via the projection map $\pi_{k} \colon \gerbeE \to \gerbeE_{k}$. By the above lemma, this does not depend on the choice of $k$, and so we denote this torsor simply by $Z_{\bar{\lambda}}$. We are now in a position to prove the main result. The statement and proof largely follow the analogous result in \cite{Tasho}, which is that paper's Theorem 4.8. 

\begin{thm}\label{tatenak1} The assignment $\bar{\lambda} \mapsto Z_{\bar{\lambda}}$ induces an isomorphism
$$\iota \colon \overline{Y}_{+,\text{tor}} \to H^{1}(\gerbeE)$$ of functors $\mathcal{T} \to \text{AbGrp}$. This isomorphism coincides with the Tate-Nakayama isomorphism for objects $[1 \to S]$ in $\mathcal{T}$ and lifts the morphism from $\overline{Y}_{+,\text{tor}}$ to $\Hom_{F}(u,-)$ described earlier in the subsection. Moreover, $\iota$ is the unique isomorphism between these two functors satisfying the above two properties.
\end{thm}

\begin{proof} This assignment is clearly additive in $\bar{\lambda}$, and so it defines a group homomorphism from $\overline{Y}^{N}$ to $H^{1}(\gerbeE, Z \to S)$ for any object $[Z \to S]$ of $\mathcal{T}$. Moreover, any morphism $[Z \to S] \xrightarrow{h} [Z' \to S']$ in $\mathcal{T}$ induces the morphism $H^{1}(\gerbeE, Z \to S) \to H^{1}(\gerbeE, Z' \to S')$ sending the class of $\pi_{k}^{*}Z_{\bar{\lambda},k}$ (for suitable $k$, as discussed above) to that of $\pi_{k}^{*}( Z_{\bar{\lambda},k} \times^{h, S} S')$, and so it is enough to show that $Z_{\bar{\lambda},k} \times^{h, S} S$ is isomorphic  to $Z_{h^{\sharp}(\bar{\lambda}),k}$ as $\xi_{k}$-twisted $S'$-torsors. Note that $Z_{\bar{\lambda},k} \times^{h, S} S' $ is evidently trivial as an $S'_{\bar{F}}$-torsor, and has $u_{k}$-action given by $h \circ \phi_{\bar{\lambda},k}$, whereas $Z_{h^{\sharp}(\bar{\lambda}),k}$ has $u_{k}$-action given by $  \phi_{h^{\sharp}\bar{\lambda},k} =  h \circ \phi_{\bar{\lambda},k} $, since if $\phi_{\bar{\lambda},k} \circ \delta_{e} = n_{k} \bar{\lambda}$, then $h \circ (\phi_{\bar{\lambda},k} \circ \delta_{e})= h \circ n_{k} \bar{\lambda} = h^{\sharp}\bar{\lambda}$. Finally, one checks by a similar argument that $h(z_{\bar{\lambda}, k}) = z_{h^{\sharp}\bar{\lambda}, k}$, giving the desired equality of torsors, and hence that the assignment of the theorem gives a morphism of functors from $\overline{Y}^{N}$ to $H^{1}(\gerbeE)$.

We need to check that for $[Z \to S]$ in $\mathcal{T}$ fixed, the homomorphism $\overline{Y}^{N} \to H^{1}(\gerbeE, Z \to S)$ descends to the quotient $\overline{Y}_{+, \text{tor}} = \overline{Y}^{N}/IY$. To this end, suppose that $\bar{\lambda} \in \overline{Y}^{N}$ lies in $Y$. Then (choosing $k$ large enough) by \S 4.1, $\phi_{\bar{\lambda}, k}$ is trivial, and moreover $$z_{\bar{\lambda},k} = \widetilde{l_{k} c_{k}} \underset{E_{k}/F}{\sqcup} n_{k} \bar{\lambda} = c_{k}  \underset{E_{k}/F}{\sqcup} \bar{\lambda}.$$ Note that $c_{k} \in \mathbb{G}_{m}(E_{k} \otimes_{F} E_{k})$, and hence by Proposition \ref{comparecup}, this unbalanced cup product may be computed using the definition given in \cite{Tasho}, working with Galois cohomology. By \cite[\S 4.3]{Tasho}, this coincides with the usual cup product in finite Tate cohomology with respect to the group $\Gamma_{E_{k}/F}$, and thus yields the image of $\bar{\lambda}$ induced by the Tate-Nakayama isomorphism $X_{*}(S)^{N_{k}} \to [X_{*}(S)/IX_{*}(S)]^{N_{k}} \xrightarrow{\sim} H^{1}(\Gamma_{E_{k}/F}, S(E_{k})) = H^{1}(F, S)$. As a consequence, if $\bar{\lambda} \in IY$, then $z_{\bar{\lambda},k} = 1$, and so $Z_{\bar{\lambda},k}$ is given by the trivial $S_{\bar{F}}$-torsor with trivial $u_{k}$-action and gluing map equal to the identity, thus yielding the trivial $\xi_{k}$-twisted $S$-torsor on $\gerbeE_{k}$, as desired.

The argument of the above paragraph also shows that if we take $[1 \to S] \in \mathcal{T}$, then $\overline{Y}_{+, \text{tor}}[1 \to S] = Y/IY$ and the homomorphism $Y/IY \to H^{1}(\gerbeE, 1 \to S) = H^1(F,S)$ is exactly the Tate-Nakayama isomorphism. For the morphism of functors on $\mathcal{T}$ from $\overline{Y}_{+, \text{tor}}$ to $\Hom_{F}(u,-)$ sending $\bar{\lambda}$ to $\phi_{\bar{\lambda},k} \circ p_{k}$, we have already discussed that the image of $\pi_{k}^{*}Z_{\bar{\lambda},k}$ under the restriction morphism $H^1(\gerbeE, Z \to S) \to \Hom_{F}(u, Z)$ equals $\phi_{\bar{\lambda},k} \circ p_{k}$, giving the desired compatibility of morphisms of functors to $\Hom_{F}(u,-)$.

The final thing to show is that for $[Z \to S]$ fixed, the assignment of the theorem yields an isomorphism from $\overline{Y}_{+,\text{tor}}$ to $H^{1}(\gerbeE, Z \to S)$. As in \cite{Tasho}, consider the diagram 
\begin{center}
\begin{tikzcd}
  0 \arrow[r] & H^{1}(F,S) \arrow[r] & H^{1}(\gerbeE, Z \to S)  \arrow[r] &\Hom_{F}(u,Z) \arrow[r] & H^{2}(F,S)   \\

  0 \arrow[r] &Y_{\Gamma,\text{tor}} \arrow[r] \arrow[u] & \overline{Y}_{+,\text{tor}} \arrow[u]  \arrow[r] & \varinjlim[\overline{Y}/Y]^{N_{k}} \arrow[r] \arrow[u] & \varinjlim Y^{\Gamma}/N_{k}(Y) \arrow[u],
\end{tikzcd}
\end{center}
where the top horizontal sequence is just inflation-restriction, the first lower-horizontal map is induced by the inclusion $X_{*}(S) \to X_{*}(\overline{S})$, the second is induced by the maps $\overline{Y}_{+,\text{tor}} = \overline{Y}^{N_{k}}/I_{k}Y \to [\overline{Y}/Y]^{N_{k}}$, and the third is induced by the maps $[\overline{Y}/Y]^{N_{k}} \to Y^{\Gamma}/N_{k}(Y)$ given by $[\bar{\lambda}] \mapsto [N_{k}(\bar{\lambda})]$. It's a straightforward exercise in group cohomology to check that the bottom horizontal sequence is exact. The first vertical map is the Tate-Nakayama isomorphism, the second vertical map is the assignment $\bar{\lambda} \mapsto Z_{\bar{\lambda}}$, the third vertical map is induced by the system of maps $[\overline{Y}/Y]^{N_{k}} \to \Hom_{F}(u_{k},Z) \to \Hom_{F}(u,Z)$ discussed in \S 4.1, and the final vertical map is induced by the system of negative Tate-Nakayama isomorphisms $H^{0}_{\text{Tate}}(\Gamma_{E_{k}/F}, Y) \xrightarrow{\sim} H^2(\Gamma_{E_{k}/F}, S(E_{k})) \xrightarrow{\text{Inf}} H^{2}(F, S)$. 

We claim that this diagram commutes; the first square commutes by our above discussion of the compatibility with the Tate-Nakayama isomorphism, and the middle square commutes by compatibility between the two morphisms of functors to $\Hom_{F}(u, -)$. Thus, we only need to show that the right-hand square commutes. It's enough to do this for a sufficiently large fixed $k$ and $u$ replaced by $u_{k}$, because any $\phi \colon u \to Z$ factors through some $\phi_{k} \colon u_{k} \to Z$, and then $\phi(\xi) = \phi_{k}(p_{k}(\xi))= [\phi_{k}(\xi_{k})]$ in $H^{2}(F,Z)$, since $[p_{k}(\xi)] = [\xi_{k}]$ in $\check{H}^{2}(\bar{F}/F, u_{k})$ (by construction). Fix $\bar{\lambda} \in \overline{Y}$ whose norm lies in $Y$. Then its image in $\Hom_{F}(u_{k}, Z)$ is $\phi_{\bar{\lambda},k}$, which, by Lemma \ref{translem}, maps under the transgression map to the image of the class $[\phi_{\bar{\lambda},k}(\xi_{k})] \in H^2(F, Z)$ in $H^2(F,S)$, which equals the class of $(d\widetilde{l_{k} c_{k}}) \underset{E_{k}/F}{\sqcup} n_{k} \bar{\lambda}$, since we may pull $\phi_{\bar{\lambda},k}$ inside the cup product defining $\xi_{k}$ by Lemma \ref{pullcup}.

On the other hand, if we take $N_{k}(\bar{\lambda}) \in Y^{\Gamma} = Y^{\Gamma_{E_{k}/F}}$, then its image under the Tate-Nakayama map $ Y^{\Gamma_{E_{k}/F}} \to H^2(\Gamma_{E_{k}/F}, S(E_{k}))$ is obtained by taking the cup product with the class $[c_{k}] \in H^{2}(\Gamma_{E_{k}/F}, E_{k}^{*}) \xrightarrow{\text{Inf}} H^{2}(F, \mathbb{G}_{m})$. I.e., we obtain the class of the cocycle $(c_{k} \cup N_{k}(\bar{\lambda}))^{-1}$ in $H^{2}(F, S)$. By Proposition \ref{bigcup}, $(d\widetilde{l_{k} c_{k}}) \underset{E_{k}/F}{\sqcup} n_{k} \bar{\lambda}$ is cohomologous in $H^{2}(F,S)$ to $[\widetilde{l_{k}c_{k}} \cup d(n_{k} \bar{\lambda})]^{-1}$, which, since $N_{k}(\bar{\lambda}) \in Y$, equals $(c_{k} \cup N_{k}(\bar{\lambda}))^{-1}$, giving the claim.

The first and third vertical maps are isomorphisms, and the last vertical map is injective, and so by the five-lemma we get that the second vertical map is an isomorphism. The uniqueness of $\iota$ satisfying the two properties of the theorem follows from the argument for the analogous result in the characteristic zero case in \cite[\S 4.2]{Tasho}.
\end{proof}

\section{Extending to reductive groups}

In order to apply the above cohomological results to the local Langlands correspondence, it is necessary to extend the above constructions to connected reductive groups over a local function field $F$. We use the same notation as above; $\gerbeE$ will always be a $u$-gerbe split over $\bar{F}$ with $[\gerbeE] = \xi$. We start by briefly recalling some fundamental cohomological results on reductive algebraic groups over $F$ a local function field. 

\begin{thm}\label{kneser} For any simply-connected reductive group $G$ over a local field $F$, $H^{1}(F,G) = 0$. 
\end{thm}
This is \cite[Theorem 5]{Serre1}.

\begin{thm}\label{connect} Let $G$ be a semisimple group over $F$ a local field, and let $C$ denote the kernel of the central isogeny $G_{\text{sc}} \to G$. Then the natural map $H^{1}(F, G) \to H^{2}(F, C)$ is a bijection, thus endowing $H^{1}(F,G)$ with the canonical structure of an abelian group.
\end{thm}
This is \cite[Theorem 2.4]{Thang1}. 

\noindent The arguments in \cite{Tasho} which extend the Tate-Nakayama isomorphism of \S 4 to reductive groups rely heavily on the existence of elliptic/fundamental maximal tori (see \cite[\S 10]{Kott86}), and their corresponding cohomological properties. 

\begin{thm} Every semisimple algebraic group over a local function field $F$ contains a maximal $F$-torus $T$ which is anisotropic over $F$.
\end{thm}
This follows from \cite[\S 2.4]{Debacker}. It follows immediately that every reductive group $G$ contains a maximal $F$-torus which is $F$-anisotropic modulo $Z(G)^{\circ}$; this will be an elliptic maximal torus. 

Moreover, we have the following result for $G$ a connected reductive group over $F$, implied by the proof of \cite[Lemma 10.2]{Kott86} and Theorem \ref{connect}: 

\begin{thm}\label{Kott10.2} If $T$ is an elliptic maximal torus of $G$, then $H^{1}(F, T) \to H^{1}(F, G)$ is surjective. 
\end{thm}

We also have the following, which is a generalization of \cite[Theorem 1.2]{Kott86}; it concerns the functor $\mathcal{A}$ from the category of connected reductive $F$-groups to abelian groups, defined by $\mathcal{A}(G) = \pi_{0}(Z(\widehat{G})^{\Gamma})^{*}$, where $\widehat{G}$ denotes a Langlands dual group of $G$. Recall that Tate-Nakayama duality gives us an isomorphism $H^{1}(F, T) \xrightarrow{\sim} \pi_{0}(\widehat{T}^{\Gamma})^{*}$ for any $F$-torus $T$ (this will be reviewed in more detail in \S 6.1).

\begin{thm}\label{kottwitzpairing} There is a unique extension of the above isomorphism of functors to an isomorphism of functors on the category of reductive $F$-groups, given by a natural transformation $$\alpha_{G} \colon H^{1}(F, G) \to \mathcal{A}(G).$$ 
\end{thm}

\noindent This is \cite[Theorem 2.1]{Thang2}.  

We are now ready to extend our previous constructions on $\mathcal{T}$ to the category $\mathcal{R}$. For the most part, the arguments from \cite{Tasho} carry over verbatim, since most depend on the structure theory of reductive groups, in particular the part of the theory that deals with character and cocharacter modules, which is uniform for local fields of any characteristic. The purpose of the remainder of this section is to summarize those results and fill in certain arguments which are different in the case of a local function field.

\begin{prop}\label{bigdiag2} Proposition \ref{bigdiag} holds for $[Z \to G]$ in $\mathcal{R}$, ignoring the $H^2(F,G)$ terms and replacing exactness of sequences of abelian groups with exactness of sequences of pointed sets.
\end{prop}

\begin{proof} Everything from the proof of \ref{bigdiag} holds, except for the use of the five-lemma to give the surjectivity of $H^{1}(\gerbeE, Z \to G) \to H^{1}(F, \overline{G})$. Instead, we may use the analogous argument used in \cite[Proposition 3.6]{Tasho}, using the existence of an elliptic maximal torus in $G$ and replacing the use of \cite[Lemma 10.2]{Kott86} with Theorem \ref{Kott10.2}, its analogue for local function fields. 
\end{proof}

\begin{prop}\label{tashocor} (Analogue of \cite[Corllary 3.7]{Tasho}) \begin{enumerate}
\item{If G possesses anisotropic maximal tori, then the map $H^{1}(\gerbeE, Z \to G) \to \Hom_{F}(u,Z)$ defined above is surjective.}
\item{If $S \subset G$ is an elliptic maximal torus, then the map $$H^1(\gerbeE, Z \to S) \to H^1(\gerbeE, Z \to G)$$ is surjective.}
\end{enumerate}
\end{prop}

\begin{proof} The same proof as in \cite{Tasho} works, again replacing the use of \cite[Lemma 10.2]{Kott86} with Theorem \ref{Kott10.2}.\end{proof}

Let $[Z \to G] \in \mathcal{R}$. We need to extend the functor $\overline{Y}_{+, \text{tor}}$ defined in \S 4. Following \cite{Tasho}, $\overline{Y}_{+, \text{tor}}[Z \to G]$ is taken to be the limit over all maximal $F$-tori $S$ of $G$ of the following colimit: $$\varinjlim \frac{[X_{*}(S/Z)/X_{*}(S_{\text{\text{sc}}})]^{N}}{I (X_{*}(S)/X_{*}(S_{\text{\text{sc}}}))},$$
where the colimit is taken over the set of Galois extensions $E/F$ splitting $S$ and the superscript $N$ denotes the kernel of the norm map. We need to explain what the limit maps are between the above objects for varying $S$. For two such tori $S_{1}, S_{2}$, picking $g \in G(F^{s})$ such that $\text{Ad}(g)(S_{1})_{F^{s}} = (S_{2})_{F^{s}}$ induces an isomorphism $$\text{Ad}(g) \colon X_{*}(S_{1}/Z)/X_{*}((S_{1})_{\text{\text{sc}}}) \to X_{*}(S_{2}/Z)/X_{*}((S_{2})_{\text{\text{sc}}})$$ 
which is independent of the choice of $g$, by \cite[Lemma 4.2]{Tasho}, and is thus $\Gamma$-equivariant. It follows that these maps may be used to define the desired limit maps for varying maximal $F$-tori in $G$. 

The fact that a morphism $[Z \to G] \to [C \to H]$ in $\mathcal{R}$ induces a morphism $\overline{Y}_{+, \text{tor}}[Z \to G] \to \overline{Y}_{+, \text{tor}}[C \to H]$ follows from the identical argument as in \cite{Tasho} (the one following Lemma 4.2 loc. cit.) after noting that for an arbitrary field $F$, we can always uniquely lift a morphism $G \xrightarrow{f} H$ between two connected reductive groups over $F$ to a morphism $G_{\text{sc}} \to H_{\text{sc}}$. To see this, we may immediately reduce to the semisimple case, where the claim follows from \cite[Proposition 9.3.2]{Conrad}, using that $H_{\text{sc}}$ is a central extension of $H$ by a multiplicative group scheme.

We now extend the isomorphism of functors $\overline{Y}_{+, \text{tor}} \xrightarrow{\sim} H^{1}(\gerbeE)$ on $\mathcal{T}$ given in Theorem \ref{tatenak1} to the category $\mathcal{R}$. The strategy will be as follows: we will show that \cite[Lemmas 4.9, 4.10]{Tasho} hold in our setting, and then the result will follow from the proof of \cite[Theorem 4.11]{Tasho}, using the existence of elliptic maximal tori, as argued above, Proposition \ref{tashocor}, and the aforementioned lemmas. As in \S 4.3, we work with the specific choice of $\gerbeE$ given by $\varprojlim_{k} \gerbeE_{\xi_{k}}$ for $\xi_{k}$ as in \S 4; by the uniqueness of $H^{1}(\gerbeE, Z \to G)$ up to canonical isomorphism, this will prove the result for an arbitrary choice of $\gerbeE$. 

\begin{lem}\label{tasholem4.9} (Analogue of \cite[Lemma 4.9]{Tasho}) Let $[Z \to G] \in \mathcal{R}$ and $S \subset G$ a maximal torus. The fibers of the composition $$\overline{Y}_{+, \text{tor}}[Z \to S] \to H^{1}(\gerbeE, Z \to S) \to H^{1}(\gerbeE, Z \to G)$$ are torsors under the image of $X_{*}(S_{\text{sc}})_{\Gamma, \text{tor}}$ in  $\overline{Y}_{+, \text{tor}}[Z \to S]$.
\end{lem}

\begin{proof} The argument of \cite{Tasho} works here, replacing \cite[Theorem 1.2]{Kott86} with the analogue for local function fields, namely \cite[Theorem 2.1]{Thang2}.
\end{proof}

\begin{lem}\label{hardlem} (Analogue of \cite[Lemma 4.10]{Tasho}) Let $[Z \to G] \in \mathcal{R}$, and let $S_{1}, S_{2} \subset G$ be maximal tori defined over $F$. Let $g \in G(\bar{F})$ with $\text{Ad}(g)(S_{1})_{\bar{F}} = (S_{2})_{\bar{F}}$. If $\bar{\lambda}_{i} \in \overline{Y}_{i}^{N}$ are such that $\bar{\lambda}_{2} = \text{Ad}(g)\bar{\lambda}_{1}$, then the images of $\iota_{[Z \to S_{1}]}(\bar{\lambda}_{1})$ and  $\iota_{[Z \to S_{2}]}(\bar{\lambda}_{2})$ in $H^{1}(\gerbeE, Z \to G)$ coincide. 
\end{lem}

\begin{proof} This argument will require more substantial adjustments, so we recall some details of the argument in \cite{Tasho}. If $P_{i}^{\vee} := X_{*}(S_{i, \text{\text{ad}}})$, the isogeny $S_{i}/Z \to S_{i}/(Z \cdot Z(\mathscr{D}(G)))$ provides an injection $\overline{Y}_{i} \to P_{i}^{\vee} \oplus X_{*}(G/ Z \cdot \mathscr{D}(G))$; we write $\bar{\lambda}_{i} = u_{1} + z$ according to this decomposition, and so $\bar{\lambda}_{2} = u_{2} + z$, with $u_{2} = \text{Ad}(g)u_{1}$. As in \cite{Tasho}, we choose $k$ large enough so that $n_{k} u_{1} \in Q_{1}^{\vee} := X_{*}(S_{1,\text{sc}})$ and $n_{k} z \in X_{*}(Z(G)^{\circ})$ [via the isogeny $Z(G)^{\circ} \to G/ Z \cdot \mathscr{D}(G)$].

Our goal will be to show that $z_{\bar{\lambda}_{2},k} = p_{1}^{\sharp}(x) z_{\bar{\lambda}_{1},k} p_{2}^{\sharp}(x)^{-1}$ for some $x \in G_{\text{sc}}(\bar{F})$ (recall from \S 2.5 that this is what it means for two twisted \v{C}ech cocycles to be equivalent). We have that $\phi_{\bar{\lambda}_{1},k} = \phi_{\bar{\lambda}_{2},k}$ and $\widetilde{l_{k}c_{k}} \underset{E_{k}/F}{\sqcup} n_{k}z \in Z(G)^{\circ}(U_{2})$, and hence by decomposing $n_{k} \bar{\lambda}_{i} = n_{k}u_{i} + n_{k}z$ we see that it's enough to show that $a_{2} = p_{1}^{\sharp}(x) a_{1} p_{2}^{\sharp}(x)^{-1}$ for some $x \in G_{\text{sc}}(\bar{F})$, where $a_{i} := \widetilde{l_{k}c_{k}} \underset{E_{k}/F}{\sqcup} n_{k} u_{i}$ (this will show that the classes of $\iota_{[Z \to S_{1}]}(\bar{\lambda}_{1})$ and $\iota_{[Z \to S_{2}]}(\bar{\lambda}_{1})$ are equal in $H^{1}(\gerbeE_{a_{k}}, Z \to G)$, and hence have the same pullback to $H^{1}(\gerbeE, Z \to G)$).

The image of $a_{1} \in S_{1,\text{sc}}(U_{1})$ in $S_{1,\text{\text{ad}}}$ is equal to $c_{k} \cup u_{1}$ (the usual Galois cohomology cup product), and is thus a Galois 1-cocycle, so we can twist the $\Gamma$-structure on $G_{\text{sc}}$ using it, obtaining the twisted structure $G_{\text{sc}}^{1}$. By basic descent theory (see, for example, \cite[\S 4.5]{Poonen}), we have an $\bar{F}$-group isomorphism $$\phi \colon (G_{\text{sc}})_{\bar{F}} \xrightarrow{\sim} (G_{\text{sc}}^{1})_{\bar{F}}$$ satisfying $p_{1}^{*}\phi^{-1} \circ p_{2}^{*}\phi = \text{Ad}(a_{1})$ on $(G_{\text{sc}})_{U_{1}}$.

We claim now that $p_{1}^{*}\phi(a_{2} \cdot a_{1}^{-1})$ is a cocycle in $G^{1}_{\text{sc}}(U_{1})$. It's enough to check that the differential post-composed with the group isomorphism $q_{1}^{*}\phi^{-1}$ sends this element to the identity in $G_{\text{sc}}(U_{2})$. 

One computes (using the non-abelian \v{C}ech differential formulas, see \S 2.2, equation \eqref{cechboundaries}) that $$q_{1}^{*}\phi^{-1}(dp_{1}^{*}\phi(a_{2} \cdot a_{1}^{-1})) = q_{1}^{*}\phi^{-1}[p_{12}^{*}p_{1}^{*}\phi(p_{12}^{\sharp}(a_{2} \cdot a_{1}^{-1})) \cdot p_{23}^{*}p_{1}^{*}\phi(p_{23}^{\sharp}(a_{2} \cdot a_{1}^{-1})) \cdot (p_{13}^{*}p_{1}^{*}\phi(p_{13}^{\sharp}(a_{2} \cdot a_{1}^{-1})))^{-1}]. $$
Rewriting each composition of pullbacks in the usual way, this may be rewritten as: 
$$q_{1}^{*}\phi^{-1}[q_{1}^{*}\phi(p_{12}^{\sharp}(a_{2} \cdot a_{1}^{-1})) \cdot q_{2}^{*}\phi(p_{23}^{\sharp}(a_{2} \cdot a_{1}^{-1})) \cdot (q_{1}^{*}\phi(p_{13}^{\sharp}(a_{2} \cdot a_{1}^{-1})))^{-1}].$$
Now distributing $q_{1}^{*}\phi^{-1}$ to each term (since $\phi$ is a morphism of group sheaves) gives: 
$$p_{12}^{\sharp}(a_{2} \cdot a_{1}^{-1}) \cdot (q_{1}^{*}\phi^{-1} \circ q_{2}^{*}\phi)(p_{23}^{\sharp}(a_{2} \cdot a_{1}^{-1})) \cdot (p_{13}^{\sharp}(a_{2} \cdot a_{1}^{-1}))^{-1}.$$ Since $(q_{1}^{*}\phi^{-1} \circ q_{2}^{*}\phi) = p_{12}^{*}(p_{1}^{*}\phi^{-1} \circ p_{2}^{*}\phi) = p_{12}^{*}\text{Ad}(a_{1})$, the above element becomes 
$$p_{12}^{\sharp}(a_{2})p_{12}^{\sharp}(a_{1})^{-1} p_{12}^{\sharp}(a_{1}) p_{23}^{\sharp}(a_{2}) [p_{23}^{\sharp}(a_{1})^{-1} p_{12}^{\sharp}(a_{1})^{-1} p_{13}^{\sharp}(a_{1})] p_{13}^{\sharp}(a_{2})^{-1}.$$ The bracketed terms all lie in $S_{1,\text{sc}}(U_{2})$ and hence may be rearranged to give $da_{1}^{-1} \in Z(G_{\text{sc}})(U_{2})$. By centrality, this may then be moved to the front, yielding $da_{2} \in Z(G_{\text{sc}})(U_{2})$, giving us $da_{2} \cdot da_{1}^{-1}$. However, we know that $$da_{1} =  d\widetilde{l_{k}c_{k}} \underset{E_{k}/F}{\sqcup} n_{k} u_{1} = d\widetilde{l_{k}c_{k}} \underset{E_{k}/F}{\sqcup} n_{k} u_{2} = da_{2},$$ because the images of $u_{1}$ and $u_{2}$ under $P_{i}^{\vee} \to P_{i}^{\vee}/Q_{i}^{\vee} \to \Hom(\mu_{n}, Z(G_{\text{sc}}))$ coincide, showing the cocycle claim.

Since $G_{\text{sc}}^{1}$ is simply-connected, Theorem \ref{kneser} tells us that $p_{1}^{*}\phi(a_{2} \cdot a_{1}^{-1}) = d(\phi(x))$, some $x \in G_{\text{sc}}(\bar{F})$. One computes easily (using a similar but simpler calculation) as above that $$a_{2} \cdot a_{1}^{-1} = p_{1}^{*}\phi^{-1} d(\phi(x)) = p_{1}^{\sharp}(x)^{-1} a_{1} p_{2}^{\sharp}(x) a_{1}^{-1},$$ as desired. 
\end{proof}

We are now ready to prove the main result of the section.

\begin{thm}\label{tatenak2} (\cite[Theorem 4.11]{Tasho}) The isomorphism $\iota$ of Theorem \ref{tatenak1} extends to an isomorphism $$\iota \colon \overline{Y}_{+,\text{tor}} \to H^{1}(\gerbeE)$$ of functors $\mathcal{R} \to \text{Sets}$ which lifts the morphism of functors on $\mathcal{R}$ from $\overline{Y}_{+,\text{tor}} \to \Hom_{F}(u,-).$ 
\end{thm}

\begin{proof} We define the map in this proof for a fixed $[Z \to G] \in \mathcal{R}$; the fact that this map satisfies the statement of the theorem follows from the proof of the analogous result in \cite{Tasho} (the arguments loc. cit. work in our setting because of the above lemmas). Defining this isomorphism of functors will first require defining, for a fixed elliptic maximal torus $S$ of $G$ defined over $F$, a bijection 
$$ \varinjlim \frac{[X_{*}(S/Z)/X_{*}(S_{\text{\text{sc}}})]^{N}}{I (X_{*}(S)/X_{*}(S_{\text{\text{sc}}}))} \xrightarrow{\sim} H^{1}(\gerbeE, Z \to G).$$ For any maximal torus $S'$ and $E$ splitting $S'$, we have an exact sequence 
\[
\begin{tikzcd}
\frac{X_{*}(S'_{\text{\text{sc}}})^{N}}{IX_{*}(S'_{\text{\text{sc}}})} \arrow{r} & \frac{[X_{*}(S'/Z)]^{N}}{IX_{*}(S')} \arrow{r} & \frac{[X_{*}(S'/Z)/X_{*}(S'_{\text{\text{sc}}})]^{N}}{I (X_{*}(S')/X_{*}(S'_{\text{\text{sc}}}))} \arrow{r} & \frac{X_{*}(S'_{\text{\text{sc}}})^{\Gamma}}{N(X_{*}(S'_{\text{\text{sc}}}))},
\end{tikzcd}
\]
where the last map sends an element represented by $x \in X_{*}(S'/Z)$ to $N(x)$; when $S = S'$ is our elliptic maximal torus, $X_{*}(S_{\text{sc}})^{\Gamma} = 0$, and so this sequence gives an isomorphism $$\overline{Y}_{+,\text{tor}}[Z \to S]/(X_{*}(S_{\text{\text{sc}}})^{N}/IX_{*}(S_{\text{\text{sc}}})) \to  \varinjlim \frac{[X_{*}(S/Z)/X_{*}(S_{\text{\text{sc}}})]^{N}}{I (X_{*}(S)/X_{*}(S_{\text{\text{sc}}}))}.$$

Note that we also have a bijection $$\overline{Y}_{+,\text{tor}}[Z \to S]/(X_{*}(S_{\text{\text{sc}}})^{N}/IX_{*}(S_{\text{\text{sc}}})) \to H^{1}(\gerbeE, Z \to G)$$ induced by the composition $\overline{Y}_{+,\text{tor}}[Z \to S] \xrightarrow{\sim} H^{1}(\gerbeE, Z \to S) \twoheadrightarrow H^{1}(\gerbeE, Z \to G)$, where the first map is from Theorem \ref{tatenak1} and the surjectivity of the second map is from Proposition \ref{tashocor}. The induced bijection is an immediate consequence of Lemma \ref{tasholem4.9}. We thus obtain the desired bijection.

For this to be well-defined across the inverse limit, we need to check that if $S_{1}$, $S_{2}$ are two elliptic maximal $F$-tori in $G$ and we take $g \in G(F^{s})$ such that $\text{Ad}(g)(S_{1})_{F^{s}} = (S_{2})_{F^{s}}$, then an element $\bar{\lambda} \in \varinjlim \frac{[X_{*}(S_{1}/Z)/X_{*}((S_{1})_{\text{\text{sc}}})]^{N}}{I (X_{*}(S_{1})/X_{*}((S_{1})_{\text{\text{sc}}})}$ maps to the same element in $H^{1}(\gerbeE, Z \to G)$ as its isomorphic image (via $\text{Ad}(g)$) in the same direct limit with $S_{2}$ instead of $S_{1}$.

This follows because, by what we did above, we may lift $\bar{\lambda}$ to $\dot{\lambda} \in \frac{[X_{*}(S_{1}/Z)]^{N}}{IX_{*}(S_{1})} = \overline{Y}_{+,\text{tor}}[Z \to S_{1}]$ and then map to $H^{1}(\gerbeE, Z \to G)$ via $H^{1}(\gerbeE, Z \to S_{1})$, and may analogously lift the image of $\bar{\lambda}$ in  $\varinjlim \frac{[X_{*}(S_{2}/Z)/X_{*}((S_{2})_{\text{\text{sc}}})]^{N}}{I (X_{*}(S_{2})/X_{*}((S_{2})_{\text{\text{sc}}})}$ to $\text{Ad}(g)\dot{\lambda} \in \frac{[X_{*}(S_{2}/Z)]^{N}}{IX_{*}(S_{2})}$ and then map to $H^{1}(\gerbeE, Z \to G)$ via $H^{1}(\gerbeE, Z \to S_{2})$. Now Lemma \ref{hardlem} implies that these images coincide. 
\end{proof}

\begin{cor}\label{uniqueisomorphism} The isomorphism of functors constructed in Theorem \ref{tatenak2} is unique satisfying the hypotheses.
\end{cor}

\begin{proof} This follows from the discussion in \cite[\S 4.2]{Tasho}, which relies on the existence of elliptic maximal tori and Corollary 3.7 loc. cit, both of which we have established in our situation.
\end{proof}
We conclude by citing one more result of \cite{Tasho} that holds here, which will be used in \S 7.

\begin{prop}\label{existtwist} Let $G$ be a connected reductive group defined over $F$, let $Z$ be the center of $\mathscr{D}(G)$, and set $\overline{G} = G/Z$. Then both natural maps $$H^{1}(\gerbeE, Z \to G) \to H^{1}(F, \overline{G}) \to H^{1}(F, G_{\text{ad}})$$ are surjective. If $G$ is split, then the second map is bijective and the first map has trivial kernel.
\end{prop} 

\begin{proof} See the proof of \cite[Corollary 3.8]{Tasho}, replacing the use of \cite[Theorem 1.2]{Kott86} with \cite[Theorem 2.1]{Thang2}.
\end{proof}

\section{The local transfer factor}
In order to apply the concepts we have developed, we need to define the local transfer factor, as defined in \cite{LS}, for reductive groups over local function fields. For expository purposes, we make this section entirely self-contained.

\subsection{Notation and preliminaries}

We will always take $G$ to be a connected reductive group defined over $F$, a local field of characteristic $p > 0$. Let $G^{*}$ be a quasi-split group over $F$ such that we have $\psi \colon G_{F^{s}} \xrightarrow{\sim} G^{*}_{F^{s}}$ satisfying $\psi^{-1} \circ \prescript{\sigma}{}\psi = \text{Ad}(u_{\sigma})$ for some $u_{\sigma} \in G_{\text{ad}}(F^{s})$ for all $\sigma$ in $\Gamma$. That is to say, $G^{*}$ is a \textit{quasi-split inner form of $G$} over $F$. One important difference that emerges here in the positive characteristic case is that such a $u_{\sigma}$ need not have a lift in $G(F^{s})$, due to the potential non-smoothness of $Z(G)$. Such lifts are useful for computational purposes, and so to combat the smoothness issue we give an equivalent characterization of inner forms in the fppf language.

Again for $G^{*}$ a quasi-split group over $F$, we say that $G^{*}$ is a quasi-split inner form of $G$ if there is an isomorphism $\psi \colon G_{F^{s}} \xrightarrow{\sim} G^{*}_{F^{s}}$ satisfying $p_{1}^{*}\psi^{-1} \circ p_{2}^{*}\psi = \text{Ad}(\bar{u})$ for some $\bar{u} \in G_{\text{ad}}(F^{s} \otimes_{F} F^{s})$. Since $H^{1}(\bar{F} \otimes_{F} \bar{F}, Z(G)) = 0$ (by Proposition \ref{vanishingcohomology}), we may always lift $\bar{u}$ to an element $u \in G(\bar{F} \otimes_{F} \bar{F})$. Recall that $p_{i}$ denotes the ith projection map from $\text{Spec}\bar{F} \times_{F} \text{Spec}\bar{F}$ to $\text{Spec}\bar{F}$. We will frequently treat inner forms using this approach, as it enables computations using the \v{C}ech cohomology of the fpqc cover $\Spec(\bar{F}) \to \Spec(F)$ (see, for example, \S 6.3.3).

We fix some \textit{dual group} $\widehat{G}$ corresponding to $G$, in the sense of \cite[\S 1.5]{Kott84}, and define $\prescript{L}{}G := \widehat{G}(\mathbb{C}) \rtimes W_{F}$ the associated $L$-group of $G$, where $W_{F}$ denotes the absolute Weil group of $F$. This is a topological group, where $\widehat{G}(\mathbb{C})$ is given the analytic topology in the usual way. Associated to $\widehat{G}$ is a $\Gamma$-equivariant bijection $\Psi(G)^{\vee} \to \Psi(\widehat{G})$ of based root data (see \cite[\S 1.1]{Kott84}), and we define a bijection $\Psi(G^{*})^{\vee} \xrightarrow{\psi} \Psi(G)^{\vee} \to \Psi(\widehat{G})$, which, along with the data of $\widehat{G}$ with its given $\Gamma$-action, also defines a dual group for $G^{*}$---note that this new bijection is still $\Gamma$-equivariant precisely because $G$ and $G^{*}$ are inner forms. 

\begin{Def} We call a tuple $(H, \mathcal{H}, s, \eta)$ an \textit{endoscopic datum} for $G$ if $H$ is a quasi-split reductive group defined over $F$ with a choice of dual group $\widehat{H}$, $\mathcal{H}$ is a split extension of $W_{F}$ by $\widehat{H}(\mathbb{C})$, and $\eta \colon \mathcal{H} \to \prescript{L}{}G$ is a map such that:
\begin{enumerate}
\item{The conjugation action by $W_{F}$ on $\widehat{H}$ induced by a section $W_{F} \to \mathcal{H}$ and any $\Gamma$-splitting of $\widehat{H}$ coincides with the $L$-group $W_{F}$-action on $\widehat{H}$;}
\item{The element $s$ lies in $Z(\widehat{H})(\mathbb{C})$;}
\item{The map $\eta$ is a morphism of $W_{F}$-extensions which restricts to an isomorphism of algebraic groups $\widehat{H} \xrightarrow{\sim} Z_{\widehat{G}}(\eta(s))^{\circ}$;}
\item{We have $s \in Z(\widehat{H})^{\Gamma} \cdot \eta^{-1}(Z(\widehat{G}))$.}
\end{enumerate}
\end{Def}

This is formulated slightly differently from the exposition in \cite[\S 1.2]{LS}; it is easily checked that this definition is equivalent to the one given there. An \textit{isomorphism of endoscopic data} from $(H, \mathcal{H}, s, \eta)$ to $(H', \mathcal{H}', s', \eta')$ is an element $g \in \widehat{G}(\mathbb{C})$ such that $g \eta(\mathcal{H}) g^{-1} = \eta'(\mathcal{H}')$, thus inducing an isomorphism $\beta \colon \mathcal{H} \xrightarrow{\eta'^{-1} \circ \text{Ad}(g) \circ \eta} \mathcal{H}'$, which we further require to satisfy that $\beta(s)$ and $s'$ are equal modulo $Z(\widehat{H'})^{\Gamma, \circ} \cdot \eta'^{-1}(Z(\widehat{G}))$. One checks that this agrees with the analogous definition in \cite{LS}. 

Fix an endoscopic datum $(H, \mathcal{H}, s, \eta)$ for $G$. If we fix two Borel pairs $(B_{G}, T_{G}), (\mathscr{B}_{G}, \mathscr{T}_{G})$ in $G_{F^{s}}, \widehat{G}$ (respectively), then the bijection of based root data gives an isomorphism $\widehat{T_{G}} \to \mathscr{T}_{G}$. The associated isomorphism $X_{*}(T_{G}) \to X^{*}(\mathscr{T}_{G})$ transports the coroot system $R^{\vee}$ of $T_{G}$ to the root system of $\mathscr{T}_{G}$ mapping the $B_{G}$-simple coroots to the $\mathscr{B}_{G}$-simple roots, and identifies the Weyl group $W(G_{F^{s}}, T_{G})$ with the Weyl group $W(\widehat{G}, \mathscr{T}_{G})$. Moreover, if $(\mathscr{T}_{H}, \mathscr{B}_{H})$ is a pair in $\widehat{H}$, then we may find $g \in \widehat{G}(\mathbb{C})$ such that $(\text{Ad}(g) \circ \eta)(\mathscr{T}_{H}) = \mathscr{T}_{G}$ and $\text{Ad}(g) \circ \eta$ maps $\mathscr{B}_{H}$ into $\mathscr{B}_{G}$. This means that if we fix a pair  $(T_{H}, B_{H})$ in $H_{F^{s}}$, then we have an isomorphism $\widehat{T_{H}} \to \mathscr{T}_{H} \to \mathscr{T}_{G} \to \widehat{T_{G}}$, inducing an isomorphism $T_{H} \to T_{G}$. This isomorphism transports $R_{H}, R_{H}^{\vee}, W(H_{F^{s}}, T_{H})$ into $R, R^{\vee}, W(G_{F^{s}}, T_{G})$. 

Suppose that we fix such a $T_{H}, T_{G}$, but now require that they are defined over $F$. An $F$-isomorphism $T_{H} \to T_{G}$ is called \textit{admissible} if it is obtained as in the above paragraph (this is not unique---we chose four Borel subgroups in the above construction). We sometimes also call this an \textit{admissible embedding} of $T_{H}$ in $G$. Such an embedding is unique up to conjugacy by an element of the set $\tilde{\mathfrak{A}}(T_{G})$, defined by
$$\tilde{\mathfrak{A}}(T_{G}) = \{\bar{g} \in G_{\text{ad}}(F^{s}) \colon \restr{\text{Ad}(\bar{g}^{-1} \prescript{\sigma}{}(\bar{g}))}{(T_{G})_{F^{s}}} = \text{id}_{(T_{G})_{F^{s}}}\hspace{1mm} \forall \sigma \in \Gamma \}.$$
Another way of describing this set is those points $\bar{g} \in G_{\text{ad}}(F^{s})$ such that $\restr{\text{Ad}(\bar{g})}{(T_{G})_{F^{s}}}$ is defined over $F$. Note that given such a $\bar{g}$, we may always find some $g \in G(F^{s})$ such that $\restr{\text{Ad}(g)}{(T_{G})_{F^{s}}} = \restr{\text{Ad}(\bar{g})}{(T_{G})_{F^{s}}}$. 
Indeed, if $g \in G(\bar{F})$ is such that $\restr{\text{Ad}(g)}{(T_{G})_{\bar{F}}}$ is defined over $F$, then we may find a point $g' \in G(F^{s})$ such that $\text{Ad}(g) = \text{Ad}(g')$ on $T_{G}$---this follows since $H^{1}(F^{s}, T_{G}) = 0$, which means that $N_{G}(T_{G})(F^{s}) \to [N_{G}(T_{G})/T_{G}](F^{s}) \ni \restr{\text{Ad}(g)}{T_{G}}$ is surjective. Thus, such an embedding is also unique up to conjugacy by an element of the set
$$\mathfrak{A}(T_{G}) = \{g \in G(F^{s}) \colon g^{-1} \cdot \prescript{\sigma}{}g \in T_{G}(F^{s}) \hspace{1mm} \forall \sigma \in \Gamma \}.$$

Given any $g \in \mathfrak{A}(T_{G})$, we may also find a point in $G_{\text{\text{sc}}}(F^{s})$ inducing the same map on $T_{G}$, where $G_{\text{\text{sc}}}$ denotes the simply connected cover of $\mathscr{D}(G)$. Suppose that $\text{Ad}(g)$ sends $T$ to $T'$, where $T$ and $T'$ are two maximal $F$-tori. Then we may take the preimages $(T_{\text{\text{sc}}})_{\bar{F}}, (T'_{\text{\text{sc}}})_{\bar{F}}$ in $(G_{\text{\text{sc}}})_{\bar{F}}$, and fix a preimage $\tilde{g} \in G_{\text{\text{sc}}}(\bar{F})$ of the image of $g$ in $G_{\text{ad}}(F^{s})$ so that $\text{Ad}(\tilde{g}) \colon (T_{\text{\text{sc}}})_{\bar{F}} \xrightarrow{\sim} (T'_{\text{\text{sc}}})_{\bar{F}}$. This isomorphism is defined over $F^{s}$, i.e., we get a descent to an isomorphism $(T_{\text{\text{sc}}})_{F^{s}} \xrightarrow{\sim} (T'_{\text{\text{sc}}})_{F^{s}}$, which is given by $\text{Ad}(x)$ for some $x \in G_{\text{\text{sc}}}(F^{s})$ by an identical $H^{1}$-vanishing argument as in the above paragraph; then $x$ satisfies $\restr{\text{Ad}(x)}{T_{F^{s}}} = \restr{\text{Ad}(\tilde{g})}{T_{F^{s}}}$, as desired.

We call an element $\gamma \in G(\bar{F})$ \textit{strongly regular} if it is semisimple and its centralizer is a maximal torus (there is a notion of strong regularity for non-semisimple elements but we will not need it here); denote the subset of strongly regular $F$-points of $G$ by $G_{\text{sr}}(F)$. We call an element $\gamma_{H} \in H(F)$ \textit{strongly $G$-regular} if it is the preimage of a strongly regular $\gamma_{G} \in G(F)$ under an admissible isomorphism. In such a case, $\gamma_{H}$ is itself strongly regular in $H$, and the admissible isomorphism between centralizers $T_{H} \xrightarrow{\sim} T_{G}$ sending $\gamma_{H}$ to $\gamma_{G}$ is unique; denote this subset of $H(F)$ by $H_{G-\text{sr}}(F)$, and call such a pair of elements $\gamma_{H}, \gamma_{G}$ \textit{related}. To discuss admissible embeddings (and for later applications), we need a generalization of \cite[Corollary 2.2]{Kott82}, which says:

\begin{lem}\label{kottlem} Let $G^{*}$ be a quasi-split reductive group over $F$ and $i: T_{F^{s}} \to G^{*}_{F^{s}}$ be an embedding over $F^{s}$ of an $F$-torus $T$ into $G^{*}$ such that $i(T_{F^{s}})$ is a maximal torus of $G^{*}_{F^{s}}$ and such that $\prescript{\sigma}{}i$ is conjugate under $G^{*}(F^{s})$ to $i$ for all $\sigma \in \Gamma$. Then some $G^{*}(F^{s})$-conjugate of $i$ is defined over $F$.
\end{lem}

\begin{proof} The proof of this result in \cite{Kott82} depends on first proving the following result (Lemma 2.1 loc. cit.): Let $S$ be a maximal torus of $G^{*}$ over $F$ and $w \colon \Gamma \to W(G^{*}_{F^{s}}, S_{F^{s}})$ be a 1-cocycle of $\Gamma$ in the absolute Weyl group of $S$ (acting in the usual way), and choose an arbitrary lift $n_{\sigma} \in N_{G^{*}}(S)(F^{s})$ of $w(\sigma)$ for all $\sigma \in \Gamma$. Then we may use it to twist $S$, obtaining an $F$-torus $\prescript{*}{}S$ which is an $F^{s}$-form of $S$, and to twist the $F$-variety $G^{*}/S$, obtaining the $F$-variety $\prescript{*}{}(G^{*}/S)$ which is an $F^{s}$-form of $G^{*}/S$. The claim is then that $\prescript{*}{}(G^{*}/S)(F) \neq \emptyset$. As in \cite{Kott82}, this will follow if we can find some $s \in S_{\text{sr}}(F^{s})$ and $g \in G^{*}(F^{s})$ such that $gsg^{-1} \in G^{*}(F)$. We will view $(\prescript{*}{}S)_{F^{s}}$ as a subtorus of $G^{*}_{F^{s}}$ via the isomorphism $(\prescript{*}{}S)_{F^{s}}\xrightarrow{\phi} S_{F^{s}}$ coming from its construction as an $F^{s}$-form of $S$.

To this end, we know by unirationality that $\prescript{*}{}S(F)$ is Zariski-dense in $(\prescript{*}{}S)_{\bar{F}}$, and also that the locus of strongly regular elements in $S(\bar{F})$ forms a Zariski-open and dense subset of $S_{\bar{F}}$, by \cite[Theorem 1.3.a]{Steinberg}, and hence there is some element $s \in (\prescript{*}{}S)(F)$ that lies in $S_{\text{sr}}(\bar{F})$; such a point necessarily lies in $S(F^{s})$, since $\phi$ maps $\prescript{*}{}S(F^{s})$ into $S(F^{s})$. Then \cite[8.6]{BorelSpringer}, (which is a generalization of \cite[Theorem 1.7]{Steinberg} to imperfect fields) shows that we may find a point in $G^{*}_{\text{sr}}(F)$ which is $G^{*}(\bar{F})$-conjugate to $s$, which we know is equivalent to $G^{*}(F^{s})$-conjugacy. This gives the claim; with this in hand, the argument in \cite[Lemma 2.1]{Kott82}, carries over verbatim to show that $\prescript{*}{}(G^{*}/S)(F) \neq \emptyset$.

Now we prove the main lemma, following \cite{Kott82}. We may assume that $i(T_{F^{s}})$ is defined over $F$, with $F$-descent denoted by $T'$,  by conjugating by an appropriate element of $G^{*}(F^{s})$. Choose $n_{\sigma} \in N_{G^{*}}(T')(F^{s})$ such that $\text{Ad}(n_{\sigma}) \circ i = \prescript{\sigma}{}i$ with image $w(\sigma) \in W(G^{*}_{F^{s}}, T'_{F^{s}})$ independent of choice of $n_{\sigma}$. Now apply the above claim to the $F$-torus $T'$ and the cocycle $\sigma \mapsto w(\sigma)$, thus obtaining $\bar{g} \in \prescript{*}{}(G^{*}/T')(F) \subset (G^{*}/T')(F^{s}) = G^{*}(F^{s})/T(F^{s})$ (containment via the defining isomorphism of the twisted form). This last equality comes from the fact that for every $t \in (G^{*}/T')(F^{s})$, if $\pi \colon G^{*}_{F^{s}} \to (G^{*}/T)_{F^{s}}$ denotes the canonical quotient map, the (scheme-theoretic) fiber $\pi^{-1}(t) \hookrightarrow G^{*}_{F^{s}}$ is a $T_{F^{s}}$-torsor, which is split over $F^{s}$ and thus contains an $F^{s}$-point. The upshot is that we have some $g \in G^{*}(F^{s})$ which satisfies $g^{-1} \prescript{\sigma}{}g n_{\sigma} \in i(T_{F_{s}})(F^{s})$ for all $\sigma \in \Gamma$, which means that $\text{Ad}(g) \circ i$ is defined over $F$.  
\end{proof}

\begin{cor}\label{admiss} Let $T_{H}$ denote the centralizer of $\gamma_{H} \in H_{G-\text{sr}}(F)$ There exists an admissible embedding $T_{H} \hookrightarrow G^{*}$.
\end{cor}

\begin{proof} By assumption we already have an admissible isomorphism $T_{H} \to T_{G}$, where $T_{G}$ is a maximal $F$-torus of $G$. It is easy to see that it then suffices to find an admissible embedding of $T_{G}$ into $G^{*}$, which is immediate from Lemma \ref{kottlem} (applied to $i = \restr{\psi}{T_{G}}$).
\end{proof}

\subsubsection{The Tits section}

We need to discuss the \textit{Tits section}, which is a (non-multiplicative) map $n \colon W(G_{F^{s}},T_{F^{s}}) \hookrightarrow N_{G}(T)(F^{s})$. To do this, we must fix a Borel subgroup $B$ of $G_{F^{s}}$ (corresponding to a root basis $\Delta$) and a basis $\{X_{\alpha}\}$ of the root space $\mathfrak{g}_{\alpha} \subset \text{Lie}(G_{F^{s}})$ for each $\alpha \in \Delta$. Let $G_{\alpha}$ be the Levi subgroup of $\mathscr{D}(G_{F^{s}})$ corresponding to the root $\alpha$; then there is a unique embedding $\zeta_{\alpha} \colon SL_{2} \to G_{\alpha}$ which (on Lie algebras) sends $\begin{bmatrix} 0 & 1 \\ 0 & 0 \end{bmatrix}$ to $X_{\alpha}$ and such that the image of $\zeta_{\alpha}(\begin{bmatrix} 0 & 1 \\ -1 & 0 \end{bmatrix})$ in $W(G_{F^{s}}, T_{F^{s}})$ is the reflection $r_{\alpha}$ defined by $\alpha$ (see \cite[\S 2.1]{KS2}). We then map $r_{\alpha}$ to the image of $\begin{bmatrix} 0 & 1 \\ -1 & 0 \end{bmatrix}$ under $\zeta_{\alpha}$. We may then lift any element of $W(G_{F^{s}}, T_{F^{s}})$ by considering their reduced expression in terms of $\Delta$. 

\subsubsection{Duality results}

We recall Langlands' reinterpretation of Tate-Nakayama duality. Let $T$ be an $F$-torus; the usual Tate-Nakayama duality theorem gives a perfect $\Z$-pairing $$H^{1}(F, T) \times H^{1}(\Gamma, X^{*}(T)) \to \mathbb{Q}/\Z,$$ see for example \cite[I.2.4]{Milne}. Consider the short exact sequence of abelian groups
\begin{center}
 \begin{tikzcd}
  0 \arrow[r] & \Z \arrow[r] & \mathbb{C}  \arrow[r, "\text{exp}"] & \mathbb{C}^{*}  \arrow[r] &1.
\end{tikzcd}
\end{center}
Tensoring this sequence over $\Z$ with $X_{*}(\widehat{T}) = X^{*}(T)$ preserves exactness, and thus yields the exact sequence
\begin{center}
 \begin{tikzcd}
  0 \arrow[r] & X^{*}(T) \arrow[r] & \text{Lie}(\widehat{T}) \arrow[r] & \widehat{T}(\mathbb{C}) \arrow[r] &1,
\end{tikzcd}
\end{center}
which then gives a canonical identification $H^{1}(\Gamma,X^{*}(T)) \xrightarrow{\sim} \pi_{0}(\widehat{T}^{\Gamma}),$ and hence a perfect pairing \begin{equation}\label{complextatenak}
H^{1}(F, T) \times \pi_{0}(\widehat{T}^{\Gamma}) \to \mathbb{Q}/\Z.
\end{equation}

In fact, the above identification $H^{1}(F, T) =  [\pi_{0}(\widehat{T}^{\Gamma})]^{*}$ is true for $\mathbb{C}$ replaced by an arbitrary algebraically closed field $C$ of characteristic zero (where $\widehat{T}$ is defined in the exact same way) via the sequence of identifications:
$$H^{1}(F, T) \xrightarrow{\sim} X_{*}(T)_{\Gamma,\text{tor}} = X^{*}(\widehat{T})_{\Gamma,\text{Tor}} = [\pi_{0}(\widehat{T}^{\Gamma})]^{*},$$
where the first map is the Tate isomorphism from local class field theory.

Returning to the setting of a connected reductive  group $G$, note that if $T$ is any maximal $F$-torus of $G$, for any maximal torus $\mathscr{T}$ of $\widehat{G}$, we have an isomorphism $\mathscr{T} \to \widehat{T}$ which is unique up to precomposing with conjugation by an element of $N_{\widehat{G}}(\mathscr{T})(\mathbb{C})$, so we get a canonical embedding $Z(\widehat{G}) \hookrightarrow \widehat{T}$, which clearly also does not depend on the choice of $\mathscr{T}$ (any two such tori are $\widehat{G}(\mathbb{C})$-conjugate). Denote $\widehat{T}/Z(\widehat{G})$ by $\widehat{T}_{\text{\text{ad}}}$. Assume for the moment that $G$ is semisimple. One checks using the basic theory of (co)character groups and root systems that (via the above embedding) $X^{*}(Z(\widehat{G}))$ corresponds to the quotient $X_{*}(T)/ \Z R(G_{F^{s}}, T_{F^{s}})^{\vee}$ of $X^{*}(\widehat{T}) = X_{*}(T)$. Whence, we have a canonical identification of $X^{*}(\widehat{T}_{\text{\text{ad}}})$ with $X_{*}(T_{\text{\text{sc}}})$, where $T_{\text{\text{sc}}}$ is the preimage of $T$ in $G_{\text{\text{sc}}}$, giving a $\Gamma$-isomorphism $\widehat{T_{\text{\text{sc}}}} \xrightarrow{\sim} \widehat{T}_{\text{\text{ad}}}$. For general $G$, one checks easily that a similar argument yields a canonical isomorphism $\widehat{T_{\text{\text{sc}}}} \xrightarrow{\sim} \widehat{T}_{\text{\text{ad}}}$, where now $T_{\text{\text{sc}}}$ denotes the preimage of the maximal $F$-torus $T \cap \mathscr{D}(G) \subset \mathscr{D}(G)$ in $G_{\text{\text{sc}}}$, the simply connected cover of $\mathscr{D}(G)$. We conclude that Tate-Nakayama then gives a perfect pairing $$H^{1}(F, T_{\text{\text{sc}}}) \times \pi_{0}(\widehat{T}_{\text{\text{ad}}}^{\Gamma}) \to \mathbb{Q}/\Z.$$
We may replace $\mathbb{Q}/\Z$ by $\mathbb{C}^{*}$ by means of the embedding $\mathbb{Q}/\Z \xrightarrow{\text{exp}} \mathbb{C}^{*}$.

Before proceeding further, we need to define $H^{-2}_{\text{Tate}}(\Gamma, U)$ for $U$ a finite, discrete $\Gamma$-module, where $\Gamma$ is the absolute Galois group of $F$, following \cite[\S VI.I]{La83}. Let $K'/K/F$ be a sequence of finite Galois extensions of $F$; recall that for a finite group $G$, elements of $H^{-2}_{\text{Tate}}(G, U)$ may be viewed as equivalence classes of $1$-cycles, which are functions $G \xrightarrow{f} U$ such that $$\sum_{g \in G} (g^{-1} - 1)f(g) = 0.$$ We then define $H^{-2}_{\text{Tate}}(\Gamma_{K'/F}, U) \to H^{-2}_{\text{Tate}}(\Gamma_{K/F}, U)$ on chains by sending $f$ to $\bar{f} \colon \sigma \mapsto \sum_{\sigma' \mapsto \sigma} f(\sigma')$; one checks that this preserves cycles, is surjective, and respects the appropriate equivalence relation (cf. \cite[\S VI]{La83}). We then define $$H^{-2}_{\text{Tate}}(\Gamma, U) = \varprojlim_{K/F} H^{-2}_{\text{Tate}}(\Gamma_{K/F}, U);$$ note that, unlike in the characteristic zero case, this projective system need not stabilize.

Recall, for an $F$-torus $T$ split over $E/F$ a finite Galois extension, we have the classical Tate isomorphism $H^{-1}_{\text{Tate}}(\Gamma_{E/F}, X_{*}(T)) \xrightarrow{\sim} H^{1}(F,T)$ induced by taking the cup product with the canonical class (see \cite{Tate66}). The following useful duality result generalizes this to finite multiplicative group schemes over $F$. 

\begin{prop}\label{duality} Let $T$ be an $F$-torus and $S$ the quotient of $T$ by a finite $F$-subgroup $Z$. We have a canonical isomorphism $$H^{-2}_{\text{Tate}}(\Gamma, X_{*}(S)/X_{*}(T)) \xrightarrow{\sim} H^{1}(F, Z)$$ which is compatible with the Tate isomorphism $H^{-1}_{\text{Tate}}(\Gamma, X_{*}(T)) \xrightarrow{\sim} H^{1}(F, T)$.
\end{prop}

\begin{proof} Cohomology in negative degrees will always be Tate cohomology, and we omit the ``Tate" notation in such cases. We have an exact sequence of character groups
\[
\begin{tikzcd}
0 \arrow{r} & X^{*}(S) \arrow{r} & X^{*}(T) \arrow{r} & X^{*}(Z) \arrow{r} & 0 
\end{tikzcd}
\]
which, by applying the functor $\Hom(-, \Z)$, yields the short exact sequence (of $\Gamma$-modules)
\[
\begin{tikzcd}
0 \arrow{r} & X_{*}(T) \arrow{r} & X_{*}(S) \arrow["\delta"]{r} & \text{Ext}^{1}_{\Z}(X^{*}(Z), \Z) \arrow{r} & 0.
\end{tikzcd}
\]
By basic homological algebra, we have a canonical isomorphism (as $\Gamma$-modules) $$\text{Ext}^{1}_{\Z}(X^{*}(Z), \Z) \cong \Hom_{\Z}(X^{*}(Z), \mathbb{Q}/\Z).$$ We make these identifications in what follows without comment. For an abelian group $M$, we set $\Hom_{\Z}(M, \mathbb{Q}/\Z) =: M^{*}$. Let $\Gamma' = \Gamma_{K/F}$ for $K/F$ a finite Galois extension splitting $T$ such that $|Z|$ and $|H^{1}(\Gamma, X^{*}(T))|$ divide $[K \colon F]$ (for finiteness of the latter, see \cite[III.6]{Milne}). Since all sufficiently large $K$ satisfy this property, we may take the limit defining $H^{-2}(\Gamma, X_{*}(S)/X_{*}(T))$ over the set of all such $K$. We then have identifications $H^{-1}(\Gamma', \mathbb{Q}/\Z) = \Z/[K \colon F]\Z$, and $H^{0}_{\text{Tate}}(\Gamma', \Z) = \Z/[K \colon F] \Z$. By Proposition VII.7.1 and Exercise VII.7.3 (respectively) in \cite{Brown}, we have the following duality pairings of $\Gamma'$-modules induced by the cup product and these identifications:
$$ H^{-2}(\Gamma', X^{*}(Z)^{*}) \times H^{1}(\Gamma', X^{*}(Z)) \to \Z/[K \colon F] \Z,$$
$$H^{-1}(\Gamma', X_{*}(T)) \times H^{1}(\Gamma', X^{*}(T)) \to \Z/[K \colon F] \Z. $$
Note that the group $H^{1}(\Gamma', X^{*}(Z))$ is $|Z|$-torsion, so that $$H^{1}(\Gamma', X^{*}(Z))^{*} = \Hom_{\Z}(H^{1}(\Gamma', X^{*}(Z)), \Z/[K \colon F] \Z),$$ analogously for $H^{1}(\Gamma', X^{*}(T))$. 

As a consequence, we have a canonical isomorphism $H^{-2}(\Gamma', X_{*}(S)/X_{*}(T)) \xrightarrow{\sim} H^{1}(\Gamma', X^{*}(Z))^{*}.$ 
For $K'/K/F$, the diagram
\[
\begin{tikzcd} 
H^{-2}(\Gamma_{K'/F}, X_{*}(S)/X_{*}(T)) \arrow{d} \arrow["\sim"]{r} & H^{1}(\Gamma_{K'/F}, X^{*}(Z))^{*} \arrow{d} \\
H^{-2}(\Gamma_{K/F}, X_{*}(S)/X_{*}(T)) \arrow["\sim"]{r} & H^{1}(\Gamma_{K/F}, X^{*}(Z))^{*}
\end{tikzcd}
\]
commutes, where the right-hand map is dual to inflation. To see this, first note that the relevant pairing (for $\Gamma_{K/F}$) can be given explicitly by sending the pair $(f, c)$ to $\sum_{\sigma \in \Gamma_{K/F}} f(\sigma)[c(\sigma)] \in \mathbb{Q}/\Z$. Starting with the $1$-cycle $f$ for $\Gamma_{K'/F}$ and then going right and down gives the function on $H^{1}(\Gamma_{K/F}, X^{*}(Z))$ sending the $1$-cocycle $c$ to $$\sum_{\gamma \in \Gamma_{K'/F}} f(\gamma)[c(\bar{\gamma})] = \sum_{\sigma \in \Gamma_{K/F}} \sum_{\sigma' \mapsto \sigma} f(\sigma')[c(\sigma)],$$ which is exactly where $c$ gets sent if one follows the diagram in the other direction. It follows that the above isomorphisms splice to give an isomorphism $$H^{-2}(\Gamma, X_{*}(S)/X_{*}(T)) \xrightarrow{\sim} \varprojlim_{K/F} H^{1}(\Gamma_{K/F}, X^{*}(Z))^{*}.$$ We then obtain an isomorphism
$$H^{-2}(\Gamma, X_{*}(S)/X_{*}(T)) \xrightarrow{\sim} \varprojlim_{K/F} H^{1}(\Gamma_{K/F}, X^{*}(Z))^{*} = [\varinjlim_{K/F} H^{1}(\Gamma_{K/F}, X^{*}(Z))]^{*} \xrightarrow{\sim} H^{1}(F, Z),$$
where the last isomorphism comes from identifying $\varinjlim_{K/F} H^{1}(\Gamma_{K/F}, X^{*}(Z)) = H^{1}(\Gamma, X^{*}(Z))$ with $H^{1}(F, \underline{X}^{*}(Z))$ and then applying the Poitou-Tate duality pairing for finite commutative group schemes over arbitrary local fields, see \cite[Theorem III.6.10]{Milne}, which is induced by the cup-product followed by the invariant map.  

We now get a commutative diagram 
\[
\begin{tikzcd} H^{-2}(\Gamma, X^{*}(Z)^{*}) \arrow{r} \arrow["\sim"]{d} & H^{-1}(\Gamma, X_{*}(T)) \arrow["\sim"]{d} \\
H^{1}(\Gamma, X^{*}(Z))^{*} \arrow{r} \arrow["\sim"]{d} & H^{1}(\Gamma, X^{*}(T))^{*} \arrow["\sim"]{d} \\
H^{1}(F, Z) \arrow{r} & H^{1}(F, T);
\end{tikzcd}
\]
to see that the top square commutes, note that this can be checked at the finite level (cf. the above argument using the inverse limit of $H^{-2}_{\text{Tate}}$), where it follows from the functoriality of the cup product in Tate cohomology with respect to connecting homomorphisms (cf. \cite[Proposition V.3.3]{Brown}), and the bottom square commutes by the discussion in \cite[\S III.6]{Milne}; see in particular the diagram used in the proof of Lemma 6.11 loc. cit. The right-hand column equals the classical Tate isomorphism discussed in \cite{Tate66}, by the associativity of the cup product in Tate cohomology. 
\end{proof}

\begin{remark}\label{altpairing} This remark concerns how the above discussion relates to the Tate-Nakayama pairing involving $\pi_{0}(\widehat{T}^{\Gamma})$ discussed earlier. Identifying $H^{1}(\Gamma, X^{*}(T)) = H^{1}(\Gamma, X_{*}(\widehat{T}))$ with $\widehat{T}(\mathbb{C})^{\Gamma}/(\widehat{T}(\mathbb{C})^{\Gamma})^{\circ}$ as above, we note that there is a natural pairing 
\begin{equation}\label{twopairings}
H^{-1}(\Gamma, X_{*}(T)) \times \frac{\widehat{T}(\mathbb{C})^{\Gamma}}{(\widehat{T}(\mathbb{C})^{\Gamma})^{\circ}}  = H^{-1}(\Gamma, X^{*}(\widehat{T})) \times \frac{\widehat{T}(\mathbb{C})^{\Gamma}}{(\widehat{T}(\mathbb{C})^{\Gamma})^{\circ}} \to \mathbb{C}^{*}
\end{equation}
 given by evaluating an element on a character. One checks that the following diagram commutes:
\[
\begin{tikzcd}
H^{1}(F, T)  \times  \pi_{0}(\widehat{T}^{\Gamma}) \arrow["f \times \text{id}"]{d} \arrow{r} & \mathbb{C}^{*} \arrow[equals]{d} \\
H^{-1}(\Gamma, X^{*}(\widehat{T})) \times \pi_{0}(\widehat{T}^{\Gamma}) \arrow{r} & \mathbb{C}^{*},
\end{tikzcd}
\]
where the top pairing is the one from \eqref{complextatenak}, the bottom pairing is as in \eqref{twopairings}, and we are using $f$ to denote the isomorphism $H^{1}(F, T) \to H^{-1}(\Gamma, X_{*}(T))$ constructed above. The above argument also holds if we replace $\mathbb{C}$ by an arbitrary algebraically closed field of characteristic zero.
\end{remark}

We conclude this subsection by recalling Langlands duality for tori, which is the following result:

\begin{thm}\label{toriduality} For an $F$-torus $T$, $F$ a local field, we have a canonical isomorphism $$H^{1}_{\text{cts}}(W_{F}, \widehat{T}(\mathbb{C})) \xrightarrow{\sim} \Hom_{\text{cts}}(T(F), \mathbb{C}^{*}).$$ This isomorphism induces a pairing $$H^{1}_{\text{cts}}(W_{F}, \widehat{T}(\mathbb{C}))  \times T(F) \to \mathbb{C}^{*}$$ which is functorial with respect to $F$-morphisms of tori and respects restriction of scalars.
\end{thm}
\noindent For the proof, see \cite[Theorem 2.a]{La} and \cite[\S 9, \S 10]{LBorel}.

\subsection{Setup}
This section completely follows \cite[\S2]{LS} and \cite[\S 2]{KS2}; its purpose is to explain why the results proved therein still work in our section.

\subsubsection{The splitting invariant}

Fix a connected reductive $F$-group $G$ which we assume to be quasi-split over $F$, and an $F$-splitting $(B_{0}, T_{0}, \{X_{\alpha}\})$, along with an arbitrary maximal $F$-torus $T$ in $G$. Assume further that $G$ is semisimple and simply-connected. For a root $\alpha \in R :=R(G_{F^{s}},T_{F^{s}})$, we take $\Gamma_{\alpha}, \Gamma_{\pm \alpha}$ to be the stabilizers of $\alpha$ and $\{\alpha, -\alpha\}$, respectively, with $F_{\alpha} \supset F_{\pm \alpha}$ the corresponding fixed fields. An \textit{a-data} $\{a_{\alpha}\}_{\alpha \in R}$ for the $\Gamma$-action on $R$ is an element $a_{\alpha} \in F_{\alpha}^{*}$ for each $\alpha \in R$ satisfying $\sigma(a_{\alpha}) = a_{\sigma \alpha}$ for all $\sigma \in \Gamma$ and $a_{-\alpha} = -a_{\alpha}$. It is easy to check that $a$-data exist for our $\Gamma$ action on $R$ above; fix such a datum $\{a_{\alpha}\}_{\alpha \in R}$. Our goal is to define the \textit{splitting invariant} $\lambda_{\{a_{\alpha}\}}(T) \in H^{1}(F, T)$. 

We first choose a Borel subgroup $B$ of $G_{F^{s}}$ containing $T$, and take some $h \in G(F^{s})$ such that $h$ conjugates the pair $((B_{0})_{F^{s}}, (T_{0})_{F^{s}})$ to $(B_{F^{s}}, T_{F^{s}})$. Denote by $\sigma_{T}$ the action of $\sigma \in \Gamma$ on $T_{F^{s}}$ and its transport to $(T_{0})_{F^{s}}$ via $\text{Ad}(h)^{-1}$. For ease of notation, let $\Omega$ denote the absolute Weyl group $W(G_{F^{s}},(T_{0})_{F^{s}})$, with Tits section $n \colon \Omega \to N_{G}(T_{0})(F^{s})$. We then have (as automorphisms of the root system $R(G, T_{0})$) 
$$ \sigma_{T} = \omega_{T}(\sigma) \rtimes \sigma_{T_{0}} \in \Omega \rtimes \Gamma, $$
where $\omega_{T}(\sigma) := \text{Ad}(h \cdot \sigma(h)^{-1})$, viewed as an element of $\Omega$. We may view our $a$-data $\{a_{\alpha}\}_{\alpha \in R}$ as an $a$-data for the (transported) action of $\Gamma$ on $R(G, T_{0})$, and denote it also by $\{a_{\alpha}\}_{\alpha }$. 

For any automorphism $\zeta$ of $R(G, T_{0})$, we define the element $x(\zeta) \in T_{0}(F^{s})$ by 
$$x(\zeta) = \prod_{\alpha \in R(\zeta)} \alpha^{\vee}(a_{\alpha}), $$ where $R(\zeta) = \{\alpha \in R(G, T_{0}) | \alpha > 0, \zeta^{-1}\alpha < 0\}$ where the ordering on $R(G, T_{0})$ is from the base $\Delta$ corresponding to the Borel subgroup $B_{0}$. 

Then the function $$m(\sigma) := x(\sigma_{T})n(\omega_{T}(\sigma))$$ is a 1-cocycle of $\Gamma$ in $N_{G}(T_{0})(F^{s})$ and $$t(\sigma) := h m(\sigma) \sigma(h)^{-1}$$ is a 1-cocycle of $\Gamma$ in $T(F^{s})$, whose class we take to be the splitting invariant $\lambda_{\{a_{\alpha}\}}(T) \in H^{1}(F, T)$---for a proof, see \cite[\S 2.3]{LS} which as \cite{KS2} explains, works in any characteristic. The same references show that $\lambda_{\{a_{\alpha}\}}(T)$ is independent of the choice of $h$ and the Borel subgroup of $G_{F^{s}}$ containing $T_{F^{s}}$. However, it does depend on the $F$-splitting of $G$. 

\subsubsection{$\chi$-data and $L$-embeddings} 

The following discussion is essentially a summary of \cite[\S 2.4-2.6]{LS}. To more closely align with \cite[\S 2.5]{LS}, we replace $F^{s}$ by a finite Galois extension $L$ and denote $\Gamma_{L/F}$ by $\Gamma$ and $W_{L/F}$, the relative Weil group, by $W$. We will fix an arbitrary $\Gamma$-module $X$ which is finitely-generated and free over $\Z$, along with a finite subset $\Gamma$-stable subset $\mathscr{R} \subset X$ closed under inversion. Any $\Gamma$-set is also a $W$-set by means of inflation along the surjection $W \to \Gamma$. Set $\Gamma' := \Gamma \times \Z/2\Z$, where $\Z/2\Z$ acts on $X$ by inversion. As in \S 6.2.1, for $\lambda \in \mathscr{R}$ we define $\Gamma_{+\lambda}$ (resp. $\Gamma_{\pm \lambda}$) to be the stabilizer of $\{\lambda\}$ (resp. $\{\pm \lambda\}$), with corresponding fixed field $F_{\lambda} \subset L$ (resp. $F_{\pm \lambda}$). The reason we want to work in this increased generality is to allow our theory to encompass the actions of $\Gamma$ on the character groups of tori in $\widehat{G}$, a Langlands dual of the connected reductive $F$-group $G$. Define a \textit{gauge} on $\mathscr{R}$ to be a function $p \colon \mathscr{R} \to \{\pm1\}$ such that $p(-\lambda) = -p(\lambda)$.

\begin{Def} We say that a collection of continuous characters $\{\chi_{\lambda} \colon F_{\lambda}^{*} \to \mathbb{C}^{*}\}_{\lambda \in \mathscr{R}}$ is a \textit{$\chi$-data} if it satisfies $\chi_{-\lambda} = \chi_{\lambda}^{-1}$ and $\chi_{\lambda} \circ \sigma^{-1} = \chi_{\sigma \lambda}$ for all $\sigma \in \Gamma$, and if $[F_{\lambda} \colon F_{\pm \lambda}] = 2$, $\chi_{\lambda}$ extends the quadratic character $F_{\pm \lambda}^{*} \to \{\pm 1\}$ associated to the quadratic extension $F_{\lambda}$ that we obtain from local class field theory. 
\end{Def}
\noindent It is straightforward to check that we can always find a $\chi$-data; fix such a $\chi$-data $\{\chi_{\lambda}\}_{\lambda \in \mathscr{R}}$. 

Assume for the moment that $\Gamma'$ acts transitively on $\mathscr{R}$; fix $\lambda \in \mathscr{R}$, set $\Gamma_{\pm} := \Gamma_{\pm \lambda}$, and choose representatives $\sigma_{1}, \dots \sigma_{n}$ for $\Gamma_{\pm} \setminus \Gamma$. We set $W_{+} := W_{L/F_{+}}, W_{\pm} = W_{L/F_{\pm}}$. We may view the character $\chi_{\lambda}$ as a (continuous) character on $W_{+}$, by taking $\chi_{\lambda} \circ \textbf{a}_{L/F_{+}}$, where $\textbf{a}_{L/F_{+}} \colon W_{+} \to F_{+}^{*}$ is the Artin reciprocity map. 

Define a gauge $p_{0}$ on $\mathscr{R}$ by $p_{0}(\lambda') = 1$ if and only if $\lambda' = \sigma_{i}^{-1} \lambda$ for some $1 \leq i \leq n$. Choose $w_{1}, \dots, w_{n} \in W$ such that $w_{i}$ maps to $\sigma_{i}$ under the surjection $W \to \Gamma$. If $W_{\pm}$ (resp. $W_{+}$) denotes the stabilizer of $\{\pm \lambda\}$ (resp. $\{\lambda\}$) under the inflated $W$-action, then the $w_{i}$ are representatives for the quotient $W_{\pm} \setminus W$. For $w \in W$, define $u_{i}(w) \in W_{\pm}$ by $$w_{i}w = u_{i}(w)w_{j}, \hspace{1cm} i=1, \dots, n$$ (for some $1 \leq j \leq n$ depending on $i$). Choose representatives $v_{0} \in W_{+}$ and $v_{1} \in W_{\pm}$ for $W_{+} \setminus W_{\pm}$ if $[F_{\lambda} \colon F_{\pm \lambda}] = 2$, and otherwise just pick some $v_{0} \in W_{+}$. For $u \in W_{\pm}$ we define $v_{0}(u) \in W_{+}$ by $v_{0} \cdot u = v_{0}(u) \cdot v_{i'},$ where $i'=0$ or $1$ depending on if $u \in W_{+}$ or not. For $w \in W$ we set $$r_{p_{0}}(w) = \prod_{i=1, \dots, n} [\chi_{\lambda}(v_{0}(u_{i}(w))) \otimes \lambda_{i}] \in \mathbb{C}^{*} \otimes_{\Z} X,$$ where $\lambda_{i} := \sigma_{i}^{-1} \lambda$ and we view $\mathbb{C}^{*} \otimes_{\Z} X$ as a $\Gamma$-module (and thus a $W$-module) via the trivial action on the first tensor factor. We view $r_{p_{0}}$ as a 1-cochain of $W$ valued in $\mathbb{C}^{*} \otimes_{\Z} X.$ 

For any two gauges $p,p'$ on $\mathscr{R}$, define for $\sigma \in \Gamma$:
$$s_{p/p'}(\sigma) := (\prod_{(p)} -1 \otimes \lambda) \cdot  \prod_{(p')} (-1 \otimes \lambda) \in \mathbb{C}^{*} \otimes_{\mathbb{Z}} X$$
where $(p)$ is the set of $\lambda \in \mathscr{R}$ such that $p(\lambda)=1, p(\sigma^{-1}\lambda) = -1$ and $p'(\lambda) = p'(\sigma^{-1} \lambda) = 1$, and $(p')$ is the set of $\lambda$ such that $p(\lambda) = p(\sigma^{-1}\lambda) = 1$, $p'(\lambda) = -1$, and $p'(\sigma^{-1} \lambda) = 1$. For an arbitrary gauge $p$, define $r_{p} = s_{p/p_{0}} r_{p_{0}}$.

We have the following result, which will be used when we look at the uniqueness of our $L$-embeddings:

\begin{lem}\label{cor2.5} Suppose $\{\xi_{\lambda}\}_{\lambda \in \mathscr{R}}$ satisfies the conditions of a $\chi$-data, except that for $\lambda$ with $[F_{\lambda} \colon F_{\pm \lambda}] = 2$ we require that $\xi_{\lambda}$ is trivial on $F_{\pm \lambda}^{*}$ rather than extending the quadratic character. Then $$c(w) = \prod_{i=1, \dots, n} [\xi_{\lambda}(v_{0}(u_{i}(w))) \otimes \lambda_{i}] \in \mathbb{C}^{*} \otimes_{\Z} X$$ is a 1-cocycle of $W$ in $\mathbb{C}^{*}\otimes_{\Z} X$ whose cohomology class does not depend on any choices. 
\end{lem}

\begin{proof} This is \cite[Corollary 2.5.B]{LS}, which follows from Lemma 2.5.A loc. cit. These results, along with the auxiliary Lemma 2.4.A, are proved in a purely group-cohomological setting, and thus the same proofs work verbatim.
\end{proof}
\noindent If the action of $\Gamma'$ is not transitive, then we define $r_{p_{0}}$, $r_{p}$, and $c$ for each of the $\Gamma'$-orbits on $\mathscr{R}$ and take the product of these functions over all such orbits; the resulting functions on $W$ are again denoted by $r_{p}$ and $c$. 

We now take $G$ a connected reductive group defined over $F$ with maximal $F$-torus $T$ with root system $R := R(G_{F^{s}}, T_{F^{s}})$ and a Langlands dual group $\widehat{G}$. In addition, we fix a $\Gamma$-stable splitting $(\mathscr{B}, \mathscr{T}, \{\textbf{X}\})$ of $\widehat{G}$. We shall attach to a $\chi$-data $\{\chi_{\alpha}\}_{\alpha \in R}$ for $T$ a canonical $\widehat{G}$-conjugacy class of \textit{admissible embeddings} $\prescript{L}{}T \to \prescript{L}{}G$; recall that a homomorphism of $W$-extensions $\xi \colon \prescript{L}{}T \to \prescript{L}{}G$ is called an admissible embedding if the map $\widehat{T} \to \mathscr{T}$ induced by $\xi$ corresponds to the isomorphism coming from the pair $(\mathscr{B}, \mathscr{T})$ and a choice of Borel subgroup $B$ of $G_{F^{s}}$ containing $T_{F^{s}}$. We replace $F^{s}$ by a finite Galois extension $L/F$ splitting $T$; there is no harm in doing this for the purposes of constructing such an admissible embedding. The $\widehat{G}$-conjugacy class of such an embedding is independent of the choice of $B$ and splitting of $\widehat{G}$. 

Fix a Borel subgroup $B$ of $G_{F^{s}}$ containing $T_{F^{s}}$ as above, giving an isomorphism $\widehat{T} \xrightarrow{\xi} \mathscr{T}$. It is clear that such an embedding $\xi \colon \prescript{L}{}T \to \prescript{L}{}G$, is determined by its values on $W$ (via the canonical splitting $W \to \widehat{T} \rtimes W$). As in \S 6.2.1, we may use $\xi$ to transport the $\Gamma$-action on $\widehat{T}$ to $\mathscr{T}$, and for $\gamma \in \Gamma$ will denote this automorphism of $\mathscr{T}$ by $\sigma_{T}$. We have that $w \in W$ transports via $\xi$ to an action on $\mathscr{T}$ given by $$\omega_{T}(\sigma) \rtimes w,$$ where $w \mapsto \sigma \in \Gamma$ and $\omega_{T}(\sigma) \in W(\widehat{G}, \mathscr{T}).$

Our goal will be to construct a homomorphism $\xi \colon W \to \prescript{L}{}G$ giving rise to our desired embedding. As explained in \cite{LS}, it's enough that each $\text{Ad}(\xi(w))$ acts on $\mathscr{T}$ as $\sigma_{T}$, where $w \mapsto \sigma \in \Gamma$. First, we note that our $\chi$-data for the action of $\Gamma$ on $R$ yields a $\chi$-data for the $\xi$-transported action of $\Gamma$ on $R(\widehat{G}, \mathscr{T})^{\vee}$; we define a gauge $p$ on the $\Gamma$-set $R(\widehat{G}, \mathscr{T})^{\vee}$ by setting $p(\beta^{\vee}) = 1$ if and only if $\beta$ is a root of $\mathscr{T}$ in $\mathscr{B}$, and (along with our transported $\chi$-data) get an associated 1-cochain $r_{p} \colon W \to \mathbb{C}^{*} \otimes_{\Z} X_{*}(\mathscr{T})$, which we view as a 1-cochain $r_{p} \colon W \to \mathscr{T}(\mathbb{C})$ using the canonical pairing. Let $n \colon W(\widehat{G}, \mathscr{T}) \to N_{\widehat{G}}(\mathscr{T})(\mathbb{C})$ denote the Tits section associated to our splitting of $\widehat{G}$. 
Finally, for $w \in W$ we set $$\xi(w) = [r_{p}(w) \cdot n(\omega_{T}(\sigma))] \rtimes w \in \prescript{L}{}G.$$ We claim that this map satisfies the desired properties.

The verification that this map works comes down to a 2-cocycle arising from the Tits section. For $w \in W$, set $n(w) := n(\omega_{T}(\sigma)) \rtimes w$; we have for $w_{1}, w_{2} \in W$ the equality $$n(w_{1})n(w_{2})n(w_{1}w_{2})^{-1} = t_{p}(\sigma_{1}, \sigma_{2}),$$ where $w_{i} \mapsto \sigma_{i}$ and $t_{p}$ is the 2-cocycle of $\Gamma$ valued in $\mathbb{C}^{*} \otimes_{\Z} X_{*}(\mathscr{T}) \mapsto \mathscr{T}(\mathbb{C})$ as defined in \cite[Lemma 2.1.B]{LS}, Lemma 2.1.B. We then have the following crucial identity:

\begin{lem}\label{rpeq} In our above situation, the differential of $r_{p}^{-1} \in C^{1}(W, \mathscr{T}(\mathbb{C}))$ equals $\text{Inf}(t_{p}) \in Z^{2}(W, \mathscr{T}(\mathbb{C}))$ (where the above groups are given the $\xi$-transported $W$-action). 
\end{lem}

\begin{proof} After applying \cite[Lemma 2.1.A]{LS}, this reduces to a special case of Lemma 2.5.A loc. cit., which is proved in an purely group-cohomological setting. Note that, as explained in \S 2.6 loc. cit., Lemma 2.5.A applies because $\{\chi_{\alpha}^{-1}\}_{\alpha}$ is also a $\chi$-datum for the $\xi$-transported action of $\Gamma$ on $R(\widehat{G}, \widehat{T})^{\vee}$, and $r_{p}^{-1}$ is the $1$-cochain constructed using this $\chi$-datum as above. The proof of \cite[Lemma 2.1.A]{LS} is root-theoretic, and therefore works in our setting as well.
\end{proof}

With the above lemma in hand, it is straightforward to check that our $\xi \colon W \to \prescript{L}{}G$ defined above is a homomorphism that induces an admissible embedding $\xi \colon \prescript{L}{}T \to \prescript{L}{}G$. We conclude this section with a discussion of how the admissible embedding $\xi$ depends on the choices we have made during its construction. 

\begin{fact}\label{fact1} Suppose that we replace our $\Gamma$-splitting by the $g \in \widehat{G}^{\Gamma}$-conjugate $(\mathscr{B}^{g}, \mathscr{T}^{g}, \{X^{g}\})$ (see \cite[1.7]{Kott84}). If $\text{Ad}(g)^{\sharp} \colon X_{*}(\mathscr{T}) \to X_{*}(\mathscr{T}^{g})$ is the induced isomorphism of cocharacter groups, then for $\lambda \in X_{*}(\mathscr{T})$ the trivial equality $\prescript{\sigma_{T}}{}(\text{Ad}(g^{-1})^{\sharp}\lambda) = \text{Ad}(g^{-1})^{\sharp}(\prescript{\sigma_{T}}{}\lambda)$ gives that for $w \in W$, $r_{p^{g}}(w) = g r_{p}(w) g^{-1}$. One checks that $n(w)$ is also replaced by $g n(w) g^{-1}$, and so the embedding $\xi$ is replaced by $\text{Ad}(g) \circ \xi$, which is in the same $\widehat{G}^{\Gamma}$-conjugacy class as $\xi$. 
\end{fact}

\begin{fact}\label{fact2} The conjugacy class of $\xi$ is also independent of our choice of Borel subgroup $T_{F^{s}} \subset B \subset G_{F^{s}}$. If $B'$ is another such subgroup, we may find $v \in N_{G}(T)(F^{s})$ such that $vBv^{-1} = B$, and denote the corresponding admissible embedding by $\xi'$. Transporting $\restr{\text{Ad}(v)}{T}$ to $W(\widehat{G}, \mathscr{T})$ using $\xi$, we obtain an element $\mu \in W(\widehat{G}, \mathscr{T})$. Then it is proved in \cite[Lemma 2.6.A]{LS} (the proof of which relies on Lemmas 2.1.A and 2.3.B loc. cit.---we have already discussed the former. The latter depends on torus normalizers, root theory, $a$-data, and the Tits section, which may be dealt with over $F^{s}$, so the proof loc. cit. works verbatim) that we have the equality $$\text{Ad}(g^{-1}) \circ \xi = \xi',$$ where $g \in N_{\widehat{G}}(\mathscr{T})(\mathbb{C})$ acts on $\mathscr{T}$ as $\mu$, giving the claim. 
\end{fact}

\begin{fact}For dependence on the $\chi$-data $\{\chi_{\alpha}\}$ for the $\Gamma$-action on $R(G_{F^{s}}, T_{F^{s}})$, we fix another $\chi$-data $\{\chi'_{\alpha}\}$, and we write $\chi'_{\alpha} = \zeta_{\alpha} \cdot \chi_{\alpha}$, where $\zeta_{\alpha}$ is a character of $F_{\alpha}$. The set $\{\zeta_{\alpha}\}_{\alpha \in R}$ then satisfies the hypotheses of Lemma \ref{cor2.5} (where, in the notation of that lemma, $\mathscr{R} = X_{*}(\mathscr{T})$ with $\xi$-transported $\Gamma$-action); we then obtain a 1-cocycle $c \in Z^{1}(W, \mathscr{T}(\mathbb{C}))$ whose class $[c] \in H^{1}(W, \mathscr{T}(\mathbb{C}))$ is independent of any choices made in the construction of $c$ from $\{\zeta_{\alpha}\}$. Then it's immediate from the construction of $c$ that the embedding $\xi$ is replaced by $t \rtimes w \mapsto c(w) \cdot \xi(t \rtimes w)$. 
\end{fact}

\begin{fact}\label{changeoftori}Finally, suppose that we take another $F$-torus $T'$, and take $g \in G(F^{s})$ such that $\text{Ad}(g)$ is an $F$-isomorphism from $T$ to $T'$. Note that $\text{Ad}(g)$ identifies a $\chi$-data $\{\chi_{\alpha}\}$ for $T$ with $\chi$-data $\{\chi'_{\beta}\}$ for $T'$, since the induced map on character groups is $\Gamma$-equivariant; take $\{\chi'_{\beta}\}$ to be the $\chi$-data for $T'$ used to construct any admissible $L$-embeddings. The map $\text{Ad}(g)$ extends to an isomorphism of $L$-groups $\lambda_{g} \colon \prescript{L}{}T \to \prescript{L}{}T'$. Let $\xi$ be the embedding $\prescript{L}{}T \to \prescript{L}{}G$ constructed above, determined by a choice of Borel subgroup $B$ containing $T_{F^{s}}$. Then we have the equality of admissible embeddings $\xi \circ \lambda_{g} = \xi'$, where $\xi'$ is the admissible embedding $\prescript{L}{}T' \to \prescript{L}{}G$ constructed above corresponding to the $\chi$-data $\{\chi'_{\beta}\}$ and the Borel subgroup $gBg^{-1}$ containing $(T')_{F^{s}}$. We conclude that the $\widehat{G}$-conjugacy class of embeddings $\prescript{L}{}T \to \prescript{L}{}G$ attached to the $\chi$-data $\{\chi_{\alpha}\}$ for $T$ is equivalent to the class of embeddings $\prescript{L}{}T' \to \prescript{L}{}G$ attached to $\{\chi'_{\beta}\}$ for $T'$ via $\lambda_{g}$. 
\end{fact}

\subsection{The local transfer factor}
We construct one factor at a time, following \cite[\S 3]{LS} and \cite[\S 3]{KS2}. Recall that $G$ is a fixed connected reductive group over $F$ a local field of positive characteristic and $\psi \colon G_{F^{s}} \to G^{*}_{F^{s}}$ is a quasi-split inner form of $G$. We fix an endoscopic datum $(H, \mathcal{H}, \eta, s)$ of $G$, which may also be viewed as an endoscopic datum for $G^{*}$, since we are taking the dual group of $G^{*}$ to be $\widehat{G}$ with bijection of based root data given by $\Psi(G^{*})^{\vee} \xrightarrow{\psi} \Psi(G)^{\vee} \to \Psi(\widehat{G})$. Let $\gamma_{H}, \bar{\gamma}_{H} \in H_{G-\text{sr}}(F)$ with corresponding images $\gamma_{G}, \bar{\gamma}_{G} \in G_{\text{sr}}(F)$. Denote by $T_{H}, \bar{T}_{H}$ the centralizers in $H$ of $\gamma_{H}, \bar{\gamma}_{H}$ respectively; these are maximal $F$-tori. By Corollary \ref{admiss}, we may fix two admissible embeddings $T_{H} \xrightarrow{\sim} T \hookrightarrow G^{*}$, $\bar{T}_{H} \xrightarrow{\sim} \bar{T} \hookrightarrow G^{*}$. Recall that such embeddings are unique up to conjugation by elements of $\mathfrak{A}(T), \mathfrak{A}(\bar{T})$---denote by $\gamma, \bar{\gamma} \in T(F), \bar{T}(F)$ the images of $\gamma_{H}, \bar{\gamma}_{H}$ under the above embeddings. 

Set $R:=R(G^{*}_{F^{s}}, T_{F^{s}}), \bar{R} = R(G^{*}_{F^{s}}, \bar{T}_{F^{s}})$, similarly with $R^{\vee}, \bar{R}^{\vee}$. Fix $a$- and $\chi$-data for the standard $\Gamma$ actions on $R$ and $\bar{R}$---these may also be viewed as data for the $\Gamma$-action on $R^{\vee}, \bar{R}^{\vee}$, and data for the $\Gamma$-action on $R((G_{\text{\text{sc}}}^{*})_{F^{s}}, (T_{\text{\text{sc}}})_{F^{s}}), R((G_{\text{\text{sc}}}^{*})_{F^{s}}, (\bar{T}_{\text{\text{sc}}})_{F^{s}})$, where $G^{*}_{\text{\text{sc}}}$ denotes the simply-connected cover of $\mathscr{D}(G^{*})$, and $T_{\text{\text{sc}}}$ denotes the preimage of $T \cap \mathscr{D}(G^{*})$ in this group (analogously for $\bar{T}$). If we replace the embedding $T_{H} \to G^{*}$ by a $\mathfrak{A}(T)$-conjugate $T_{H} \to T'$, then we may view the $a$- and $\chi$-data as data for $R(G^{*}_{F^{s}}, T'_{F^{s}})$. Our goal will be to define a value $$\Delta(\gamma_{H}, \gamma_{G}; \bar{\gamma}_{H}, \bar{\gamma}_{G}) \in \mathbb{C}$$ which will be constructed purely from the admissible embeddings, the map $\psi$, and the $a$- and $\chi$-data, but which only depends on the four inputs. As such, we need to examine the following two things: \begin{enumerate} 
\item{How $\Delta$ changes when we replace the admissible embeddings $T_{H} \to G^{*}$, $\bar{T}_{H} \to G^{*}$ by $\mathfrak{A}(T), \mathfrak{A}(\bar{T})$-conjugates, and use the translated $a$- and $\chi$-data;}
\item{How $\Delta$ changes when we keep the admissible embeddings the same but pick different $a$- and $\chi$-data.}
\end{enumerate}

In light of these observations, we may fix $\Gamma$-splittings $(\mathscr{B}, \mathscr{T},\{X\}), (\mathscr{B}_{H}, \mathscr{T}_{H},\{X^{H}\})$ of $\widehat{G}$, $\widehat{H}$, respectively, that give rise to our admissible embeddings $T_{H} \to T$, $\bar{T}_{H} \to \bar{T}$, since choosing different splittings only serves to conjugate the admissible embeddings by $\mathfrak{A}(T)$, $\mathfrak{A}(\bar{T})$, which is included in condition (1). Implicit in the construction of the admissible embedding $T_{H} \to G^{*}$ is also the choice of $g \in \widehat{G}(\mathbb{C})$ such that $\text{Ad}(g)[\eta(\mathscr{T}_{H})] = \mathscr{T}$ and $\text{Ad}(g)[\eta(\mathscr{B}_{H})] \subset \mathscr{B}$; thus, if we replace the endoscopic datum by $(H, \mathcal{H}, \text{Ad}(g) \circ \eta, s)$, then $\gamma_{H}, \bar{\gamma}_{H} \in H(F)$ are still strongly $G$-regular, and so if we carry out the construction of $\Delta$ for this datum, the admissible embeddings and $a$- and $\chi$-data are unaffected, and hence our value of $\Delta$ will be the same. If we choose a different $g \in \widehat{G}(\mathbb{C})$ satisfying the above properties, it again only serves to replace our admissible embeddings with $\mathfrak{A}$-conjugates. The upshot is that we may assume that $\eta$ carries $\mathscr{T}_{H}$ to $\mathscr{T}$ and $\mathscr{B}_{H}$ into $\mathscr{B}$. We also note that although our admissible embeddings are given by (non-unique) choices of Borel subgroups of $G^{*}$ and $H$, our constructions (apart from that of $\Delta_{III_{2}}$, as addressed in \S 6.3.4) only depend on the admissible embeddings themselves, so in general we do not fix such choices.

Suppose we have a fixed admissible embedding $T_{H} \xrightarrow{f} T$, dual to $\widehat{T_{H}} \xrightarrow{\hat{f}} \widehat{T}$. Recall that we have our element $s \in \widehat{H}(\mathbb{C})$ from the endoscopic datum. Let $B_{H}$ be a Borel subgroup containing $(T_{H})_{F^{s}}$ which is used to induce $f$ (again, there is no such unique $B_{H}$ in general).  Since by assumption $s \in Z(\widehat{H})(\mathbb{C})$, it lies in $\mathscr{T}_{H}(\mathbb{C})$ and its preimage under the map $\widehat{T}_{H} \xrightarrow{\sim} \mathscr{T}_{H}$ induced by $B_{H}$ (and our fixed $(\mathscr{B}_{H}, \mathscr{T}_{H})$) is independent of choice of $B_{H}$. We conclude that the image of $s$ in $\widehat{T}(\mathbb{C})$, denoted by $s_{T}$, only depends on the choice of admissible embedding $T_{H} \to T$. In the definition of an endoscopic datum, it is assumed that $s \in Z(\widehat{H})^{\Gamma} \cdot \eta^{-1}(Z(\widehat{G}))$, and hence the preimage of $s$ in $\widehat{T_{H}}(\mathbb{C})$ lies in $\iota(Z(\widehat{H}))^{\Gamma} \cdot \widehat{f}^{-1}(\iota(Z(\widehat{G})))$, where we have pedantically denoted the canonical embeddings $Z(\widehat{H}) \to \widehat{T_{H}}, Z(\widehat{G}) \to \widehat{T}$ by $\iota$, and have also used the fact that $Z(\widehat{H}) \to \widehat{T_{H}}$ is canonical to obtain $\Gamma$-equivariance. This implies (since $\widehat{f}$ is $\Gamma$-equivariant) that $s_{T}$ lies in $\widehat{T}_{\text{\text{ad}}}^{\Gamma}$, and we set $\textbf{s}_{T}$ to be its image in $\pi_{0}(\widehat{T}_{\text{\text{ad}}}^{\Gamma})$. 

We make the assumption throughout this section that for any endoscopic datum, $\mathcal{H} = \prescript{L}{}H$ with embedding $\widehat{H} \to \prescript{L}{}H$ the canonical embedding; this assumption will only be necessary in \S 6.3.4. We will discuss how to deal with general $\mathcal{H}$ in \S 6.4.

\subsubsection{The factor $\Delta_{I}$}
We set  $$\Delta_{I}(\gamma_{H}, \gamma_{G}) := \langle \lambda_{\{a_{\alpha}\}}(T_{\text{\text{sc}}}), \textbf{s}_{T} \rangle, $$
where we view the $a$-data for $T$ as an $a$-data for $T_{\text{\text{sc}}}$, the pairing $\langle - , - \rangle$ is from Tate-Nakayama duality, and $\lambda_{\{a_{\alpha}\}}(T_{\text{\text{sc}}})$ is the splitting invariant associated to the maximal $F$-torus $T_{\text{\text{sc}}} \hookrightarrow G^{*}_{\text{\text{sc}}}$, a fixed $F$-splitting $\mathscr{S}$ of $G_{\text{\text{sc}}}^{*},$ and the $a$-data $\{a_{\alpha}\}$. 

\begin{lem}\label{splitting} The value $$\frac{\Delta_{I}(\gamma_{H}, \gamma_{G})}{\Delta_{I}(\bar{\gamma}_{H}, \bar{\gamma}_{G})}$$ is independent of the splitting $\mathscr{S}$. 
\end{lem}

\begin{proof} Suppose that we replace $\mathscr{S} = (B, S, \{X_{\alpha}\})$ by another $F$-splitting $\mathscr{S}' = (B', S', \{X_{\alpha}'\})$ of $G^{*}_{\text{\text{sc}}}$. It will be necessary to use fppf cohomology here, since these two splittings need not be $G^{*}(F^{s})$-conjugate. Accordingly, take $z \in G_{\text{\text{sc}}}^{*}(\bar{F})$ such that $z\mathscr{S}'z^{-1} = \mathscr{S}$ and $p_{1}^{\sharp}(z) p_{2}^{\sharp}(z)^{-1} \in Z_{\text{\text{sc}}}(\bar{F} \otimes_{F} \bar{F}):= Z(G_{\text{\text{sc}}}^{*})(\bar{F} \otimes_{F} \bar{F})$ (this last inclusion is because $\text{Ad}(z) \in G_{\text{ad}}(F)$). Then if $B_{T}$ is a fixed Borel subgroup containing $(T_{\text{\text{sc}}})_{F^{s}}$ and $h \in G_{\text{\text{sc}}}^{*}(F^{s})$ carries $(B_{F^{s}}, S_{F^{s}})$ to $(B_{T}, (T_{\text{\text{sc}}})_{F^{s}})$, then $hz$ carries $(B'_{F^{s}},S'_{F^{s}})$ to $(B_{T}, (T_{\text{\text{sc}}})_{F^{s}})$, and for all $\sigma \in \Gamma$, we have $n_{S'}(\omega_{T}(\sigma)) = \text{Ad}(z^{-1})n_{S}(\omega_{T}(\sigma)) \in N_{G^{*}_{\text{\text{sc}}}}(S')(F^{s})$ (notation as in the definition of the splitting invariant, where $n_{S}, n_{S'}$ denote the Tits sections corresponding to $\mathscr{S}, \mathscr{S}'$), similarly for $x(\sigma)$. We need to be careful here, since we defined the splitting invariant in terms of a Galois cocycle and it is not in general true that $z \in G^{*}_{\text{\text{sc}}}(F^{s})$. However, recall the definition of the splitting invariant: the cocycle $m$ is still a Galois cocycle for us, since $x(\sigma) \in G^{*}_{\text{\text{sc}}}(F^{s})$ and $n(\omega_{T}(\sigma)) \in N_{G^{*}_{\text{\text{sc}}}}(F^{s})$, and we may view it as a \v{C}ech cocycle $m \in G_{\text{\text{sc}}}^{*}(\bar{F} \otimes_{F} \bar{F})$. Then we may set $$\lambda_{\{a_{\alpha}\}}(T):= p_{1}^{\sharp}(h) m p_{2}^{\sharp}(h)^{-1} \in T_{\text{\text{sc}}}(\bar{F} \otimes_{F} \bar{F}),$$ and get the same definition as in \S6.2.1. However, this modified definition allows us to compute that if $c' \in T_{\text{\text{sc}}}(\bar{F} \otimes_{F} \bar{F})$ is the cocycle used to defined the splitting invariant for $\mathscr{S}'$, then $m' = p_{1}^{\sharp}(z)^{-1}mp_{1}^{\sharp}(z) \in G_{\text{\text{sc}}}^{*}(\bar{F} \otimes_{F} \bar{F})$, and so we have: $$c' = p_{1}^{\sharp}(h) p_{1}^{\sharp}(z)p_{1}^{\sharp}(z)^{-1}mp_{1}^{\sharp}(z)p_{2}^{\sharp}(z)^{-1}p_{2}^{\sharp}(h)^{-1} = p_{1}^{\sharp}(z)p_{2}^{\sharp}(z)^{-1}(p_{1}^{\sharp}(h)mp_{1}^{\sharp}(h)^{-1}),$$ where in the above equality we have used the centrality of $p_{1}^{\sharp}(z)p_{1}^{\sharp}(z)^{-1}$, and we conclude that $\lambda_{\{a_{\alpha}\}}$ computed with respect to $\mathscr{S}'$ differs from the one computed with respect to $\mathscr{S}$ by left-translation by the class $\textbf{z}_{T}$ in $H^{1}(F, T)$ represented by $p_{1}^{\sharp}(z)p_{2}^{\sharp}(z)^{-1}.$  Whence, to prove the lemma, it's enough to show that $$\langle \textbf{z}_{T}, \textbf{s}_{T} \rangle = \langle \textbf{z}_{\bar{T}}, \textbf{s}_{\bar{T}} \rangle.$$

By Proposition \ref{duality} (with $\Gamma = \Gamma_{F^{s}/F}$, recall the definition of the groups $H^{-2}(\Gamma, M)$ for finite $M$ given loc. cit.), we have the following commutative diagram with exact columns
\[
\begin{tikzcd}
H^{1}(F, Z_{\text{\text{sc}}}) \arrow[r, "\sim"] \arrow[d] & H^{-2}(\Gamma, X_{*}(T_{\text{\text{ad}}})/X_{*}(T_{\text{\text{sc}}})) \arrow[d] \\
H^{1}(F, T_{\text{\text{sc}}}) \arrow[r, "\sim"] \arrow[d] & H^{-1}(\Gamma, X_{*}(T_{\text{\text{sc}}})) \arrow[d] \\
H^{1}(F, T_{\text{\text{ad}}}) \arrow[r, "\sim"] & H^{-1}(\Gamma, X_{*}(T_{\text{\text{ad}}})),
\end{tikzcd}
\]
with horizontal isomorphisms induced by Tate-Nakayama duality, as discussed in \S 6.1.3. From here, one may deduce the result from the argument in the proof of \cite[Lemma 3.2.A]{LS}, which looks at the images of $\textbf{z}_{T}, \textbf{z}_{\bar{T}}$ in the right-hand column and then uses group-cohomological calculations, along with the alternative characterization of the Tate-Nakayama pairing that we discussed in Remark \ref{altpairing} (replacing the use of duality results loc. cit. with our Proposition \ref{duality}). Note that although in our Proposition \ref{duality} the groups $H^{-2}(\Gamma_{K/F},X_{*}(T_{\text{\text{ad}}})/X_{*}(T_{\text{\text{sc}}}))$ (for finite Galois $K/F$) do not in general stabilize, their images in $H^{-1}(\Gamma, X_{*}(T_{\text{\text{sc}}}))$ (which stabilizes at a finite level) do, and hence the finite-level arguments in \cite[Lemma 3.2.A]{LS} work in this setting.
\end{proof}

We now discuss how $\Delta_{I}$ changes under conjugation by $\mathfrak{A}(T_{\text{\text{sc}}})$ and another choice of $a$-data.

\begin{lem}\label{change1} The factor $\Delta_{I}$ satisfies:
\begin{enumerate} \item{If $T_{H} \to T$ is replaced by its conjugate under $g \in \mathfrak{A}(T_{\text{\text{sc}}})$, with corresponding transported $a$-data, then $\Delta_{I}(\gamma_{H}, \gamma_{G})$ is multiplied by $\langle \textbf{g}_{T}, \textbf{s}_{T} \rangle^{-1}$, where $\textbf{g}_{T}$ is the class of $\sigma \mapsto g \sigma(g)^{-1}$ in  $H^{1}(F, T_{\text{\text{sc}}})$. }
\item{Suppose that the $a$-data $\{a_{\alpha}\}$ is replaced by $\{a_{\alpha}'\}$. Set $b_{\alpha} = a_{\alpha}'/a_{\alpha}$. Then the term $\Delta_{I}(\gamma_{H}, \gamma_{G})$ is multiplied by the sign $$\prod_{\alpha} \text{sgn}_{F_{\alpha}/F_{\pm \alpha}}(b_{\alpha}),$$
where the product is taken over a set of representatives for the symmetric $\Gamma$-orbits (the orbit of $\alpha$ is \textbf{symmetric} if it contains $-\alpha$, otherwise it is \textbf{asymmetric}) in $R$ that lie outside $R(H_{F^{s}}, (T_{H})_{F^{s}})$.}
\end{enumerate}
\end{lem}

\begin{proof} Part (1) is the analogue of \cite[Lemma 3.2.B]{LS}, and the proof loc. cit. works in our situation, since all elements of $\mathfrak{A}(T_{\text{sc}})$ are separable points, the construction of the splitting invariant only uses separable points, and the Tate-Nakayama duality pairing for tori works the same way in positive characteristic. 

For (2), we first note that the expression $\text{sgn}_{F_{\alpha}/F_{\pm \alpha}}(b_{\alpha})$ makes sense, since $b_{\alpha}$ is fixed by $\Gamma_{\pm \alpha}$, and thus lies in $F_{\pm \alpha}$. Our result is exactly \cite[Lemma 3.4.1]{KS2}, which is proved without assumptions on the characteristic of $F$. 
\end{proof}

\subsubsection{The factor $\Delta_{II}$} We define 
\begin{equation}
\Delta_{II}(\gamma_{H}, \gamma_{G}) = \prod \chi_{\alpha} \left( \frac{\alpha(\gamma) - 1}{a_{\alpha}} \right ),
\end{equation}
where the product is over representatives $\alpha$ for the orbits of $\Gamma$ in $R$ the lie outside $R(H_{F^{s}}, (T_{H})_{F^{s}})$. This is easily checked to be independent of the representatives chosen.

\begin{lem}\label{a2} The factor $\Delta_{II}(\gamma_{H}, \gamma_{G})$ is unaffected by replacing the admissible embedding $T_{H} \to T$ by an $\mathfrak{A}(T)$-conjugate (and the transporting the $\chi$- and $a$-data accordingly). Moreover, replacing the $a$-data $\{a_{\alpha}\}$ by a different data $\{a_{\alpha}'\}$ serves to multiply $\Delta_{II}(\gamma_{H}, \gamma_{G})$ by $$\prod_{\alpha} \text{sgn}_{F_{\alpha}/F_{\pm \alpha}}(b_{\alpha})^{-1},$$ where $b_{\alpha} = a'_{\alpha}/a_{\alpha}$ and the product is over representatives for the symmetric orbits outside $R(H_{F^{s}}, (T_{H})_{F^{s}})$.
\end{lem}

\begin{proof} The arguments in \cite[Lemmas 3.3.B, 3.3.C]{LS} are purely root-theoretic and work verbatim here.
\end{proof}

It remains to check the dependency of $\Delta_{II}$ on the $\chi$-data. Suppose the $\chi$-data $\{\chi_{\alpha}\}$ are replaced by $\{ \chi_{\alpha}' \}$, and set $\zeta_{\alpha} := \chi_{\alpha}'/\chi_{\alpha}$. Note that $\zeta_{\alpha}$ restricts to the trivial character on $F_{\pm \alpha}^{*}$. To analyze this dependency, we will need to introduce some new notation, following \cite[\S 3.3]{LS}. Let $\mathcal{O}$ be a symmetric orbit of $\Gamma$ on $R$, with a gauge $q$, $X^{\mathcal{O}}$ the free abelian group on the elements $\mathcal{O}_{+} = \{\alpha \in \mathcal{O} \colon q(\alpha)=1\}$ ,with inherited $\Gamma$-action, and $X^{\alpha}$ the $\Z$-submodule generated by some $\alpha \in \mathcal{O}_{+}$, which is preserved by $\Gamma_{\pm \alpha}$, and so $X^{\mathcal{O}} = \text{Ind}_{\Gamma_{\pm \alpha}}^{\Gamma} (X^{\alpha})$. We obtain a corresponding $F_{\pm \alpha}$-torus $T^{\alpha}$ which is one-dimensional, anisotropic, and split over $F_{\alpha}$, and corresponding $F$-torus $T^{\mathcal{O}}$ which satisfies $T^{\mathcal{O}} = \text{Res}_{F_{\pm \alpha}/F} T^{\alpha}$. 

We have a natural $\Gamma$-homomorphism $X^{\mathcal{O}} \to X^{*}(T)$ which induces a morphism of $F$-tori $T \to T^{\mathcal{O}}$ that maps $T(F)$ into $T^{\alpha}(F_{\pm \alpha})$; denote by $\gamma^{\alpha}$ the image of $\gamma$ in $T^{\alpha}(F_{\pm \alpha})$. 
Note that the norm map $T^{\alpha}(F_{\alpha}) \to T^{\alpha}(F_{\pm \alpha})$ is surjective, since we have the exact sequence of $F_{\pm \alpha}$-tori 
\begin{center}
 \begin{tikzcd}
  0 \arrow[r] & T'\arrow[r] & \text{Res}_{F_{\alpha}/F_{\pm \alpha}}(T^{\alpha}_{F_{\alpha}})  \arrow[r, "\text{Norm}"] & T^{\alpha}  \arrow[r] & 0,
\end{tikzcd}
\end{center}
where $T'$ is a split $F_{\pm \alpha}$-torus, and so taking the long exact sequence in cohomology (along with Hilbert 90) gives the desired surjectivity. Whence, we may write $$\gamma_{\alpha} = \delta^{\alpha} \overline{\delta^{\alpha}},$$ where $\delta^{\alpha} \in T^{\alpha}(F_{\alpha})$ and the bar denotes the map from $T^{\alpha}(F_{\alpha})$ to itself induced by the unique automorphism of $F_{\alpha}/F_{\pm \alpha}$. 

If $\mathcal{O}$ is an asymmetric $\Gamma$-orbit in $R$, then $X^{\pm \mathcal{O}}$ is defined to be the free abelian group on $\mathcal{O}$ with inherited $\Gamma$-action and $X^{\alpha}$ is the subgroup generated by some $\alpha \in \mathcal{O}$, which again carries a $\Gamma_{\pm \alpha} = \Gamma_{\alpha}$-action. We get a corresponding split 1-dimensional $F_{\alpha}$-torus $T^{\alpha}$ and $F$-torus $T^{\pm \mathcal{O}}$, with $T^{\mathcal{O}} = \text{Res}_{F_{\alpha}/F} T^{\alpha}$. We again obtain a map $T \to T^{\pm \mathcal{O}}$, inducing a map $T(F) \to T^{\alpha}(F_{\alpha})$; denote the image of $\gamma$ under this map by $\gamma^{\alpha}$. We are now ready to state how $\Delta_{II}$ changes when we alter the $\chi$-data.

\begin{lem}\label{chi2} If the $\chi$-data $\{\chi_{\alpha}\}$ are replaced by $\{ \chi_{\alpha}' \}$, with $\zeta_{\alpha} = \chi_{\alpha}'/\chi_{\alpha}$, then $\Delta_{II}(\gamma_{H}, \gamma_{G})$ is multiplied by 
$$\prod_{\text{aysmm}} \zeta_{\alpha}(\gamma^{\alpha}) \cdot \prod_{\text{symm}} \zeta_{\alpha}(\delta^{\alpha}),$$ where $\prod_{\text{asymm}}$ denotes the product over representatives $\alpha$ for pairs $\pm \mathcal{O}$ of asymmetric orbits of $R$ outside $H$, and to make sense of $\zeta_{\alpha}(\gamma^{\alpha})$, we are using the canonical isomorphism $T^{\alpha} \xrightarrow{\sim} \mathbb{G}_{m}$ given on character groups by $1 \mapsto \alpha$, and $\prod_{\text{symm}}$ is the product over representatives $\alpha$ for the symmetric orbits of $R$ outside $H$, and to make sense of $\zeta_{\alpha}(\delta^{\alpha})$ we are using the canonical isomorphism $T^{\alpha}_{F_{\alpha}} \xrightarrow{\sim} \mathbb{G}_{m}$ given on character groups by $1 \mapsto \alpha$.
\end{lem}

\begin{proof} This is \cite[Lemma 3.3.D]{LS}, the proof of which (along with the proof of Lemma 3.3.A loc. cit.) carries over to our setting verbatim.
\end{proof}

\subsubsection{The factor $\Delta_{III_{1}}$ (or $\Delta_{1}$)}
The construction of this factor is the only part of the construction of the relative local transfer factor that involves fppf cohomology rather than Galois cohomology. For the moment, we will assume that $G$ is quasi-split over $F$, with $\psi = \text{id}$; the construction of $\Delta_{1}$ in this case can be done using Galois cohomology, but in order to match more closely with the general case, we work in the setting of fppf cohomology. By construction, the admissible embedding $T_{H} \to G$ is obtained by first taking $T_{H} \xrightarrow{\sim} T_{G}$ determined by $\gamma_{H}, \gamma_{G}$ and then conjugating an embedding $(T_{G})_{F^{s}} \to G_{F^{s}}$ induced by a choice of Borel subgroup containing $(T_{G})_{F^{s}}$ and $(\mathscr{B}, \mathscr{T})$ by some appropriate $g \in G_{\text{sc}}(\bar{F})$. As a consequence, we see that $\gamma_{G}$ and $\gamma$ are conjugate by some $h \in G_{\text{\text{sc}}}(\bar{F})$ such that $p_{1}^{\sharp}(h) p_{2}^{\sharp}(h)^{-1} \in T_{\text{\text{sc}}}(\bar{F} \otimes_{F} \bar{F})$. We then set $v = p_{1}^{\sharp}(h)p_{2}^{\sharp}(h)^{-1}$ and denote the class of $v$ in $H^{1}(F, T_{\text{\text{sc}}})$ by $\text{inv}(\gamma_{H}, \gamma_{G})$; this class is independent of the choice of $h$, since if we choose any other $h' \in G_{\text{\text{sc}}}(\bar{F})$ with $h' \gamma_{G} h'^{-1} = \gamma$, then $h^{-1} h' \in T_{\text{\text{sc}}}(\bar{F})$, since $\gamma$ is strongly regular. We then set $$\Delta_{1}(\gamma_{H}, \gamma_{G}) = \langle \text{inv}(\gamma_{H}, \gamma_{G}), \textbf{s}_{T} \rangle ^{-1}.$$

Now we return to the setting of a general connected reductive group $G$ over $F$ with $\psi \colon G_{F^{s}} \to G^{*}_{F^{s}}$ the quasi-split inner form of $G$ over $F$ with the assumptions stated in the beginning of \S 6.3. In particular, we have two pairs of elements $\gamma_{H}, \gamma_{G}$ and $\bar{\gamma}_{H}, \bar{\gamma}_{G}.$ As in the quasi-split case, we may find $h, \bar{h} \in G^{*}_{\text{\text{sc}}}(\bar{F})$ such that $$h \psi(\gamma_{G}) h^{-1} = \gamma, \hspace{1mm} \bar{h} \psi(\bar{\gamma}_{G}) \bar{h}^{-1} = \bar{\gamma}.$$ One could take $h, \bar{h} \in G^{*}_{\text{sc}}(F^{s})$, but since we will be using these elements to construct fppf \v{C}ech cocycles, we want to view them as $\bar{F}$-points anyway. Further, let $u \in G^{*}_{\text{\text{sc}}}(\bar{F} \otimes_{F} \bar{F})$ be such that $p_{1}^{*}\psi \circ p_{2}^{*}\psi^{-1} = \text{Ad}(u)$ on $G^{*}_{\bar{F} \otimes_{F} \bar{F}}$; the fact that one cannot in general find such a $u$ in $G^{*}_{\text{sc}}(F^{s} \otimes_{F} F^{s})$ is the reason we need to use fppf cohomology to define the $\Delta_{III_{1}}$ factor. We then obtain two (\v{C}ech) cochains, $$v := p_{1}^{\sharp}(h)up_{2}^{\sharp}(h)^{-1} \in T_{\text{\text{sc}}}(\bar{F} \otimes_{F} \bar{F}), \hspace{1mm} \bar{v} := p_{1}^{\sharp}(\bar{h})up_{2}^{\sharp}(\bar{h})^{-1} \in \bar{T}_{\text{\text{sc}}}(\bar{F} \otimes_{F} \bar{F}) ;$$ we have that $v \in T_{\text{\text{sc}}}(\bar{F} \otimes_{F} \bar{F})$ because (since $\gamma, \gamma_{G}$ are $F$-points) $$v \gamma v^{-1} = p_{1}^{\sharp}(h)(p_{1}^{*}\psi \circ p_{2}^{*}\psi^{-1}(p_{2}^{\sharp}(\psi(\gamma_{G}))))p_{1}^{\sharp}(h)^{-1} = p_{1}^{\sharp}(h) p_{1}^{\sharp}(\psi(\gamma_{G})) p_{1}^{\sharp}(h)^{-1} = \gamma,$$ similarly for $\bar{v}$.

By construction, we have $dv = d\bar{v} = du \in Z_{\text{\text{sc}}}(\bar{F} \otimes_{F} \bar{F}),$ where recall that $Z_{\text{\text{sc}}} := Z(G_{\text{\text{sc}}}^{*})$, and by $d$ we are denoting the \v{C}ech differential. We have an embedding $Z_{\text{\text{sc}}} \to T_{\text{\text{sc}}} \times \bar{T}_{\text{\text{sc}}}$ defined by $i^{-1} \times j$, where $i$ and $j$ denote the obvious inclusions. Set $$U(T, \bar{T}) = U := \frac{T_{\text{\text{sc}}} \times \bar{T}_{\text{\text{sc}}}}{Z_{\text{\text{sc}}}},$$ which is an $F$-torus. We have the following easy lemma:

\begin{lem}The image of $(v, \bar{v}) \in T_{\text{\text{sc}}}(\bar{F} \otimes_{F} \bar{F}) \times \bar{T}_{\text{\text{sc}}}(\bar{F} \otimes_{F} \bar{F})  = (T_{\text{\text{sc}}} \times \bar{T}_{\text{\text{sc}}})(\bar{F} \otimes_{F} \bar{F})$ in $U(\bar{F} \otimes_{F} \bar{F})$ is a 1-cocycle, whose cohomology class, denoted by
\begin{equation} \text{inv} \left( \frac{\gamma_{H}, \gamma_{G}}{\bar{\gamma}_{H}, \bar{\gamma}_{G}} \right ) \in H^{1}(F,U),
\end{equation}
is independent of the choices of $u, h, \bar{h}.$
\end{lem}

\begin{proof} The fact that the above defines a 1-cocycle is trivial, since $$U(\bar{F} \otimes_{F} \bar{F}) = \frac{T_{\text{\text{sc}}}(\bar{F} \otimes_{F} \bar{F}) \times \bar{T}_{\text{\text{sc}}}(\bar{F} \otimes_{F} \bar{F})}{Z_{\text{\text{sc}}}(\bar{F} \otimes_{F} \bar{F})},$$ using the fact that $H^{1}(\bar{F} \otimes_{F} \bar{F}, Z_{\text{\text{sc}}}) = 0$, and the construction of $v$, $\bar{v}$, and $U$. Replacing $u$ by $u'$ satisfies $u'  = uz$, $z \in Z_{\text{\text{sc}}}(\bar{F})$, and so the new element $(v', \bar{v}') \in T_{\text{\text{sc}}} \times \bar{T}_{\text{\text{sc}}}$ is equivalent to $(v, \bar{v})$ modulo $Z_{\text{\text{sc}}}$. Replacing $h$ by $h' = ht$, where $t \in T_{\text{\text{sc}}}(\bar{F})$, gives $v' = d(t) \cdot v \in T_{\text{\text{sc}}}(\bar{F} \otimes_{F} \bar{F})$, and so the image of $(v' , \bar{v})$ in $U$ differs from the image of $(v, \bar{v})$ by $(d(t),1)$, a coboundary, similarly with the element $\bar{h}$. 
\end{proof}
Note that if $G$ is quasi-split and $\pi$ denotes the quotient map defining $U$, then 
\begin{equation}
\left( \frac{\gamma_{H}, \gamma_{G}}{\bar{\gamma}_{H}, \bar{\gamma}_{G}} \right ) = \pi[ (\text{inv}(\gamma_{H}, \gamma_{G})^{-1}, \text{inv}(\bar{\gamma}_{H}, \bar{\gamma}_{G})) ].
\end{equation}

\noindent Now let $\hat{T}_{\text{\text{sc}}}$ denote the torus dual to $T_{\text{\text{ad}}} = T/Z(G)$, and set $\widehat{Z}_{\text{\text{sc}}} := Z(\widehat{G}_{\text{\text{sc}}})$. The homomorphism $X_{*}(T) \to X_{*}(T_{\text{\text{ad}}})$ induces a morphism of $\hat{T}_{\text{\text{sc}}} \to \widehat{T} \hookrightarrow \widehat{G}$ (using an isomorphism $\widehat{T} \to \mathscr{T}$ giving our admissible embedding) which factors through $\mathscr{D}(\widehat{G}) \cap \widehat{T}$ by dimension and root system considerations. From this, one obtains $\hat{T}_{\text{\text{sc}}} \to \mathscr{D}(\widehat{G})$ which further factors through an embedding $\hat{T}_{\text{\text{sc}}} \to \mathscr{D}(\widehat{G})_{\text{\text{sc}}}$ that identifies $\hat{T}_{\text{\text{sc}}}$ with a maximal torus of $\widehat{G}_{\text{\text{sc}}}$, giving an embedding $\widehat{Z}_{\text{\text{sc}}} \hookrightarrow \hat{T}_{\text{\text{sc}}}$ which is canonical (because of centrality, this does not depend on our initial embedding of $\widehat{T}$ in $\widehat{G}$). The same result holds for $\hat{\bar{T}}_{\text{\text{sc}}}$.

With this in hand, we set $$\widehat{U} := \frac{\hat{T}_{\text{\text{sc}}} \times \hat{\bar{T}}_{\text{\text{sc}}}}{\widehat{Z}_{\text{\text{sc}}}},$$ where now $\widehat{Z}_{\text{\text{sc}}}$ is embedded diagonally. The $\mathbb{Q}$-pairing $\mathbb{Q}R^{\vee} \times \mathbb{Q}R \to \mathbb{Q}$ gives a pairing $X^{*}(\hat{T}_{\text{\text{sc}}}) \times X^{*}(T_{\text{\text{sc}}}) \to \mathbb{Q}$ which, together with the analogue for $\bar{T}$, yields a $\mathbb{Q}$-pairing between $X^{*}(\bar{T}_{\text{\text{sc}}} \times \hat{\bar{T}}_{\text{\text{sc}}})$ and $X^{*}(T_{\text{\text{sc}}} \times \bar{T}_{\text{\text{sc}}})$, which further induces a perfect $\mathbb{Z}$-pairing between $X^{*}(\widehat{U})$ and $X^{*}(U)$, identifying $\widehat{U}$ with the dual of $U$, see \cite[\S 3.4]{LS}. 

Take the projection of $\eta(s) \in \mathscr{T}(\mathbb{C})$ in $\mathscr{T}_{\text{\text{ad}}}(\mathbb{C})$, and then pick an arbitrary preimage  $\tilde{s}$ of this projection in $\mathscr{T}_{\text{\text{sc}}}(\mathbb{C})$. We have isomorphisms $\hat{T}_{\text{\text{sc}}} \to \mathscr{T}_{\text{\text{sc}}}$, $\hat{\bar{T}}_{\text{\text{sc}}} \to \mathscr{T}_{\text{\text{sc}}}$ induced by choices of isomorphisms $\widehat{T}, \widehat{\bar{T}} \to \mathscr{T}$ giving our admissible embeddings, and the respective preimages of $\tilde{s}$, denoted by $\tilde{s}_{T}, \tilde{s}_{\bar{T}}$, only depend on choice of $\tilde{s}$ and the admissible isomorphisms $T_{H} \to T, \bar{T}_{H} \to \bar{T}$. We then set $s_{U}:= (\tilde{s}_{T}, \tilde{s}_{\bar{T}}) \in \widehat{U}(\mathbb{C})$. Note that a different choice of $\tilde{s}$ corresponds to replacing $\tilde{s}_{T}, \tilde{s}_{\bar{T}}$ by $\tilde{s}_{T}z_{T}, \tilde{s}_{\bar{T}}z_{\bar{T}}$, where $z \in \widehat{Z}_{\text{\text{sc}}}(\mathbb{C})$ and $z_{T}, z_{\bar{T}}$ denote the images of $z$ under the canonical embeddings of $\widehat{Z}_{\text{\text{sc}}}$ in $\hat{T}_{\text{\text{sc}}}, \hat{\bar{T}}_{\text{\text{sc}}}$. Thus, $s_{U}$ is independent of the choice of $\tilde{s}$. Then one can show that $s_{U} \in \widehat{U}^{\Gamma}$, see for example the discussion of the $\Delta_{III_{1}}$ factor in \cite{Tasho}, proof of Proposition 5.6. Hence, it makes sense to define $\textbf{s}_{U}$ to be the image of $s_{U}$ in $\pi_{0}(\widehat{U}^{\Gamma})$. We then set 
\begin{equation}
\Delta_{III_{1}}(\gamma_{H}, \gamma_{G}; \bar{\gamma}_{H}, \bar{\gamma}_{G}) := \langle \text{inv} \left( \frac{\gamma_{H}, \gamma_{G}}{\bar{\gamma}_{H}, \bar{\gamma}_{G}} \right ), \textbf{s}_{U} \rangle .
\end{equation}
By what we have done, it is clear that if $G$ is quasi-split over $F$, then $$\Delta_{III_{1}}(\gamma_{H}, \gamma_{G}; \bar{\gamma}_{H}, \bar{\gamma}_{G}) = \langle \text{inv}(\gamma_{H}, \gamma_{G}), \textbf{s}_{T} \rangle ^{-1} \langle \text{inv}(\bar{\gamma}_{H}, \bar{\gamma}_{G}), \textbf{s}_{\bar{T}} \rangle.$$

\begin{lem}\label{embedding3} If $T_{H} \to T$ and $\bar{T}_{H} \to \bar{T}$ are replaced by their $g$- and $\bar{g}$-conjugates, $g, \bar{g} \in \mathfrak{A}(T_{\text{\text{sc}}}), \mathfrak{A}(\bar{T}_{\text{\text{sc}}})$, then $\Delta_{III_{1}}(\gamma_{H}, \gamma_{G}; \bar{\gamma}_{H}, \bar{\gamma}_{G})$ is multiplied by $$\langle \textbf{g}_{T}, \textbf{s}_{T} \rangle \langle \textbf{g}_{\bar{T}}, \textbf{s}_{\bar{T}} \rangle^{-1},$$ where $\textbf{g}_{T}$ is the class of the 1-cocycle $p_{1}^{\sharp}(g)p_{2}^{\sharp}(g)^{-1} \in T_{\text{\text{sc}}}(\bar{F} \otimes_{F} \bar{F})$, analogously for $\textbf{g}_{\bar{T}}$.
\end{lem}

\begin{proof} Denote the $g^{-1}, \bar{g}^{-1}$-conjugates of $T, \bar{T}$ by $T'$, $\bar{T}'$. One checks that $v$ as defined above is replaced by $p_{1}^{\sharp}(g)^{-1} v p_{2}^{\sharp}(g) \in T'_{\text{\text{sc}}}(\bar{F} \otimes_{F} \bar{F})$ and conjugating this element by $p_{1}^{\sharp}(g)$ yields the element $v (p_{1}^{\sharp}(g) p_{2}^{\sharp}(g)^{-1})^{-1}$, analogously for $\bar{T}$ and $\bar{v}$. Similarly, $\tilde{s}_{T}, \tilde{s}_{\bar{T}}$ can be taken to be $\text{Ad}(g)\tilde{s}_{T'}$ (by $\text{Ad}(g)$, we mean the induced dual map $\hat{T'}_{\text{\text{sc}}} \to \hat{T}_{\text{\text{sc}}}$) and $\text{Ad}(\bar{g})\tilde{s}_{\bar{T}'}$. The functoriality of the Tate-Nakayama pairing then gives the result.
\end{proof}

\subsubsection{The factor $\Delta_{III_{2}}$} To construct this factor, we will fix Borel subgroups $B \supset T_{F^{s}}$, $B_{H} \supset (T_{H})_{F^{s}}$ which (along with our fixed $(\mathscr{B}, \mathscr{T}), (\mathscr{B}_{H}, \mathscr{T}_{H})$) determine the admissible isomorphism $T_{H} \to T$; note that our $\chi$- and $a$-data also serve the $\Gamma$-action on $R(H_{F^{s}}, (T_{H})_{F^{s}}) \subset R$. Then, according to \S 6.2.2, we obtain from our $\chi$-data $\{\chi_{\alpha}\}$ (viewed as a $\chi$-data for $T$ and for $T_{H}$) admissible embeddings $\xi_{T} \colon \prescript{L}{}T \to \prescript{L}{}G$ extending the map $\widehat{T} \to \mathscr{T}$ and $\xi_{T_{H}} \colon \prescript{L}{}T_{H} \to \prescript{L}{}H$ extending $T_{H} \to \mathscr{T}_{H}$. We then obtain $$\eta \circ \xi_{T_{H}} = a \cdot \xi_{T},$$ where we view $\xi_{T}$ as a map on $\prescript{L}{}T_{H}$ by means of the isomorphism $\prescript{L}{}T_{H} \to \prescript{L}{}T$ induced by the admissible isomorphism $T_{H} \to T$ and $a$ is a 1-cocycle in $\mathscr{T}(\mathbb{C})$ for the $\widehat{T}$-transported $W_{F}$-action. Its class $\textbf{a}$ in $H^{1}(W_{F}, \widehat{T}(\mathbb{C}))$ (after applying the fixed isomorphism $\mathscr{T} \to \widehat{T}$ to $a$) is independent of the choice of $B_{H}$ and $B$, as well as the $\Gamma$-splittings $(\mathscr{B}, \mathscr{T},\{X\})$ and $(\mathscr{B}_{H}, \mathscr{T}_{H},\{X^{H}\})$ by Facts \ref{fact2} and \ref{fact1} from \S 6.2, respectively. 

Suppose now that $T_{H} \to T$ (and the corresponding data) is replaced with a $g \in \mathfrak{A}(T_{\text{\text{sc}}})$-conjugate $T' = \text{Ad}(g^{-1})T$ with admissible embedding $\xi_{T'}$. Then Fact \ref{changeoftori} from \S 6.2 shows that the induced isomorphism $\lambda_{g} \colon \prescript{L}{}T' \to \prescript{L}{}T$ satisfies $\xi_{T} \circ \lambda_{g} = \xi_{T'}$, and so it follows that the class $\textbf{a}$ is the image of $\textbf{a}' \in H^{1}(W_{F}, \widehat{T'}(\mathbb{C}))$ under the isomorphism $H^{1}(W_{F}, \widehat{T'}(\mathbb{C})) \xrightarrow{\text{Ad}(g)} H^{1}(W_{F}, \widehat{T}(\mathbb{C}))$. The dependence on the $\chi$-data will be addressed later.

We then set $$\Delta_{III_{2}}(\gamma_{H}, \gamma_{G}) := \langle \textbf{a}, \gamma \rangle,$$ where the above pairing comes from Langlands duality for tori, as in Theorem \ref{toriduality}. By the functoriality of the pairing (Theorem \ref{toriduality}) and our above remarks on the cocycle $\textbf{a}$, it is immediate that this number does not change if the admissible embedding $T_{H} \to T$ (and corresponding data) is changed by a $\mathfrak{A}(T_{\text{\text{sc}}})$-conjugate.

\begin{lem}\label{chi3} Suppose that the $\chi$-data $\{\chi_{\alpha}\}$ is replaced by $\{\chi_{\alpha}'\}$, with $\zeta_{\alpha} := \chi_{\alpha}'/\chi_{\alpha}$. Then $\Delta_{III_{2}}(\gamma_{H}, \gamma_{G})$ is multiplied by 
$$\prod_{\text{aysmm}} \zeta_{\alpha}(\gamma^{\alpha})^{-1} \cdot \prod_{\text{symm}} \zeta_{\alpha}(\delta^{\alpha})^{-1},$$ where $\gamma^{\alpha}$ and $\delta^{\alpha}$ are defined as in \S 6.3.2.
\end{lem}

\begin{proof} This result is \cite[Lemma 3.5.A]{LS}. The proof loc. cit. depends on our Lemma \ref{cor2.5} (which is Corollary 2.5 loc. cit.) as well as the general discussion of our \S 6.2.2, Galois-cohomological computations similar to the ones done in our \S 6.3.2,  and the fact that the pairing coming from Langlands duality for tori is functorial and respects restriction of scalars. All of these facts/techniques are unchanged in our setting, and therefore the same argument works.
\end{proof}

\subsubsection{The factor $\Delta_{IV}$} We denote the (normalized) absolute value on $F$ by $| \cdot |$. For our $\gamma \in T(F)$, we set \begin{equation}
D_{G^{*}}(\gamma) :=  | \prod_{\alpha \in R} (\alpha(\gamma) -1 ) |^{1/2}.
\end{equation}
Note that this is well-defined because $\prod_{\alpha \in R} (\alpha(\gamma) -1) \in F$. Then we set $$\Delta_{IV}(\gamma_{H}, \gamma_{G}) := D_{G^{*}}(\gamma) \cdot D_{H}(\gamma_{H})^{-1}.$$ This is clearly unchanged if the admissible embedding is replaced by a $\mathfrak{A}(T_{\text{\text{sc}}})$-conjugate. 

\subsubsection{The local transfer factor} We are now ready to define the absolute transfer factor for quasi-split connected reductive groups $G$ over $F$ a local function field and the relative transfer factor for arbitrary connected reductive groups over $F$. Fix two pairs $\gamma_{G}, \gamma_{H}$, $\bar{\gamma}_{H}, \bar{\gamma}_{G}$ as in the beginning of \S 6.3.

For quasi-split $G$ over $F$, we set $$\Delta_{0}(\gamma_{H}, \gamma_{G}) = \Delta_{I}(\gamma_{H}, \gamma_{G}) \Delta_{II}(\gamma_{H}, \gamma_{G}) \Delta_{1}(\gamma_{H}, \gamma_{G}) \Delta_{III_{2}}(\gamma_{H}, \gamma_{G}) \Delta_{IV}(\gamma_{H}, \gamma_{G}).$$ For general $G$, we set 
\begin{equation}
\Delta(\gamma_{H}, \gamma_{G}; \bar{\gamma}_{H}, \bar{\gamma}_{G}) := \frac{\Delta_{I}(\gamma_{H}, \gamma_{G})}{\Delta_{I}(\bar{\gamma}_{H}, \bar{\gamma}_{G})} \cdot \frac{\Delta_{II}(\gamma_{H}, \gamma_{G})}{\Delta_{II}(\bar{\gamma}_{H}, \bar{\gamma}_{G})} \cdot \frac{\Delta_{III_{2}}(\gamma_{H}, \gamma_{G})}{\Delta_{III_{2}}(\bar{\gamma}_{H}, \bar{\gamma}_{G})} \cdot \frac{\Delta_{IV}(\gamma_{H}, \gamma_{G})}{\Delta_{IV}(\bar{\gamma}_{H}, \bar{\gamma}_{G})} \cdot \Delta_{III_{1}}(\gamma_{H}, \gamma_{G} ; \bar{\gamma}_{H}, \bar{\gamma}_{G}).
\end{equation}

We have the following results that discuss the dependence of $\Delta_{0}, \Delta$ on the admissible embeddings and $\chi$- and $a$-data.

\begin{thm}\label{factor} The factor $\Delta(\gamma_{H}, \gamma_{G}; \bar{\gamma}_{H}, \bar{\gamma}_{G})$ is independent of the choice of admissible embeddings, $a$-data, and $\chi$-data.
\end{thm}

\begin{proof} If the admissible embeddings are replaced by $g^{-1} \in \mathfrak{A}(T_{\text{\text{sc}}})$ and $\bar{g}^{-1} \in \mathfrak{A}(\bar{T}_{\text{\text{sc}}})$-conjugate embeddings (with translated $a$- and $\chi$-data), $\Delta_{I}(\gamma_{H}, \gamma_{G})$ is multiplied by $\langle \textbf{g}_{T}, \textbf{s}_{T} \rangle^{-1}$ by Lemma \ref{change1} (similarly for $\bar{\gamma}_{H}, \bar{\gamma}_{G}$), $\Delta_{II}(\gamma_{H}, \gamma_{G})$ is unchanged, $\Delta_{III_{1}}$ is multiplied by $\langle \textbf{g}_{T}, \textbf{s}_{T} \rangle \langle \textbf{g}_{\bar{T}}, \textbf{s}_{\bar{T}} \rangle ^{-1}$ by Lemma \ref{embedding3}, and $\Delta_{III_{2}}$, $\Delta_{IV}$ are unaffected. Thus, $\Delta$ is unaffected.

If we change the $a$- and $\chi$-data to $\{a'_{\alpha}\}$, $\{\chi_{\alpha}'\}$ with $b_{\alpha} := a_{\alpha}'/a_{\alpha}$ and $\zeta_{\alpha} := \chi_{\alpha}'/\chi_{\alpha}$, then the change in $\Delta_{I}(\gamma_{H}, \gamma_{G})$ induced by the new $a$-data cancels with the change in $\Delta_{II}(\gamma_{H}, \gamma_{G})$ induced by the new $a$-data, by Lemmas \ref{change1} and \ref{a2}. The change in $\Delta_{II}(\gamma_{H}, \gamma_{G})$ induced by the new $\chi$-data is cancelled by the change in $\Delta_{III_{2}}(\gamma_{H}, \gamma_{G})$ induced by the new $\chi$-data, by Lemmas \ref{chi2} and \ref{chi3}. All the other factors are unaffected.
\end{proof}
Note that by Lemma \ref{splitting}, $\Delta(\gamma_{H}, \gamma_{G}; \bar{\gamma}_{H}, \bar{\gamma}_{G})$ is also independent of the $F$-splitting chosen for $G^{*}_{\text{\text{sc}}}$ in the construction of the splitting invariant used to define $\Delta_{I}$.

\begin{cor}\label{delta0} The factor $\Delta_{0}(\gamma_{H}, \gamma_{G})$ only depends on the chosen $F$-splitting of $G^{*}_{\text{\text{sc}}}$.
\end{cor}

\begin{proof} This is immediate after using the above proof and replacing Lemma \ref{embedding3} with the observation that conjugating the admissible embedding $T_{H} \to T$ by $g^{-1} \in \mathfrak{A}(T_{\text{\text{sc}}})$ serves to multiply $\Delta_{1}(\gamma_{H}, \gamma_{G})$ by $\langle \textbf{g}_{T}, \textbf{s}_{T} \rangle$, cancelling the corresponding new factor from $\Delta_{I}(\gamma_{H}, \gamma_{G})$.
\end{proof}

\subsection{Addendum: $z$-pairs}
We continue with the same notation as \S 6.3. In particular, $G$ is a connected reductive group over $F$ with quasi-split inner twist $G^{*}$ and endoscopic datum $\mathfrak{e}$. Our goal in this section is to extend the definition of the (relative) transfer factor $\Delta$ to the case where $\widehat{H} \to \mathcal{H}$ is not necessarily equal to the canonical embedding $\widehat{H} \to \prescript{L}{}H$. To do this, we need to introduce the concept of a $z$-pair.

\begin{Def} A \textit{z-pair} $\mathfrak{z} = (H_{\mathfrak{z}}, \eta_{\mathfrak{z}})$ for the endoscopic datum $\mathfrak{e}$ is an $F$-group $H_{\mathfrak{z}}$ that is an extension of $H$ by an induced central torus such that $\mathscr{D}(H_{\mathfrak{z}})$ is simply-connected, and a map $\eta_{\mathfrak{z}} \colon \mathcal{H} \to \prescript{L}{}H_{\mathfrak{z}}$ that is an $L$-embedding extending the embedding $\widehat{H} \to \widehat{H}_{\mathfrak{z}}$ dual to $H_{\mathfrak{z}} \to H$. We call an element of $H_{\mathfrak{z}}(F)$ \textit{strongly $G$-regular semisimple} if its image in $H(F)$ is strongly $G$-regular and semisimple, as we defined above; this set will be denoted by $H_{\mathfrak{z}, G-\text{sr}}(F)$.
\end{Def}

The following result explains the usefulness of this concept:

\begin{prop} A $z$-pair $(H_{\mathfrak{z}}, \eta_{\mathfrak{z}})$ for $\mathfrak{e}$ always exists.
\end{prop}

\begin{proof} The group $H_{\mathfrak{z}}$ without the data of $\eta_{\mathfrak{z}}$ is called a \textit{$z$-extension} of $H$. Such a $z$-extension exists in any characteristic, using \cite[Proposition 3.1]{Milne2}, Proposition 3.1; although the proposition loc. cit. is stated for local fields of characteristic zero, the proof works in the local function field setting as well. Once we have such an extension, \cite[Lemma 2.2.A]{KS1} shows that we can find an $\eta_{\mathfrak{z}}$ satisfying the desired properties (the proof loc. cit. does not depend on the characteristic of $F$ either).
\end{proof}

We will now discuss how to extend the relative transfer factor to a function 
$$ \Delta \colon H_{\mathfrak{z}, G-\text{sr}}(F) \times G_{\text{sr}}(F) \times H_{\mathfrak{z}, G-\text{sr}}(F) \times G_{\text{sr}}(F) \to \mathbb{C},$$ satisfying all the desired properties enjoyed by the factor $\Delta$ defined above. This discussion is taken from the proof of \cite[Proposition 5.6]{Tasho}. Let $\gamma_{\mathfrak{z}}, \bar{\gamma}_{\mathfrak{z}} \in H_{\mathfrak{z}, G-\text{sr}}(F)$ with images $\gamma_{H}, \bar{\gamma}_{H}$ in $H_{G-\text{sr}}(F)$, related to $\gamma_{G}, \bar{\gamma}_{G} \in G_{\text{sr}}(F)$. The factors $\Delta_{I}(\gamma_{\mathfrak{z}}, \gamma_{G})$, $\Delta_{II}(\gamma_{\mathfrak{z}}, \gamma_{G})$, $\Delta_{III_{1}}(\gamma_{\mathfrak{z}}, \gamma_{G}; \bar{\gamma}_{\mathfrak{z}}, \bar{\gamma}_{G})$, and $\Delta_{IV}(\gamma_{\mathfrak{z}}, \gamma_{G})$ are all defined to be the same factors with $\gamma_{\mathfrak{z}}, \bar{\gamma}_{\mathfrak{z}}$ replaced by their images $\gamma_{H}, \bar{\gamma}_{H}$. It remains to define $\Delta_{III_{2}}(\gamma_{\mathfrak{z}}, \gamma_{G})$. Consider the following diagram:
\[
\begin{tikzcd}
\prescript{L}{}H_{\mathfrak{z}} & \prescript{L}{}T_{H_{\mathfrak{z}}} \arrow[hook]{l} &  \prescript{L}{}T_{H_{\mathfrak{z}}} \arrow[dotted]{l} & \prescript{L}{}T_{H} \arrow[hook]{l} \\
\mathcal{H} \arrow[hook, "\eta_{\mathfrak{z}}"]{u} \arrow[hook, "\eta"]{r} & \prescript{L}{}G & \prescript{L}{}T,  \arrow[hook]{l} \arrow["\prescript{L}{}\phi_{\gamma_{H}, \gamma}"]{ur} & 
\end{tikzcd}
\]
where we are denoting the centralizer of $\gamma_{\mathfrak{z}}$ by $T_{H_{\mathfrak{z}}}$, the map $\prescript{L}{}T \to \prescript{L}{}G$ is the one corresponding to a choice of $\chi$-data for $T$, as discussed in \S 6.3.4 and \S 6.2.2, we are denoting the choice of admissible embedding $T_{H} \to T$ by $\phi_{\gamma_{H}, \gamma}$, and the embedding $\prescript{L}{}T_{H_{\mathfrak{z}}} \hookrightarrow \prescript{L}{}H_{\mathfrak{z}}$ is obtained by transporting the $\chi$-data to $T_{H}$ and then to $T_{H_{\mathfrak{z}}}$ via the projection $T_{H_{\mathfrak{z}}} \to T_{H}$ (this makes sense because $H_{\mathfrak{z}}$ is a central extension of $H$, so that $T_{H}$ and $T_{H_{\mathfrak{z}}}$ have the same root systems). The dotted arrow is the unique $L$-homomorphism extending the identity on $\widehat{T_{H_{\mathfrak{z}}}}$ and making the diagram commute; its restriction to $W_{F}$ gives a 1-cocycle $a \colon W_{F} \to \widehat{T_{H_{\mathfrak{z}}}}(\mathbb{C})$; for an explanation of why such an $L$-homomorphism exists, as well as the fact that this is a cocycle, see \cite[\S 4.4]{KS1}. We then set $\Delta_{III_{2}}(\gamma_{\mathfrak{z}}, \gamma_{G}) := \langle a, \gamma_{\mathfrak{z}} \rangle$, where as in \S6.3.4 the pairing is from Langlands duality for tori. 

We then define $\Delta(\gamma_{\mathfrak{z}}, \gamma_{G}; \bar{\gamma}_{\mathfrak{z}}, \bar{\gamma}_{G})$ identically as in \S 6.3, except with our new $\Delta_{III_{2}}$ factor.  We may also use this to define an analogous factor $\Delta_{0}(\gamma_{\mathfrak{z}}, \gamma_{G})$ in the quasi-split case, where we simply replace the $\Delta_{III_{2}}$ factor in the definition given in \S 6.3 with the factor we defined above (and take the image of $\gamma_{\mathfrak{z}}$ in $H(F)$ to define the other $\Delta_{i}$-factors). 

\begin{prop} The above factor does not depend on the choice of admissible embeddings, $\chi$-data, or $a$-data.
\end{prop}

\begin{proof} This is \cite[Theorem 4.6.A]{KS1}. In view of the proof of Theorem \ref{factor}, it suffices to check that $\Delta(\gamma_{\mathfrak{z}}, \gamma_{G}; \bar{\gamma}_{\mathfrak{z}}, \bar{\gamma}_{G})$ is unaffected by changing the $\chi$-data for $T$. Verifying this comes down to examining the new $\Delta_{III_{2}}$-factor, which is not affected by the characteristic of $F$, so the proof loc. cit. works in our situation as well.
\end{proof}

\section{Applications to the Langlands conjectures}
This section applies the theory we have constructed in order to state the local Langlands conjectures for connected reductive groups over local fields of positive characteristic. Again, in this section $F$ is a local field of characteristic $p > 0$, $G$ is a connected reductive group over $F$, and $\gerbeE$ is a $u$-gerbe split over $\bar{F}$ with $[\gerbeE] = -1 \in \check{H}^{2}(\bar{F}/F, u)$. Recall from \S 1 that our goal is to generalize the notion of \textit{rigid inner forms}, introduced in \cite{Tasho}, in order to work with the representations of all inner forms of $G$ simultaneously.

\subsection{Rigid inner twists} In order to assign to inner twists of $G$ the ``correct" automorphism group (i.e., one such that automorphisms preserve $F$-conjugacy classes and $F$-representations), we need to refine the data of an inner twist  to that of a rigid inner twist. For a $\GE$-torsor $\mathscr{T}$, we denote the $(G_{\text{ad}})_{\gerbeE}$-torsor $\mathscr{T} \times^{\GE} (G_{\text{ad}})_{\gerbeE}$ by $\overline{\mathscr{T}}$. We can view an inner twist $G_{F^{s}} \xrightarrow{\xi} G'_{F^{s}}$ (up to equivalence) as an element of $[\underline{\text{Isom}}(G, G')/G_{\text{ad}}](F)$, and its fiber in $\underline{\text{Isom}}(G, G')$ (the inner class of $\xi$) is a $G_{\text{ad}}$-torsor, denoted by $T_{\xi}$.  

\begin{Def}
\begin{enumerate}
\item{A \textit{rigid inner twist} of $G$ is a triple $(\xi, \mathscr{T}, \bar{h})$ of an inner twist $\xi \colon G_{F {s}} \to G'_{F^{s}}$, a $Z$-twisted $\GE$-torsor $\mathscr{T}$ for some finite central $Z$, and an isomorphism of $(G_{\text{ad}})_{\gerbeE}$-torsors $\bar{h} \colon \overline{\mathscr{T}} \to (T_{\xi})_{\mathcal{E}}$ (which we call a \textit{rigidification}---note that since both torsors descend to $F$, we could also define $\bar{h}$ to be an isomorphism between their descents). If we demand that $\mathscr{T}$ is $Z$-twisted for some fixed finite central $Z$, then we say further that the rigid inner twist is a \textit{$Z$-rigid inner twist}. }

\item{An \textit{isomorphism of rigid inner twists} $(f, \Psi) \colon (\xi_{1}, \mathscr{T}_{1} ,\bar{h}_{1}) \to (\xi_{2}, \mathscr{T}_{2}, \bar{h}_{2})$ is a pair consisting of an isomorphism $f \colon G_{1} \to G_{2}$ defined over $F$ and an isomorphism of $\GE$-torsors $\Psi \colon \mathscr{T}_{1} \to \mathscr{T}_{2}$; note that such an isomorphism induces an automorphism $\bar{h}_{2} \circ \Psi \circ \bar{h}_{1}^{-1}  \colon (T_{\xi})_{\mathcal{E}} \to (T_{\xi})_{\mathcal{E}}$, giving an element $\bar{\delta} \in G_{\text{ad}}(\bar{F})$ (by looking at where this automorphism sends $\xi \in T_{\xi}(\bar{F})$) which we require to satisfy $\xi_{2}^{-1} \circ f \circ \xi_{1} = \text{Ad}(\bar{\delta})$. It is clear that $\Psi$ uniquely determines $f$, so one could just as well denote the isomorphism by $\Psi$; nevertheless, we will still use the notation $(f, \Psi)$ (for one, this is consistent with its analogue in \cite{Tasho}).} \end{enumerate}
\end{Def}

Note that in the above definition the rigidification $\bar{h}$ together with the canonical element $\xi \in T_{\xi}(\bar{F})$ defines a trivialization $\bar{h}'$ of $\overline{\mathscr{T}}_{\bar{F}}$.

Denote by $RI(G,\gerbeE)$ (resp. $RI_{Z}(G,\gerbeE)$) the category whose objects are rigid inner twists of $G$ (resp. $Z$-rigid inner-twists of $G$) and morphisms are isomorphisms of rigid inner twists. It is clear that the natural functor $RI_{Z}(G,\gerbeE) \to RI(G,\gerbeE)$ is fully faithful and $RI(G,\gerbeE) = \varinjlim RI_{Z}(G,\gerbeE)$, where the colimit is taken over all finite central $Z$. Note that for every inner twist $\psi \colon G \to G'$, there exists a $Z(\mathscr{D}(G))$-twisted $\GE$-torsor $\mathscr{T}$ and rigidification $\bar{h}$ such that $(\psi, \mathscr{T}, \bar{h})$ is a rigid inner twist, by Proposition \ref{existtwist}. For computational purposes, we reformulate the above definition in the case $\gerbeE = \gerbeE_{a}$ for $[a] = -1 \in \check{H}^{2}(\bar{F}/F, u)$:
\begin{Def}
\begin{enumerate}
\item{For $a \in u(U_{2})$ such that $[a] = -1 \in \check{H}^{2}(\bar{F}/F, u)$, an \textit{$a$-normalized rigid inner twist} of $G$ is a pair $(\xi, (x,\phi))$ of an inner twist $\xi \colon G \to G'$ and $(x,\phi) \in Z^{1}(\gerbeE_{a}, Z \to G)$ for some finite central $Z$ such that the image of $(x,\phi)$ in $Z^{1}(F, G_{\text{ad}})$, denoted by $\bar{x}$, satisfies $\text{Ad}(\bar{x}) = p_{1}^{*}\xi^{-1} \circ p_{2}^{*}\xi$. If we demand that $\phi$ factors through some fixed finite central $Z$, then we say further that the $a$-normalized rigid inner twist is an \textit{$a$-normalized $Z$-rigid inner twist}. }

\item{An \textit{isomorphism of $a$-normalized rigid inner twists} $(f, \delta) \colon (\xi_{1}, (x_{1}, \phi_{1})) \to (\xi_{2}, (x_{2}, \phi_{2}))$ for $\phi_{1} = \phi_{2}$, is a pair consisting of an isomorphism $f \colon G_{1} \to G_{2}$ defined over $F$ and $\delta \in G(\bar{F})$ such that $\xi_{2}^{-1} \circ f \circ \xi_{1} = \text{Ad}(\delta)$ and $x_{1} = p_{1}^{\sharp}(\delta)^{-1} x_{2} p_{2}^{\sharp}(\delta)$.}
\end{enumerate}
\end{Def}

Denote by $RI(G, a)$ (resp. $RI_{Z}(G, a)$) the category whose objects are $a$-normalized rigid inner twists of $G$ (resp. $a$-normalized $Z$-rigid inner twists of $G$) and morphisms are isomorphisms of $a$-normalized rigid inner twists. The following result explains the relationship between the categories $RI(G,\gerbeE)$ and $RI(G, a)$:

\begin{prop} For $\mathcal{E} = \mathcal{E}_{a}$, we have an equivalence of categories between $RI(G,\gerbeE)$ and $RI(G, a)$ which is canonical on isomorphism classes and restricts to an equivalence between $RI_{Z}(G,\gerbeE)$ and $RI_{Z}(G, a)$.
\end{prop}

\begin{proof}
Let $s \colon \text{Sch}/\bar{F} \to \gerbeE_{a}$ denote the section constructed in Lemma \ref{section}. Then if $(\xi, \mathscr{T}, \bar{h})$ is a $Z$-rigid inner twist, by the proof of Proposition \ref{correspondence1}, after setting $\phi := \text{Res}(\mathscr{T})$, choosing a trivialization $h$ of $s^{*}\mathscr{T}$ \textit{lifting $\bar{h}'$}, by which we mean is such that the diagram 
\[
\begin{tikzcd}
s^{*}\mathscr{T} \arrow["h"]{r} \arrow{d} & G_{\bar{F}} \arrow{d}  \\
s^{*}\overline{\mathscr{T}} \arrow["s^{*}\bar{h}' "]{r} & (G_{\text{ad}})_{\bar{F}}
\end{tikzcd}
\]
commutes (such an $h$ evidently always exists), gives an $a$-twisted $Z$-cocycle $(x, \phi)$ valued in $G$, and by construction we also have that $\text{Ad}(\bar{x}) =  p_{1}^{*}\xi^{-1} \circ p_{2}^{*}\xi$. Thus, we have a way of associating to a $Z$-rigid inner twist an $a$-normalized $Z$-rigid inner twist.

Moreover, given any isomorphism $(f, \Psi)$ between the $Z$-rigid inner twists $(\xi_{1}, \mathscr{T}_{1}, \bar{h}_{1})$ and $(\xi_{2}, \mathscr{T}_{2}, \bar{h}_{2})$, choices $h_{i}$ of $\bar{F}$-trivializations lifting $s^{*}\bar{h}_{i}'$ give an automorphism $h_{2} \circ s^{*}\Psi \circ h_{1}^{-1} \colon G_{\bar{F}} \xrightarrow{\sim} G_{\bar{F}}$ which is left-translation by a unique $\delta \in G(\bar{F})$. Then we may define an isomorphism $(\xi_{1}, (x_{1}, \phi_{1})) \to (\xi_{2}, (x_{2}, \phi_{2}))$ between the corresponding twisted cocycles (obtained using $h_{1}$ and $h_{2}$) given by $(f, \delta)$. 

By Proposition \ref{correspondence1} (see also Proposition \ref{gerbesheaf}), every $a$-normalized $Z$-rigid inner twist is isomorphic to the image of some $Z$-rigid inner twist under the above map. By the discussion following the proof of Lemma \ref{uniqueness}, if the image of two rigid inner twists are isomorphic as $a$-normalized rigid inner twists, then they are isomorphic as rigid inner twists (using that the condition on $\delta$ and $\bar{\delta}$ is the same).
\end{proof}

A similar argument using Lemma \ref{trivialize} shows that for arbitrary $\gerbeE$, we have a canonical bijection between isomorphism classes in $RI(G,\gerbeE)$ and $RI(G, \gerbeE_{a})$, and hence also between classes in $RI(G,\gerbeE)$ and isomorphism classes in $RI(G,a)$.  Everything said above applies if we restrict ourselves to $Z$-rigid inner forms for some fixed $Z$ as well. 

We have the following important fact about automorphisms of rigid inner forms:

\begin{prop}\label{automorphisms} The automorphism group of a fixed $a$-normalized rigid inner twist $(\xi, (x, \phi))$ for $\xi \colon G_{F^{s}} \to (G')_{F^{s}}$ is canonically isomorphic to $G'(F)$ by the map $(f, \delta) \mapsto \xi(\delta)$. 
\end{prop}

\begin{proof} One computes the 0-differential of $\xi(\delta)$ to be $p_{1}^{*}\xi(p_{1}^{\sharp}\delta^{-1}) \cdot p_{2}^{*}\xi(p_{2}^{\sharp}\delta)$, and post-composing with $p_{1}^{*}\xi^{-1}$ yields $$p_{1}^{\sharp}\delta^{-1} \cdot x \cdot p_{2}^{\sharp}\delta \cdot x^{-1} = e,$$ giving $\xi(\delta) \in G'(F)$, showing that the above map is well-defined. From here it is straightforward to check that it defines an isomorphism.
\end{proof}

\begin{cor}\label{automorphisms2} The automorphism group of a fixed rigid inner twist $(\xi, \mathscr{T}, \bar{h})$ for $\xi \colon G_{F^{s}} \to (G')_{F^{s}}$ is canonically isomorphic to $G'(F)$.
\end{cor}

\begin{proof} Fix a section $s \colon (\text{Sch}/\bar{F}) \to \gerbeE$, as well as a trivialization $h$ of $s^{*}\mathscr{T}$ lifting $\bar{h}'$ (terminology as above). Note that any two choices of $h$ differ by precomposing by an automorphism of $s^{*}\mathscr{T}$ induced by an automorphism of $\mathscr{T}$ given by right-translation by some $z \in Z(G)(\bar{F})$. The map $h \circ s^{*}\Psi \circ h^{-1}$ is left-translation by an element $\delta \in G(\bar{F})$, and any different choice of $h$ yields the same $\delta$, by the $\GE$-equivariance of $\Psi$. We may thus define our desired isomorphism to send $(f,\Psi)$ to $\xi(\delta)$, which lies in $G'(F)$, by the proof of Proposition \ref{automorphisms}. It is straightforward to verify that the element $\delta$ does not depend on the choice of section $s$. 
\end{proof}

We now define rational and stable conjugacy of elements of rigid inner forms. Let $(\xi_{1}, \mathscr{T}_{1}, \bar{h}_{1})$ and $(\xi_{2}, \mathscr{T}_{2}, \bar{h}_{2})$ be two $Z$-rigid inner twists for some fixed $Z$ corresponding to the groups $G_{1}, G_{2}$, and let $\delta_{i} \in G_{i,\text{sr}}(F)$ for $i=1,2$. We say that $(G_{1}, \xi_{1}, \mathscr{T}_{1}, \bar{h}_{1}, \delta_{1})$ and $(G_{2}, \xi_{2}, \mathscr{T}_{2}, \bar{h}_{2}, \delta_{2})$ are \textit{rationally conjugate} if there exists an isomorphism $(f, \Psi) \colon (\xi_{1}, \mathscr{T}_{1}, \bar{h}_{1}) \to (\xi_{2}, \mathscr{T}_{2}, \bar{h}_{2})$ such that $f(\delta_{1}) = \delta_{2}$. We say that they are \textit{stably conjugate} if $\xi_{1}^{-1}(\delta_{1})$ is $G(\bar{F})$-conjugate to $\xi_{2}^{-1}(\delta_{2})$. The arguments used in \S 6.1 show that the latter condition is equivalent to $\xi_{1}^{-1}(\delta_{1})$ being $G(F^{s})$-conjugate to $\xi_{2}^{-1}(\delta_{2})$ (this centers on the fact that the Weyl group scheme of a maximal torus in an algebraic group is \'{e}tale). Define rational and stable conjugacy identically for elements of $a$-normalized rigid inner twists.

We need the following lemma:

\begin{lem}\label{Fconj} Assume that $G$ is quasi-split. For any $(G_{1}, \xi_{1}, \mathscr{T}_{1}, \bar{h}_{1}, \delta_{1})$ (resp. $(G_{1}, \xi_{1}, (x_{1}, \phi_{1}), \delta_{1})$) as above, there exists $\delta \in G_{\text{sr}}(F)$ such that $(G_{1}, \xi_{1}, \mathscr{T}_{1}, \bar{h}_{1}, \delta_{1})$ (resp. $(G_{1}, \xi_{1}, (x_{1}, \phi_{1}), \delta_{1})$)  is stably conjugate to  $(G, \text{id}_{G}, \GE, \text{id}_{\bar{F}}, \delta)$ (resp. $(G, \text{id}_{G}, (e, 0), \delta)$).
\end{lem} 

\begin{proof} This follows immediately from Lemma \ref{kottlem}.
\end{proof}

We continue to assume that $G$ is quasi-split. For any $(G_{1}, \xi_{1}, \mathscr{T}_{1}, \bar{h}_{1}, \delta_{1})$, there exists $\delta \in G_{\text{sr}}(F)$ such that $(G_{1}, \xi_{1}, \mathscr{T}_{1}, \bar{h}_{1}, \delta_{1})$ is stably conjugate to $(G, \text{id}_{G}, \GE, \text{id}_{\bar{F}}, \delta)$, by the above lemma. As in \cite{Tasho}, we now fix $\delta \in G_{\text{sr}}(F)$ and consider the category $\mathcal{C}_{Z}(\delta, \gerbeE)$ whose objects are points $(G_{1}, \xi_{1}, \mathscr{T}_{1}, \bar{h}_{1}, \delta_{1})$ which are stably conjugate to $(G, \text{id}_{G}, \GE, \text{id}_{\bar{F}}, \delta)$ such that $(\xi_{1}, \mathscr{T}_{1}, \bar{h}_{1})$ is a $Z$-rigid inner twist and whose morphisms $(G_{1}, \xi_{1}, \mathscr{T}_{1}, \bar{h}_{1}, \delta_{1}) \to (G_{2}, \xi_{2}, \mathscr{T}_{2}, \bar{h}_{2}, \delta_{2})$ are isomorphisms of rigid inner twists $(f, \Psi)$ such that $f(\delta_{1}) = \delta_{2}$. We interpret this category as the stable conjugacy class of $(G, \text{id}_{G}, \GE, \text{id}_{\bar{F}}, \delta)$, and it is clear that the isomorphism classes within $\mathcal{C}_{Z}(\delta, \gerbeE)$ give the rational conjugacy classes within this stable conjugacy class. We define the category $\mathcal{C}_{Z}(\delta, a)$ using $a$-normalized $Z$-rigid inner twists completely analogously. By our previous discussion, it is clear that the isomorphism classes of $\mathcal{C}_{Z}(\delta, \gerbeE_{a})$ are in canonical bijection with those of $\mathcal{C}_{Z}(\delta, a)$, as are the isomorphism classes of $\mathcal{C}_{Z}(\delta, \gerbeE)$.

Set $S := Z_{G}(\delta)$, a maximal torus. We will now define a map from the isomorphism classes of $\mathcal{C}_{Z}(\delta, \gerbeE)$ to $H^{1}(\gerbeE, Z \to S)$, denoted by $\text{inv}(-,\delta)$. The simplest way to do this for general $\gerbeE$ is to first define it for isomorphism classes in $\mathcal{C}_{Z}(\delta, a)$ for $a$ representing $[\gerbeE]$, invoke the canonical bijection between the isomorphism classes in $\mathcal{C}_{Z}(\delta, \gerbeE)$ and those of $\mathcal{C}_{Z}(\delta, a)$ and then check that for $a$ cohomologous to $a'$, the diagram
\begin{equation}\label{isomclasses}
\begin{tikzcd}
\text{Isom}[\mathcal{C}_{Z}(\delta, a)] \arrow{r}  \arrow{d} &  H^{1}(\gerbeE_{a}, Z \to S) \arrow{d} \\
\text{Isom}[\mathcal{C}_{Z}(\delta, a')] \arrow{r} & H^{1}(\gerbeE_{a'}, Z \to S)
\end{tikzcd}
\end{equation}
commutes, where $\text{Isom}[\mathcal{C}_{Z}(\delta, a)]$ denotes the set of isomorphism classes in $\mathcal{C}_{Z}(\delta, a)$, and the vertical arrows are the canonical bijections induced by any choice of $y \in u(U_{1})$ such that $dy \cdot a = a'$ (cf. Construction \ref{changeofgerbe}, the canonicity comes from the fact that $H^{1}(F, u) = 0$, cf. Corollary \ref{punchline}). This last condition ensures that the map we define is canonical. 

Fix $(G_{1}, \xi_{1}, (x_{1}, \phi_{1}), \delta_{1}) \in \mathcal{C}_{Z}(\delta,a)$, and choose $g \in G(F^{s})$ such that $\xi_{1}(g \delta g^{-1}) = \delta_{1}$. The map sending this element to the $a$-twisted cocycle $(p_{1}(g)^{-1}x_{1}p_{2}(g), \phi_{1})$ gives a map $\mathcal{C}_{Z}(\delta,a) \to Z^{1}(\gerbeE_{a}, Z \to S)$, since translating by $g$ does not affect the differential of $x_{1}$. This induces a map $\text{inv}(-,\delta) \colon \mathcal{C}_{Z}(\delta,a) \to H^{1}(\gerbeE_{a}, Z \to S)$, which does not depend on the choice of $g$, by construction of the equivalence relation defined on $a$-twisted cocycles. The following result shows that the cohomology set $H^{1}(\gerbeE_{a}, Z \to S)$ parametrizes the rational classes within the stable class of $\delta$.

\begin{prop} The map $\text{inv}(-,\delta)$ induces a bijection from the isomorphisms classes of $\mathcal{C}_{Z}(\delta, a)$ to $H^{1}(\gerbeE_{a}, Z \to S)$ .
\end{prop}

\begin{proof}
First note that if $(G_{1}, \xi_{1}, (x_{1}, \phi_{1}), \delta_{1}) \in \mathcal{C}_{Z}(\delta,a)$ and $(G_{2}, \xi_{2}, (x_{2}, \phi_{2}), \delta_{2}) \in \mathcal{C}_{Z}(\delta,a)$ are isomorphic via $(f,g)$ then if we take $g_{i}$ satisfying $\xi_{1}(g_{i} \delta_{1} g_{i}^{-1}) = \delta_{i}$, we have $\phi_{1} = \phi_{2}$ (by definition) and $g_{1}^{-1}g^{-1}g_{2} \in S(\bar{F})$, since $$\text{Ad}(g_{1}^{-1}g^{-1}g_{2})\delta = \text{Ad}(g_{1}^{-1})(\xi_{1}^{-1} \circ f^{-1} \circ \xi_{2})(\text{Ad}(g_{2})(\delta)) = \text{Ad}(g_{1}^{-1})(\xi_{1}^{-1}(\delta_{1})) = \delta,$$ giving that $[(p_{1}^{\sharp}(g_{1})^{-1}x_{1}p_{2}^{\sharp}(g_{1}), \varphi_{1})] = [(p_{1}^{\sharp}(g_{2})^{-1}x_{2}p_{2}^{\sharp}(g_{2}), \varphi_{1})]$ in $H^{1}(\gerbeE_{a}, Z \to S)$. This shows that the invariant map is constant on isomorphism classes. 

For injectivity, we note that if $[(p_{1}^{\sharp}(g_{1})^{-1}x_{1}p_{2}^{\sharp}(g_{1}), \phi)] = [(p_{1}^{\sharp}(g_{2})^{-1}x_{2}p_{2}^{\sharp}(g_{2}), \phi)]$ in $H^{1}(\gerbeE_{a}, Z \to S)$, then if we take $g \in S(\bar{F})$ realizing this equivalence of cocycles, the (fppf descent of the) map $G_{1} \to G_{2}$ defined by $\xi_{2} \circ \text{Ad}(g_{2}gg_{1}^{-1}) \circ \xi_{1}^{-1}$ defines an isomorphism from  $(G_{1}, \xi_{1}, (x_{1}, \phi_{1}), \delta_{1})$ to $(G_{2}, \xi_{2}, (x_{2}, \phi_{2}), \delta_{2})$ in $\mathcal{C}_{Z}(\delta)$. 

For surjectivity, if we fix $[(x,\phi)] \in  H^{1}(\gerbeE_{a}, Z \to S)$, then since $dx \in Z(G)$, we may twist $G$ by $x$ to obtain $G^{x}$, with the usual isomorphism $\xi \colon G \xrightarrow{\sim} G^{x}$ satisfying $p_{1}^{*}\xi^{-1} \circ p_{2}^{*}\xi = \text{Ad}(x)$, and then (since $x$ commutes with $\delta$) the tuple $(G^{x}, \xi, (x, \phi), \xi(\delta))$ lies in $\mathcal{C}_{Z}(\delta,a)$ and trivially maps to $(x,\phi) \in Z^{1}(\gerbeE_{a}, Z \to S)$.
\end{proof}

\begin{lem} The diagram \eqref{isomclasses} commutes.
\end{lem}

\begin{proof} If $y \in u(U_{1})$ is such that $dy \cdot a = a'$, then the map $\mathcal{C}_{Z}(\delta, a) \to \mathcal{C}_{Z}(\delta, a')$ may be defined by sending $(G_{1}, \xi_{1}, (x_{1}, \phi_{1}), \delta_{1})$ to $(G_{1} ,\xi_{1}, (x_{1} \cdot \phi_{1}(y), \phi_{1}),\delta_{1})$. This maps to the equivalence class of the $a'$-twisted cocycle $(p_{1}^{\sharp}(g)^{-1}\phi_{1}(y)xp_{2}^{\sharp}(g), \phi_{1})$ in $H^{1,*}(\gerbeE_{a'}, Z \to G)$. Going the other direction, the class of $(G_{1}, \xi_{1}, (x_{1}, \phi_{1}), \delta_{1})$ maps to the equivalence class of the $a$-twisted cocycle $(p_{1}^{\sharp}(g)^{-1}xp_{2}^{\sharp}(g), \phi_{1})$, which then maps to the class of $(p_{1}^{\sharp}(g)^{-1}\phi_{1}(y)xp_{2}^{\sharp}(g), \phi_{1})$, by the centrality of $Z$. 
\end{proof}

Because of the above result, in the context of the invariant map it will be harmless to denote $\mathcal{C}_{Z}(\delta, \gerbeE)$ for a choice of $\gerbeE$ simply by $\mathcal{C}_{Z}(\delta)$, and for computational purposes to identify $\mathcal{C}_{Z}(\delta)$ with $\mathcal{C}_{Z}(\delta, a)$ for a choice of $a$. Note that if $Z \to S$ factors through another finite central $Z' \to S$, then we have a canonical functor $\iota_{Z,Z'} \colon \mathcal{C}_{Z}(\delta) \to \mathcal{C}_{Z'}(\delta)$ which is fully faithful. Moreover, the two invariant maps to $H^{1}(\gerbeE, Z \to S), H^{1}(\gerbeE, Z' \to S)$ commute with the natural inclusion $H^{1}(\gerbeE, Z \to S) \to H^{1}(\gerbeE, Z' \to S)$; thus, the invariant map does not depend on the choice of $Z$.

The last thing we do in this subsection is define a representation of a rigid inner form.

\begin{Def}\label{representation} A \textit{representation} of a rigid inner twist of $G$ is a tuple $(G_{1}, \xi_{1},\mathscr{T}_{1}, \bar{h}_{1}, \pi_{1})$, where $(\xi_{1}, \mathscr{T}_{1}, \bar{h}_{1})$ is a rigid inner twist of $G$ and $\pi_{1}$ is an admissible representation of $G_{1}(F)$. An \textit{isomorphism of representations} of rigid inner twists $(G_{1}, \xi_{1},\mathscr{T}_{1}, \bar{h}_{1}, \pi_{1}) \to (G_{2}, \xi_{2},\mathscr{T}_{2}, \bar{h}_{2}, \pi_{2})$ is an isomorphism of rigid inner twists $(f,\Psi) \colon (\xi_{1}, \mathscr{T}_{1}, \bar{h}_{1}) \to (\xi_{2}, \mathscr{T}_{2}, \bar{h}_{2})$ such that the $G_{1}(F)$-representations $\pi_{1}$ and $\pi_{2} \circ f$ are isomorphic.

\end{Def}

One verifies easily that two representations $(G_{1}, \xi_{1},\mathscr{T}_{1}, \bar{h}_{1}, \pi_{1}) $ and $(G_{1}, \xi_{1},\mathscr{T}_{1}, \bar{h}_{1}, \pi_{2})$ are isomorphic in the above sense if and only if $\pi_{1}$ and $\pi_{2}$ are isomorphic as $G_{1}(F)$-representations.

\subsection{Local transfer factors and endoscopy} Let $[Z \to G] \in \mathcal{R}$ and let $\widehat{G}$ be a complex Langlands dual group for $G$. We have an isogeny $G \to \overline{G}$ which dualizes to an isogeny $\widehat{\overline{G}} \to \widehat{G}$, inducing a homomorphism $Z(\widehat{\overline{G}}) \to Z(\widehat{G})$. Identifying these complex varieties with their $\mathbb{C}$-points, we define $Z(\widehat{\overline{G}})^{+} \subset Z(\widehat{\overline{G}})$ to be the preimage of $Z(\widehat{G})^{\Gamma}$ under this isogeny. We thus obtain a functor $\mathcal{R} \to \text{FinAbGrp}$ by sending $G$ to $\pi_{0}(Z(\widehat{\overline{G}})^{+})^{*}$; this can be seen as an analogue of functor introduced in \cite[Theorem 1.2]{Kott86}.

\begin{prop} We have a functorial isomorphism $$Y_{+,\text{tor}}(Z \to G) \xrightarrow{\sim} \pi_{0}(Z(\widehat{\overline{G}})^{+})^{*}.$$
\end{prop}

\begin{proof} We describe what the construction of this map is; the proof that this construction indeed is a functorial isomorphism is identical to the one given in \cite[Proposition 5.3]{Tasho}.

Recall that for $[Z \to G] \in \mathcal{R}$, the group $Y_{+,\text{tor}}(Z \to G)$ is an inverse limit as $S$ ranges over all maximal $F$-tori of $G$ of groups of the form $$\varinjlim \frac{(X_{*}(\bar{S})/ X_{*}(S_{\text{sc}}))^{N}}{I(X_{*}(S)/X_{*}(S_{\text{sc}}))},$$ where each direct limit is over all finite Galois extensions of $F$ splitting $S$. For a fixed $S$, we have a commutative square of multiplicative groups corresponding to the commutative square of character groups:
\[
\begin{tikzcd}
Z(\widehat{\overline{G}}) \arrow{r} \arrow{d} & Z(\widehat{G}) \arrow{d} & & \frac{X_{*}(\bar{S})}{X_{*}(S_{\text{sc}})}  & \frac{X_{*}(S)}{X_{*}(S_{\text{sc}})}  \arrow{l} \\
\widehat{\bar{S}} \arrow{r} & \widehat{S} & & X_{*}(\bar{S})  \arrow{u}& X_{*}(S). \arrow{l} \arrow{u}
\end{tikzcd}
\]

Under the canonical embedding $Z(\widehat{G}) \to \widehat{S}$, the subgroup inclusion $Z(\widehat{G})^{\Gamma} \subset Z(\widehat{G})$ corresponds at the level of character groups to the quotient map $$X_{*}(S)/X_{*}(S_{\text{sc}}) \to [X_{*}(S)/X_{*}(S_{\text{sc}})] / I [X_{*}(S)/X_{*}(S_{\text{sc}})],$$ and it follows that the subgroup $Z(\widehat{\overline{G}})^{+} \subset Z(\widehat{\overline{G}})$ has character group $$X^{*}(Z(\widehat{\overline{G}})^{+}) = [X_{*}(\bar{S})]/ [I X_{*}(S)+X_{*}(S_{\text{sc}})].$$ Finally, passing to the component group corresponds to taking the torsion subgroup, which (for a Galois extension splitting $S$) contains $[X_{*}(\bar{S})/X_{*}(S_{\text{sc}})]^{N} / I [X_{*}(S)/X_{*}(S_{\text{sc}})]$. This gives a natural inclusion $$\frac{(X_{*}(\bar{S})/ X_{*}(S_{\text{sc}}))^{N}}{I(X_{*}(S)/X_{*}(S_{\text{sc}}))} \hookrightarrow \pi_{0}(Z(\widehat{\overline{G}})^{+})^{*},$$ since we have the obvious identification $X^{*}(\pi_{0}(Z(\widehat{\overline{G}})^{+})) = \pi_{0}(Z(\widehat{\overline{G}})^{+})^{*}$. These maps glue for varying Galois extensions of $F$, and then induce an isomorphism on the direct limit over all extensions $E$ (see \cite[Proposition 5.3]{Tasho}).
\end{proof}

The analogue of \cite[Corollary 5.4]{Tasho} makes precise our earlier statement comparing this new functor to the one defined in \cite[Theorem 1.2]{Kott86}:

\begin{cor}\label{pairing} There is a perfect pairing $$H^{1}(\gerbeE, Z \to G) \times \pi_{0}(Z(\widehat{\overline{G}})^{+}) \to \mathbb{Q}/\Z,$$ which is functorial in $[Z \to G] \in \mathcal{R}$. Moreover, if $Z$ is trivial then this pairing coincides with the one stated in Theorem \ref{kottwitzpairing}. 
\end{cor}

We now recall the notion of a \textit{refined endoscopic datum}, introduced in \cite[\S 5]{Tasho}. As before, assume that we have some fixed finite central $Z \to G$, and denote $G/Z$ by $\overline{G}$. First, let $(H, \mathcal{H}, s, \eta)$ be an endoscopic datum for $G$. We may always replace this datum with an equivalent $\widehat{G}(\mathbb{C})$-conjugate datum $(H, \mathcal{H}, s', \eta')$ such that $s' \in Z(\widehat{H})^{\Gamma}$ without affecting the value of the transfer factors $\Delta, \Delta_{0}$ involving $(H, \mathcal{H}, s, \eta)$ (see the beginning \S 6.3). We will always assume that our endoscopic datum has this form.

Choices of maximal tori in $\widehat{H}$, $\widehat{G}$, $H$, and $G$ give embeddings $Z_{F^{s}} \to Z(H)_{F^{s}}$ which differ by pre- and post-composing with inner automorphisms induced by $G(F^{s})$, $H(F^{s})$, and hence are all the same, meaning that we have a canonical $F$-embedding $Z \hookrightarrow{H}$ (the $\Gamma$-equivariance follows from the fact that the maps $\widehat{T} \to \mathscr{T}$ for $T$ maximal in $G$, $\mathscr{T}$ maximal in $\widehat{G}$, are $\Gamma$-equivariant up to the action of the Weyl group---analogously for $H$). It thus makes sense to define $\overline{H} := H/Z$, which gives rise to the isogeny $\widehat{\overline{H}} \to \widehat{H}$. 

As above, we define $Z(\widehat{\overline{H}})^{+}$ to be the preimage of $Z(\widehat{H})^{\Gamma}$ in $Z(\widehat{\overline{H}})$, and declare that $(H, \mathcal{H}, \dot{s}, \eta)$ is a \textit{refined endoscopic datum} if $H, \mathcal{H}$, and $\eta$ are defined as for an endoscopic datum, and $\dot{s} \in Z(\widehat{\overline{H}})^{+}$ is such that $(H, \mathcal{H}, s, \eta)$ is an endoscopic datum, where $s$ is the image of $\dot{s}$ under the map $Z(\widehat{\overline{H}})^{+} \to Z(\widehat{H})^{\Gamma}$. An isomorphism of two refined endoscopic data $(H, \mathcal{H}, \dot{s}, \eta), (H', \mathcal{H}', \dot{s}', \eta')$ is an element $\dot{g} \in \widehat{\bar{G}}(\mathbb{C})$ such that its image $g$ in $\widehat{G}(\mathbb{C})$ satisfies $g \eta(\mathcal{H}) g^{-1} = \eta'(\mathcal{H}')$, inducing $\beta \colon \mathcal{H} \to \mathcal{H}'$ and the restriction $\beta \colon \widehat{H} \to \widehat{H}'$, which (by basic properties of central isogenies) lifts uniquely to a map $\bar{\beta} \colon \widehat{\overline{H}} \to \widehat{\overline{H'}}$, and such that the images of $\bar{\beta}(\dot{s})$ and $\dot{s}'$ in $\pi_{0}(Z(\widehat{\overline{H'}})^{+})$ coincide. It is clear that every endoscopic datum lifts to a refined endoscopic datum, and that every isomorphism of refined endoscopic data induces an isomorphism of the associated endoscopic data.

 Let $\dot{\mathfrak{e}} = (\mathcal{H}, H, \eta, \dot{s})$ be a refined endoscopic datum for $G$ with associated endoscopic datum $\mathfrak{e} = (\mathcal{H}, H, \eta, s)$, which is also an endoscopic datum for $G^{*}$. Let $\mathfrak{z} = (H_{\mathfrak{z}}, \eta_{\mathfrak{z}})$ be a $z$-pair for $\mathfrak{e}$. As discussed in \S 6, have two functions
$$ \Delta[\mathfrak{e}, \mathfrak{z}] \colon H_{\mathfrak{z}, G-\text{sr}}(F) \times G^{*}_{\text{sr}}(F) \times H_{\mathfrak{z}, G-\text{sr}}(F) \times G^{*}_{\text{sr}}(F) \to \mathbb{C},$$
$$\Delta[\mathfrak{e}, \mathfrak{z}, \psi] \colon H_{\mathfrak{z}, G-\text{sr}}(F) \times G_{\text{sr}}(F) \times H_{\mathfrak{z}, G-\text{sr}}(F) \times G_{\text{sr}}(F) \to \mathbb{C},$$
where the first equation makes sense because strongly $G$-regular elements of $H(F)$ are strongly $G^{*}$-regular via choices of admissible embeddings $T_{H} \xrightarrow{\sim} T$, $T_{\bar{H}} \xrightarrow{\sim} \bar{T}$, as in our discussion of the local transfer factor. As in \cite{Tasho}, we have added terms in the brackets to show what each factor depends on. We set the above function to zero if either of the pairs of inputs consist of two elements which are not related.

For our arbitrary $G$, we say that an \textit{absolute transfer factor} is a function $$ \Delta[\mathfrak{e}, \mathfrak{z}]_{\text{abs}} \colon H_{\mathfrak{z}, G-\text{sr}} \times G_{\text{sr}}(F) \to \mathbb{C},$$ which is nonzero for any pair $(\gamma_{\mathfrak{z}}, \delta)$ of related elements and satisfies the relation
$$\Delta[\mathfrak{e}, \mathfrak{z}]_{\text{abs}}(\gamma_{\mathfrak{z}, 1}, \delta_{1}) \cdot \Delta[\mathfrak{e}, \mathfrak{z}]_{\text{abs}}(\gamma_{\mathfrak{z}, 2}, \delta_{2})^{-1} = \Delta[\mathfrak{e}, \mathfrak{z}](\gamma_{\mathfrak{z},1}, \delta_{1}; \gamma_{\mathfrak{z},2}, \delta_{2})$$
for pairs $(\gamma_{\mathfrak{z}, 1}, \delta_{1})$, $(\gamma_{\mathfrak{z}, 2}, \delta_{2})$ of related elements.

By \S 6, if $G$ is quasi-split, setting $\Delta[\mathfrak{e}, \mathfrak{z}] = \Delta_{0}$ (and zero if the pair is unrelated) satisfies these properties. As we noted in Corollary \ref{delta0}, this function is not unique, depending on a choice of $F$-splitting of $G_{\text{sc}}$. Our next goal will be to use the notions of refined endoscopic data and $Z$-rigid inner forms to construct an absolute transfer factor in the non quasi-split case which is associated to some splitting of the quasi-split inner form $G^{*}$, extending the absolute transfer factor in the quasi-split case. This follows the corresponding construction in \cite[\S 5.3]{Tasho}.

We return to the setting of arbitrary $G$ connected reductive over $F$ with quasi-split inner form $\psi \colon G_{F^{s}} \to G^{*}_{F^{s}}$ and fixed $Z$-rigid inner form $(\xi, \mathscr{T}, \bar{h})$, $\xi := \psi^{-1}$, for some fixed finite central $Z$ defined over $F$. Let $\delta' \in G(F)$ and $\gamma_{\mathfrak{z}} \in H_{\mathfrak{z}, G-\text{sr}}(F)$ be related elements, and let $\gamma_{H}$ be the image of $\gamma_{\mathfrak{z}}$ in $H(F)$. Then, by Lemma \ref{Fconj}, we may find $\delta \in G^{*}(F)$ such that $\dot{\delta}' := (G, \xi, \mathscr{T}, \bar{h}, \delta')$ lies in $\mathcal{C}_{Z}(\delta)$; note that by strong regularity, the induced isomorphism of centralizers $\text{Ad}(g) \circ \psi \colon Z_{G}(\delta')_{F^{s}} \to Z_{G^{*}}(\delta)_{F^{s}}$, some $g \in G^{*}(F^{s})$, is defined over $F$. 

Let $S_{H}$ denote the centralizer of $\gamma_{H}$ in $H$, and $S$ denote the centralizer of $\delta$ in $G^{*}$. Since $\gamma_{H}$ and $\delta'$ are related, we have an admissible isomorphism $S_{H} \to Z_{G}(\delta')$ sending $\gamma_{H}$ to $\delta'$. Post-composing this map with the $F$-isomorphism $Z_{G}(\delta') \to S$ gives an admissible isomorphism $\phi_{\gamma_{H}, \delta} \colon S_{H} \to S$ which sends $\gamma_{H}$ to $\delta$, and is unique with these properties. This isomorphism identifies the canonically embedded copies of $Z$ in both of the tori, and therefore induces an isomorphism $\bar{\phi}_{\gamma_{H}, \delta} \colon \overline{S}_{H} \to \overline{S}$. If $[\widehat{\overline{S}_{H}}]^{+}$ denotes the preimage of $\widehat{S_{H}}^{\Gamma}$ under the isogeny $\widehat{\overline{S}_{H}} \to \widehat{S_{H}}$,  then the canonical ($\Gamma$-equivariant) embeddings $Z(\widehat{H}) \hookrightarrow \widehat{S_{H}}$, $Z(\widehat{\overline{H}}) \hookrightarrow \widehat{\overline{S}_{H}}$ induce a canonical embedding $Z(\widehat{\overline{H}})^{+} \hookrightarrow [\widehat{\overline{S}_{H}}]^{+}$. If the group $[\widehat{\overline{S}}]^{+}$ is defined analogously, we have that $\bar{\phi}_{\gamma_{H}, \delta}^{-1}$ dualizes to a map $[\widehat{\overline{S}_{H}}]^{+} \to [\widehat{\overline{S}}]^{+}$ (since $\bar{\phi}_{\gamma_{H}, \delta}$ is defined over $F$) which further induces an embedding $$Z(\widehat{\overline{H}})^{+} \hookrightarrow [\widehat{\overline{S}}]^{+}.$$ 

We thus obtain from $\dot{s} \in Z(\widehat{\overline{H}})^{+}$ associated to our refined endoscopic datum an element $\dot{s}_{\gamma_{H}, \delta} \in  [\widehat{\overline{S}}]^{+}.$ Then we set 
\begin{equation}\label{absolutetransfer}
\Delta[\dot{\mathfrak{e}}, \mathfrak{z}, \psi, (\mathscr{T}, \bar{h}) ]_{\text{abs}}(\gamma_{\mathfrak{z}}, \delta') := \Delta[\mathfrak{e}, \mathfrak{z}]_{\text{abs}}(\gamma_{\mathfrak{z}}, \delta) \cdot \langle \text{inv}(\delta, \dot{\delta}'), \dot{s}_{\gamma_{H}, \delta} \rangle^{-1},
\end{equation}
where the pairing $\langle -, - \rangle$ is as in Corollary \ref{pairing} with $G = S$. 

It is clear that we could have replaced $(\xi, \mathscr{T}, \bar{h})$ with any $a$-normalized $Z$-rigid inner twist $(\xi, (y, \phi^{*}))$ in its isomorphism class from the start, and defined the transfer factor using the invariant of the corresponding class of  $(G,\xi, (y, \phi^{*})), \delta')$ in $\mathcal{C}_{Z}(\delta,a)$. The last main goal of this paper will be to prove that \eqref{absolutetransfer} defines an absolute transfer factor on $G$.  In light of the above discussion, it is enough to work entirely with $a$-normalized $Z$-rigid inner twists for some fixed choice of $a \in u(U_{2})$ with $[a] = -1$. In this context, $\dot{\delta}'$ will denote the element $(G, \xi, (y, \phi^{*}), \delta') \in \mathcal{C}_{Z}(\delta, a)$, and we denote the function from \eqref{absolutetransfer} by $\Delta[\dot{\mathfrak{e}}, \mathfrak{z}, \psi, (y, \phi^{*}) ]$.

Before we prove this, we discuss the dependency of this factor on $Z$. Let $Z'$ be another finite central $F$-subgroup of $G$ which contains $F$, viewed also as a finite central $F$-subgroup of $G^{*}$. We denote by $(y, (\phi')^{*}) \in Z^{1}(F, Z' \to G)$ the image of $(y,\phi^{*})$ under the natural inclusion map, so that $(\phi^{*})'$ is $\phi^{*} \colon u \to Z$ post-composed with the inclusion map, defining a $Z'$-rigid inner twist $(\xi, (y, (\phi')^{*}))$. As with $Z$, we have a canonical $F$-embedding $Z' \hookrightarrow H$ which commutes with our embedding of $Z$ and the inclusion map, and we set $\overline{H}' := H/Z'$. Now we have an isogeny $\overline{H} \to \overline{H}'$ which dualizes to an isogeny $\widehat{\overline{H}'} \to \widehat{\overline{H}}$, inducing a canonical surjection $Z(\widehat{\overline{H}'})^{+} \to Z(\widehat{\overline{H}})^{+}.$ Choose a preimage $\ddot{s}$ in $Z(\widehat{\overline{H}'})^{+}$ of $\dot{s}$, giving a refined endoscopic datum $\ddot{\mathfrak{e}} := (H, \mathcal{H}, \ddot{s}, \eta)$. Note that the point $\ddot{\delta'} := (G, \xi, (y, (\phi')^{*}), \delta')$ equals $\iota_{Z,Z'}(\dot{\delta}') \in \mathcal{C}_{Z'}(\delta)$. As we discussed in \S 7.1, we then have that $$\text{inv}(\delta, \iota_{Z, Z'}(\dot{\delta'})) = i (\text{inv}(\delta, \dot{\delta}'))$$ in $H^{1}(\gerbeE_{a}, Z' \to S)$, where $i$ is the natural map $H^{1}(\gerbeE_{a}, Z \to S) \to H^{1}(\gerbeE_{a}, Z' \to S)$. One checks easily that $\ddot{s}_{\gamma_{H}, \delta}$ maps to $\dot{s}_{\gamma_{H}, \delta}$ under the dual surjection $\widehat{\overline{S}'} \to \widehat{\overline{S}}$. The functoriality of the pairing from Corollary \ref{pairing} then gives us that $$\langle  i (\text{inv}(\delta, \dot{\delta})), \ddot{s}_{\gamma_{H}, \delta} \rangle = \langle \text{inv}(\delta, \dot{\delta}), \dot{s}_{\gamma_{H}, \delta} \rangle.$$ Since this factor is the only part of $\Delta[\ddot{\mathfrak{e}}, \mathfrak{z}, \psi, (y, (\phi')^{*})]_{\text{abs}}$ that depends on $Z$, we see that 

\begin{equation}\label{changeZ} \Delta[\dot{\mathfrak{e}}, \mathfrak{z}, \psi, (y,\phi^{*})]_{\text{abs}}(\gamma_{\mathfrak{z}}, \delta') = \Delta[\ddot{\mathfrak{e}}, \mathfrak{z}, \psi, (y,(\phi')^{*}))]_{\text{abs}}(\gamma_{\mathfrak{z}}, \delta').
\end{equation}

\begin{prop} The value of $\Delta[\dot{\mathfrak{e}}, \mathfrak{z}, \psi, (y,\phi^{*})]_{\text{abs}}(\gamma_{\mathfrak{z}}, \delta')$ does not depend on the choice of $\delta$, and the function $\Delta[\dot{\mathfrak{e}}, \mathfrak{z}, \psi, (y,\phi^{*})]_{\text{abs}}$ is an absolute transfer factor. Moreover, this function does not change if we replace $\dot{\mathfrak{e}}$ by an equivalent refined endoscopic datum, or if we replace $(G, \xi, (y,\phi^{*}))$ by an isomorphic ($a$-normalized) $Z$-rigid inner twist of $G^{*}$. 
\end{prop}

\begin{proof} We follow the proof of \cite[Proposition 5.6]{Tasho}. For the independence of the choice of $\delta$ let $\delta_{0} \in G^{*}_{\text{sr}}(F)$ be another element such that  $(G, \xi, (y, \phi^{*}), \delta') \in \mathcal{C}_{Z}(\delta_{0})$ and $\text{Ad}(g') \circ \psi$, for some $g' \in G^{*}(F^{s})$, induces an $F$-isomorphism $Z_{G}(\delta') \to Z_{G^{*}}(\delta_{0})$. By taking the composition $(\text{Ad}(g') \circ \psi) \circ (\psi^{-1} \circ \text{Ad}(g^{-1}))$, we see that $\delta$ and $\delta_{0}$ are conjugate by an element $c \in \mathfrak{A}(S) \subset G^{*}(F^{s})$, notation as in \S 6. Similarly, the element realizing the stable conjugacy of $\delta$ and $\delta'$ may be chosen to lie in $G^{*}(F^{s})$. From here, the same argument used in \cite{Tasho} for the corresponding part of the proof of Proposition 5.6 works in our setting---we can still use Galois cohomology and our analysis of the local transfer factor in \S 6 lines up exactly with that of \cite[\S 3]{LS}.

As is remarked in \cite{Tasho}, invariance under isomorphisms of rigid inner twists is immediate from the fact that $\text{inv}(\delta, \dot{\delta}')$ depends only on the isomorphism class of $\dot{\delta}'$ in $\mathcal{C}_{Z}(\delta)$. Similar to our justification of the fact that our function is independent of choice of $\delta$, our discussion in \S 6 can be substituted for \S 3 of \cite{LS} and then the corresponding argument in \cite{Tasho}, Proposition 5.6 carries over verbatim to show that our function is invariant under isomorphisms of refined endoscopic data. 

The only work we need to do here is to show that $\Delta[\dot{\mathfrak{e}}, \mathfrak{z}, \psi, (y,\phi^{*})]_{\text{abs}}$ is indeed an absolute transfer factor. This means that we need to show that $$ \Delta[\dot{\mathfrak{e}}, \mathfrak{z}, \psi, (y,\phi^{*})]_{\text{abs}}(\gamma_{\mathfrak{z},1}, \delta_{1}') \cdot  \Delta[\dot{\mathfrak{e}}, \mathfrak{z}, \psi, (y,\phi^{*})]_{\text{abs}}(\gamma_{\mathfrak{z},2}, \delta_{2}')^{-1} = \Delta[\mathfrak{e}, \mathfrak{z}, \psi](\gamma_{\mathfrak{z},1}, \delta_{1}' ; \gamma_{\mathfrak{z},2}, \delta_{2}').$$
We emphasize that we still follow the corresponding argument in \cite{Tasho}, Proposition 5.6, closely.  Replacing $\dot{\mathfrak{e}}$ by an appropriate refined endoscopic datum as in our construction of $\ddot{\mathfrak{e}}$ above, we may assume, using the identity \eqref{changeZ}, that $Z$ contains $Z(\mathscr{D}(G))$. Choose $\delta_{1}, \delta_{2} \in G^{*}(F)$ which are stably conjugate to $\delta_{1}', \delta_{2}'$. It's enough to show that $$\frac{\langle \text{inv}(\delta_{1}, \dot{\delta_{1}}'), \dot{s}_{\gamma_{1}, \delta_{1}} \rangle^{-1}}{\langle \text{inv}(\delta_{2}, \dot{\delta_{2}}'), \dot{s}_{\gamma_{2}, \delta_{2}} \rangle^{-1}} = \frac{\Delta[\mathfrak{e}, \mathfrak{z}, \psi](\gamma_{\mathfrak{z},1}, \delta_{1}' ; \gamma_{\mathfrak{z},2}, \delta_{2}')}{\Delta[\mathfrak{e}, \mathfrak{z}](\gamma_{\mathfrak{z},1}, \delta_{1}; \gamma_{\mathfrak{z},2}, \delta_{2})},$$ where we are using $\gamma_{i}$ to denote the image of $\gamma_{\mathfrak{z},i}$ in $H_{G-\text{sr}}(F)$. To simplify the right-hand side, note that in the definition of the bottom factor, we may choose our admissible embeddings $Z_{H}(\gamma_{i}) \hookrightarrow G^{*}$ to be the unique ones from  $Z_{H}(\gamma_{i})$ to $G^{*}$ that map $\gamma_{i}$ to $\delta_{i}$. Then, as in the definition of the factor $\Delta_{1}$ in the quasi-split case (see 6.3.3), we have that $\gamma_{G^{*}} = \gamma$, and hence we can take $h=\text{id}$ and so $\text{inv}(\gamma_{i}, \delta_{i}) = 0 \in H^{1}(F, Z_{G^{*}}(\delta_{i}))$, giving $\Delta_{1}(\gamma_{1}, \delta_{1}; \gamma_{2}, \delta_{2}) = 1$. All of the $\Delta_{I}, \Delta_{II}, \Delta_{III_{2}}$, and $\Delta_{IV}$ factors of the numerator and denominator of the right-hand side coincide, and so all we're left with is 
\begin{equation}
\Delta_{III_{1}}(\gamma_{1}, \delta_{1}' ; \gamma_{2}, \delta_{2}') := \langle \text{inv} \left( \frac{\gamma_{1}, \delta_{1}'}{\gamma_{2}, \delta_{2}'} \right ), \textbf{s}_{U} \rangle,
\end{equation}
where all the notation is as defined in 6.3.3.

Set $Z_{H}(\gamma_{i}) := S^{H}_{i}$, $Z_{G}(\delta_{i}') := S_{i}'$, and $Z_{G^{*}}(\delta_{i}):= S_{i}$; these are all maximal $F$-tori. Set $$V:= \frac{S_{1} \times S_{2}}{Z(G^{*})},$$ where $Z(G^{*}) \hookrightarrow S_{1} \times S_{2}$ via $i^{-1} \times j$ . The homomorphism $S_{1} \times S_{2} \to V$ defines a morphism $[Z \times Z \to S_{1} \times S_{2}] \to [(Z \times Z)/Z \to V]$ in the category $\mathcal{T}$. We claim that the image in $H^{1}(\gerbeE_{a}, (Z \times Z)/Z \to V)$ of the element
$$(\text{inv}(\delta_{1}, \dot{\delta_{1}}')^{-1}, \text{inv}(\delta_{2}, \dot{\delta_{2}}')) \in H^{1}(\gerbeE_{a}, Z \times Z \to S_{1} \times S_{2}),$$ where $\text{inv}(\delta_{i}, \dot{\delta_{i}}')$ is defined as in \S 7.1, lies inside $H^{1}(F, V)$ (embedded in $H^{1}(\gerbeE_{a}, (Z \times Z)/Z \to V)$  via the ``inflation" map).

It is clear that the restriction maps $H^{1}(\gerbeE_{a}, Z \to S_{i}) \to \Hom_{F}(u,Z)$ factor as a composition of the maps $H^{1}(\gerbeE_{a}, Z \to S_{i}) \to H^{1}(\gerbeE_{a}, Z \to G^{*})$ and $H^{1}(\gerbeE_{a}, Z \to G^{*}) \xrightarrow{\text{Res}} \Hom_{F}(u, Z)$. Moreover, the image of $\text{inv}(\delta_{i}, \delta_{i}')$ in $H^{1}(\gerbeE_{a}, Z \to G^{*})$ is the class of the twisted $a$-cocycle $(p_{1}^{\sharp}(g_{i})yp_{2}^{\sharp}(g_{i})^{-1}, \phi^{*})$, where $g_{i} \in G^{*}(\bar{F})$ is such that $\text{Ad}(g_{i})\psi(\delta_{i}') = \delta_{i}$, which is just the class of the twisted cocycle $(y, \phi^{*}) \in Z^{1}(\gerbeE_{a}, Z \to G)$. This means that the image of $(\text{inv}(\delta_{1}, \dot{\delta_{1}}')^{-1}, \text{inv}(\delta_{2}, \dot{\delta_{2}}'))$ in $\Hom_{F}(u, Z \times Z) = \Hom_{F}(u,Z) \times \Hom_{F}(u,Z)$ equals $(\text{Res}((y, \phi^{*}))^{-1}, \text{Res}((y, \phi^{*}))) = (-\phi^{*}, \phi^{*})$ which is zero in $\Hom_{F}(u, (Z \times Z)/Z)$. Whence, the exact sequence $$H^{1}(F, V) \to H^{1}(\gerbeE_{a}, (Z \times Z)/Z \to V) \to \Hom_{F}(u, Z)$$ gives the claim.
 
 Recall from \S 6.3.3 that $U:= ((S_{1})_{\text{\text{sc}}} \times (S_{2})_{\text{\text{sc}}})/Z_{\text{\text{sc}}}$ where $Z_{\text{\text{sc}}}$ embeds via $i^{-1} \times j$ (here we are taking our admissible embeddings $S^{H}_{i} \to G^{*}$ to be the unique ones that send $\gamma_{i}$ to $\delta_{i}$); there is an obvious homomorphism $U \to V$. We now claim that the image of $\text{inv}(\frac{\gamma_{1}, \delta_{1}'}{\gamma_{2}, \delta_{2}'}) \in H^{1}(F, U)$ in $H^{1}(F, V)$ coincides with the image of $(\text{inv}(\delta_{1}, \dot{\delta_{1}}')^{-1}, \text{inv}(\delta_{2}, \dot{\delta_{2}}'))$. From the rigidifying element $(y, \phi^{*}) \in Z^{1}(\gerbeE_{a}, Z \to G^{*})$, $y \in G^{*}(\bar{F} \otimes_{F} \bar{F})$, $\psi^{-1} \colon Z \to G^{*}$, we extract the \v{C}ech 1-cochain $y$, which we will factor as $\bar{u} \cdot z$ with $\bar{u} \in \mathscr{D}(G^{*})(\bar{F} \otimes_{F} \bar{F})$ and $z \in Z(G^{*})(\bar{F} \otimes_{F} \bar{F})$; we can do this because the central isogeny decomposition for $G^{*}$ is surjective on $\bar{F} \otimes_{F} \bar{F}$-points, owing to the fact that $H^{1}(\bar{F} \otimes_{F} \bar{F}, Z(\mathscr{D}(G^{*}))) = 0$. Let $u \in G^{*}_{\text{\text{sc}}}(\bar{F} \otimes_{F} \bar{F})$ be a lift of $\bar{u}$. By construction (see 6.3.3, using the fact that $\text{Ad}(u) = \text{Ad}(\bar{u}) = \text{Ad}(y) = p_{1}^{*}\psi \circ p_{2}^{*}\psi^{-1}$ on $G^{*}_{\bar{F} \otimes_{F} \bar{F}}$), we have the equality 
 \begin{equation}
\text{inv} \left( \frac{\gamma_{1}, \delta_{1}'}{\gamma_{2}, \delta_{2}'} \right ) = ([p_{1}^{\sharp}(g_{1})up_{2}^{\sharp}(g_{1})^{-1}]^{-1}, p_{1}^{\sharp}(g_{2})up_{2}^{\sharp}(g_{2})^{-1}) \in U(\bar{F} \otimes_{F} \bar{F}),
\end{equation}
whose image in $V(\bar{F} \otimes_{F} \bar{F})$ coincides with the image of $([p_{1}^{\sharp}(g_{1})yp_{2}^{\sharp}(g_{1})^{-1}]^{-1}, p_{1}^{\sharp}(g_{2})yp_{2}^{\sharp}(g_{2})^{-1})$, because, by design, $y = \bar{u} \cdot z$ for $z \in Z(G^{*})(\bar{F} \otimes_{F} \bar{F})$. 
This gives the claim.

%
Suppose there is some $v \in [\widehat{\overline{V}}]^{+}$ whose image in $[\widehat{\overline{S_{1}}}]^{+} \times [\widehat{\overline{S_{2}}}]^{+}$ via the map $\widehat{\overline{V}} \to \widehat{\overline{S_{1}}} \times \widehat{\overline{S_{2}}}$ dual to the projection map $\overline{S_{1}} \times \overline{S_{2}} \to \overline{V}$, where $\overline{V} := \frac{V}{(Z \times Z)/Z},$ is equal to $(\dot{s}_{\gamma_{1}, \delta_{1}}, \dot{s}_{\gamma_{2}, \delta_{2}})$. Additionally suppose that the image of this same $v$  in $[\widehat{\overline{U}}]^{+}$ maps to $s_{U}$ under the isogeny $[\widehat{\overline{U}}]^{+} \to \widehat{U}^{\Gamma}$, where $\overline{U}$ is formed from the object $[Z(G^{*}_{\text{\text{sc}}}) \to U] \in \mathcal{T}$.

Then, given this $v$, the equality \begin{equation}
\frac{\langle \text{inv}(\delta_{1}, \dot{\delta_{1}}'), \dot{s}_{\gamma_{1}, \delta_{1}} \rangle^{-1}}{\langle \text{inv}(\delta_{2}, \dot{\delta_{2}}'), \dot{s}_{\gamma_{2}, \delta_{2}} \rangle^{-1}} = \langle \text{inv} \left( \frac{\gamma_{1}, \delta_{1}'}{\gamma_{2}, \delta_{2}'} \right ), \textbf{s}_{U} \rangle, 
\end{equation} follows from the diagram

\[
\begin{tikzcd}[column sep=0.05mm]
\text{inv} \left( \frac{\gamma_{1}, \delta_{1}'}{\gamma_{2}, \delta_{2}'} \right ) \in H^{1}(F,U) \arrow[mapsto]{d}  & s_{U} \in \widehat{U}^{\Gamma}  \\
\pi( (\text{inv}(\delta_{1}, \dot{\delta_{1}'})^{-1}, \text{inv}(\delta_{2}, \dot{\delta_{2}}))) \in H^{1}(F,V)  & v \in [\widehat{\overline{V}}]^{+} \arrow[mapsto]{u}  \arrow[mapsto]{d} \\
(\text{inv}(\delta_{1}, \dot{\delta_{1}}')^{-1}, \text{inv}(\delta_{2}, \dot{\delta_{2}}')) \in H^{1}(\gerbeE_{a}, Z \times Z \to S_{1} \times S_{2}) \arrow[mapsto, "\pi"]{u}  & (\dot{s}_{\gamma_{1}, \delta_{1}}, \dot{s}_{\gamma_{2}, \delta_{2}}) \in [\widehat{\overline{S_{1}}}]^{+} \times [\widehat{\overline{S_{2}}}]^{+},
\end{tikzcd}
\]
where the top pair of elements are the inputs of the pairing in the right-hand side of our main desired equality, the bottom pair of elements are the inputs of the pairing in the left-hand side of that equality, and by functoriality their pairings both equal the pairing of the two elements in the middle line.

The argument for the fact that we can find such an element $v$ of $[\widehat{\overline{V}}]^{+}$ is identical to the corresponding argument in \cite{Tasho}, proof of Proposition 5.6, finishing the proof of our result.
\end{proof}

\subsection{The Langlands conjectures} We now use our constructions to discuss the Langlands correspondence for an arbitrary connected reductive group defined over a local function field $F$. This section is a summary of \cite[\S 5.4]{Tasho}. 

Let $G^{*}$ be a connected, reductive, and quasi-split group over $F$ with finite central $F$-subgroup $Z$ which is an inner form of our fixed arbitrary connected reductive group $G$. Fix a \textit{Whittaker datum} $\mathfrak{w}$ for $G^{*}$, which recall is a $G^{*}(F)$-conjugacy class of pairs $(B, \zeta_{B})$ consisting of an $F$-Borel subgroup $B \subset G^{*}$ and a non-degenerate character $\zeta_{B} \colon B_{u}(F) \to \mathbb{C}^{*}$, where the subscript $u$ denotes the unipotent radical. We may view the group $Z$ as a finite central $F$-subgroup of $G$, also denoted by $Z$, with $\overline{G} := G/Z$ as before.

\begin{Def} Given a quasi-split connected reductive group $G^{*}$ over $F$, we write $\Pi^{\text{rig}}(G^{*})$ for the set of isomorphism classes of irreducible admissible representations of rigid inner twists of $G^{*}$ (see Definition \ref{representation}). Define the subset $\Pi_{\text{temp}}^{\text{rig}}(G^{*})$ to be those representations which are tempered. 
\end{Def}

Let $\varphi \colon W_{F}' \to \prescript{L}{}{G}$ be a \textit{tempered Langlands parameter}, which means that it's a homomorphism of $W_{F}$-extensions that is continuous on $W_{F}$, restricts to a morphism of algebraic groups on $SL_{2}(\mathbb{C})$, and sends $W_{F}$ to a set of semisimple elements of $\prescript{L}{}{G}$ that project onto a bounded subset of $\widehat{G}(\mathbb{C})$. Setting $S_{\varphi} = Z_{\widehat{G}}(\varphi)$, and $S_{\varphi}^{+}$ its preimage in $\widehat{\overline{G}}$, we have an inclusion $Z(\widehat{\overline{G}})^{+} \subset S_{\varphi}^{+}$ which induces a map $\pi_{0}(Z(\widehat{\overline{G}})^{+}) \to \pi_{0}(S_{\varphi}^{+})$ with central image. Denote by $\text{Irr}(\pi_{0}(S_{\varphi}^{+}))$ the set of irreducible representations of the finite group $\pi_{0}(S_{\varphi}^{+})$.

\begin{conj} There is a finite subset $\Pi_{\varphi} \subset \Pi_{\text{temp}}^{\text{rig}}(G^{*})$ and a commutative diagram 
\[
\begin{tikzcd}
\Pi_{\varphi} \arrow["\iota_{\mathfrak{w}}"]{r} \arrow{d} & \text{Irr}(\pi_{0}(S_{\varphi}^{+})) \arrow{d} \\
H^{1}(\gerbeE, Z \to G^{*}) \arrow{r} & \pi_{0}(Z(\widehat{\overline{G}})^{+})^{*}
\end{tikzcd}
\]
in which the top map is a bijection, the bottom map is given by the pairing of Corollary \ref{pairing}, the right map assigns to each irreducible representation the restriction of its central character to $\pi_{0}(Z(\widehat{\overline{G}}))^{+}$ , and the left map sends a representation $(G_{1}, \xi_{1}, \mathscr{T}_{1}, \bar{h}, \pi)$ to the class of $\mathscr{T}_{1}$. We also expect that there is a unique element $(G, \text{id}_{G}, \GE, \text{id}_{\bar{F}}, \pi)$ of $\Pi_{\varphi}$ such that $\pi$ is $\mathfrak{w}$-generic and the map $\iota_{\mathfrak{w}}$ identifies this element with the trivial irreducible representation, see \cite[\S 9]{Shahidi}.
\end{conj}

For $\dot{\pi} := (G_{1}, \xi_{1}, \mathscr{T}_{1}, \bar{h}_{1}, \pi_{1}) \in \Pi_{\varphi}$, denote by $\langle -, \dot{\pi} \rangle$ the conjugation-invariant function on $\pi_{0}(S_{\varphi}^{+})$ given by the trace of the irreducible representation $\iota_{\mathfrak{w}}(\dot{\pi})$. We let $\Theta_{\dot{\pi}}$ denote the distribution on $G_{2}(F)$ for any isomorphic rigid inner twist $(G_{2}, \xi_{2}, \mathscr{T}_{2}, \bar{h}_{2})$ given by transporting the Harish-Chandra character $\Theta_{\pi_{1}}$ associated to $\pi_{1}$ to $G_{2}(F)$ via any choice of isomorphism of rigid inner twists---note that by Corollary \ref{automorphisms2} this distribution does not depend on the choice of isomorphism, justifying the notation. 

\begin{remark} There are no issues in defining the Harish-Chandra character $\Theta_{\pi}$ of an admissible representation $\pi$ for local function fields, see \cite[Appendix A]{Cluckers}, which is all that is needed in the setting of this paper. The difficulty that arises in the function field setting is that it is not known in general if $\Theta_{\pi}$ is representable by a locally integrable function, although this is known under certain assumptions (see \cite{Cluckers}, in particular Theorem 2.2).
\end{remark}

We define the virtual character
\begin{equation}
S \Theta_{\varphi, \text{id}, (\GE, \text{id})} = e(G) \sum_{\dot{\pi} \in \Pi_{\varphi}, \dot{\pi} \mapsto [\mathscr{T}]} \langle 1, \dot{\pi} \rangle \Theta_{\dot{\pi}},
\end{equation}
and for a fixed rigid inner twist $(\xi, \mathscr{T}, \bar{h}) \colon G^{*} \to G$ enriching our inner twist $\psi^{-1} \colon G^{*}_{F^{s}} \xrightarrow{\sim} G_{F^{s}}$, and semisimple $\dot{s} \in S_{\varphi}^{+}(\mathbb{C})$ we define 
\begin{equation}
 \Theta^{\dot{s}}_{\varphi, \mathfrak{w}, \xi, (\mathscr{T}, \bar{h})} = e(G) \sum_{\dot{\pi} \in \Pi_{\varphi}, \dot{\pi} \mapsto [\mathscr{T}]} \langle \dot{s}, \dot{\pi} \rangle \Theta_{\dot{\pi}}.
\end{equation}
Here $e(G)$ denotes the sign defined in \cite{Kott83}; we expect $S \Theta_{\varphi, \text{id}, (\GE, \text{id})}$ to be a stable distribution on $G(F)$, as defined in \cite[I.4]{La83}, and that (as the notation suggests) it is independent of $\mathfrak{w}$.

The element $\dot{s}$ also defines a refined endoscopic datum $\dot{\mathfrak{e}}$ as follows: Let $s \in S_{\varphi}(\mathbb{C})$ be the image of $\dot{s}$, set $\widehat{H} = Z_{\widehat{G}}(s)^{\circ}$, set $\mathcal{H} = \widehat{H}(\mathbb{C}) \cdot \varphi(W_{F})$, and take $\eta \colon \mathcal{H} \to \prescript{L}{}G$ to be the natural inclusion, and define $\dot{\mathfrak{e}} = (H, \mathcal{H}, \eta, \dot{s})$. Take also a $z$-pair $(H_{\mathfrak{z}}, \eta_{\mathfrak{z}})$ corresponding to the endoscopic datum $\mathfrak{e}$ associated to the refined datum $\dot{\mathfrak{e}}$, which induces a tempered Langlands parameter $\varphi_{\mathfrak{z}} := \eta_{\mathfrak{z}} \circ \varphi$. 

Because the above conjectural parametrization depends on a choice of Whittaker datum $\mathfrak{w}$ (and we expect that the distribution $S \Theta_{\varphi, \text{id}, (\GE, \text{id})}$ does not), we need to normalize our transfer factor $\Delta$ to reflect this choice before comparing $S \Theta_{\varphi, \text{id}, (\GE, \text{id})}$ and $\Theta^{\dot{s}}_{\varphi, \mathfrak{w}, \xi, (\mathscr{T}, \bar{h})}$. According to \cite[\S 5.5]{KS2}, we may define a \textit{Whittaker normalization} of the absolute transfer factor for quasi-split groups, denoted by $\Delta^{'}[\mathfrak{e}, \mathfrak{z}, \mathfrak{w}] \colon H_{\mathfrak{z}, G-\text{sr}}(F) \times G^{*}_{\text{sr}}(F) \to \mathbb{C}$ associated to our Whittaker datum $\mathfrak{w}$. We briefly describe this normalization: using the notation of \S 6, we set 
$$\Delta^{'}[\mathfrak{e}, \mathfrak{z}, \mathfrak{w}]  := \epsilon_{L}(V, \psi_{F}) (\Delta_{I} \Delta_{1})^{-1} \Delta_{II} \Delta_{IV},$$ where $\epsilon_{L}(V, \psi_{F})$ is a 4th root of unity associated to a specific  virtual representation $V$ of $\Gamma$ (and thus of $W_{F}$) coming from $\mathfrak{e}$ and $\mathfrak{w}$, together with a choice of additive character $\psi_{F} \colon F \to \mathbb{C}^{*}$; for details, see \cite[\S 5.3]{KS1}. The important takeaway is that the construction of the normalization factor $\epsilon_{L}(V, \psi_{F})$ can be done uniformly for all non-archimedean local fields. One deduces from the arguments in \cite[\S 5.3]{KS1} that this still defines an absolute transfer factor for related strongly regular elements of $H_{\mathfrak{z}}$ and $G^{*}$ which depends only on $\mathfrak{w}$.

As a consequence, we may combine this normalization with our new absolute transfer factor \eqref{absolutetransfer} to obtain a normalized absolute transfer factor for general connected reductive groups over $F$; we use the same notation as in our transfer factor formula \eqref{absolutetransfer}. We then set
\begin{equation}
\Delta^{'}[\dot{\mathfrak{e}}, \mathfrak{z}, \mathfrak{w}, \psi, (\mathscr{T}, \bar{h})](\gamma_{\mathfrak{z}}, \delta') = \Delta^{'}[\mathfrak{e}, \mathfrak{z}, \mathfrak{w}](\gamma_{\mathfrak{z}}, \delta) \langle \text{inv}(\delta, \dot{\delta}'), \dot{s}_{\gamma, \delta} \rangle.
\end{equation}
Note that we have switched the sign of $\langle \text{inv}(\delta, \dot{\delta}'), \dot{s}_{\gamma, \delta} \rangle$ so that our formula agrees with the sign changes in the factors defining $\Delta^{'}[\mathfrak{e}, \mathfrak{z}, \mathfrak{w}].$

Then if $f^{\dot{\mathfrak{e}}}$ and $f$ are smooth compactly supported functions on $H_{\mathfrak{z}}(F)$ and $G(F)$ respectively, whose orbital integrals are $\Delta^{'}[\dot{\mathfrak{e}}, \mathfrak{z}, \mathfrak{w}, \psi, (\mathscr{T}, \bar{h})]$-matching we then expect to have the equality
$$S \Theta_{\varphi_{\mathfrak{z}}, \text{id}, (\GE, \text{id})}(f^{\dot{\mathfrak{e}}}) = \Theta^{\dot{s}}_{\varphi, \mathfrak{w}, \xi, (\mathscr{T}, \bar{h})}(f).$$

\begin{remark} Given $f \in C^{\infty}_{c}(G(F))$, the existence of a $\Delta^{'}[\dot{\mathfrak{e}}, \mathfrak{z}, \mathfrak{w}, \psi, (\mathscr{T}, \bar{h})]$-matching $f^{\dot{\mathfrak{e}}}$ is an open problem (for arbitrary connected reductive $G$). This is in contrast to the $p$-adic case, where the analogous result follows from the fundamental lemma. Nevertheless, this result is expected to be true in our setting as well, and there are several situations in which it has already been proved. For example, the case of unitary $G$ follows from the results of \cite{Yun} for sufficiently large characteristic. In addition, it has been proved for endoscopic groups associated to depth-zero characters of the maximal compact subgroup of a maximal split torus of $G(F)$ for split $G$ when $f$ is in the center of the Hecke algebra associated to that character in \cite{LM} (in fact, the matching functions constructed loc. cit. are ``distinguished," which we means that they are both spherical with respect to hyperspecial subgroups).
\end{remark}

\bibliographystyle{alpha}

\end{document}